%% file: sts_thesis.tex
\newif\iftwosided  \twosidedtrue	
\newif\iftestpage  \testpagefalse	
\newif\iftweakfont \tweakfontfalse	

\iftwosided
\documentstyle[11pt,twoside,harvard]{report}
\else
\documentstyle[11pt,harvard]{report}
\fi

\font\seal=harvard64
\def\harvard{{\seal H}}

\newdimen\pagehoffset	
\newdimen\pagevoffset	
\advance\oddsidemargin by-\pagehoffset
\evensidemargin=\oddsidemargin
\newskip\thesisbaselineskip
\textheight=\thesisbaselineskip \multiply\textheight by33
\advance\textheight by\topskip                      

\advance\dimen0 by-\headsep \advance\dimen0 by-\footskip
\divide\dimen0 by2 \advance\dimen0 by-\pagevoffset
\topmargin=\dimen0

\iftwosided
\advance\oddsidemargin by-\pagehoffset
\advance\dimen0 by-\oddsidemargin \advance\dimen0 by-2\pagehoffset
\evensidemargin=\dimen0
\fi 

\makeatletter 

\newif\ifch@pterhasfigures \ch@pterhasfiguresfalse
\newif\ifch@pterhastables \ch@pterhastablesfalse

\def\ch@ptern@me{}
\def\runningchaptername#1{\xdef\ch@ptern@me{#1}\chaptermark{#1}}

\def\@chapter[#1]#2{\ifnum \c@secnumdepth >\m@ne
  \refstepcounter{chapter}
  \gdef\ch@ptern@me{#2}
  \typeout{\@chapapp\space\thechapter.}
  \addcontentsline{toc}{chapter}{\protect
  \numberline{\thechapter}#1}\else
  \addcontentsline{toc}{chapter}{#1}\fi
  \global\ch@pterhasfiguresfalse
  \global\ch@pterhastablesfalse
  \chaptermark{#1}
  \if@twocolumn\@topnewpage[\@makechapterhead{#2}] 
    \else \@makechapterhead{#2}
    \@afterheading \fi}
\def\@schapter#1{{\let\@mkboth\@gobbletwo
  \gdef\ch@ptern@me{#1}%
  \typeout{#1.}%
  }%
  \if@twocolumn \@topnewpage[\@makeschapterhead{#1}]
  \else \@makeschapterhead{#1} 
  \@afterheading\fi}

\def\@makechapterhead#1{\vspace*{50pt}{\leftskip=0pt
  \rightskip0pt plus1fil \parindent0pt\parskip0pt\parfillskip=0pt
  \ifnum \c@secnumdepth >\m@ne
    \LARGE\rm{\Large\rm
    \edef\temp{\@chapapp{}}\uppercase\expandafter{\temp}}\ \thechapter\par
    \vskip20pt\fi \Large\bf \uppercase{#1}\par
  \nobreak\vskip40pt plus5pt }}

\def\@makeschapterhead#1{\vspace*{50pt}{\leftskip=0pt
  \rightskip0pt plus1fil \parindent0pt\parskip0pt\parfillskip=0pt
  \Large\bf \uppercase\expandafter{#1}\par 
  \nobreak\vskip40pt plus5pt }}

\def\section{\@startsection{section}{1}{\z@}{-3.5ex plus -1ex minus
 -.2ex}{1.2ex plus .2ex}{\large\bf\boldmath}}
\def\subsection{\@startsection{subsection}{2}{\z@}{-2.5ex plus -1ex minus
 -.2ex}{1ex plus .2ex}{\normalsize\bf\boldmath}}

\def\@sect#1#2#3#4#5#6[#7]#8{\ifnum #2>\c@secnumdepth
     \def\@svsec{}\else 
     \refstepcounter{#1}\edef\@svsec{\csname the#1\endcsname.\hskip.5em }\fi
     \@tempskipa #5\relax
      \ifdim \@tempskipa>\z@ 
        \begingroup #6\relax
          \@hangfrom{\hskip #3\relax\@svsec}{\interlinepenalty \@M #8\par}%
        \endgroup
       \csname #1mark\endcsname{#7}\addcontentsline
         {toc}{#1}{\ifnum #2>\c@secnumdepth \else
                      \protect\numberline{\csname the#1\endcsname}\fi
                    #7}\else
        \def\@svsechd{#6\hskip #3\relax  
                   \@svsec #8\csname #1mark\endcsname
                      {#7}\addcontentsline
                           {toc}{#1}{\ifnum #2>\c@secnumdepth \else
                             \protect\numberline{\csname the#1\endcsname}\fi
                       #7}}\fi
     \@xsect{#5}}

\def\thesection{\arabic{section}}
\def\theequation{\arabic{equation}}

\def\thebibliography#1{\chapter*{References\@mkboth
 {REFERENCES}{REFERENCES}}\list
 {[\arabic{enumi}]}{\settowidth\labelwidth{[#1]}\leftmargin\labelwidth
 \advance\leftmargin\labelsep
 \usecounter{enumi}}
 \def\newblock{\hskip .11em plus .33em minus .07em}
 \sloppy\clubpenalty4000\widowpenalty4000
 \sfcode`\.=1000\relax}

\def\l@chapter#1#2{\addpenalty{-\@highpenalty}
 \vskip .5em plus 1pt \@tempdima 1.5em \begingroup
 \parindent \z@ \rightskip \@pnumwidth
 \parfillskip -\@pnumwidth 
 \bf \leavevmode \advance\leftskip\@tempdima \hskip -\leftskip #1\nobreak\hfil
   \nobreak\hbox to\@pnumwidth{\hss #2}\par
 \penalty\@highpenalty \endgroup}
\def\l@section{\@dottedtocline{1}{1.25em}{1.25em}}
\def\l@figure{\@dottedtocline{1}{1.25em}{4.5em}}
\let\l@table\l@figure

\def\numberline#1{\hbox to\@tempdima{#1.\hfil}}

\def\@pnumwidth{1em}
\def\@dotsep{10}
\def\hg@ldenr@ti@{0.309017}

\def\@dottedtocline#1#2#3#4#5{\ifnum #1>\c@tocdepth \else
  {\leftskip #2\relax \rightskip \@tocrmarg \parfillskip -\rightskip
    \parindent #2\relax\@afterindenttrue
   \interlinepenalty\@M
   \leavevmode
   \@tempdima #3\relax \advance\leftskip \@tempdima \hbox{}\hskip -\leftskip
    #4\nobreak\leaders\hbox{$\m@th
      \kern\hg@ldenr@ti@\baselineskip.\kern\hg@ldenr@ti@\baselineskip$}\hfill
    \nobreak \hbox to\@pnumwidth{\hfil\rm #5}\par}\fi}

\def\l@frontmatter#1#2{\addpenalty{-\@highpenalty}
 \addvspace{.25em plus 1pt}\@tempdima 1.5em \begingroup
 \parindent \z@ \rightskip \@pnumwidth 
 \parfillskip 0pt plus 1fil
 \rm \leavevmode \advance\leftskip\@tempdima \hskip -\leftskip #1\nobreak\quad
 \nobreak#2\par
 \penalty\@highpenalty \endgroup}

\def\thefigure{\@arabic\c@figure}
\def\thetable{\@arabic\c@table}

\def\listoffigures{\@restonecolfalse\if@twocolumn\@restonecoltrue\onecolumn
 \fi\chapter*{Figures\@mkboth
 {FIGURES}{FIGURES}}\@starttoc{lof}\if@restonecol
 \twocolumn\fi}

\def\listoftables{\@restonecolfalse\if@twocolumn\@restonecoltrue\onecolumn
 \fi\chapter*{Tables\@mkboth
 {TABLES}{TABLES}}\@starttoc{lot}\if@restonecol
 \twocolumn\fi}

\def\endfigure{\vskip 0pt plus5pt \end@float}

\long\def\@caption#1[#2]#3{\par
  \expandafter\ifx \csname ext@#1\endcsname \ext@figure
    \ifch@pterhasfigures\else \global\ch@pterhasfigurestrue
      \addtocontents{\csname ext@#1\endcsname}{\protect\addvspace{10pt}\protect
      \leftline{\bf Chapter \thechapter.}}\fi
    \addcontentsline{\csname 
       ext@#1\endcsname}{#1}{\protect\numberline{Figure \csname 
      the#1\endcsname}{\ignorespaces #2}}\begingroup
      \@parboxrestore
      \normalsize
      \@makecaption{{\sc\csname fnum@#1\endcsname}}{\ignorespaces #3}\par
    \endgroup
  \else \expandafter\ifx \csname ext@#1\endcsname \ext@table
    \ifch@pterhastables\else \global\ch@pterhastablestrue
      \addtocontents{\csname ext@#1\endcsname}{\protect\addvspace{10pt}\protect
      \leftline{\bf Chapter \thechapter.}}\fi
    \addcontentsline{\csname 
       ext@#1\endcsname}{#1}{\protect\numberline{Table \csname
      the#1\endcsname}{\ignorespaces #2}}\begingroup
      \@parboxrestore
      \normalsize
      \@maketablecaption{{\sc\csname fnum@#1\endcsname}}{\ignorespaces #3}\par
    \endgroup
  \else
    \typeout{Don't know about #1s!!!}\fi\fi}

\long\def\@makecaption#1#2{
  \vskip 10pt 
  \setbox\@tempboxa\hbox{#1.\ \small#2}
  \ifdim \wd\@tempboxa >\hsize #1.\ \small#2\par
  \else \hbox to\hsize{\hfil\box\@tempboxa\hfil}\fi}

\long\def\@maketablecaption#1#2{\centerline{\strut #1.}\medskip
  \centerline{\strut\it #2}\medskip}

\newtheorem{THEOREM}{Theorem}[section]
\newtheorem{PROPOSITION}[THEOREM]{Proposition}
\newtheorem{DEFINITION}[THEOREM]{Definition}
\newtheorem{LEMMA}[THEOREM]{Lemma}
\newtheorem{COROLLARY}[THEOREM]{Corollary}
\newtheorem{EXAMPLE}[THEOREM]{Example}
\newtheorem{REMARK}[THEOREM]{Remark}
\newtheorem{REMARKS}[THEOREM]{Remarks}
\newtheorem{DISCUSSION}[THEOREM]{}
\newtheorem{ALGORITHM}[THEOREM]{Algorithm}
\newtheorem{PROBLEM}[THEOREM]{Problem}

\def\thmstretch{plus.5em minus.1em}
\def\thmskip{0pt \thmstretch}
\def\thmstart{\hskip\thmskip\ignorespaces}
\let\thmtextfont=\it
\newif\ifthmtitle \thmtitletrue
\newenvironment{theorem}%
  {\let\thmtextfont=\it \thmtitletrue \begin{THEOREM}}%
{\end{THEOREM}}
\newenvironment{proposition}%
  {\let\thmtextfont=\it \thmtitletrue \begin{PROPOSITION}}%
  {\end{PROPOSITION}}
\newenvironment{definition}%
  {\let\thmtextfont=\rm \thmtitletrue \begin{DEFINITION}}%
  {\end{DEFINITION}}
  {\let\thmtextfont=\it \thmtitletrue \begin{LEMMA}}%
  {\end{LEMMA}}
\newenvironment{corollary}%
  {\let\thmtextfont=\it \thmtitletrue \begin{COROLLARY}}%
  {\end{COROLLARY}}
\newenvironment{example}%
  {\let\thmtextfont=\rm \thmtitletrue \begin{EXAMPLE}}%
  {\end{EXAMPLE}}
\newenvironment{remark}%
  {\let\thmtextfont=\rm \thmtitletrue \begin{REMARK}}%
  {\end{REMARK}}
  {\let\thmtextfont=\rm \thmtitletrue \begin{REMARKS}}%
  {\end{REMARKS}}
  {\let\thmtextfont=\rm \thmtitlefalse \begin{DISCUSSION}}%
  {\end{DISCUSSION}}
\newenvironment{algorithm}%
  {\let\thmtextfont=\rm \thmtitletrue \begin{ALGORITHM}}%
  {\end{ALGORITHM}}
\newenvironment{problem}%
  {\let\thmtextfont=\rm \thmtitletrue \begin{PROBLEM}}%
  {\end{PROBLEM}}

\newenvironment{steps}{\begin{description}\interlinepenalty=1000
  \abovedisplayskip=9pt plus 3pt minus 3pt
  \abovedisplayshortskip=1pt plus 2pt
  \belowdisplayskip=9pt plus 3pt minus 3pt
  \belowdisplayshortskip=6pt plus 3pt minus 1pt }%
  {\end{description}}
\def\step[#1]{\item[\rm{\it Step\/}~#1]}

\def\theoremhooks{\thmtextfont \topsep=9pt plus 3pt minus 6pt }

\def\@begintheorem#1#2{\theoremhooks \trivlist \ifthmtitle
  \item[\hskip \labelsep{\bf #1\ #2.}]\else
  \item[\hskip \labelsep{\bf #2.}]\fi\thmstart}
\def\@opargbegintheorem#1#2#3{\thmtextfont \trivlist
      \item[\hskip\labelsep{\bf #1\ #2\ {\rm(#3)}.}]\thmstart}

\newif\ifmathtomb \mathtombfalse

\def\tombstone{\unskip\penalty50   
  \hskip 0pt plus-1fill \null\nobreak\hskip 0pt plus1fill
  \enskip \vrule width.3333em height.7em depth.2em
  \ifmmode \global\mathtombtrue \else \global\mathtombfalse \fi}

\newenvironment{proof}
  {\futurelet\next\pr@oftext}
  {\ifmathtomb \else \tombstone \fi \widowpenalty=10000  
   \par \ifmathtomb \else \addvspace{\medskipamount}\fi \global\mathtombfalse}
\def\pr@oftext{\ifx\next[\let\temp\opr@@ftext\else\let\temp\pr@@ftext\fi\temp}
\def\pr@@ftext{\beginpr@@f{Proof}}
\def\opr@@ftext[#1]{\beginpr@@f{#1}}
\def\beginpr@@f#1{\par \addvspace{\bigskipamount}%
  \noindent{\bf #1.}\hskip\labelsep\thmstretch\ignorespaces}

\gdef\hysep@chicago{\ }\gdef\hisep@chicago{;}\gdef\heyesep@chicago{,}

\def\citationstyle#1{%
\global\@namedef{@hysep}{\csname hysep@#1\endcsname}%
\global\@namedef{@hisep}{\csname hisep@#1\endcsname}%
\global\@namedef{@heyesep}{\csname heyesep@#1\endcsname}}

\let\3=\&  

\def\@citexasnoun[#1]#2{%
\if@filesw\immediate\write\@auxout{\string\citation{#2}}\fi%
\@citeasnoun{{\@ifundefined%
{b@#2}%
{{\bf ?}\@warning{Citation `#2' on page \thepage \space undefined}}%
{{{\def\&{and}\def\3{\unskip,\ and}\csname b@#2\endcsname}\ (\csname bhy@#2\endcsname}%
\global\@namedef{b@#2}{\csname bha@#2\endcsname}}%
}}{#1}}

\def\@citexyear[#1]#2{\if@filesw\immediate\write\@auxout{\string\citation{#2}}\fi
  \def\@citeayear{}\@cite{\@for\@citebyear:=#2\do
    {\@citeayear\def\@citeayear{\@heyesep\penalty\@m\ }\@ifundefined
       {b@\@citebyear}{{\bf ?}\@warning
       {Citation `\@citebyear' on page \thepage \space undefined}}%
{{\csname bhy@\@citebyear\endcsname}%
}%
}}{#1}}

\citationstyle{chicago}

\def\@chapabb{Chap.}
\def\@sectionabb{\S}

\def\runningsize{\ixpt}
\def\numbersize{\tenrm}
\def\pagenumsize{\tenrm}

\if@twoside
\def\ps@headings{\let\@mkboth\markboth
  \def\@oddfoot{}\def\@evenfoot{}%
  \def\@evenhead{{\pagenumsize\thepage}\quad{\runningsize\leftmark}\hfil}%
  \def\@oddhead{\null\hfil{\runningsize\rightmark}\quad{\pagenumsize\thepage}}%
  \def\chaptermark##1{\markboth{\uppercase{##1\hfill\ifnum\c@secnumdepth>\m@ne
    \@chapabb\ {\numbersize\thechapter}.\fi}}{\hfill\uppercase{##1}}}%
  \def\sectionmark##1{\markboth{\uppercase{\ch@ptern@me\hfill
    \ifnum\c@secnumdepth>\m@ne\@chapabb\ 
      {\numbersize\thechapter},\fi}}{\ifnum\c@secnumdepth>\m@ne
    \@sectionabb\thinspace{\numbersize\thesection}.\fi\hfill\uppercase{##1}}}}
\else
\def\ps@headings{\let\@mkboth\markboth
  \def\@oddfoot{}\def\@evenfoot{}%
  \def\@oddhead{\null\hfil{\runningsize\rightmark}\quad{\pagenumsize\thepage}}%
  \def\chaptermark##1{\markright{\ifnum\c@secnumdepth>\m@ne
    \uppercase{\@chapabb\ {\numbersize\thechapter}.\fi \hfill ##1}}}
  \def\sectionmark##1{\markright{\ifnum\c@secnumdepth>\m@ne
    \uppercase{\@chapabb\ {\numbersize\thechapter},\ \ 
      \@sectionabb\thinspace{\numbersize\thesection}.\fi
    \hfill \ch@ptern@me}}}}
\fi
\def\ps@plain{\let\@mkboth\@gobbletwo
  \def\@oddhead{}\def\@oddfoot{\hfil\pagenumsize\thepage}%
  \def\@evenhead{}\def\@evenfoot{\pagenumsize\thepage\hfil}}

\iftweakfont
\def\ps@headings{\let\@mkboth\markboth
  \def\@oddfoot{\xpt\tt\hfil z\hfil}\def\@evenfoot{\xpt\tt\hfil z\hfil}%
  \def\@oddhead{\null\hfil{\runningsize\rightmark}\quad{\pagenumsize\thepage}}%
  \def\chaptermark##1{\markright{\ifnum\c@secnumdepth>\m@ne
    \uppercase{\@chapabb\ {\numbersize\thechapter}.\fi \hfill ##1}}}
  \def\sectionmark##1{\markright{\ifnum\c@secnumdepth>\m@ne
    \uppercase{\@chapabb\ {\numbersize\thechapter},\ \ 
      \@sectionabb\thinspace{\numbersize\thesection}.\fi
    \hfill \ch@ptern@me}}}}
\def\ps@plain{\let\@mkboth\@gobbletwo
  \def\@oddhead{}\def\@oddfoot{\hfil{\xpt\tt z}\hfil\pagenumsize\thepage}%
  \def\@evenhead{\hfil{\xpt\tt z}\hfil}\def\@evenfoot{\pagenumsize\thepage\hfil}}
\fi 

\def\thinskip{\hskip .16667em }

{\catcode`|=\active \catcode`\!=\active
\gdef\references{\catcode`|=\active \catcode`\!=\active
  \def\!{\char`\!}%
  \def|{\bgroup\sc \let|=\egroup}     
  \def!{\bgroup\it \let!=\egroup}     
  \def\<##1>{{\bf##1}\futurelet\next\number@ptional}  
  \def\number@ptional{\ifx\next(\def\@temp{\n@mber}\else
    \ifx\next:\def\@temp{\p@ges}\else\def\@temp{}\fi\fi \@temp}%
  \def\n@mber(##1){\thinspace(##1)\futurelet\next\page@ptional}%
  \def\page@ptional{\ifx\next:\def\@temp{\p@ges}\else
    \def\@temp{}\fi \@temp}%
  \def\p@ges:{\thinspace:\penalty-20\thinskip}%
  \frenchspacing}}
\def\authorbar{$\vcenter{\hbox{\vrule width3em height.4pt}}\,$}

\def\gnulog 1e{\futurelet\next\gnul@g}
\def\gnul@g{\ifx\next+\let\temp\gnul@gp\else\let\temp\gnul@gm\fi\temp}
\def\gnul@gp+{\afterassignment\gnul@gg\count10=}
\def\gnul@gm-{\afterassignment\gnul@gg-\count10=}
\def\gnul@gg{\the\count10}

\newenvironment{gnuplot}[1]{\vskip-.25in \hbox to\textwidth\bgroup
  \kern#1\xpt}{\hss\egroup \vskip-.15in}

\newenvironment{gnupicture}[1]{\hbox to#1\bgroup\xpt\hss}{\hss\egroup}

\def\eqnarray{\stepcounter{equation}\let\@currentlabel=\theequation
  \global\@eqnswtrue
  \global\@eqcnt\z@\tabskip\@centering\let\\=\@eqncr
  $$\halign to \displaywidth\bgroup\@eqnsel\hskip\@centering
    $\displaystyle\tabskip\z@{##}$&\global\@eqcnt\@ne 
    ${{}##{}}$&\global\@eqcnt\tw@$\displaystyle\tabskip\z@{##}$\hfil
    \tabskip\@centering&\llap{##}\tabskip\z@\cr}

\font\fivrm=cmr5
\font\fivmi=cmmi5
\font\fivsy=cmsy5

\font\sixrm=cmr6
\font\sixmi=cmmi6
\font\sixsy=cmsy6

\font\sevrm=cmr7
\font\sevit=cmti7
\font\sevsy=cmsy7
\font\sevmi=cmmi7

\font\egtrm=cmr8
\font\egtit=cmti8
\font\egtmi=cmmi8
\font\egtsy=cmsy8

\font\ninrm=cmr9
\font\ninbf=cmbx9
\font\ninit=cmti9
\font\nintt=cmtt9
\font\ninmi=cmmi9
\font\ninsy=cmsy9

\font\tenrm=cmr10
\font\tenit=cmti10
\font\tenbf=cmbx10
\font\tentt=cmtt10
\font\tensl=cmsl10
\font\tensf=cmss10
\font\tensy=cmsy10
\font\tenmi=cmmi10
\font\eightex=cmex8
\font\sevenbf=cmbx7
\font\fivebf=cmbx5

\font\elvex=cmex10 scaled\magstephalf
\font\eightex=cmex8
\font\eightbf=cmbx8
\font\sixbf=cmbx6

\font\twlrm=cmr12
\font\twlbf=cmbx12
\font\twltt=cmtt12
\font\twlsf=cmss12
\font\twlsl=cmsl10 at 12pt
\font\twlit=cmti10 at 12pt
\font\twlmi=cmmi10 at 12pt
\font\twlex=cmex10 at12pt
\font\twlsy=cmsy10 at 12pt
\font\eightex=cmex8
\font\eightbf=cmbx8
\font\sixbf=cmbx6

\font\frtnrm=cmr12 scaled\magstep1
\font\frtnmi=cmmi12 scaled\magstep1
  \skewchar\frtnmi='177
\font\frtnbf=cmbx12 scaled\magstep1
\font\frtnex=cmex10 scaled\magstep2
\font\tenex=cmex10
\font\sevenex=cmex7
\font\tenbf=cmbx10
\font\sevenbf=cmbx7

\font\svtnmi=cmmi12 scaled\magstep2
  \skewchar\svtnmi='177

\font\twtymi=cmmi12 scaled\magstep3
  \skewchar\twtymi='177

\font\fivly=lasy5
\font\sixly=lasy6
\font\sevly=lasy7
\font\egtly=lasy8
\font\ninly=lasy9
\font\tenly=lasy10
\font\elvly=lasy10 scaled\magstephalf
\font\twlly=lasy10 scaled\magstep1

\newif\if@bold
\newfam\lyfam
\def\@setstrut{\setbox\strutbox=\hbox{\vrule \@height .7\baselineskip
    \@depth .3\baselineskip \@width\z@}}
 
\def\ixpt{\textfont\z@\ninrm
  \scriptfont\z@\sixrm \scriptscriptfont\z@\fivrm
\textfont\@ne\ninmi \scriptfont\@ne\sixmi \scriptscriptfont\@ne\fivmi
\textfont\tw@\ninsy \scriptfont\tw@\sixsy \scriptscriptfont\tw@\fivsy
\textfont\thr@@\tenex \scriptfont\thr@@\tenex \scriptscriptfont\thr@@\tenex
\def\prm{\fam\z@\ninrm}%
\def\unboldmath{\everymath{}\everydisplay{}\@nomath\unboldmath
    \@boldfalse}\@boldfalse
\def\boldmath{\@subfont\boldmath\unboldmath}%
\def\pit{\fam\itfam\ninit}\textfont\itfam\ninit
   \scriptfont\itfam\sevit \scriptscriptfont\itfam\sevit
\def\psl{\@getfont\psl\slfam\@ixpt{cmsl9}}%
\def\pbf{\fam\bffam\ninbf}\textfont\bffam\ninbf
   \scriptfont\bffam\ninbf \scriptscriptfont\bffam\ninbf
\def\ptt{\fam\ttfam\nintt}\textfont\ttfam\nintt
   \scriptfont\ttfam\nintt \scriptscriptfont\ttfam\nintt
\def\psf{\@getfont\psf\sffam\@ixpt{cmss9}}%
\def\psc{\@getfont\psc\scfam\@ixpt{\@mcsc \@ptscale9}}%
\def\ly{\fam\lyfam\ninly}\textfont\lyfam\ninly
   \scriptfont\lyfam\sixly \scriptscriptfont\lyfam\fivly
  \ninebmit
  \ninefrak
\@setstrut \rm}

\def\xpt{\textfont\z@\tenrm
  \scriptfont\z@\sevrm \scriptscriptfont\z@\fivrm
\textfont\@ne\tenmi \scriptfont\@ne\sevmi \scriptscriptfont\@ne\fivmi
\textfont\tw@\tensy \scriptfont\tw@\sevsy \scriptscriptfont\tw@\fivsy
\textfont\thr@@\tenex \scriptfont\thr@@\eightex \scriptscriptfont\thr@@\eightex
\def\unboldmath{\everymath{}\everydisplay{}\@nomath\unboldmath
          \textfont\@ne\tenmi
          \textfont\tw@\tensy \textfont\lyfam\tenly
          \@boldfalse}\@boldfalse
\def\boldmath{\@ifundefined{tenmib}{\global\font\tenmib\@mbi
   \global\font\tensyb\@mbsy
   \global\font\tenlyb\@lasyb\relax\@addfontinfo\@xpt
   {\def\boldmath{\everymath{\mit}\everydisplay{\mit}\@prtct\@nomathbold
        \textfont\@ne\tenmib \textfont\tw@\tensyb
        \textfont\lyfam\tenlyb \@prtct\@boldtrue}}}{}\@xpt\boldmath}%
\def\prm{\fam\z@\tenrm}%
\def\pit{\fam\itfam\tenit}\textfont\itfam\tenit \scriptfont\itfam\sevit
    \scriptscriptfont\itfam\sevit
\def\psl{\fam\slfam\tensl}\textfont\slfam\tensl
     \scriptfont\slfam\tensl \scriptscriptfont\slfam\tensl
\def\pbf{\fam\bffam\tenbf}\textfont\bffam\tenbf
    \scriptfont\bffam\sevenbf \scriptscriptfont\bffam\fivebf
\def\ptt{\fam\ttfam\tentt}\textfont\ttfam\tentt
    \scriptfont\ttfam\tentt \scriptscriptfont\ttfam\tentt
\def\psf{\fam\sffam\tensf}\textfont\sffam\tensf
    \scriptfont\sffam\tensf \scriptscriptfont\sffam\tensf
\def\psc{\@getfont\psc\scfam\@xpt{\@mcsc}}%
\def\ly{\fam\lyfam\tenly}\textfont\lyfam\tenly
   \scriptfont\lyfam\sevly \scriptscriptfont\lyfam\fivly
  \tenbmit
  \tenfrak
  \tenbfit
\@setstrut \rm}

\def\xipt{\textfont\z@\elvrm
  \scriptfont\z@\egtrm \scriptscriptfont\z@\sixrm
\textfont\@ne\elvmi \scriptfont\@ne\egtmi \scriptscriptfont\@ne\sixmi
\textfont\tw@\elvsy \scriptfont\tw@\egtsy \scriptscriptfont\tw@\sixsy
\textfont\thr@@\elvex \scriptfont\thr@@\eightex \scriptscriptfont\thr@@\eightex
\def\unboldmath{\everymath{}\everydisplay{}\@nomath\unboldmath
          \textfont\@ne\elvmi \textfont\tw@\elvsy
          \textfont\lyfam\elvly \@boldfalse}\@boldfalse
\def\boldmath{\@ifundefined{elvmib}{\global\font\elvmib\@mbi\@halfmag
         \global\font\elvsyb\@mbsy\@halfmag
         \global\font\elvlyb\@lasyb\@halfmag\relax\@addfontinfo\@xipt
         {\def\boldmath{\everymath{\mit}\everydisplay{\mit}\@prtct\@nomathbold
                \textfont\@ne\elvmib \textfont\tw@\elvsyb
                \textfont\lyfam\elvlyb\@prtct\@boldtrue}}}{}\@xipt\boldmath}%
\def\prm{\fam\z@\elvrm}%
\def\pit{\fam\itfam\elvit}\textfont\itfam\elvit
   \scriptfont\itfam\egtit \scriptscriptfont\itfam\sevit
\def\psl{\fam\slfam\elvsl}\textfont\slfam\elvsl
    \scriptfont\slfam\tensl \scriptscriptfont\slfam\tensl
\def\pbf{\fam\bffam\elvbf}\textfont\bffam\elvbf
   \scriptfont\bffam\eightbf \scriptscriptfont\bffam\sixbf
\def\ptt{\fam\ttfam\elvtt}\textfont\ttfam\elvtt
   \scriptfont\ttfam\nintt \scriptscriptfont\ttfam\nintt
\def\psf{\fam\sffam\elvsf}\textfont\sffam\elvsf
    \scriptfont\sffam\tensf \scriptscriptfont\sffam\tensf
\def\psc{\@getfont\psc\scfam\@xipt{\@mcsc\@halfmag}}%
\def\ly{\fam\lyfam\elvly}\textfont\lyfam\elvly
   \scriptfont\lyfam\egtly \scriptscriptfont\lyfam\sixly
  \setbox\strutbox=\hbox{\vrule height9pt depth4pt width\z@}%
  \def\big##1{{\hbox{$\left##1\vbox to9\p@{}\right.\n@space$}}}%
  \def\Big##1{{\hbox{$\left##1\vbox to12.5\p@{}\right.\n@space$}}}%
  \def\bigg##1{{\hbox{$\left##1\vbox to16\p@{}\right.\n@space$}}}%
  \def\Bigg##1{{\hbox{$\left##1\vbox to19\p@{}\right.\n@space$}}}%
  \elevenbmit
  \elevenfrak
  \elevenbfit
\@setstrut \rm}
 
\def\xiipt{\textfont\z@\twlrm
  \scriptfont\z@\egtrm \scriptscriptfont\z@\sixrm
\textfont\@ne\twlmi \scriptfont\@ne\egtmi \scriptscriptfont\@ne\sixmi
\textfont\tw@\twlsy \scriptfont\tw@\egtsy \scriptscriptfont\tw@\sixsy
\textfont\thr@@\twlex \scriptfont\thr@@\eightex \scriptscriptfont\thr@@\eightex
\def\unboldmath{\everymath{}\everydisplay{}\@nomath\unboldmath
          \textfont\@ne\twlmi
          \textfont\tw@\twlsy \textfont\lyfam\twlly
          \@boldfalse}\@boldfalse
\def\boldmath{\@ifundefined{twlmib}{\global\font\twlmib\@mbi\@magscale1\global
        \font\twlsyb\@mbsy \@magscale1\global\font
         \twllyb\@lasyb\@magscale1\relax\@addfontinfo\@xiipt
              {\def\boldmath{\everymath
                {\mit}\everydisplay{\mit}\@prtct\@nomathbold
                \textfont\@ne\twlmib \textfont\tw@\twlsyb
                \textfont\lyfam\twllyb\@prtct\@boldtrue}}}{}\@xiipt\boldmath}%
\def\prm{\fam\z@\twlrm}%
\def\pit{\fam\itfam\twlit}\textfont\itfam\twlit \scriptfont\itfam\egtit
   \scriptscriptfont\itfam\sevit
\def\psl{\fam\slfam\twlsl}\textfont\slfam\twlsl
     \scriptfont\slfam\tensl \scriptscriptfont\slfam\tensl
\def\pbf{\fam\bffam\twlbf}\textfont\bffam\twlbf
   \scriptfont\bffam\eightbf \scriptscriptfont\bffam\sixbf
\def\ptt{\fam\ttfam\twltt}\textfont\ttfam\twltt
   \scriptfont\ttfam\nintt \scriptscriptfont\ttfam\nintt
\def\psf{\fam\sffam\twlsf}\textfont\sffam\twlsf
    \scriptfont\sffam\tensf \scriptscriptfont\sffam\tensf
\def\psc{\@getfont\psc\scfam\@xiipt{\@mcsc\@magscale1}}%
\def\ly{\fam\lyfam\twlly}\textfont\lyfam\twlly
   \scriptfont\lyfam\egtly \scriptscriptfont\lyfam\sixly
  \setbox\strutbox=\hbox{\vrule height10pt depth4pt width\z@}%
  \def\big##1{{\hbox{$\left##1\vbox to10\p@{}\right.\n@space$}}}%
  \def\Big##1{{\hbox{$\left##1\vbox to13\p@{}\right.\n@space$}}}%
  \def\bigg##1{{\hbox{$\left##1\vbox to17.5\p@{}\right.\n@space$}}}%
  \def\Bigg##1{{\hbox{$\left##1\vbox to21\p@{}\right.\n@space$}}}%
  \twelvebmit
  \twelvefrak
  \twelvebfit
  \@setstrut \rm}

\def\xivpt{\textfont\z@\frtnrm
  \scriptfont\z@\tenrm \scriptscriptfont\z@\sevrm
\textfont\@ne\frtnmi \scriptfont\@ne\tenmi \scriptscriptfont\@ne\sevmi
\textfont\tw@\frtnsy \scriptfont\tw@\tensy \scriptscriptfont\tw@\sevsy
\textfont\thr@@\frtnex \scriptfont\thr@@\tenex \scriptscriptfont\thr@@\sevenex
\def\unboldmath{\everymath{}\everydisplay{}\@nomath\unboldmath
          \textfont\@ne\frtnmi \textfont\tw@\frtnsy
          \textfont\lyfam\frtnly \@boldfalse}\@boldfalse
\def\boldmath{\@ifundefined{frtnmib}{\global\font
        \frtnmib\@mbi\@magscale2\global\font\frtnsyb\@mbsy\@magscale2
         \global\font\frtnlyb\@lasyb\@magscale2\relax\@addfontinfo\@xivpt
               {\def\boldmath{\@prtct\@nomathbold
  \textfont\z@\frtnbf   \scriptfont\z@\tenbf   \scriptscriptfont\z@\sevenbf
  \textfont\@ne\frtnib  \scriptfont\@ne\tenib  \scriptscriptfont\@ne\sevenib
  \textfont\tw@\frtnsyb \scriptfont\tw@\tensyb \scriptscriptfont\tw@\sevsy
  \def\it{\fam\@ne}%
  \textfont\lyfam\frtnlyb\@prtct\@boldtrue}}}{}\@xivpt\boldmath}%
\def\prm{\fam\z@\frtnrm}%
\def\pit{\@getfont\pit\itfam\@xivpt{cmti10\@magscale2}}%
\def\psl{\@getfont\psl\slfam\@xivpt{cmsl10\@magscale2}}%
\def\pbf{\fam\bffam\frtnbf}\textfont\bffam\frtnbf
   \scriptfont\bffam\tenbf \scriptscriptfont\bffam\sixbf
\def\ptt{\@getfont\ptt\ttfam\@xivpt{cmtt10\@magscale2}\@nohyphens\ptt\@xivpt}%
\def\psf{\@getfont\psf\sffam\@xivpt{\@mss\@magscale2}}%
\def\psc{\@getfont\psc\scfam\@xivpt{\@mcsc\@magscale2}}%
\def\ly{\fam\lyfam\frtnly}\textfont\lyfam\frtnly
   \scriptfont\lyfam\tenly \scriptscriptfont\lyfam\sevly
  \setbox\strutbox=\hbox{\vrule height12pt depth4.5pt width0pt}%
  \def\big##1{{\hbox{$\left##1\vbox to12\p@{}\right.\n@space$}}}%
  \def\Big##1{{\hbox{$\left##1\vbox to16.5\p@{}\right.\n@space$}}}%
  \def\bigg##1{{\hbox{$\left##1\vbox to21\p@{}\right.\n@space$}}}%
  \def\Bigg##1{{\hbox{$\left##1\vbox to25\p@{}\right.\n@space$}}}%
\fourteenbmit
\fourteenfrak
\fourteenbfit
\@setstrut \rm}

\def\eqalign#1{\null\,\vcenter{\openup\jot\m@th
  \ialign{\strut\hfil$\displaystyle{##}$&$\displaystyle{{}##}$\hfil
      \crcr#1\crcr}}\,}
\def\eqalignno#1{\displ@y \tabskip\@centering
  \halign to\displaywidth{\hfil$\@lign\displaystyle{##}$\tabskip\z@skip
    &$\@lign\displaystyle{{}##}$\hfil\tabskip\@centering
    &\llap{$\@lign##$}\tabskip\z@skip\crcr
    #1\crcr}}
\def\eqaligncond#1{\null\,\vcenter{\openup\jot\m@th
  \ialign{\strut\hfil\rm##\quad&\hfil$\displaystyle{##}$&$\displaystyle
      {{}##}$\hfil&\quad##\hfil\crcr#1\crcr}}\,}
\def\doubleeqalign#1{\null\,\vcenter{\openup\jot\m@th
  \ialign{\strut\hfil$\displaystyle{##}$&$\displaystyle{{}##}$\hfil
         &\qquad\hfil$\displaystyle{##}$&$\displaystyle{{}##}$\hfil
      \crcr#1\crcr}}\,}
\def\diffeqalign#1{\null\,\vcenter{\openup\jot\m@th
  \ialign{\strut\hfil$\displaystyle{##}$&$\displaystyle{{}##}$\hfil
      &$\displaystyle\qquad##$\hfil\crcr#1\crcr}}\,}

\def\dosuperejectLaTeX{\ifnum\@floatpenalty<\z@ 
  \line{}\kern-\topskip\nobreak\vfill\supereject\fi}

\makeatother 

\font\frtnib=cmmib10 scaled\magstep2  \font\sevenib=cmmib7
\font\frtnsyb=cmbsy10 scaled\magstep2 \font\tensyb=cmbsy10

\font\twlib=cmmib10 at12pt \font\eightib=cmmib8 \font\sixib=cmmib6
\font\elvib=cmmib10 scaled\magstephalf          \font\fiveib=cmmib5
\font\tenib=cmmib10 \font\nineib=cmmib9

\font\twlsyb=cmbsy10 scaled\magstep1

\skewchar\frtnib='177
\skewchar\twlib='177
\skewchar\elvib='177
\skewchar\tenib='177
\skewchar\nineib='177
\skewchar\eightib='177
\skewchar\sevenib='177
\skewchar\sixib='177
\skewchar\fiveib='177

\skewchar\frtnsyb='60 
\skewchar\twlsyb='60
\skewchar\tensyb='60

\newfam\bmitfam
\def\fourteenbmit{\textfont\bmitfam=\frtnib
  \scriptfont\bmitfam=\tenib \scriptscriptfont\bmitfam=\sevenib
  \def\bmit{\fam\bmitfam\frtnib}}
\def\twelvebmit{\textfont\bmitfam=\twlib
  \scriptfont\bmitfam=\eightib \scriptscriptfont\bmitfam=\sixib
  \def\bmit{\fam\bmitfam\twlib}}
\def\elevenbmit{\textfont\bmitfam=\elvib
  \scriptfont\bmitfam=\sevenib \scriptscriptfont\bmitfam=\fiveib
  \def\bmit{\fam\bmitfam\elvib}}
\def\tenbmit{\textfont\bmitfam=\tenib
  \scriptfont\bmitfam=\sevenib \scriptscriptfont\bmitfam=\fiveib
  \def\bmit{\fam\bmitfam\tenib}}
\def\ninebmit{\textfont\bmitfam=\nineib
  \scriptfont\bmitfam=\sixib \scriptscriptfont\bmitfam=\fiveib
  \def\bmit{\fam\bmitfam\nineib}}

\let\smallbmit=\tenbmit
\elevenbmit

{\boldmath}  

\def\redeflcgreek{\mathchardef\alpha="710B
  \mathchardef\beta="710C
  \mathchardef\gamma="710D
  \mathchardef\delta="710E
  \mathchardef\epsilon="710F
  \mathchardef\zeta="7110
  \mathchardef\eta="7111
  \mathchardef\theta="7112
  \mathchardef\iota="7113
  \mathchardef\kappa="7114
  \mathchardef\lambda="7115
  \mathchardef\mu="7116
  \mathchardef\nu="7117
  \mathchardef\xi="7118
  \mathchardef\pi="7119
  \mathchardef\rho="711A
  \mathchardef\sigma="711B
  \mathchardef\tau="711C
  \mathchardef\upsilon="711D
  \mathchardef\phi="711E
  \mathchardef\chi="711F
  \mathchardef\psi="7120
  \mathchardef\omega="7121
  \mathchardef\varepsilon="7122
  \mathchardef\vartheta="7123
  \mathchardef\varpi="7124
  \mathchardef\varrho="7125
  \mathchardef\varsigma="7126
  \mathchardef\varphi="7127
  \mathchardef\imath="717B
  \mathchardef\jmath="717C
  \mathchardef\ell="7160
  \mathchardef\wp="717D
  \mathchardef\partial="7140
  \mathchardef\flat="715B
  \mathchardef\natural="715C
  \mathchardef\sharp="715D }

\redeflcgreek

\font\frtneuf=eufm10 scaled\magstep2 \font\teneuf=eufm10 \font\seveneuf=eufm7

\font\twleuf=eufm10 at12pt \font\eighteuf=eufm8 \font\sixeuf=eufm6
\font\elveuf=eufm10 scaled\magstephalf
\font\teneuf=eufm10
\font\nineeuf=eufm9
\font\seveneuf=eufm7
\font\fiveeuf=eufm5

\newfam\frakfam
\def\fourteenfrak{\textfont\frakfam=\frtneuf \scriptfont\frakfam=\teneuf
  \scriptscriptfont\frakfam=\seveneuf}
\def\twelvefrak{\textfont\frakfam=\twleuf \scriptfont\frakfam=\eighteuf
  \scriptscriptfont\frakfam=\sixeuf}
\def\elevenfrak{\textfont\frakfam=\elveuf \scriptfont\frakfam=\eighteuf
  \scriptscriptfont\frakfam=\sixeuf}
\def\tenfrak{\textfont\frakfam=\teneuf
  \scriptfont\frakfam=\seveneuf \scriptscriptfont\frakfam=\fiveeuf}
\def\ninefrak{\textfont\frakfam=\nineeuf
  \scriptfont\frakfam=\sixeuf \scriptscriptfont\frakfam=\fiveeuf}
\def\frak{\fam\frakfam}

\let\smallfrak=\tenfrak
\elevenfrak

\font\frtnbit=cmmib10 scaled\magstep2
\font\twlbit=cmbxti10 at12pt
\font\elvbit=cmbxti10 scaled\magstephalf
\font\tenbit=cmbxti10

\newfam\bitfam
\def\fourteenbfit{\textfont\bitfam=\frtnbit
  \def\bfit{\fam\bitfam\frtnbit}}
\def\twelvebfit{\textfont\bitfam=\twlbit
  \def\bfit{\fam\bitfam\twlbit}}
\def\elevenbfit{\textfont\bitfam=\elvbit
  \def\bfit{\fam\bitfam\elvbit}}
\def\tenbfit{\textfont\bitfam=\tenbit
  \def\bfit{\fam\bitfam\tenbit}}

\elevenbfit

\def\R{{\bf R}}                     
\def\vf{{\frak X}}                  
\def\SO(#1){\mathord{\bfit S\!O}({#1})}
\def\GL(#1){\mathord{\bfit G\!L}({#1})}
\def\Tri(#1){\mathord{\bmit T}({#1})}
\let\ScandOslash=\O
\def\O{\ifmmode \def\temp{\OLiegp}\else \def\temp{\ScandOslash}\fi \temp}
\def\OLiegp(#1){\mathord{\bmit O}({#1})}
\let\Scandoslash=\o
\def\o{\ifmmode \def\temp{\oLiealg}\else \def\temp{\Scandoslash}\fi \temp}
\def\oLiealg(#1){\mathord{\frak o}({#1})}
\def\so(#1){\mathord{\frak so}({#1})}
\def\gl(#1){\mathord{\frak gl}({#1})}
\def\tri(#1){\mathord{\frak t}({#1})}
\def\g{{\frak g}}
\def\h{{\frak h}}
\def\k{{\frak k}}
\def\m{{\frak m}}
\def\stiefel(#1,#2){{V_{#1,#2}}}
\def\grassmann(#1,#2){{G_{#1,#2}}}
\def\svd{\bigl(\OLiegp(n)\times\OLiegp(k)\bigr)/\Dg(D)\OLiegp(n-k)}
\def\Dg(#1){{\mathord{\mit\Delta}_{#1}}}
\def\Sym(#1){\mathord{\bmit S}_{\bmit #1}}
\def\K(#1){\mathord{\bmit K}_{\bmit #1}}

\def\phat{{\hat p}}
\def\xhat{{\hat x}}
\def\Thetahat{{\hat\Theta}}
\def\Xtilde{{\tilde X}}
\def\Ytilde{{\tilde Y}}
\def\Htilde{{\tilde H}}
\def\Omegatilde{{\tilde\Omega}}

\def\etal{et~al.\null}

\catcode`|=\active
\def|{\ifmmode \vert \else \bgroup\footnotesize\def|{\null\egroup}\fi}

\def\conditions#1{\vbox{\halign{\rm\hfil##&\quad##\hfil&\qquad##\hfil
  &&\qquad\rm\hfil##&\quad##\hfil&\qquad##\hfil\crcr
    #1\crcr}}}

\def\T{{\scriptscriptstyle\rm T}}   
\def\by{\ifmmode $\hbox{-by-}$\else \leavevmode\hbox{-by-}\fi}
\def\hyphen{\ifmmode\hbox{-}\else-\fi}
\def\f{{\!f}}
\let\(=\langle
\let\)=\rangle
\def\kf{\varphi}  
\def\geodesic#1{\gamma_{\lower1pt\hbox{$\scriptstyle#1$}}}
\def\ddt{{d\over dt}\Big|_{t=0}}
\def\covD#1{\nabla_{\!#1}}          
\def\D#1{{\nabla\!#1}}              
\def\Dsqr#1{{\nabla^2\!#1}}         
\def\grad{\mathop{\rm grad}\nolimits}
\def\tr{\mathop{\rm tr}\nolimits}
\def\diag{\mathop{\rm diag}\nolimits}
\def\sign{\mathop{\rm sign}\nolimits}
\def\argmax{\mathop{\rm arg\,max}}

\def\Riemann(#1,#2){\mathop{R(#1,#2)}}
\def\Aut{\mathord{\rm Aut}}
\def\End{\mathord{\rm End}}
\def\id{\mathop{\rm id}\nolimits}
\def\Ad{\mathop{\rm Ad}\nolimits}
\def\ad{\mathop{\rm ad}\nolimits}
\mathchardef\rdot="0201             
\def\Ball{{B_\epsilon(\phat)}}
\def\Diff{{\mit\Delta}}
\def\vrem{{\mit\Xi}}
\def\hot{{\rm h.o.t.}}
\def\nupushf{{\nu_{\mskip-1.5mu*}\mskip-2mu f}}
\def\lyapunov{\tr\Theta^\T Q\Theta N}
\def\adjarrow{\mathop{\vcenter{\hbox{\dimen1=40pt \dimen3=1.5pt
  \raise\dimen3\rlap{\hbox to\dimen1{\rightarrowfill}}\lower\dimen3
  \hbox to\dimen1{\leftarrowfill}}}}}
\def\smallchoice#1{\mathchoice{{\textstyle#1}}{{\textstyle#1}}%
  {{\scriptstyle#1}}{{\scriptscriptstyle#1}}}
\def\oneoverx#1{{1\over#1}}
\def\half{{\smallchoice{\oneoverx2}}}
\def\third{{\smallchoice{\oneoverx3}}}
\def\quarter{{\smallchoice{\oneoverx4}}}

\def\[{[\mkern-3mu[} \def\]{]\mkern-3mu]}

\begin{document}

\baselineskip=\thesisbaselineskip
\parskip=0pt plus 2pt
\setcounter{page}{1}
\pagenumbering{roman}




\setcounter{page}{1}
\pagenumbering{roman}
\thispagestyle{empty}
\begingroup \leftskip=0pt plus1fil \rightskip=\leftskip
\parindent=0pt \parskip=0pt \parfillskip=0pt
\baselineskip=15pt
\obeylines

\null

\vfil

\Large\bf
Geometric Optimization Methods for Adaptive Filtering

\vfil

\xiipt
A thesis presented
by

\vskip.25in

\Large
Steven Thomas Smith

\vskip.25in

\xiipt
to
The Division of Applied Sciences

\vskip.1in

in partial fulfillment
for the degree of Doctor of Philosophy
in the subject of
Applied Mathematics

\vfil

Harvard University
Cambridge, Massachusetts

\vskip.25in

May 1993

\vfil

\endgroup
\eject

\thispagestyle{empty}
\null\vfil
\centerline{\harvard}
\vfil
\vbox to0in{\vss
\leftline{Copyright \copyright\ 1993 by Steven T. Smith. All rights reserved.}}
\eject

\thispagestyle{empty}
\begingroup \leftskip=0pt \rightskip=0pt plus1fil
\parindent=0pt \parskip=0pt \parfillskip=0pt
\large\it\baselineskip=14pt
\obeylines

\null\vskip2in

To Laura

\vfil

\ixpt\it\baselineskip=10pt
\def\\{\vskip 4pt}

L'aura soave al sole spiega et vibra
l'auro ch' Amor di sua man fila et tesse;
l\`a da' belli occhi et de le chiome stesse
lega 'l cor lasso e i lievi spirti cribra.
\\
Non \`o medolla in osso o sangue in fibra
ch' i' non senta tremar pur ch' i' m'apresse
dove \`e chi morte et vita inseme, spesse
volte, in frale bilancia apprende et libra,
\\
vedendo ardere i lumi ond' io m'accendo,
et folgorare i nodi ond' io son preso
or su l'omero destro et or sul manco.
\\
I' nol posso ridir, ch\'e nol comprendo,
da ta' due luci \`e l'intelletto offeso
et di tanta dolcezza opresso et stanco.
\\\\
\noindent\hskip1in Petrarch, Rime sparse  

%
%
%
%

\vskip1in
\endgroup
\eject

\begingroup

\makeatletter  
\def\@makeschapterhead#1{\vspace*{0pt}{\leftskip=0pt
  \rightskip0pt plus1fil \parindent0pt\parskip0pt\parfillskip=0pt
  \Large\bf \uppercase{#1}\par 
  \nobreak\vskip5pt plus5pt }}
\makeatother

\chapter*{Acknowledgments}
\addcontentsline{toc}{frontmatter}{Acknowledgments}

I would like to thank my advisor Professor Roger Brockett for
directing me in three directions that eventually coalesced: subspace
tracking, gradient flows on Lie groups, and conjugate gradient methods
on symmetric spaces.  I am also indebted to Tony Bloch for the
invitation to speak at the Fields Institute workshop on Hamiltonian
and gradient flows in April~1992, and to the Fields Institute itself
for their generous support during my visit.  I greatly appreciated the
opportunity to present my work to the distinguished participants of
this workshop.  My exposition was influenced by the teaching style of
Professor Guillemin, and I benefited from many helpful discussions
with Professor Anderson, who has always been generous with his time
and advice.  I would also like to thank the members of my thesis
committee: Professors James Clark, Petros Maragos, and David Mumford,
not only for their evaluation of this work, but also for the
intellectual and social qualities that they and their students brought
to the Harvard Robotics Laboratory.

My life in graduate school would have been much more difficult without
the pleasant and supportive atmosphere that my classmates provided.
The help and advice of my predecessors was invaluable in learning the
graduate student ropes. I will always remember sitting on the bank of
Walden Pond with Bill Nowlin, Nicola Ferrier, and Ann Stokes on my
first day off after a very long first year.  ``I'm supposed to pick a
project this summer,'' I said.  Bill was quick with his advice: ``The
summer after first year is just a waste.''  I also remember Nicola
bringing cookies to me and Bob Hewes late one night because she
remembered how heavy the work load was.  Ken Keeler and Frank Park
enriched the cultural life of Pierce~G14 by founding an ever growing
museum. Peter Hallinan and Tai~Sing Lee taught me some climbing
techniques at Acadia Natl.\ Park. Gaile Gordon hosted many enjoyable
and memorable social events.  Ed~Rak convinced me that I could hack
\TeX\ macros and still graduate. Ann Stokes's friendship made office
life enjoyable even in the batcave. John Page and George Thomas worked
very hard to ensure that the lab ran smoothly.  Leonid Faybusovich
taught me about the gradient.  I thoroughly enjoyed the discussions I
shared with Dan Friedman, Jeff Kosowsky, Bob Hewes, Peter Belhumeur,
and others.  We all owe a debt to Navin Saxena for making our
Wednesday afternoon lab tea a thriving event.\looseness=-1\par

Finally, I thank my family, especially my parents for their
encouragement, empathy, and love.  Most of all, I thank my wife Laura,
whom I love very much.  Her constant thoughtfulness, support,
patience, and love made many difficult days good, and many good days
great.\looseness=-1\par

\eject 
\endgroup

\chapter*{Abstract}
\addcontentsline{toc}{frontmatter}{Abstract}

\begingroup

The techniques and analysis presented in this thesis provide new
methods to solve optimization problems posed on Riemannian manifolds.
These methods are applied to the subspace tracking problem found in
adaptive signal processing and adaptive control.  A new point of view
is offered for the constrained optimization problem.  Some classical
optimization techniques on Euclidean space are generalized to
Riemannian manifolds.  Several algorithms are presented and their
convergence properties are analyzed employing the Riemannian structure
of the manifold.  Specifically, two new algorithms, which can be
thought of as Newton's method and the conjugate gradient method on
Riemannian manifolds, are presented and shown to possess quadratic and
superlinear convergence, respectively.  These methods are applied to
several eigenvalue and singular value problems, which are posed as
constrained optimization problems.  New efficient algorithms for the
eigenvalue problem are obtained by exploiting the special homogeneous
space structure of the constraint manifold.  It is shown that Newton's
method applied to the Rayleigh quotient on a sphere converges
cubically, and that the Rayleigh quotient iteration is an efficient
approximation of Newton's method.  The Riemannian version of the
conjugate gradient method applied to this function gives a new
algorithm for finding the eigenvectors corresponding to the extreme
eigenvalues of a symmetric matrix.  The Riemannian version of the
conjugate gradient method applied to a generalized Rayleigh quotient
yields a superlinearly convergent algorithm for computing the $k$
eigenvectors corresponding to the extreme eigenvalues of an $n\by n$
matrix.  This algorithm requires $O(nk^2)$ operations and $O(k)$
matrix-vector multiplications per step. Several gradient flows are
analyzed that solve eigenvalue and singular value problems.  The new
optimization algorithms are applied to the subspace tracking problem
of adaptive signal processing.  A new algorithm for subspace tracking
is given, which is based upon the conjugate gradient method applied to
the generalized Rayleigh quotient.  The results of several numerical
experiments demonstrating the convergence properties of the new
algorithms are given.\parfillskip=0pt

\endgroup

\tableofcontents
\vfil\eject

\listoffigures
\addcontentsline{toc}{frontmatter}{Figures}
\vfil\eject

\listoftables
\addcontentsline{toc}{frontmatter}{Tables}
\vfil\eject

\clearpage
\setcounter{page}{1}
\pagenumbering{arabic}

\input introduction     

\input chap-geom        

\input chap-gradflow    

\input chap-orm         

\input chap-af          

\input conclusions      

\input references       

\end{document}

%% file: introduction.tex
\chapter{Introduction}

Optimization is the central idea behind many problems in science and
engineering.  Indeed, determination of ``the best'' is both a
practical and an aesthetic problem that is encountered almost
universally.  Thus it is not surprising to find in many areas of study
a variety of optimization methods and vocabulary.  While the statement
of the optimization problem is simple---given a set of points and an
assignment of a real number to each point, find the point with the
largest or smallest number---its solution is not.  In general, the
choice of optimization algorithm depends upon many factors and
assumptions about the underlying set and the real-valued function
defined on the set.  If the set is discrete, then a simple search and
comparison algorithm is appropriate.  If the discrete set is endowed
with a topology, then a tree searching algorithm can yield a local
extremum.  If the set is a finite-dimensional vector space and the
function is continuous, then the simplex method can yield a local
extremum.  If the set is a Euclidean space, i.e., a finite-dimensional
vector space with inner product, and the function is differentiable,
then gradient-based methods may be used.  If the set is a polytope,
i.e., a subset of Euclidean space defined by linear inequality
constraints, and a linear function, then linear programming techniques
are appropriate.  This list indicates how successful optimization
techniques exploit the given structure of the underlying space.  This
idea is an important theme of this thesis, which explains how the
metric structure on a manifold may be used to develop effective
optimization methods on such a space.

Manifolds endowed with a metric structure, i.e., Riemannian manifolds,
arise naturally in many applications involving optimization problems.
For example, the largest eigenvalue of a symmetric matrix corresponds
to the point on a sphere maximizing the Rayleigh quotient.  This
eigenvalue problem and its generalizations are encountered in diverse
fields: signal processing, mechanics, control theory, estimation
theory, and others.  In most cases the so-called principal invariant
subspace of a matrix must be computed.  This is the subspace spanned
by the eigenvectors or singular vectors corresponding to the largest
eigenvalues or singular values, respectively.  Oftentimes there is an
adaptive context so that the principal invariant subspaces change over
time and must be followed with an efficient tracking algorithm.  Many
algorithms rely upon optimization techniques such as gradient
following to perform this tracking.

A few analytic optimization methods are quite old, but, as in most
computational fields, the invention of electronic computers was the
impetus for the development of modern optimization theory and
techniques.  Newton's method has been a well-known approach for
solving optimization problems of one or many variables for centuries.
The method of steepest descent to minimize a function of several
variables goes back to Cauchy.  Its properties and performance are
well-known; see, e.g., the books of \citeasnoun{Polak},
\citeasnoun{Luenberger}, or \citeasnoun{Fletcher} for a description
and analysis of this technique.  Modern optimization algorithms
appeared in the middle of this century, with the introduction of
linear and quadratic programming algorithms, the conjugate gradient
algorithm of \citeasnoun{HestenesStiefel}, and the variable metric
algorithm of \citeasnoun{Davidon}.  It is now understood how these
algorithms may be used to compute the point in~$\R^n$ at which a
differentiable function attains its maximum value, and what
performance may be expected of them.

Of course, not all optimization problems are posed on a Euclidean
space, and much research has been done on the constrained optimization
problem, specifically when the underlying space is defined by equality
constraints on Euclidean space.  Because all Riemannian manifolds may
be defined in this way, this approach is general enough for the
purposes of this thesis.  What optimization algorithms are appropriate
on such a space?  \citeasnoun[pp.~254ff\/]{Luenberger} considers this
question in his exposition of the constrained optimization problem. He
describes an idealized steepest descent algorithm on the constraint
surface that employs geodesics in gradient directions, noting that
this approach is in general not computationally feasible.  For this
reason, other approaches to the constrained optimization problem have
been developed. All of these methods depend upon the imbedding of the
constraint surface in $\R^n$.  Projective methods compute a gradient
vector tangent to the constraint surface, compute a minimum in $\R^n$
along this direction, then project this point onto the constraint
surface.  Lagrange multiplier methods minimize a function defined
on~$\R^n$ constructed from the original function to be minimized and
the distance to the constraint surface.  However, this so-called
extrinsic approach ignores the intrinsic structure that the manifold
may have.  With specific examples, such as a sphere and others to be
discussed later, intrinsic approaches are computationally feasible,
but the study of intrinsic optimization algorithms is absent from the
literature.

Optimization techniques have long been applied to the fields of
adaptive filtering and control.  There is a need for such algorithms
in these two fields because of their reliance on error minimization
techniques and on the minimax characterization of the eigenvalue
problem.  Also, many scenarios in adaptive filtering and control have
slowly varying parameters which corresponding to the minimum point of
some function that must be estimated and tracked.  Gradient-based
algorithms are desirable in this situation because the minimum point
is ordinarily close to the current estimate, and the gradient provides
local information about the direction of greatest decrease.

Many researchers have applied constrained optimization techniques to
algorithms that compute the static or time varying principal invariant
subspaces of a symmetric matrix.  This problem may be viewed as the
problem of computing $k$ orthonormal vectors in~$\R^n$ that maximize a
generalized form of the Rayleigh quotient.  Orthonormality imposes the
constraint surface.  \citeasnoun{BradFletch} propose a projective
formulation of the constrained conjugate gradient method to solve the
symmetric eigenvalue problem.  \citeasnoun{Fried} proposes a very
similar method for application to finite element eigenvalue problems.
\citeasnoun{Chenetal} are the first to apply this projective conjugate
gradient method to the problem of adaptive spectral estimation for
signal processing.  However, these conjugate gradient algorithms are
based upon the classical unconstrained conjugate gradient method on
Euclidean space. They apply this algorithm to the constrained problem
without accounting for the curvature terms that naturally arise.  In
general, the superlinear convergence guaranteed by the classical
conjugate gradient method is lost in the constrained case when these
curvature terms are ignored.

\citeasnoun{FuhrLiu} recognize this fact in their constrained
conjugate gradient algorithm for maximizing the Rayleigh quotient on
a sphere.  They correctly utilize the curvature of the sphere to
develop a conjugate gradient algorithm on this space analogous to the
classical superlinearly convergent conjugate gradient algorithm.
Insofar as they use maximization along geodesics on the sphere instead
of maximization along lines in~$\R^n$ followed by projection, their
approach is the first conjugate gradient method employing instrinsic
ideas to appear.  However, they use an azimuthal projection to
identify points on the sphere with points in tangent planes, which is
not naturally defined because it depends upon the choice of imbedding.
Thus their method is extrinsic.  Although the asymptotic performance
of their constrained conjugate gradient algorithm is the same as one
to be presented in this thesis, their dependence on azimuthal
projection does not generalize to other manifolds.  We shall see that
completely intrinsic approaches on arbitrary Riemannian manifolds are
possible and desirable.

There are many other algorithms for computing the principal invariant
subspaces that are required for some methods used in adaptive
filtering \cite{ComonGolub}.  Of course, one could apply the |QR|
algorithm at each step in the adaptive filtering procedure to obtain a
full diagonal decomposition of a symmetric matrix, but this requires
$O(n^3)$ floating point operations ($n$ is the dimension of the
matrix), which is unnecessarily expensive. Also, many applications
require only the principal invariant subspace corresponding to the $k$
largest eigenvalues, thus a full decomposition involves wasted effort.
Furthermore, this technique does not exploit previous information,
which is important in most adaptive contexts.  So other techniques for
obtaining the eigenvalue decomposition are used.  In addition to the
constrained conjugate gradient approaches mentioned in the preceding
paragraphs, pure gradient based methods and other iterative techniques
are popular.  The use of gradient techniques in adaptive signal
processing was pioneered in the 1960s. See \citeasnoun{WidrowStearns}
for background and references.  Several algorithms for the adaptive
eigenvalue problem use such gradient ideas \cite{Owsley,Larimore,Hu}.

Iterative algorithms such as Lanczos methods are very important in
adaptive subspace tracking problems.  Lanczos methods compute a
sequence of tridiagonal matrices (or bidiagonal matrices in the case
where singular vectors are required) whose eigenvalues approximate the
extreme eigenvalues of the original matrix.  The computational
requirements of the classical Lanczos algorithm are modest: only
$O(nk^2)$ operations and $O(k)$ matrix-vector multiplications are
required to compute $k$ eigenvectors of an $n\by n$ symmetric matrix.
Thus Lanczos methods are well-suited for sparse matrix extreme
eigenvalue problems.  However, the convergence properties of the
classical Lanczos methods are troublesome, and they must be modified
to yield useful algorithms \cite{ParlettScott,Parlett,GVL,CullumWill}.

This thesis arose from the study of gradient flows applied to the
subspace tracking problem as described by
\citeasnoun{Brockett:subspace}, and from the study of gradient flows
that diagonalize matrices \cite{Brockett:sort}.  While the resulting
differential equation models are appealing from the perspective of
learning theory, it is computationally impractical to implement them
on conventional computers.  A desire to avoid the ``small step''
methods found in the integration of gradient flows while retaining
their useful optimization properties led to the investigation of
``large step'' methods on manifolds, analogous to the optimization
algorithms on Euclidean space discussed above.  A theory of such
methods was established, and then applied to the subspace tracking
problem, whose homogeneous space structure allows efficient and
practical optimization algorithms.

The following contributions are contained within this thesis. In
Chapter~\ref{chap:geom}, a geometric framework is provided for a large
class of problems in numerical linear algebra.  This chapter reviews
the natural metric structure of various Lie groups and homogeneous
spaces, along with some useful formulae implied by this structure,
which will be used throughout the thesis.  This geometric framework
allows one to solve problems in numerical linear algebra, such as the
computation of eigenvalues and eigenvectors, and singular values and
singular vectors, with gradient flows on Lie groups and homogeneous
spaces.

In Chapter~\ref{chap:grad} a gradient flow that yields the extreme
eigenvalues and corresponding eigenvectors of a symmetric matrix is
given, together with a gradient flow that yields the singular value
decomposition of an arbitrary matrix.  (Functions whose gradient flows
yield the extreme singular values and corresponding left singular
vectors of an arbitrary matrix are also discussed in
Chapter~\ref{chap:af}.)

Chapter~\ref{chap:orm} develops aspects of the theory of optimization
of differentiable functions defined on Riemannian manifolds.  New
methods and a new point of view for solving constrained optimization
problems are provided.  Within this chapter, the usual versions of
Newton's method and the conjugate gradient method are generalized to
yield new optimization algorithms on Riemannian manifolds.  The method
of steepest descent on a Riemannian manifold is first analyzed.
Newton's method on Riemannian manifolds is developed next and a proof
of quadratic convergence is given.  The conjugate gradient method is
in then introduced with a proof of superlinear convergence.  Several
illustrative examples are offered throughout this chapter.  These
three algorithms are applied to the Rayleigh quotient defined on the
sphere, and the function $f(\Theta)=\tr\Theta^\T Q\Theta N$ defined on
the special orthogonal group.  It is shown that Newton's method
applied to the Rayleigh quotient converges cubically, and that this
procedure is efficiently approximated by the Rayleigh quotient
iteration.  The conjugate gradient algorithm applied to the Rayleigh
quotient on the sphere yields a new superlinearly convergent algorithm
for computing the eigenvector corresponding to the extreme eigenvalue
of a symmetric matrix, which requires two matrix-vector
multiplications and $O(n)$ operations per iteration.

Chapter~\ref{chap:af} applies the techniques developed in the
preceding chapters to the subspace tracking problem of adaptive signal
processing.  The idea of tracking a principal invariant subspace is
reviewed in this context, and it is shown how this problem may be
viewed as maximization of the generalized Rayleigh quotient on the
so-called Stiefel manifold of matrices with orthonormal columns.  An
efficient conjugate gradient method that solves this optimization
problem is developed next.  This algorithm, like Lanczos methods,
requires $O(nk^2)$ operations per step and $O(k)$ matrix-vector
multiplications.  This favorable computational cost is dependent on
the homogeneous space structure of the Stiefel manifold; a description
of the algorithms implementation is provided.  Superlinear convergence
of this algorithm to the eigenvectors corresponding to the extreme
eigenvalues of a symmetric matrix is assured by results of
Chapter~\ref{chap:orm}.  This algorithm also has the desirable feature
of maintaining the orthonormality of the $k$ vectors at each step.  A
similar algorithm for computing the largest left singular vectors
corresponding to the extreme singular values of an arbitrary matrix is
discussed.  Finally, this algorithm is applied to the subspace
tracking problem.  A new algorithm for subspace tracking is given,
which is based upon the conjugate gradient method applied to the
generalized Rayleigh quotient.  The results of several numerical
experiments demonstrating the tracking properties of this algorithm
are given.

%% file: chap-geom.tex
\chapter{Riemannian geometry of Lie groups and homogeneous
spaces}\runningchaptername{Geometry of Lie groups and homogeneous
spaces}\label{chap:geom}

Both the analysis and development of optimization algorithms presented
in this thesis rely heavily upon the geometry of the space on which
optimization problems are posed.  This chapter provides a review of
pertinent ideas from differential and Riemannian geometry, Lie groups,
and homogeneous spaces that will be used throughout the thesis.  It
may be skipped by those readers familiar with Riemannian geometry.
Sections \ref{sec:riemannian} and \ref{sec:liegphomsp} contain the
necessary theoretical background.  Section~\ref{sec:examples} provides
formulae specific to the manifolds to be used throughout this thesis,
which are derived from the theory contained in the previous sections.

\section{Riemannian manifolds}\label{sec:riemannian}

In this section the concepts of Riemannian structures, affine
connections, geodesics, parallel translation, and Riemannian
connections on a differentiable manifold are reviewed.  A background
of differentiable manifolds and tensor fields is assumed, e.g.,
Chapters 1--5 of \citeasnoun[Vol.~1]{Spivak} or the introduction of
\citeasnoun{GandG}.  The review follows Helgason's \citeyear{Helgason}
and Spivak's \citeyear[Vol.~2]{Spivak} expositions.

Let $M$ be a $C^\infty$ differentiable manifold.  Denote the set
of~$C^\infty$ functions on~$M$ by $C^\infty(M)$, the tangent plane
at~$p$ in~$M$ by~$T_p$ or $T_pM$, and the set of $C^\infty$ vector
fields on~$M$ by~$\vf(M)$.

\subsection{Riemannian structures}

\begin{definition} Let $M$ be a differentiable manifold. A {\it
Riemannian structure\/} on~$M$ is a tensor field $g$ of type~$(0,2)$
which for~all $X$, $Y\in\vf(M)$ and $p\in M$ satisfies
$$\conditions{(i)&$g(X,Y)=g(Y,X)$,\cr (ii)&$g_p\colon T_p\times
T_p\to\R$ is positive definite.\cr}$$ \end{definition}

A {\it Riemannian manifold\/} is a connected differentiable manifold
with a Riemannian structure.  For every $p$ in~$M$, the Riemannian
structure $g$ provides an inner product on~$T_p$ given by the
nondegenerate symmetric bilinear form $g_p\colon T_p\times T_p\to\R$.
The notation $\(X,Y\)=g_p(X,Y)$ and $\|X\|=g_p(X,X)^{1/2}$, where $X$,
$Y\in T_p$, is often used.  Let $t\mapsto\gamma(t)$, $t\in[a,b]$, be a
curve segment in~$M$.  The length of~$\gamma$ is defined by the
formula $$L(\gamma)=\int_a^b
g_{\gamma(t)}\bigl(\dot\gamma(t),\dot\gamma(t)\bigr)^{1/2}\,dt.$$
Because $M$ is connected, any two points $p$ and $q$ in~$M$ can be
joined by a curve. The infimum of the length of all curve segments
joining $p$ and $q$ yields a metric on~$M$ called the {\it Riemannian
metric} and denoted by~$d(p,q)$.

\begin{definition} Let $M$ be a Riemannian manifold with Riemannian
structure $g$ and $f\colon M\to\R$ a $C^\infty$ function on~$M$.  The
{\it gradient\/} of~$f$ at~$p$, denoted by $(\grad\f)_p$, is the
unique vector in~$T_p$ such that $df_p(X)=\((\grad\f)_p,X\)$ for~all
$X$ in~$T_p$.  \end{definition}

The corresponding vector field $\grad\f$ on~$M$ is clearly smooth.

The expression of the preceding ideas using coordinates is often
useful. Let $M$ be an $n\hyphen$dimensional Riemannian manifold with
Riemannian structure~$g$, and $(U,x^1,\ldots,x^n)$ a coordinate chart
on~$M$.  There exist $n^2$ functions $g_{ij}$, $1\le i,j\le n$, on~$U$
such that $$g=\sum_{i,j} g_{ij}\,dx^i\otimes dx^j.$$ Clearly
$g_{ij}=g_{ji}$ for~all $i$ and $j$.  Because $g_p$ is nondegenerate
for~all $p\in U\subset M$, the symmetric matrix $(g_{ij})$ is
invertible. The elements of its inverse are denoted by $g^{kl}$, i.e.,
$\sum_l g^{il}g_{lj}=\delta^i{}_j$, where $\delta^i{}_j$ is the
Kronecker delta.  Furthermore, given $f\in C^\infty(M)$, we have
$$df=\sum_i \Bigl({\partial\f\over\partial x^i}\Bigr)\,dx^i.$$
Therefore, from the definition of~$\grad\f$ above, we see that
$$\grad\f=\sum_{i,l} g^{il} \Bigl({\partial\f\over\partial
x^l}\Bigr)\, {\partial\over\partial x^i}.$$

\subsection{Affine connections}

Let $M$ be a differentiable manifold.  An {\it affine connection\/}
on~$M$ is a function~$\nabla$ which assigns to each vector field
$X\in\vf(M)$ an $\R$-linear map $\covD{X}\colon\vf(M)\to\vf(M)$ which
satisfies $$\conditions{(i)&$\covD{f\!X+gY}=f\covD{X}+g\covD{Y}$,\cr
(ii)&$\covD{X}(fY) =f\covD{X}Y+(X\f)Y$,\cr}$$ for~all $f\!$, $g\in
C^\infty(M)$, $X$, $Y\in\vf(M)$.  The map $\covD{X}$ may be applied to
tensors of arbitrary type.  Let $\nabla$ be an affine connection
on~$M$ and $X\in\vf(M)$.  Then there exists a unique $\R$-linear map
$A\mapsto\covD{X}A$ of $C^\infty$ tensor fields into $C^\infty$ tensor
fields which satisfies $$\conditions{(i)&$\covD{X}f=X\f\!$,\cr
(ii)&$\covD{X}Y$ is given by~$\nabla$,\cr (iii)&$\covD{X}(A\otimes B)
=\covD{X}A\otimes B +A\otimes\covD{X}B$,\cr (iv)&$\covD{X}$ preserves
the type of tensors,\cr (v)&$\covD{X}$ commutes with
contractions,\cr}$$ where $f\in C^\infty(M)$, $Y\in\vf(M)$, and $A$,
$B$ are $C^\infty$ tensor fields.  If $A$ is of type~$(k,l)$, then
$\covD{X}A$, called the {\it covariant derivative\/} of~$A$ along~$X$,
is of type~$(k,l)$, and $\D A\colon X\mapsto\covD{X}A$, called the
{\it covariant differential\/} of~$A$, is of type~$(k,l+1)$.

The expression of these ideas using coordinates is useful.  Let $M$ be
an $n\hyphen$dimensional differentiable manifold with affine
connection~$\nabla$, and $(U,x^1,\ldots,x^n)$ a coordinate chart
on~$M$.  These coordinates induce the canonical basis
$(\partial/\partial x^1)$, \dots, $(\partial/\partial x^n)$
of~$\vf(U)$.  There exist $n^3$ functions $\Gamma_{ij}^k$, $1\le
i,j,k\le n$, on~$U$ such that $$\covD{(\partial/\partial
x^i)}{\partial\over\partial x^j} = \sum_{i,j,k}\Gamma_{ij}^k\,
{\partial\over\partial x^k}.$$ The $\Gamma_{ij}^k$ are called the {\it
Christoffel symbols\/} of the connection.

The convergence proofs of later chapters require an analysis of the
second order terms of real-valued functions near critical points.
Consider the second covariant differential $\nabla\nabla\f=\nabla^2\f$
of a smooth function $f\colon M\to\R$.  If $(U,x^1,\ldots,x^n)$ is a
coordinate chart on~$M$, then this $(0,2)$ tensor takes the form
$$\nabla^2\f =\sum_{i,j} \biggl(\Bigl({\partial^2\f\over\partial
x^i\partial x^j}\Bigr) -\sum_k \Gamma_{ji}^k
\Bigl({\partial\f\over\partial x^k}\Bigr)\biggr)\, dx^i\otimes dx^j.$$

\subsection{Geodesics and parallelism}

Let $M$ be a differentiable manifold with affine connection~$\nabla$.
Let $\gamma\colon I\to M$ be a smooth curve with tangent vectors
$X(t)=\dot\gamma(t)$, where $I\subset\R$ is an open interval.  The
curve $\gamma$ is called a {\it geodesic\/} if $\covD{X}X=0$ for~all
$t\in I$.  Let $Y(t)\in T_{\gamma(t)}$ ($t\in I$) be a smooth family
of tangent vectors defined along $\gamma$.  The family $Y(t)$ is said
to be {\it parallel\/} along $\gamma$ if $\covD{X}Y=0$ for~all $t\in
I$.

For every $p$ in~$M$ and $X\ne0$ in~$T_p$, there exists a unique
geodesic $t\mapsto\geodesic{X}(t)$ such that $\geodesic{X}(0)=p$ and
$\dot\geodesic{X}(0)=X$.  We define the {\it exponential map\/}
$\exp_p\colon T_p\to M$ by $\exp_p(X)=\geodesic{X}(1)$ for~all $X\in
T_p$ such that $1$ is in the domain of~$\geodesic{X}$.  Oftentimes the
map $\exp_p$ will be denoted by ``$\exp$'' when the choice of tangent
plane is clear, and $\geodesic{X}(t)$ will be denoted by $\exp tX$.  A
neighborhood $N_p$ of~$p$ in~$M$ is a {\it normal neighborhood\/} if
$N_p=\exp N_0$, where $N_0$ is a star-shaped neighborhood of the
origin in~$T_p$ and $\exp$ maps $N_0$ diffeomorphically onto~$N_p$.
Normal neighborhoods always exist.

Given a curve $\gamma\colon I\to M$ such that $\gamma(0)=p$, for each
$Y\in T_p$ there exists a unique family $Y(t)\in T_{\gamma(t)}$ ($t\in
I$) of tangent vectors parallel along $\gamma$ such that $Y(0)=Y$.  If
$\gamma$ joins the points $p$ and~$\gamma(\alpha)=q$, the parallelism
along $\gamma$ induces an isomorphism $\tau_{pq}\colon T_p\to T_q$
defined by $\tau_{pq}Y=Y(\alpha)$.  If $\mu\in T_p^*$, define
$\tau_{pq}\mu\in T_q^*$ by the formula
$(\tau_{pq}\mu)(X)=\mu(\tau_{pq}^{-1}X)$ for~all $X\in T_q$.  The
isomorphism $\tau_{pq}$ can be extended in an obvious way to mixed
tensor products of arbitrary type.

Let $M$ be a manifold with an affine connection $\nabla$, and $N_p$ a
normal neighborhood of~$p\in M$.  Define the vector field $\Xtilde$
on~$N_p$ {\it adapted\/} to the tangent vector $X$ in~$T_p$ by putting
$\Xtilde_q=\tau_{pq}X$, the parallel translation of $X$ along the
unique geodesic segment joining $p$ and $q$.

Let $(U,x^1,\ldots,x^n)$ be a coordinate chart on an
$n\hyphen$dimensional differentiable manifold with affine
connection~$\nabla$.  Geodesics in~$U$ satisfy the $n$ second order
nonlinear differential equations $${d^2x^k\over dt^2} +\sum_{i,j}
{dx^i\over dt} {dx^j\over dt}\Gamma_{ij}^k=0.$$ For example, geodesics
on the imbedded $2$-sphere in~$\R^3$ with respect to the connection
given by $\Gamma_{ij}^k=\delta_{ij}x^k$ (the Levi-Civita connection on
the sphere), $1\le i,j,k\le3$, are segments of great circles, as shown
in Figure~\ref{fig:expmap}.  Let $t\mapsto\gamma(t)$ be a curve
in~$U$, and let $Y=\sum_k Y^k\,(\partial/\partial x^k)$ be a vector
field parallel along $\gamma$.  Then the functions $Y^k$ satisfy the
$n$ first order linear differential equations $${dY^k\over dt}
+\sum_{i,j} {dx^i\over dt} Y^j \Gamma_{ij}^k=0.$$ For example, if
$\gamma$ is a segment of a great circle on the sphere, then parallel
translation of vectors along $\gamma$ with respect to the connection
given by $\Gamma_{ij}^k=\delta_{ij}x^k$ is equivalent to rotating
tangent planes along the great circle.  The parallel translation of a
vector tangent to the north pole around a geodesic triangle on $S^2$
is illustrated in Figure~\ref{fig:part}. Note that the tangent vector
obtained by this process is different from the original tangent
vector.

\makeatletter
\def\addfig#1{\refstepcounter{figure}\let\@currentlabel=\thefigure
  \addcontentsline{lof}{figure}{\protect\numberline{Figure \thefigure}%
    {\ignorespaces #1}}}
\makeatother

\begin{figure}
\makeatletter
\ifch@pterhasfigures\else \global\ch@pterhasfigurestrue
  \addtocontents{lof}{\protect\addvspace{10pt}\protect
    \leftline{\bf Chapter \thechapter.}}\fi
\makeatother
\def\clap#1{\hbox to0pt{\hss#1\hss}}
\hbox to\textwidth{\hfil\dimen0=1.75in
\clap{\pdfximage width 1.75in {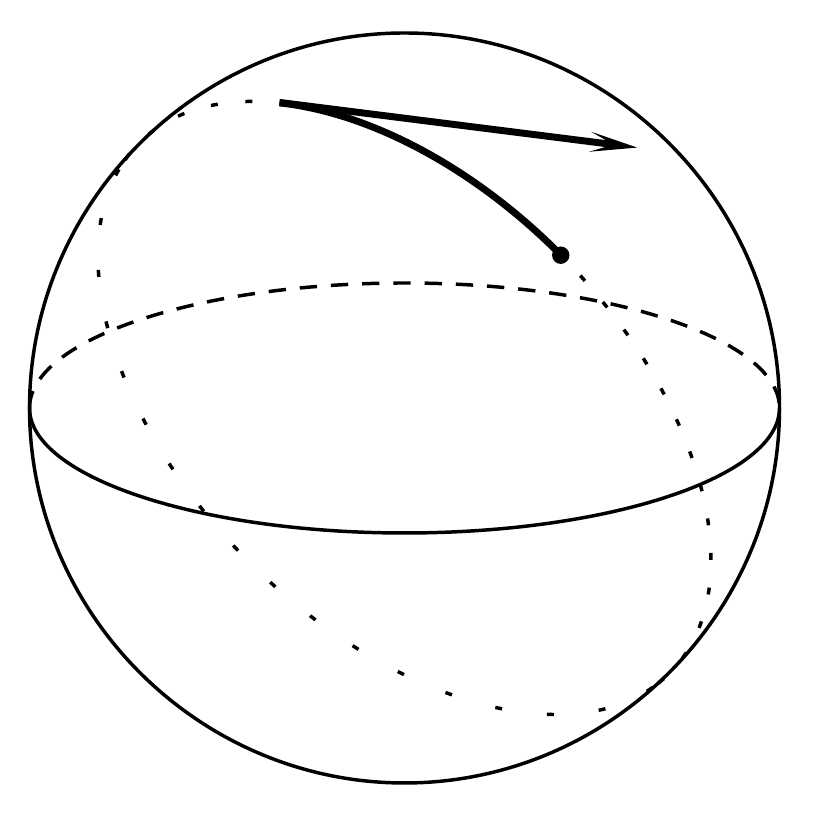}\pdfrefximage\pdflastximage}\hfil\hfil
\clap{\pdfximage width 1.75in {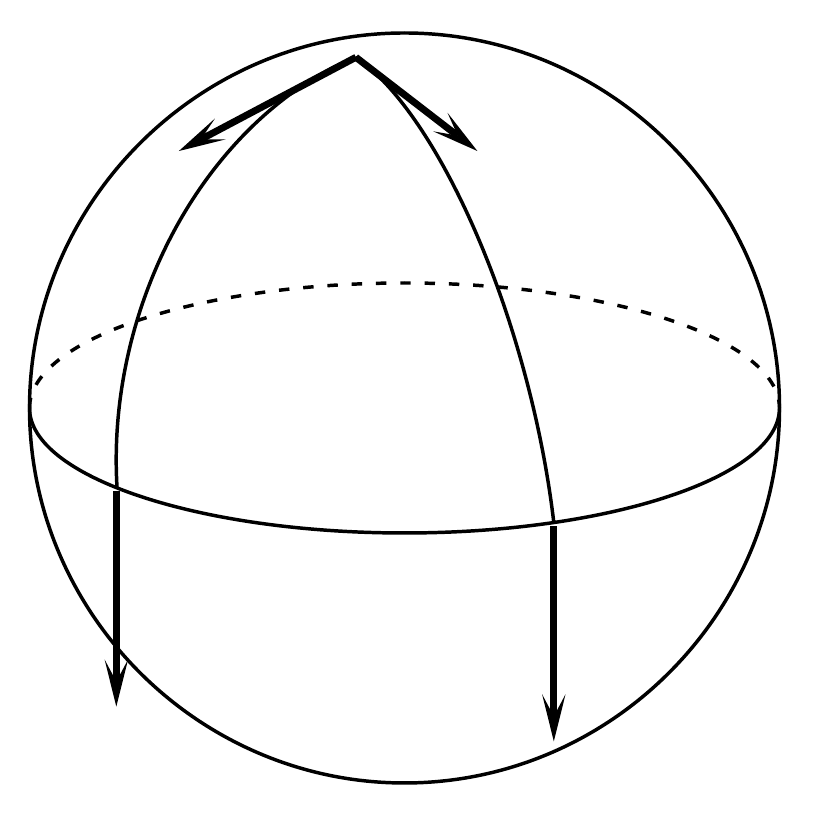}\pdfrefximage\pdflastximage}\hfil}
\vskip18pt
\hbox to\textwidth{\hfil
\clap{{\sc Figure} \addfig{Geodesic on a sphere}\thefigure.\
\small Geodesic on a sphere\label{fig:expmap}}\hfil\hfil
\clap{{\sc Figure} \addfig{Parallel translation on a sphere}\thefigure.\
\small Parallel translation on a sphere\label{fig:part}}\hfil}
\end{figure}


\begingroup \dimen0=2.5in
\hangafter=2\hangindent=\dimen0\indent Parallel translation
and covariant differentiation are related in the following way.  Let
$X$ be a vector field on~$M$, and $t\mapsto \gamma(t)$ an integral
curve of~$X$.  Denote the parallelism along~$\gamma$
from~$p=\gamma(0)$ to~$\gamma(h)$, $h$ small, by $\tau_h$. Then for an
arbitrary tensor field $A$ on~$M$, \begin{equation}
(\covD{X}A)_p=\lim_{h\to0}{1\over h} (\tau_h^{-1} A_{\gamma(h)}
-A_p). \end{equation} The covariant differentiation of a vector field
$Y$ along a vector field $X$ is illustrated in
Figure~\ref{fig:covd} at the left.\parfillskip=0pt\par
\nobreak\nointerlineskip
\leftline{\vbox to0pt{\vss
\begin{gnupicture}{\dimen0}\input covdiff
\end{gnupicture}
\hbox to\dimen0{\strut\hss{\sc Figure} 
  \addfig{Covariant differentiation}\thefigure.\label{fig:covd}\hss}}}
\endgroup

\subsection{Riemannian connections}

Given a Riemannian structure $g$ on a differentiable manifold $M$,
there exists a unique affine connection $\nabla$ on~$M$, called the
Riemannian or Levi-Civita connection, which for~all $X$, $Y\in\vf(M)$
satisfies $$\conditions{(i)&$\covD{X}Y-\covD{Y}X=[X,Y]$&($\nabla$ is
symmetric or torsion-free),\cr (ii)&$\nabla g=0$&(parallel translation
is an isometry).\cr}$$ Length minimizing curves on $M$ are geodesics
of the Levi-Civita connection.  We shall use this connection
throughout the thesis.  For every $p\in M$, there exists a normal
neighborhood $N_p=\exp N_0$ of~$p$ such that $d(p,\exp_pX)=\|X\|$
for~all $X\in N_0$, where $d$ is the Riemannian metric corresponding
to~$g$.

If $(U,x^1,\ldots,x^n)$ is a coordinate patch on~$M$, then the
Christoffel symbols $\Gamma_{ij}^k$ of the Levi-Civita connection are
related to the functions $g_{ij}$ by the formula
$$\Gamma_{ij}^k={1\over2}\sum_l g^{kl} \Bigl({\partial
g_{li}\over\partial x^j} -{\partial g_{ij}\over\partial x^l}
+{\partial g_{jl}\over\partial x^i} \Bigr).$$ By inspection it is seen
that $\Gamma_{ij}^k=\Gamma_{ji}^k$.

\section{Lie groups and homogeneous spaces}\label{sec:liegphomsp}

The basic structure of Lie groups and homogeneous spaces is reviewed
in this section, which follows Helgason's \citeyear{Helgason},
Warner's \citeyear{Warner}, Cheeger and Ebin's \citeyear{Cheegin}, and
Kobayashi and Nomizu's \citeyear[Chap.~10]{KobayashiandNomizu}
expositions.

\subsection{Lie groups}

\begin{definition} A {\it Lie group\/} $G$ is a differentiable
manifold and a group such that the map $G\times G\to G$ defined by
$(g,k)\mapsto gk^{-1}$ is $C^\infty$.
\end{definition}

The identity in~$G$ will be denoted by~$e$ in the general case, and
by~$I$ if $G$ is a matrix group.

\begin{definition} A {\it Lie algebra\/} $\g$ over~$\R$ is a vector
space over~$\R$ with a bilinear operation $[\,{,}\,]\colon
\g\times\g\to\g$ (called the {\it bracket\/}) such that for~all $x$,
$y$, $z\in\g$, $$\conditions{(i)&$[x,x]=0$&(implies anticommutivity),\cr
(ii)&$\bigl[x,[y,z]\bigr] +\bigl[y,[z,x]\bigr] +\bigl[z,[x,y]\bigr]
=0$&(Jacobi identity).\cr}$$ \end{definition}

Let $G$ be a Lie group and $g\in G$.  Left multiplication by~$g$ is
denoted by the map $l_g\colon G\to G$, $k\mapsto gk$, and similarly
for right multiplication $r_g\colon k\mapsto kg$.  Let $X$ be a vector
field on~$G$.  $X$ is said to be {\it left invariant\/} if for each
$g\in G$, $$l_g{}_*(X)=X\circ l_g.$$ The notation $f_*$ is used here
and elsewhere to denote $df$, the differential of a map~$f$.
Specifically, note that if $X$ is a left invariant vector field, then
$X_g=l_g{}_*X_e$, i.e., the value of~$X$ at any point $g\in G$ is
determined by its value at the identity~$e$.  Thus there is a
one-to-one correspondence between left invariant vector fields on~$G$
and tangent vectors in~$T_eG$.  Given a finite dimensional Lie
group~$G$, the vector space of left invariant vector fields on~$G$ or,
equivalently, the vector space $T_eG$, together with the Lie
derivative $L_XY=[X,Y]=XY-YX$ as the bracket operation yields a finite
dimensional Lie algebra~$\g$, in this thesis denoted by a lower-case
German letter.  We shall define $\g$ to be the vector space $T_eG$,
and for $X\in\g$, oftentimes denote the corresponding left invariant
vector field by~$\Xtilde$.

For every element $X$ in~$\g$, there is a unique homomorphism
$\phi\colon\R\to G$, called the {\it one-parameter subgroup\/} of~$G$
generated by~$X$, such that $\dot\phi(0)=X$.  Define the {\it
exponential map\/} $\exp\colon\g\to G$ by setting $\exp X=\phi(1)$.
The one-parameter subgroup $t\mapsto\phi(t)$ generated by~$X$ is
denoted by~$t\mapsto\exp tX$.  For matrix groups, the exponential map
corresponds to matrix exponentiation, i.e., $\exp tX
=e^{Xt}=I+tX+(t^2/2!)X^2+\cdots\,$.  It will be seen in the next
section in what sense the exponential map for a Lie group is related
to the exponential map for a manifold with an affine connection.

Let $G$ be a Lie group with Lie algebra $\g$.  Consider the action
of~$G$ on itself by conjugation, i.e., $a\colon(g,k)\mapsto gkg^{-1}$,
which has a fixed point at the identity. Denote the automorphism
$k\mapsto gkg^{-1}$ of~$G$ by~$a_g$.  Define the {\it adjoint
representation\/} $\Ad\colon G\to\Aut(\g)$ by the map
$g\mapsto(da_g)_e$, where $\Aut(\g)$ is the group of automorphisms of
the Lie algebra $\g$.  If $G$ is a matrix group with $g\in G$ and
$\omega\in\g$, we have $\Ad(g)(\omega)=g\omega g^{-1}$.  Furthermore,
we denote the differential of~$\Ad$ at the identity by~$\ad$, i.e.,
$$\ad=d\mkern-.6\thinmuskip\Ad_e$$ so that $\ad\colon\g\to\End(\g)$ is
a map from the Lie algebra $\g$ to its vector space of endomorphisms
$\End(\g)$. The notation $\Ad_g=\Ad(g)$ ($g\in G$) and $\ad_X=\ad(X)$
($X\in\g$) is often used.  It may be verified that $\ad_XY=[X,Y]$ for
$X$ and $Y$ in~$\g$.  If $G$ is a matrix group, then $\ad_XY=XY-YX$.
The functions $\Ad\colon G\to\Aut(\g)$ and $\ad\colon\g\to\End(\g)$
are related by $$\Ad\circ\exp=\exp\circ\ad,$$ i.e., for $X\in\g$,
$\Ad_{\exp X}=e^{\ad_X}$.

\begin{definition} Let $\g$ be a Lie algebra.  The {\it Killing
form\/} of~$\g$ is the bilinear form $\kf$ on~$\g\times\g$ defined by
$$\kf(X,Y)=\tr(\ad_X\circ\ad_Y).$$
\end{definition}

\subsection{Homogeneous spaces}

Let $G$ be a Lie group and $H$ a closed subgroup of~$G$.  Then the
(left) {\it coset space\/} $G/H=\{\,gH:g\in G\,\}$ admits the
structure of a differentiable manifold such that the natural
projection $\pi\colon G\to G/H$, $g\mapsto gH$, and the action of~$G$
on~$G/H$ defined by $(g,kH)\mapsto gkH$ are~$C^\infty$.  The dimension
of~$G/H$ is given by~$\dim G/H=\dim G-\dim H$.  Define the origin
of~$G/H$ by~$o=\pi(e)$.

\begin{definition} Let $G$ be a Lie group and $H$ a closed subgroup
of~$G$. The differentiable manifold $G/H$ is called a {\it homogeneous
space}.  \end{definition}

Let $\g$ and $\h$ be the Lie algebras of $G$ and $H$, respectively,
and let $\m$ be a vector subspace of~$\g$ such that $\g=\m+\h$ (direct
sum).  Then there exists a neighborhood of~$0\in\m$ which is mapped
homeomorphically onto a neighborhood of the origin~$o\in G/H$ by the
mapping $\pi\circ\exp|_\m$.  The tangent plane $T_o(G/H)$ at the
origin can be identified with the vector subspace~$\m$.

\begin{definition} A {\it Lie transformation group\/} $G$ acting on a
differentiable manifold $M$ is a Lie group $G$ which acts on~$M$ (on
the left) such that (i)~every element $g\in G$ induces a
diffeomorphism of~$M$ onto itself, denoted by $p\mapsto g\cdot p$ or
$p\mapsto l_g(p)$, (ii)~the map from~$G\times M$ to~$M$ defined
by~$(g,p)\mapsto g\cdot p$ is~$C^\infty$, and (iii)~$g\cdot(k\cdot p)=
gk\cdot p$ for $p\in M$, $g$, $k\in G$ (the action is transitive).
\end{definition}

For example, the Lie group $G$ is clearly a Lie transformation group
of the homogeneous space $G/H$.

The action of~$G$ on~$M$ is said to be {\it effective\/} if for any
$g\in G$, $l_g=\id$ on~$M$ implies that $g=e$.  Define the {\it
isotropy group\/} $H_p$ at~$p$ in~$M$ by $$H_p=\{\,g\in G:g\cdot
p=p\,\}.$$ The isotropy group at~$p$ is a closed subgroup of~$G$, and
the mapping $$g\cdot p\mapsto gH_p$$ of~$M$ onto~$G/H_p$ is a
diffeomorphism.  Therefore, we can identify $M$ with the homogeneous
space $G/H_p$.  Note that $H_p$ is not uniquely determined by~$M$, as
it may be replaced by~$H_{g\cdot p}=gH_pg^{-1}$ for any $g$ in~$G$.
The element $g$ in~$G$ is called a {\it coset representative\/} of the
point $g\cdot p$ in~$M$ and the point $gH$ in~$G/H$.  Every element
$h$ in~$H_p$ fixes~$p$, and therefore induces a linear transformation
$(dl_h)_p$ on the tangent plane $T_pM$.  The set $\tilde
H_p=\{\,(dl_h)_p:h\in H_p\,\}$ is called the {\it linear isotropy
group\/} at~$p$.

Let $G/H$ be a homogeneous space of~$G$, and $\pi\colon G\to G/H$ the
natural projection.  The tangent plane $T_o(G/H)$ at the origin
$o=\pi(e)$ may be identified with the quotient space $\g/\h$, because
for any function $f\in C^\infty(G/H)$, $$\bar f_*(\h)=0,$$ where $\bar
f$ is the unique lift in~$C^\infty(G)$ such that $\bar f=f\circ\pi$.
A tensor field~$A$ on~$G/H$ is $G$-invariant if and only if $A_o$ is
invariant under the linear isotropy group at~$o$, thus a computation
of the map $l_h{}_*\colon T_o\to T_o$ is desirable.  Let $\bar
l_g\colon G\to G$ and $l_g\colon G/H\to G/H$ denote left translation
by~$g\in G$.  Note that \begin{equation} l_g\circ\pi=\pi\circ\bar l_g,
\label{eq:pil=lpi} \end{equation} and for any $h\in H$, $g\in G$,
$\pi(hg)=\pi(hgh^{-1})$, i.e., \begin{equation} \pi\circ\bar
l_h=\pi\circ a_h, \label{eq:pil=pia} \end{equation} where $a_h$
denotes conjugation by~$h$.  Therefore, by applying
Equation~(\ref{eq:pil=lpi}) to Equation~(\ref{eq:pil=pia}) and
evaluating the differential of both sides at the identity~$e$, it is
seen that $$l_h{}_*\circ\pi_*=\pi_*\circ\Ad_h,$$ i.e., the action
of~$l_h{}_*$ on~$T_o$ corresponds to the action of~$\Ad_h$ on~$\g$,
which in turn corresponds to the action of~$\Ad_h$ on~$\g/\h$ because
$\h$ is $\Ad_H$-invariant.

Let $M$ be a differentiable manifold, and $G$ a Lie transformation
of~$M$.  To every $X\in\g$, there corresponds a unique vector
field~$\Xtilde$ on~$M$ defined by the equation \begin{equation}
\label{eq:homspvf} (\Xtilde\f)_p ={d\over dt}\Bigl|_{t=0} f(\exp
tX\cdot p) \end{equation} for $f\in C^\infty(M)$.  The vector field
$\Xtilde$ on~$M$ is said to be {\it induced\/} by the one-parameter
subgroup $\exp tX$.  For $p_0\in M$, $g\in G$, note that
\begin{equation} \Xtilde_{g\cdot p_0} =(l_g{}_*\circ \pi_*\circ
\Ad_{g^{-1}})(X), \label{eq:Xinduced} \end{equation} where $\pi$ is
the projection $g\mapsto g\cdot p_0$ from $G$ onto~$M$. Thus $\Xtilde$
is not left invariant in general. Equation~(\ref{eq:Xinduced}) will be
useful when specific vector fields on homogeneous spaces are
considered.  Furthermore, if $\Xtilde$ and $\Ytilde$ are vector fields
on~$M$ induced by $X$ and $Y$ in~$\g$, then $$[\Xtilde,\Ytilde]
=-\widetilde{[X,Y]}.$$

\begin{definition} Let $G$ be a connected Lie group, $H$ a closed
subgroup of~$G$, and $\g$ and $\h$ the Lie algebras of $G$ and $H$,
respectively.  The homogeneous space $G/H$ is said to be {\it
reductive\/} if there exists a vector subspace $\m$ of~$\g$ such that
$\g=\m+\h$ (direct sum), and $\m$ is $\Ad_H$-invariant, i.e.,
$\Ad_H(\m)\subset\m$.  \end{definition}

For example, the homogeneous space $G/H$ is reductive if $H$ is
compact.  Our interest in reductive homogeneous spaces lies solely
with this class of examples; for others, see \citeasnoun{Nomizu} or
\citeasnoun[Chap.~10]{KobayashiandNomizu}.

\subsection{Invariant affine connections}

\begin{definition} Let $G$ be a Lie transformation group acting on a
differentiable manifold $M$.  An affine connection $\nabla$ on~$M$ is
said to be {\it $G$-invariant\/} if for~all $g\in G$, $X$,
$Y\in\vf(M)$, $$l_g{}_*(\covD XY) =\covD{(l_g{}_*X)}(l_g{}_*Y).$$
\end{definition}

If $M=G$ is a Lie group with Lie algebra $\g$, we have the following
useful classification.  Let $\Xtilde$ and $\Ytilde$ be left invariant
vector fields on~$G$ corresponding to $X$ and $Y\in\g$, respectively.
There is a one-to-one correspondence between invariant affine
connections on~$G$ and the set of bilinear functions
$\alpha\colon\g\times\g\to\g$ given by the formula
$$\alpha(X,Y)=(\covD \Xtilde \Ytilde)_e.$$ Geodesics on~$G$ coincide
with one-parameter subgroups if and only if $\alpha(X,X)=0$ for~all
$X\in\g$.  The classical Cartan-Schouten invariant affine connections
on~$G$ correspond to $\alpha(X,Y)\equiv0$ (the $(-)\hyphen$connection),
$\alpha(X,Y)=\half[X,Y]$ (the $(0)\hyphen$connection), and
$\alpha(X,Y)=[X,Y]$ (the $(+)\hyphen$connection).

Let $G/H$ be a reductive homogeneous space with a fixed decomposition
of the Lie algebra $\g=\m+\h$, $\Ad_H(\m)\subset\m$, and $\pi\colon
G\to G/H$ the natural projection.  Any element $X\in\g$ can be
uniquely decomposed into the sum of elements in $\m$ and $\h$, which
will be denoted by $X_\m$ and $X_\h$, respectively.  There is a
one-to-one correspondence between invariant affine connections
on~$G/H$ and the set of bilinear functions
$\alpha\colon\m\times\m\to\m$ which are $\Ad_H$-invariant, i.e.,
$\Ad_h\cdot\alpha(X,Y)=\alpha(\Ad_hX,\Ad_hY)$ for~all $X$, $Y\in\m$,
$h\in H$.

Let $t\mapsto \exp tX$ be the one-parameter subgroup generated by~$X$
in~$\m$, and denote the curve $t\mapsto\pi(\exp tX)$ in~$G/H$ by
$t\mapsto\gamma_X(t)$.  In addition to the requirement that the
connection be complete, consider the following conditions on the
invariant affine connection on~$G/H$.

$(\rm a)$\enskip The curve $\gamma_X$ is a geodesic in~$G/H$.

$(\rm b)$\enskip Parallel translation of the tangent vector $Y\in T_o$
corresponding to $y\in\m$ along the curve $\gamma_X$ is given by the
differential of~$\exp tX$ acting on~$G/H$.

Nomizu \citeyear{Nomizu} established the following results concerning
invariant affine connections on reductive homogeneous spaces.  Recall
that the {\it torsion\/} of a connection $\nabla$ on a manifold~$M$ is
a tensor $T$ of type~$(1,2)$ defined by $T(X,Y)=\covD XY -\covD YX
-[X,Y]$, $X$,~$Y\in\vf(M)$.  The connection is said to be torsion-free
if $T\equiv0$.

\begin{theorem}[Nomizu] On a reductive homogeneous space $G/H$, there
exists a unique invariant connection which is torsion-free and
satisfies\/~$(\rm a)$.  It is defined by the function
$\alpha(X,Y)=\half[X,Y]_\m$ on~$\m\times\m$. \end{theorem}

This connection is called the {\it canonical torsion-free
connection\/} on~$G/H$ with respect to the fixed decomposition
$\g=\m+\h$.  In the case of a Lie group, it is the Cartan-Schouten
$(0)\hyphen$connection.

\begin{theorem}[Nomizu] On a reductive homogeneous space $G/H$, there
exists a unique invariant connection which satisfies\/~$(\rm b)$.  It
is defined by the function $\alpha(X,Y)\equiv0$ on~$\m\times\m$.
\end{theorem}

This connection is called the {\it canonical connection\/} on~$G/H$
with respect to the fixed decomposition $\g=\m+\h$.  In the case of a
Lie group, it is the Cartan-Schouten $(-)\hyphen$connection.

If $G/H$ is a {\it symmetric\/} homogeneous space, then these two
connections coincide. (See \citeasnoun{Helgason},
\citeasnoun{KobayashiandNomizu}, or \citeasnoun{Wolf} for background on
symmetric spaces.)  From the point of view of the applications we have
in mind, the choice of the canonical torsion-free connection on~$G/H$
facilitates the computation of geodesics on~$G/H$.  The choice of the
canonical connection on~$G/H$ facilitates the computation of parallel
translation along curves of the form $t\mapsto\pi(\exp tX)$.  In the
case of a symmetric space, the canonical connection allows both the
computation of geodesics and parallel translation along geodesics by
conditions $(\rm a)$ and $(\rm b)$ above.

\subsection{Invariant Riemannian metrics}

\begin{definition} Let $G$ be a Lie transformation group acting on a
differentiable manifold $M$.  A tensor field $A$ on~$M$ is said to be
{\it $G$-invariant\/} if for~all $g\in G$, $p\in M$, $$l_g{}_*(A)=
A\circ l_g.$$ \end{definition}

In particular, a Riemannian structure $g$ on $M$ is said to be (left)
{\it invariant\/} if it is $G$-invariant as a tensor field on~$G/H$.
That is, $$g_{k\cdot p}(l_k{}_*X,l_k{}_*Y)=g_p(X,Y)$$ for~all $p\in
M$, $k\in G$, $X$, $Y\in T_p$.  In the case of a Lie group $G$, a {\it
bi-invariant\/} metric on~$G$ is a Riemannian structure on~$G$ that is
invariant with respect to the left and right action of~$G$ on itself.

There may not exist an invariant Riemannian metric on the homogeneous
space $G/H$; however, \citeasnoun{Cheegin} provide a proposition that
describes the invariant metric structure of all homogeneous spaces
considered in this thesis.  The following proposition paraphrases
Proposition~3.16 of Cheeger and Ebin.

\begin{proposition}\label{prop:cheegin}\ignorespaces $(1)$\enspace The
set of\/~$G$-invariant metrics on $G/H$ is naturally isomorphic to the
set of bilinear forms\/ $\(\,{,}\,\)$ on\/~$\g/\h\times\g/\h$ which
are $\Ad_H$-invariant.  \par $(2)$\enskip If\/ $H$ is connected, the
bilinear form\/ $\(\,{,}\,\)$ on\/~$\g/\h\times\g/\h$ is
$\Ad_H$-invariant if and only if for~all $\eta\in\h$, $\ad_\eta$ is
skew-symmetric with respect to\/ $\(\,{,}\,\)$. \par $(3)$\enskip If\/
$G$ acts effectively on~$G/H$, then $G/H$ admits an invariant metric
if and only if the closure of the group $\Ad_H$ in~$\Aut(\g)$, the
group of automorphisms of\/~$\g$, is compact.  \par $(4)$\enskip If\/
$G$ acts effectively on~$G/H$ and $G/H$ is reductive with the fixed
decomposition\/ $\g=\m+\h$, then there is a one-to-one correspondence
between $G$-invariant metrics on~$G/H$ and $\Ad_H$-invariant bilinear
forms on\/~$\m\times\m$. If\/ $G/H$ admits a left invariant metric,
then $G$ admits a left invariant metric which is right invariant with
respect to~$H$; the restriction of this metric to~$H$ is bi-invariant.
\par Setting\/ $\m=\h^\perp$ provides such a decomposition. \par
$(5)$\enskip If\/ $H$ is connected, then the condition
$\Ad_H(\m)\subset\m$ is equivalent to\/ $[\h,\m]\subset\m$.
\end{proposition}

Let $G$ be a Lie group which admits a bi-invariant metric
$\(\,{,}\,\)$.  Then there is a corresponding left invariant metric,
called the {\it normal\/} metric, on the homogeneous space $G/H$ with
fixed decomposition $\g=\m+\h$, $\m=\h^\perp$, arising from the
restriction of~$\(\,{,}\,\)$ to~$\m$.  For example, let $G$ be a
compact semisimple Lie group with Lie algebra $\g$.  The Killing form
$\kf$ of~$\g$ is negative definite; therefore, $-\kf$ naturally defines an
invariant Riemannian metric on~$G$. The Levi-Civita connection of this
metric is the $(0)\hyphen$connection of~$G$.  Let $H$ be a closed subgroup
of~$G$ such that $G$ acts effectively on~$G/H$.  Then setting
$\m=\h^\perp$ with respect to~$-\kf$ yields a subspace $\m$ of~$\g$ such
that $\Ad_H(\m)\subset\m$, i.e., $G/H$ is a reductive homogeneous
space.  Furthermore, $-\kf$ restricted to~$\m$ yields an
$\Ad_H$-invariant bilinear form on~$\m\times\m$ and therefore yields
a left invariant Riemannian metric on~$G/H$.  The Levi-Civita
connection of this metric is the canonical torsion-free connection
on~$G/H$.

\subsection{Formulae for geodesics and parallel translation along geodesics}

Let $G$ be a Lie group with the $(0)\hyphen$connection, $\g$ the Lie
algebra of~$G$, and $X$ a left invariant vector field on~$G$
corresponding to~$x\in\g$.  Then the unique geodesic in~$G$ emanating
from $g$ in direction $X_g$ is given by the curve $t\mapsto\gamma(t)$,
where $$\gamma(t)=g\exp tx.$$ Let $t\mapsto\gamma_x(t)$ be the
geodesic in~$G$ emanating from the identity~$e$ with direction
$x\in\g$, and let $Y(t)$ be the parallel translation of~$y\in\g$
from~$e$ to~$\exp tX$ along~$\gamma_X$.  Then \begin{equation}
\label{eq:expliegp} Y(t)= l_{e^{xt}}{}_* \Ad_{e^{-{t\over2}x}} (y),
\end{equation} where $e^x=\exp x$.  For computations involving vector
fields on~$G$, it is oftentimes convenient to represent a tangent
vector $X_g$ in~$T_gG$ ($g\in G$) by a corresponding element $x_g$
in~$\g$ defined by the equation $X_g=l_g{}_*x_g$.  Letting $y(t)\in\g$
correspond to~$Y(t)\in T_{e^{xt}}G$ in this way, it is seen that
\begin{equation} y(t)=\Ad_{e^{-{t\over2}x}}(y). \label{eq:Liegppt}
\end{equation}

Let $G$ be a Lie group with bi-invariant metric $g$ also denoted by
$\(\,{,}\,\)$, and $G/H$ a reductive homogeneous space with the normal
metric and the fixed decomposition $\g=\m+\h$, $\m=\h^\perp$. Denote
the natural projection from~$G$ onto $G/H$ by~$\pi$ and let $o=\pi(e)$
be the origin in~$G/H$.  We wish to compute a formula for parallel
translation along geodesics in~$G/H$.  To do this, we view $G$ as
principal fiber bundle over~$G/H$ with structure group $H$, i.e., we
consider the fiber bundle $G(G/H,H)$ with its canonical torsion-free
connection \cite[Chap.~1, \S\thinspace5; Chap.~2]
{KobayashiandNomizu}.

For every element $x\in\m$, there is a unique element $X_o\in
T_o(G/H)$ given by Equation~(\ref{eq:homspvf}). For $t$ small enough,
define the vector field $X$ along the geodesic $t\mapsto\exp tx\cdot
o$ in~$G/H$ by setting $X_{e^{xt}}=l_{e^{xt}}{}_*X_o$. There is a
unique horizontal lift $\bar X\in\vf(G)$ of a smooth extension of~$X$.
Let $Y$ be a parallel vector field along the geodesic $t\mapsto
\exp(tx)\cdot o$ and denote $Y$ evaluated at the point $\exp(tx)\cdot
o$ by~$Y(t)$.  For each $t\in\R$, define $Y_o(t)\in T_o(G/H)$ and
$y(t)\in\m$ by the equation $$Y(t)=l_{e^{xt}}{}_*Y_o(t),$$ such that
$Y_o(t)$ corresponds to~$y(t)$. Let $\bar Y\in\vf(G)$ be the
horizontal lift a smooth extension of~$Y$, and let $\bar Z$ be the
horizontal lift in~$\vf(G)$ of a smooth extension of the vector field
$Z$ along $t\mapsto\exp tx\cdot o$ defined
by~$Z_{e^{xt}}=l_{e^{xt}}{}_*Z_o$, where $Z_o\in T_o(G/H)$ corresponds
to~$z\in\m$.  The projection onto the horizontal and vertical
components in~$\vf(G)$ will be denoted by superscript $H$ and $V$,
respectively, i.e., the vector field $A\in\vf(G)$ decomposes uniquely
as $A=A^H+A^V$.  At the identity, the horizontal and vertical
components of~$\vf(G)$ coincide with the $\m$-component and
$\h$-component of~$\g$, respectively.  The projection of~$\g$ onto
these components will be denoted by a subscript $\m$ and $\h$,
respectively, i.e., $x\in\g$ decomposes uniquely as $x=x_\m+x_\h$.

By the definition of the Levi-Civita connection, we have
\begin{equation} X\(Y,Z\)=\(\covD{X}Y,Z\) +\(Y,\covD{X}Z\).
\label{eq:lcdef} \end{equation} It is seen that \begin{equation}
X\(Y,Z\)=\(dy/dt,z\) \label{eq:lcdeflft} \end{equation} by the chain
of equalities $$\def\ddt{\mathop{(d/dt)}}
\eqalign{\bigl(X\(Y,Z\)\bigr)_{e^{xt}\cdot o} &=\ddt g_{e^{xt}\cdot
o}(l_{e^{xt}}{}_*Y_o(t),l_{e^{xt}}{}_*Z_o)\cr &=\ddt
g_o(Y_o(t),Z_o)\cr &=\ddt\(y(t),z\).\cr}$$ The vector field $Y$ is
parallel along the geodesic; therefore, \begin{equation} \covD{X}Y=0
\label{eq:lcdefz} \end{equation} by definition.  Computing the
rightmost term of Equation~(\ref{eq:lcdef}), we find that
\begin{eqnarray} \(Y,\covD{X}Z\) &=& \(\bar Y, \overline{\covD{X}Z}\)
\nonumber \\ &=& \bigl\(\bar Y, \covD{\bar X}\bar Z -\half[\bar X,
\bar Z]^V\bigr\) \qquad\hbox{(cf.\ Cheeger and Ebin, p.~67)} \nonumber
\\ &=& \bigl\(\bar Y, \half[\bar X, \bar Z] -\half[\bar X, \bar
Z]^V\bigr\) \nonumber \\ &=& \bigl\(\bar Y, \half[\bar X, \bar
Z]^H\bigr\) \nonumber \\ &=& \bigl\(y, \half[x,z]_\m\bigr\) \nonumber
\\ &=& \bigl\(-\half[x,y]_\m,z\bigr\).  \label{eq:lcdefrt}
\end{eqnarray} Combining Equations (\ref{eq:lcdef}),
(\ref{eq:lcdeflft}), (\ref{eq:lcdefz}), and (\ref{eq:lcdefrt}), we
have proved:

\begin{proposition}\label{prop:ptredhom}\ignorespaces Let $M=G/H$ be a
reductive homogeneous space which admits an invariant metric
$\(\,{,}\,\)$ and which has the fixed decomposition\/ $\g=\m+\h$,
$\m=\h^\perp$. Denote the origin of~$M$ by $o=\pi(e)$, where
$\pi\colon G\to M$ is the natural projection.  Let $x$ and $y_0$ be
vectors in\/~$\m$ corresponding to the tangent vectors $X$ and $Y_0$
in~$T_o(G/H)$.  The parallel translation $Y(t)$ of~$Y_0$ along the
geodesic $t\mapsto \exp_o tX =e^{xt}\cdot o$ in~$M$ is given in the
following way.  Define $Y_o(t)\in T_o(G/H)$ by the equation $Y(t)=
l_{e^{xt}}{}_*Y_o(t)$, and let $y(t)\in\m$ correspond to $Y_o(t)\in
T_o(G/H)$. The vector $y(t)$ satisfies the ordinary linear
differential equation \begin{equation} \label{eq:homspcpt} \dot
y=-\half[x,y]_\m;\qquad y(0)=y_0.  \end{equation} \end{proposition}

In the case of a Lie group $M=G$, we may take $\m=\g$; thus
Equation~(\ref{eq:homspcpt}) reduces to $$\dot y=-\half[x,y];\qquad
y(0)=y_0\in\g,$$ whose solution is given by
Equation~(\ref{eq:Liegppt}).  In the case where $G/H$ is a symmetric
space, the decomposition $\g=\m+\h$ satisfies the properties
$$[\m,\m]\subset\h,\qquad [\h,\m]\subset\m,\qquad [\h,\h]\subset\h.$$
Therefore $\dot y\equiv0$ because the $\m$-component of~$[x,y]$
vanishes.  Thus $y(t)\equiv y_0$.

\section{Examples}\label{sec:examples}

\subsection{The orthogonal group}

The {\it general linear group\/} $\GL(n)$ is the set of all real
$n\by n$ invertible matrices.  $\GL(n)$ is easily verified to be a
Lie group of dimension $n^2$ whose Lie algebra $\gl(n)$ is the vector
space of all $n\by n$ matrices with bracket $[X,Y]=XY-YX$.  The {\it
orthogonal group\/} $\O(n)$ is the subgroup of~$\GL(n)$ given by
$$\O(n)= \{\,\Theta\in\GL(n): \Theta^\T\Theta=I\,\}.$$ It is
well-known that $\O(n)$ is a Lie group of dimension $n(n-1)/2$ with
two connected components.  The identity component of~$\O(n)$ is called
the {\it special orthogonal group\/} $\SO(n)$, which is defined by
$$\SO(n)= \{\,\Theta\in\GL(n): \Theta^\T\Theta=I,\ \det\Theta=1\,\}.$$
The Lie algebra of~$\SO(n)$, denoted by~$\so(n)$, is the set of all
$n\by n$ skew-symmetric matrices, i.e., $$\so(n)
=\{\,\Omega\in\gl(n): \Omega+\Omega^\T =0\,\}.$$ The orthogonal group
is the group of all isometries of the vector space $\R^n$ (endowed
with the standard inner product $\(x,y\)=x^\T y=\sum_i x^iy^i$) which
fix the origin.  The special orthogonal group is the group of all
orientation preserving isometries of~$\R^n$ which fix the origin.

The orthogonal group $\O(n)$ is compact, and therefore admits a
bi-invariant metric, which is given by the negative of the Killing
form of~$\so(n)$.  A computation shows that for $X$, $Y\in\so(n)$,
\begin{equation} \label{eq:killingson} \kf(X,Y) =\tr(\ad_X\circ\ad_Y)
=(n-2)\tr XY, \end{equation} where the first trace is the trace of
endomorphisms of~$\so(n)$, and the second trace is the trace of $n\by
n$ matrices.  In case $n=2$, take the bilinear form $(X,Y)\mapsto\tr
XY$.  Furthermore, $\O(n)$ is semisimple, therefore $-\kf$ is positive
definite and thus defines a bi-invariant metric on~$\O(n)$.  This is
the natural bi-invariant metric on~$\O(n)$.

The Levi-Civita connection of this metric is the
$(0)\hyphen$connection.  The unique geodesic in~$\O(n)$ emanating from
the identity~$I$ in direction $X\in\so(n)$ is given by the formula
\begin{equation} \label{eq:geodesicson} t\mapsto e^{Xt},
\end{equation} where $e^{Xt}=I+tX+(t^2/2!)X^2+\cdots$ denotes matrix
exponentiation of~$X$.  Because the $(0)\hyphen$connection is
invariant, geodesics anywhere on~$\O(n)$ may be obtained by left
translation of geodesics emanating from the identity.  The parallel
translation $Y(t)$ of a tangent vector $Y_0\in\so(n)$ along the
geodesic $t\mapsto e^{Xt}$ is given by the formula \begin{equation}
Y_0(t)=e^{-{t\over2}X}Y_0e^{{t\over2}X}, \label{eq:SOnpt}
\end{equation} where $Y(t)\in T_{e^{xt}}$ corresponds to
$Y_0(t)\in\so(n)$ via left translation by~$e^{Xt}$, i.e.,
$Y(t)=e^{Xt}Y_0(t)$.  This formula may be used to compute the parallel
translation along any geodesic in~$\O(n)$ by the invariance of the
canonical connection.  Thus both geodesics in~$\O(n)$ and parallel
translation along geodesics in~$\O(n)$ may be computed via matrix
exponentiation of skew-symmetric matrices, for which there exist
stable efficient algorithms (Ward \& Gray 1978a, 1978b).

\subsection{The sphere}

Endow $\R^n$ with the standard inner product $\(x,y\)=x^\T y=\sum_i
x^iy^i$.  The $(n-1)$-sphere $S^{n-1}$ is an imbedded manifold
in~$\R^n$ defined by $$S^{n-1} =\{\,x\in\R^n: x^\T x=1\,\}.$$ The
standard inner product on~$\R^n$ induces a Riemannian metric
on~$S^{n-1}$.  As is well-known, geodesics on the sphere are great
circles and parallel translation along a geodesic is equivalent to
rotating the tangent plane along the corresponding great circle.
The tangent plane of the sphere at~$x$ in~$S^{n-1}$ is characterized
by $$T_xS^{n-1} =\{\,v\in\R^n: x^\T v=0\,\}.$$ Let $x\in S^{n-1}$, and
let $h\in T_x$ be any tangent vector at~$x$ having unit length, i.e.,
$h^\T h=1$, and $v\in T_x$ any tangent vector.  Then the unique
geodesic in~$S^{n-1}$ emanating from~$x$ in direction $h$, the
parallel translation of~$h$ along this geodesic, and the parallel
translation of~$v$ along this geodesic are given by the equations
$$\eqalign{\exp_x th &=x\cos t +h\sin t,\cr \tau h &=h\cos t -x\sin
t,\cr \tau v &=v-(h^\T v)\bigl(x\sin t +h(1-\cos t)\bigr),\cr}$$ where
$\tau$ is the parallelism along the geodesic $t\mapsto\exp th$.

The special orthogonal group $\SO(n)$ is a Lie transformation group of
the sphere $S^{n-1}$.  At any point on the sphere, say
$(1,0,\ldots,0)$, there is a closed subgroup $\SO(n-1)$ of~$\SO(n)$
that fixes this point. Therefore, we may make the identification
$$S^{n-1}\cong\SO(n)/\SO(n-1).$$ In fact, the homogeneous space
$\SO(n)/\SO(n-1)$ is a symmetric space.  We do not use the homogeneous
space structure of the sphere explicitly in this thesis, although the
sphere is a special case in the next example to be consider.  The
symmetric space structure of the sphere is described by
\citeasnoun[Chap.~11, \S\thinspace10]{KobayashiandNomizu}.

\subsection{The Stiefel manifold}

The compact {\it Stiefel manifold\/} $\stiefel(n,k)$ is defined to be
the set of all real $n\by k$ matrices, $k\le n$, with orthonormal
columns, i.e., $$\stiefel(n,k) =\{\,U\in\R^{n\times k}: U^\T
U=I\,\}.$$ Note that $\stiefel(n,n)=\O(n)$ and
$\stiefel(n,1)=S^{n-1}$.  The orthogonal group $\O(n)$ is naturally a
Lie transformation group of~$\stiefel(n,k)$ where the group action is
given by matrix multiplication on the left, i.e.,
$(\Theta,U)\mapsto\Theta U$.  Fix the origin $o=\bigl({I\atop0}\bigr)$
in~$\stiefel(n,k)$.  The isotropy group $H$ of this action at the
point~$o$ is the closed subgroup $$H=\left\{\,\left({I\atop0}\>{0\atop
h}\right)\in\SO(n): h\in\SO(n-k)\,\right\}.$$ Thus the Stiefel
manifold $\stiefel(n,k)$ may be identified with the homogeneous space
given by $$\stiefel(n,k)\cong \O(n)/\O(n-k),$$ which is a
differentiable manifold of dimension $k(k-1)/2+(n-k)k$.

For notational convenience, set $M=\stiefel(n,k)$, $G=\O(n)$, and
$H=\O(n-k)$ the isotropy group at~$o=\bigl({I\atop0}\bigr)$ in~$M$.
The Lie group $G$ has a bi-invariant metric, and acts transitively and
effectively on~$M$; therefore, the homogeneous space $G/H$ is
reductive with the fixed decomposition $\g=\m+\h$, where $$\g=\so(n)$$
is the Lie algebra of~$G$, $$\h=\left\{\,\left({0\atop0}\>{0\atop
\omega}\right)\in\so(n): \omega\in\so(n-k)\,\right\}$$ is the Lie
algebra of~$H$, and $\m=\h^\perp$ is the vector subspace
$$\m=\left\{\,\left({a\atop b}\>{-b^\T\atop0}\right)\in\so(n):
a\in\so(k)\,\right\}.$$ 

Let $H_p$ denote the isotropy group of an arbitrary point $p\in M$,
and let $g$ be a coset representative of~$p=g\cdot o$.  Then, as seen
above, $H_p=gH_og^{-1}$.  We identify tangent vectors in $T_pM$ with
elements of~$\m$ in the following way.  Let $\h_p$ denote the Lie
algebra of~$H_p$, and set $\m_p=\h_p^\perp$.  Then we have the
decomposition $\g=\m_p+\h_p$ (direct sum).  Clearly,
$$\m_p=\Ad_g(\m),\qquad \h_p=\Ad_g(\h).$$ An element $x$ in~$\m$
corresponds to an element $x_p$ in~$\m_p$ by the equation
$x_p=\Ad_g(x)$; the element $x_p$ induces a tangent vector $X$
in~$T_pM$ by the equation $X\f=(d/dt)_{t=0}f(e^{x_pt}\cdot p)$ for any
$f$ in~$C^\infty(M)$.  Combining these ideas, it is seen that $X$ is
defined by $$X\f={d\over dt}\Big|_{t=0} f(e^{x_pt}g\cdot o) ={d\over
dt}\Big|_{t=0} f(ge^{xt}\cdot o).$$ It is important to note that this
identification of elements $x\in\m$ with tangent vectors $X\in T_pM$
depends upon the choice of coset representative $g$.  The reason for
making this identification will be clear when we consider in
Chapter~\ref{chap:af}, Section~\ref{sec:cg-genray}, the computational
aspects of computing geodesics in $\stiefel(n,k)$.

The negative of the Killing form of~$\g$ restricted to~$\m$ yields an
invariant Riemannian metric on~$M$. The Levi-Civita connection of this
metric coincides with the canonical torsion-free affine connection
of~$G/H$.  Let $p$ be a point in~$M$, $g$ a coset representative
of~$p$ such that $p=g\cdot o$, and $X$ a tangent vector in~$T_pM$
corresponding to the element $x$ in~$\m$ as described in the preceding
paragraph.  Then the unique geodesic emanating from~$p$ in direction
$X$ is given by $$t\mapsto ge^{xt}\cdot o.$$ Thus geodesics
in~$\stiefel(n,k)$ may be computed by matrix exponentiation of
elements in~$\m$.  However, the Stiefel manifold is not a symmetric
space, so parallel translation along geodesics may not be computed as
easily as in the previous examples. Indeed, partition any element $x$
in~$\m$ as $$x=\pmatrix{x_1&-x_2^\T\cr x_2&0\cr},\qquad\hbox{$x_1$
in~$\so(k)$.}$$ The parallel translation of a tangent vector in $T_oM$
corresponding to $y_0\in\m$ along the geodesic $t\mapsto e^{xt}\cdot
o$ is given by Equation~(\ref{eq:homspcpt}).  In the case of the
Stiefel manifold, and after rescaling the parameter $t$ by~$-1/2$,
this equation becomes the pair coupled linear differential equations
\begin{equation} \label{eq:stiefelpt} \diffeqalign{\dot y_1 &=
[x_1,y_1] +y_2^\T x_2 -x_2^\T y_2 &y_1(0)=\hbox{given},\cr \dot y_2
&=x_2y_1-y_2x_1 & y_2(0)=\hbox{given}.\cr} \end{equation}

In the case $k=n$, i.e., $\stiefel(n,n)=\O(n)$, the linear operator
$y\mapsto [x,y]$ of~$\so(n)$ onto itself has eigenvalues
$\lambda_i-\lambda_j$, $1\le i,j\le n$, where the $\lambda_i$ are the
eigenvalues of the skew-symmetric matrix $x$. Thus the differential
equation in~(\ref{eq:stiefelpt}) has the relatively simple solution
given by Equation~(\ref{eq:SOnpt}). In the case $k=1$, i.e.,
$\stiefel(n,1)=S^{n-1}$, the linear operator $y\mapsto [x,y]_\m$
of~$\m$ onto itself is identically zero, thus the differential
equation of~(\ref{eq:stiefelpt}) also has a simple solution.  In all
other cases where $\stiefel(n,k)$ is not a symmetric case, i.e., $k\ne
n$ or~$1$, the solution to the differential equation of
(\ref{eq:stiefelpt}) may be obtained by exponentiating the linear
operator $y\mapsto[x,y]_\m$, which is skew-symmetric with respect to
the Killing form of~$\g$ restricted to~$\m$.  However, this
exponentiation corresponds to the problem of computing the matrix
exponential of a $\bigl(k(k-1)/2 +(n-k)k\bigr)\by\bigl(k(k-1)/2
+(n-k)k\bigr)$ skew-symmetric matrix, which is computationally much
more expensive than computing the matrix exponential of an $n\by n$
skew-symmetric matrix as in the case $\stiefel(n,n)=\O(n)$.

%% file: covdiff.tex
\setlength{\unitlength}{0.240900pt}
\ifx\plotpoint\undefined\newsavebox{\plotpoint}\fi
\begin{picture}(750,225)(60,100)
\tenrm
\put(264,124){\makebox(0,0){$\scriptstyle\bullet$}}
\put(264,104){\makebox(0,0)[t]{$p$}}
\put(264,124){\vector(4,1){200}}
\put(400,200){\makebox(0,0)[l]{$X_p$}}
\put(264,124){\vector(-1,2){100}}
\put(210,210){\makebox(0,0)[r]{$Y_p$}}
\put(264,124){\vector(1,4){60}}
\put(320,290){\makebox(0,0)[l]{$\tau_h^{-1}Y_{\gamma(h)}$}}
\put(164,324){\vector(4,1){160}}
\put(300,390){\makebox(0,0)[r]{$h(\covD XY)_p$}}
\put(264,124){\rule[-0.175pt]{0.964pt}{0.350pt}}
\put(268,125){\rule[-0.175pt]{1.204pt}{0.350pt}}
\put(273,126){\rule[-0.175pt]{0.964pt}{0.350pt}}
\put(277,127){\rule[-0.175pt]{0.964pt}{0.350pt}}
\put(281,128){\rule[-0.175pt]{0.964pt}{0.350pt}}
\put(285,129){\rule[-0.175pt]{1.204pt}{0.350pt}}
\put(290,130){\rule[-0.175pt]{0.964pt}{0.350pt}}
\put(294,131){\rule[-0.175pt]{0.964pt}{0.350pt}}
\put(298,132){\rule[-0.175pt]{0.964pt}{0.350pt}}
\put(302,133){\rule[-0.175pt]{1.204pt}{0.350pt}}
\put(307,134){\rule[-0.175pt]{0.964pt}{0.350pt}}
\put(311,135){\rule[-0.175pt]{0.964pt}{0.350pt}}
\put(315,136){\rule[-0.175pt]{0.964pt}{0.350pt}}
\put(319,137){\rule[-0.175pt]{1.204pt}{0.350pt}}
\put(324,138){\rule[-0.175pt]{1.927pt}{0.350pt}}
\put(332,139){\rule[-0.175pt]{0.964pt}{0.350pt}}
\put(336,140){\rule[-0.175pt]{1.204pt}{0.350pt}}
\put(341,141){\rule[-0.175pt]{0.964pt}{0.350pt}}
\put(345,142){\rule[-0.175pt]{2.168pt}{0.350pt}}
\put(354,143){\rule[-0.175pt]{0.964pt}{0.350pt}}
\put(358,144){\rule[-0.175pt]{1.927pt}{0.350pt}}
\put(366,145){\rule[-0.175pt]{1.204pt}{0.350pt}}
\put(371,146){\rule[-0.175pt]{1.927pt}{0.350pt}}
\put(379,147){\rule[-0.175pt]{2.168pt}{0.350pt}}
\put(388,148){\rule[-0.175pt]{1.927pt}{0.350pt}}
\put(396,149){\rule[-0.175pt]{3.132pt}{0.350pt}}
\put(409,150){\rule[-0.175pt]{1.927pt}{0.350pt}}
\put(417,151){\rule[-0.175pt]{4.336pt}{0.350pt}}
\put(435,152){\rule[-0.175pt]{16.381pt}{0.350pt}}
\put(503,151){\rule[-0.175pt]{2.891pt}{0.350pt}}
\put(515,150){\rule[-0.175pt]{3.132pt}{0.350pt}}
\put(528,149){\rule[-0.175pt]{2.168pt}{0.350pt}}
\put(537,148){\rule[-0.175pt]{1.927pt}{0.350pt}}
\put(545,147){\rule[-0.175pt]{2.168pt}{0.350pt}}
\put(554,146){\rule[-0.175pt]{1.927pt}{0.350pt}}
\put(562,145){\rule[-0.175pt]{1.204pt}{0.350pt}}
\put(567,144){\rule[-0.175pt]{1.927pt}{0.350pt}}
\put(575,143){\rule[-0.175pt]{0.964pt}{0.350pt}}
\put(579,142){\rule[-0.175pt]{1.204pt}{0.350pt}}
\put(584,141){\rule[-0.175pt]{0.964pt}{0.350pt}}
\put(588,140){\rule[-0.175pt]{1.927pt}{0.350pt}}
\put(596,139){\rule[-0.175pt]{1.204pt}{0.350pt}}
\put(601,138){\rule[-0.175pt]{0.964pt}{0.350pt}}
\put(605,137){\rule[-0.175pt]{0.964pt}{0.350pt}}
\put(609,136){\rule[-0.175pt]{1.204pt}{0.350pt}}
\put(614,135){\rule[-0.175pt]{0.964pt}{0.350pt}}
\put(618,134){\rule[-0.175pt]{0.964pt}{0.350pt}}
\put(622,133){\rule[-0.175pt]{0.482pt}{0.350pt}}
\put(624,132){\rule[-0.175pt]{0.482pt}{0.350pt}}
\put(626,131){\rule[-0.175pt]{1.204pt}{0.350pt}}
\put(631,130){\rule[-0.175pt]{0.964pt}{0.350pt}}
\put(635,129){\rule[-0.175pt]{0.964pt}{0.350pt}}
\put(639,128){\rule[-0.175pt]{0.964pt}{0.350pt}}
\put(643,127){\rule[-0.175pt]{0.602pt}{0.350pt}}
\put(645,126){\rule[-0.175pt]{0.602pt}{0.350pt}}
\put(645,126){\makebox(0,0){$\scriptstyle\bullet$}}
\put(645,106){\makebox(0,0)[t]{$\gamma(h)$}}
\put(645,126){\vector(1,4){60}}
\put(660,240){\makebox(0,0)[r]{$Y_{\gamma(h)}$}}
\put(648,125){\rule[-0.175pt]{0.964pt}{0.350pt}}
\put(652,124){\rule[-0.175pt]{0.964pt}{0.350pt}}
\put(656,123){\rule[-0.175pt]{0.482pt}{0.350pt}}
\put(658,122){\rule[-0.175pt]{0.482pt}{0.350pt}}
\put(660,121){\rule[-0.175pt]{1.204pt}{0.350pt}}
\put(665,120){\rule[-0.175pt]{0.482pt}{0.350pt}}
\put(667,119){\rule[-0.175pt]{0.482pt}{0.350pt}}
\put(669,118){\rule[-0.175pt]{0.964pt}{0.350pt}}
\put(673,117){\rule[-0.175pt]{0.482pt}{0.350pt}}
\put(675,116){\rule[-0.175pt]{0.482pt}{0.350pt}}
\put(677,115){\rule[-0.175pt]{1.204pt}{0.350pt}}
\put(682,114){\rule[-0.175pt]{0.482pt}{0.350pt}}
\put(684,113){\rule[-0.175pt]{0.482pt}{0.350pt}}
\end{picture}

%% file: chap-gradflow.tex
\chapter{Gradient flows on Lie groups and homogeneous
spaces}\label{chap:grad}

To develop a theory of optimization on smooth manifolds, it is natural
to begin with a study of gradient flows, which provide local
information about the direction of greatest increase or decrease of a
real-valued function defined on the manifold.  The study of gradient
flows is also desirable from the perspective of applications because
we will later apply optimization theory to the problem of principal
component analysis, which may be expressed as a smooth optimization
problem.  This approach has received wide attention in the fields of
adaptive signal processing
\cite{WidrowStearns,Schmidt,RoyKailath,Larimore,Fuhrmann} and neural
networks (Oja 1982, 1989; Bourland \& Kamp 1988; Baldi \& Hornik 1989;
Rubner \& Tavan 1989; Rubner \& Schulten 1990), where the problem of
tracking a principal invariant subspace is encountered
\cite{Brockett:subspace}.

Let $M$ be a Riemannian manifold with Riemannian structure~$g$, and
$f\colon M\to\R$ a smooth function on~$M$.  Then the gradient of~$f$,
denoted by~$\grad\f$, is a smooth vector field on~$M$ and the
one-parameter groups of diffeomorphisms generated by~$\grad\f$ are
called the gradient flows of~$f$.  In this chapter we will consider a
variety of problems whose solutions correspond to the stable critical
points of the gradient of a function, i.e., the problems will be
restated as local optimization problems on a manifold.  These
optimization problems will then be solved by computing an integral
curve of the gradient.  As our concern will be principal component
analysis, we shall consider the algebraic task of computing the
eigenvalues and eigenvectors of a symmetric matrix, and the singular
values and singular vectors of an arbitrary matrix.  Of course,
efficient algorithms already exist to solve these eigenvalue problems
and the methods described within this chapter---integrating
differential equations on Lie groups and homogeneous spaces---are not
in the least way competetive with standard techniques.  Our interest
in gradient flows to solve the problems in numerical linear algebra
arises in part from the intent to illuminate and provide a framework
for the practical large step optimization algorithms that will appear
in Chapter~\ref{chap:orm}.  There is also a general interest in
studying the class of problems that may be solved via dynamical
systems \cite{Brockett:subspace,Leonid,Chu:grad}.

From the perspective of optimization theory, there is a very natural
setting for the symmetric eigenvalue problem and the singular value
problem.  Indeed, finding the eigenvalues of a symmetric matrix may be
posed as an optimization problem \cite{Wilkinson,GVL}.  Let $Q$ be an
$n\by n$ symmetric matrix.  The largest (smallest) eigenvalue of~$Q$
is the maximum (resp., minimum) value taken by the Rayleigh quotient
$x^\T Qx/x^\T x$ over all vectors $x\ne0$ in~$\R^n$.  The
Courant-Fisher minimax characterization describes the general case.
Denote the $k$th largest eigenvalue of~$Q$ by~$\lambda_k$, and let
$S\subset\R^n$ be a vector subspace.  Then for $k=1$, \dots,~$n$,
$$\lambda_k=\max_{\dim S=k\vphantom\backslash}\; \min_{x\in
S\backslash\{0\}} {x^\T Qx\over x^\T x}.$$ The situation for the
singular value problem is similar.  Let $K$ be an $m\by n$ matrix,
$S\subset\R^n$ and $T\subset\R^m$ vector subspaces, and denote the
$k$th largest singular value of~$K$ by~$\sigma_k$.  Then by
Theorem~8.3-1 of \citeasnoun{GVL}, for $k=1$, \dots,~$\min(m,n)$,
$$\sigma_k=\max_{\dim S=k\vphantom\backslash \atop \dim
T=k\vphantom\backslash}\; \min_{x\in S\backslash\{0\}\atop y\in
T\backslash\{0\}} {y^\T\!Ax\over \|x\|\>\|y\|} =\max_{\dim
S=k\vphantom\backslash}\; \min_{x\in S\backslash\{0\}} {\|Ax\|\over
\|x\|}.$$ Several practical algorithms for the eigenvalue problem,
specifically Jacobi methods and Lanczos methods, can be developed on
the basis of such optimization requirements.  Thus we see that the
eigenvalue problem and singular value problems can be viewed as
optimization problems on the manifold of $k$-planes in $\R^n$, i.e.,
the Grassmann manifold $\grassmann(n,k)$.  Although this particular
minimax characterization and manifold will not be used within this
chapter, several equivalent optimization problems will be
investigated.

\section{The diagonalization of a matrix}\label{sec:doublebracket}

This section briefly describes pertinent elements of the work of
Brockett \citeyear{Brockett:match,Brockett:sort}, who provides a
gradient flow on the special orthogonal group, or under a change of
variables, on the space of symmetric matrices with fixed spectrum.
This material is covered to motivate some contributions of this thesis
that will appear in subsequent sections, and to illustrate some
techniques that will be used throughout the thesis.

In the investigation of some least squares matching problems in
computer vision, Brockett \citeyear{Brockett:match} considers the
function $f\colon\Theta\mapsto\tr\Theta^\T Q\Theta N$ on the special
orthogonal group $\SO(n)$, where $Q$ is a fixed real symmetric matrix
and $N$ is a real diagonal matrix with distinct diagonal elements.  We
wish to compute the gradient flow of~$f$, which will lead to the
eigenvalue decomposition of~$Q$.  Using the definition of the
differential, we have $df_\Theta(\Omegatilde) =(d/dt)|_{t=0}
f\bigl(c(t)\bigr)$, where $c\colon \SO(n)\to\R$ is any smooth curve
such that $c(0)=\Theta$ and $\dot c(0)=\Omegatilde\in T_\Theta\SO(n)$.
As shown in Chapter~\ref{chap:geom}, the unique geodesic in $\SO(n)$
through $\Theta$ with direction $\Omegatilde\in T_\Theta$ is given by
$\exp t\Omegatilde= \Theta e^{\Omega t}$, where $\Omega$ is the unique
vector in~$\so(n)$ determined by the equation
$\Omegatilde=L_{\Theta*}\Omega$.

Therefore, taking $c(t)=\Theta e^{\Omega t}$ and setting $H=\Theta^\T
Q\Theta$, we have $$\eqalign{df_\Theta(\Omegatilde) &= \ddt f(\Theta
e^{\Omega t})\cr &=\ddt \tr(\Theta e^{\Omega t})^\T Q(\Theta e^{\Omega
t})N\cr &=\ddt \tr\Ad_{e^{-\Omega t}}(H)N\cr &=-\tr[\Omega,H]N\cr
&=-\tr\Omega[H,N]\cr &=\bigl\(\Omega,[H,N]\bigr\).\cr}$$ The expansion
$\Ad_{e^{Xt}}(Y) =e^{t\ad_X}\cdot Y=Y +t\ad_XY
+(t^2/2!)\ad^2_XY+\cdots$ and the identity $\tr ABC=\tr BCA=\tr CAB$
are used in this chain of equalities.  Equivalently, we may also use
the fact that with respect to the Killing form on $\so(n)$,
$\(\ad_xy,z\)=-\(y,\ad_xz\)$, following Brockett
\citeyear{Brockett:grad}.  From the definition of the gradient, i.e.\
$df_p(X)=\((\grad\f)_p,X\)$ for~all $X\in T_p$, we see that with
respect to the natural invariant metric on~$\SO(n)$ the gradient
of~$f$ is given by \begin{equation} (\grad\f)_\Theta=\Theta[\Theta^\T
Q\Theta, N]. \label{eq:SOgrad} \end{equation}

Let $\Sym(\lambda)$ denote the set of real symmetric matrices with the
fixed set of eigenvalues ${\bmit\lambda} =\{\lambda_1,\ldots,
\lambda_n\}$.  If the eigenvalues are distinct, then this set is a
$C^\infty$ differentiable manifold of dimension $n(n-1)/2$.  To see
why this is so, observe that the Lie group $\SO(n)$ acts effectively
and transitively on $\Sym(\lambda)$ by the action $(\theta,s)\mapsto
\theta s\theta^{-1}$.  If the eigenvalues
${\bmit\lambda}=\{\lambda_1,\ldots,\lambda_n\}$ are distinct, the
isotropy group of this action at the point
$\diag(\lambda_1,\ldots,\lambda_n)$ is the discrete subgroup
$\diag(\pm1,\ldots,\pm1)$, which we denote by~$D$.  Therefore, we may
make the natural identification $\Sym(\lambda)\cong\SO(n)/D$.  Thus
the manifold $\Sym(\lambda)$ inherits a Riemannian structure from the
natural invariant structure on $\SO(n)$.

The so-called double bracket equation, also known as Brockett's
equation, can be obtained from Equation~(\ref{eq:SOgrad}) by making
the change of variables \begin{equation} H=\Theta^\T Q\Theta.
\label{eq:ThetatoH} \end{equation} Differentiating both sides of
Equation~(\ref{eq:ThetatoH}) and rearranging terms yields the
isospectral flow \begin{equation} \dot H=[H,[H,N]].
\label{eq:doublebracket} \end{equation} Remarkably,
Equation~(\ref{eq:doublebracket}) is equivalent to a Toda flow in the
case where $H$ is tridiagonal and $N=\diag(1,\ldots,n)$ (Bloch 1990;
Bloch et~al.\ 1990, 1992); therefore, it is an example of a flow that
is both Hamiltonian and gradient.

The fixed points of Equations (\ref{eq:SOgrad}) and
(\ref{eq:doublebracket}) may be computed in a straightforward way.
Consider the function $H\mapsto \tr HN$ on the set of real symmetric
matrices with fixed spectrum, where $N$ is a real diagonal matrix with
distinct diagonal entries.  Computing as above, we see that
$$\eqalign{{d\over dt}\tr HN&= \tr [H,[H,N]]N\cr &= -\tr [H,N]^2\cr
&=\bigl\|[H,N]\bigr\|^2.}$$ This derivative is nonnegative and bounded
from above because the set $\Sym(\lambda)$ is compact. Therefore, $\tr
HN$ has a limit and its derivative approaches zero as $$[H,N]\to0.$$
In the limit, this becomes $$h_{ij}(n_{jj}-n_{ii})=0.$$ Therefore, $H$
approaches a diagonal matrix with the prescribed eigenvalues along its
diagonal, i.e., $H=\diag(\lambda_{\pi(1)},\ldots,\lambda_{\pi(n)})$
for some permutation $\pi$ of the integers $1$, \dots,~$n$.

Inspecting the second order terms of $\tr HN$ at a critical point
$H=\diag(\lambda_{\pi(1)},\ldots,\penalty1000\lambda_{\pi(n)})$ will
show which of these $n!$ points are asymptotically stable.  Let $H$ be
the parameterized matrix $(\Theta e^{\Omega\epsilon})^\T Q(\Theta
e^{\Omega\epsilon})$, where $\Theta^\T Q\Theta
=\diag(\lambda_{\pi(1)},\ldots,\lambda_{\pi(n)})$ and
$\Omega\in\so(n)$.  The second order terms of $\tr HN$ are
\begin{equation} -\sum_{1\le i<j\le n} (n_{ii}-n_{jj})
(\lambda_{\pi(i)} -\lambda_{\pi(j)}) (\epsilon\omega_{ij})^2.
\label{eq:d2bracket} \end{equation} This quadratic form is negative
(positive) definite if and only if the sets $\{\lambda_i\}$ and
$\{n_{ii}\}$ are similarly (resp., oppositely) ordered.  Therefore, of
the $n!$ critical points of Equation~(\ref{eq:doublebracket}), one is
a sink, one is a source, and the remainder are saddle points.  Of the
$2^nn!$ critical points of Equation~(\ref{eq:SOgrad}), $2^n$ are
sinks, $2^n$ are sources, and the remainder are saddle points.

Equations (\ref{eq:SOgrad}) and (\ref{eq:doublebracket}) play a role
in the study of interior point methods for linear programming
\cite{Leonid} and the study of continuous versions of the |QR|
algorithm \cite{Lagarias,WatkinsElsner:laa}, but this work will not be
discussed here.

\section{The extreme eigenvalues of a matrix}\label{sec:extremeig}

In the previous section an optimization problem was considered whose
solution corresponds to the complete eigenvalue decomposition of a
symmetric matrix.  However, oftentimes only a few eigenvalues and
eigenvectors are required. If given an $n\by n$ symmetric matrix~$Q$
with distinct eigenvalues, the closest rank~$k$ symmetric matrix is
desired, this is determined by the sum $\sum
\lambda_ix_i^{\vphantom{\T}}x_i^\T$, $i=1$, \dots,~$k$, where
$\lambda_i$ is the $i$th largest eigenvalue of~$Q$ and $x_i$ is the
corresponding eigenvector.  Some signal processing applications
\cite{BienvenuKopp,Larimore,RoyKailath} require knowledge of the
smallest eigenvalues and corresponding eigenvectors to estimate
signals in the presence of noise.  In this section we will consider a
function whose gradient flow yields the eigenvectors corresponding to
the extreme eigenvalues of a given matrix.

\subsection{The generalized Rayleigh quotient}

Consider the compact Stiefel manifold $\stiefel(n,k)$ of real
$n\by k$ matrices, $k\le n$, with orthonormal columns.  As discussed
in Chapter~\ref{chap:geom}, Section~\ref{sec:examples},
$\stiefel(n,k)$ may be identified with the reductive homogeneous space
$\O(n)/\O(n-k)$ of dimension $k(k-1)/2+(n-k)k$.  Let $G=\O(n)$,
$o=\bigl({I\atop0}\bigr)$ the origin of~$\stiefel(n,k)$, $H=\O(n-k)$
the isotropy group at~$o$, $\g$ and $\h$ the Lie algebra of $G$ and
$H$, respectively. Set $M=\stiefel(n,k)$.  There is a subspace $\m$
of~$\g$ such that $\g=\m+\h$ (direct sum) and $\Ad_H(\m)=\m$ obtained
by choosing $\m=\h^\perp$ with respect to the Killing form of~$\g$.
The tangent plane $T_oM$ is identified with the subspace $\m$ in the
standard way.  Let $g$ be a coset representative
of~$p\in\stiefel(n,k)$, i.e., $p=g\cdot o$. Tangent vectors in $T_pM$
will be represented by vectors in~$\m$ via the correspondence
described in Chapter~\ref{chap:geom}, Section~\ref{sec:examples}.
The reductive homogeneous space structure of~$\stiefel(n,k)$ will be
exploited in this section to describe the gradient flow of a function
defined on~$\stiefel(n,k)$ and will be especially important in later
chapters when efficient algorithms for computing a few extreme
eigenvalues of a symmetric matrix are developed.

\begin{definition} Let $1\le k\le n$, $A$ be a real $n\by n$
symmetric matrix, and $N$ a real $n\by n$ diagonal matrix.  Define
the {\it generalized Rayleigh quotient\/} to be the function
$\rho\colon V_{n,k}\to\R$ given by $$\rho(p)=\tr p^\T\!ApN.$$
\end{definition}

\subsection{Gradient flows}

\begin{proposition}\label{prop:genraygrad}\ignorespaces Let $p$ be a
point in~$\stiefel(n,k)$, $A$ a real $n\by n$ symmetric matrix, and
$N$ a real $n\by n$ diagonal matrix.\par $(1)$\enskip The element $v$
in\/~$\m$ corresponding to the gradient of the generalized Rayleigh
quotient $\rho$ at~$p$ with respect to the canonical invariant metric
is given by $$v =[g^\T\!Ag,oNo^\T] =g^\T\!ApNo^\T-oNp^\T\!Ag.$$\par
$(2)$\enskip If the diagonal elements $\nu_i$ of~$N$ are distinct,
with $\nu_i>0$ for $i=1$, \dots,~$r$, and $\nu_i<0$ for $i=r+1$,
\dots~$k$, and the largest $r$ eigenvalues and smallest $k-r$
eigenvalues of~$A$ are distinct, then with the exception of certain
initial points contained within codimension\/~$1$ submanifolds
of~$\stiefel(n,k)$, the gradient flow associated with
$v=[g^\T\!Ag,oNo^\T]\in\m$ converge exponentially to points $p_\infty$
such that the first $r$ columns contain the eigenvectors of~$A$
corresponding to its largest eigenvalues, and the last $k-r$ columns
contain the eigenvectors corresponding to the smallest eigenvalues.
\end{proposition}

\begin{proof} Let $X$ a tangent vector in $T_pM$. Then for $f\in
C^\infty(M)$, $X$ corresponds to~$x\in\m$ by $$(X\f)_p=\ddt
f(ge^{xt}\cdot o).$$ By the definition of~$\rho$, it is seen that for
any $X\in T_pM$ $$\eqalign{d\rho_p(X) &=\ddt\rho(ge^{xt}\cdot o)\cr
&=\ddt\tr (ge^{xt}\cdot o)^\T\!A(ge^{xt}\cdot o)N\cr &=\ddt\tr
o^\T\!\Ad_{e^{-xt}}(g^\T\!Ag)oN\cr &=-\tr[x,g^\T\!Ag]oNo^\T\cr &=-\tr
x [g^\T\!Ag, oNo^\T]\cr &=\bigl\(x,[g^\T\!Ag, oNo^\T]\bigr\).\cr}$$
This establishes the first part.

Let $t\mapsto p_t$ be an integral curve of a gradient flow of the
$\rho$ on~$\stiefel(n,k)$, and $g_t$ a coset representative of~$p_t$
such that $p_t= g_t\cdot o$ for~all $t\in\R$.  For simplicity, denote
the $n\by n$ symmetric matrix $g_t^\T\!Ag_t$ by~$H$.  The manifold
$\stiefel(n,k)$ is compact and thus $\rho$ is bounded from above.  As
the derivative $${d\over dt}\rho(p_t) =-\tr [H,oNo^\T]^2$$ is
nonnegative, the value of~$\rho(p_t)$ has a limit and its derivative
approaches zero as $$[H,oNo^\T]\to0.$$ In the limit these asymptotics
become $$\diffeqalign{h_{ij}(\nu_j-\nu_i)&=0 &\hbox{for $1\le i,j\le
k$},\cr h_{ij}\nu_j&=0 &\hbox{for $k<i\le n$, $1\le j\le k$.}\cr}$$
Because the $\nu_i$ are assumed to be distinct, these conditions imply
that in the limit,
\def\smallmatrix#1{\vcenter{\baselineskip=.5\normalbaselineskip
\ialign{\hfil$\textstyle##$\hfil&&\,\hfil$\textstyle##$\hfil\crcr
\mathstrut\crcr\noalign{\kern-\baselineskip}#1\crcr
\mathstrut\crcr\noalign{\kern-\baselineskip}}}} $$H
=\pmatrix{\smallmatrix{\lambda_{\pi(1)}\cr
&\ddots&&\cr&&\lambda_{\pi(k)}\cr}&0\cr 0&H_1\cr},$$ where $\pi$ is a
permutation of the integers $1$, \dots,~$n$, and $H_1$ is an
$(n-k)\by(n-k)$ symmetric matrix with eigenvalues
$\lambda_{\pi(k+1)}$, \dots,~$\lambda_{\pi(n)}$.

The second order terms of~$\rho(p_t)$ at the critical points
corresponding to $H=\diag(\lambda_{\pi(1)},\penalty1000 \ldots,
\lambda_{\pi(k)}, H_1)$ indicate which of these points are
asymptotically stable.  Because the coset representative $g_t$
of~$p_t$ is arbitrary, choose $g_t$ such that
$H=g_t^\T\!Ag_t=\diag(\lambda_{\pi(1)}, \ldots, \lambda_{\pi(n)})$.
Let $X$ be tangent vector in~$T_{p_0}M$ corresponding to~$x\in\m$.
The Taylor expansion of~$\rho(p_t)$ about $t=0$ is
$$\rho(p_t)=\rho(p_0) +t(\D\rho)_{p_0}(X)
+{t^2\over2}(\Dsqr\rho)_{p_0}(X,X)+\cdots$$ (this formula will be
established rigorously in Chapter~\ref{chap:orm}).  The second order
terms of~$\rho(p_t)$ at the critical points of~$\rho$ corresponding to
$H=\diag(\lambda_{\pi(1)}, \ldots, \lambda_{\pi(n)})$ are given by the
Hessian $$\eqalign{(d^2\!\rho)_{p_0}(X,X) &=-\sum_{1\le j<i\le k}
x_{ij}^2(\lambda_{\pi(i)}-\lambda_{\pi(j)})(\nu_i-\nu_j)\cr
&\hphantom{{}={}}{-}\sum_{k<i\le n\atop 1\le j\le k} x_{ij}^2
(\lambda_{\pi(j)}-\lambda_{\pi(i)})\nu_j,\cr}$$ where $x_{ij}$ are the
elements of the matrix $x$.

This quadratic form is negative definite if and only if
\begin{itemize}
\item[(i)] The eigenvalues $\lambda_{\pi(i)}$ and the numbers $\nu_i$,
$1\le i\le k$, are similarly ordered.
\item[(ii)] If $\nu_j>0$, then $\lambda_{\pi(j)}$ is greater than all
the eigenvalues of the matrix~$H_1$; if $\nu_j<0$, then
$\lambda_{\pi(j)}$ is less than all the eigenvalues of the matrix
$H_1$.
\end{itemize}
This establishes the second part of the proposition.
\end{proof}

Note that the second equality of part~1 of
Proposition~\ref{prop:genraygrad} is more suitable for computations
because it requires $O(k)$ matrix-vector multiplications, as opposed
to the first equality which requires $O(n)$ matrix-vector
multiplications.

\begin{remark} If $A$ or~$N$ in Proposition~\ref{prop:genraygrad} has
repeated eigenvalues, then exponential stability, but not asymptotic
stability, is lost.
\end{remark}

\begin{corollary}\label{cor:raygencrit}\ignorespaces Let $A$ and $N$
be as in part\/~$(2)$ of Proposition~\ref{prop:genraygrad}.  Then the
generalized Rayleigh quotient $\rho$ has\/ $2^k\,{}_nP_k$ critical
points ($_nP_k=n!/(n-k)!$ is the number of permutations of~$n$ objects
taken $k$ at a time), of which one is a sink, one is a source, and the
remainder are saddle points.  \end{corollary}

\begin{corollary} Let $A$ and $N$ be as in part\/~$(2)$ of
Proposition~\ref{prop:genraygrad}.  Then near the critical points
corresponding to $H=\diag(\lambda_{\pi(1)}, \ldots, \lambda_{\pi(k)},
H_1)$ the gradient flow of $\rho$ has the exponential rates of
convergence $\mu_{ij}$ given by $$\mu_{ij}=\cases{-(\lambda_{\pi(i)}
-\lambda_{\pi(j)})(\nu_i-\nu_j), &for $1\le i,j\le k$;\cr
-(\lambda_{\pi(j)} -\lambda_{\pi(i)})\nu_j, &for $k<i\le n$, $1\le
j\le k$.\cr}$$
\end{corollary}

\section{The singular value decomposition}

The singular value decomposition (|SVD|) is an important decomposition
in numerical linear algebra.  It has applications in least squares
theory, matrix inversion, subspace comparisons, and spectral analysis.
Golub and Van~Loan \citeyear{GVL} provide background and examples.
There has been interest recently in the application of dynamical
systems to the solution of problems posed in the domain of numerical
linear algebra.  Brockett [1988] \citeyear{Brockett:sort} introduces
the double bracket equation $\dot H=[H,[H,N]]$, discussed in
Section~\ref{sec:doublebracket}, and shows that it can solve certain
problems of this type.  This work motivated Perkins et~al.\
\citeyear{balreal} to formulate a gradient algorithm which finds
classes of balanced realizations of finite dimensional linear systems.
In particular, they give a gradient algorithm for the |SVD|.  Also,
several researchers have constructed neuron-like networks that perform
principal component analysis.  For example, Oja \citeyear{Oja:jmb}
describes a network algorithm that extracts the principal component of
a statistically stationary signal; Rubner and Schulten
\citeyear{RubnerSchulten} generalize this method so that all principal
components are extracted.  There is a link between the matrix double
bracket equation and the least squares problems studied by
\citeasnoun{Brockett:subspace}, and the analysis of neural network
principal component analysis provided by Baldi and Hornik
\citeyear{BaldiHornik}.  Baldi and Hornik describe the level set
structure of a strictly convex function defined on real $n\by n$
matrices of rank~$k$.  This level set structure becomes identical to
that of the Lyapunov function~$-\tr HN$ if the strictly convex
function is restricted to matrices with fixed singular values.  See
also the work of Watkins and Elsner
\citeyear{WatkinsElsner:laa,WatkinsElsner:maa} for a discussion of
self-similar and self-equivalent flows and a continuous version of the
|QR|~algorithm for eigenvalues and singular values.
\citeasnoun{Chu:grad} and \citeasnoun{UweJohn} also provide gradient
flows similar to the ones described here that yield the singular value
decomposition of a matrix.  Deift et~al.\ \citeyear{DDLT:1,DDLT:1}
describe how a certain flow of bidiagonal matrices that leads to the
singular value decomposition can be viewed as a Hamiltonian flow with
respect to the so-called Sklyanin structure, which is described by
\citeasnoun{LiParmentier} and \citeasnoun{DeiftLi}. \looseness=-1\par

This section describes a gradient flow on the space of real $n\by
k$ matrices with fixed singular values whose solutions converge
exponentially to the |SVD| of a given matrix provided that its
singular values are distinct.  This dynamic system has, therefore,
potential application to the problems mentioned above.  Also, as a
generalization of the symmetric version of the matrix double bracket
equation, it inherits the capability to sort lists, diagonalize
matrices, and solve linear programming problems.  Viewed as an
algorithm for the |SVD|, this method is less efficient than the
variant of the |QR|~algorithm described by Golub and Van~Loan; however
the motivation here is to describe analog systems capable of this
task.  As opposed to Perkins et~al.'s method which requires matrix
inversion, matrix multiplication and addition are the only operations
required.  First presented are some results from differential geometry
and a suitable representation of the set of real $n\by k$ matrices
with prescribed singular values. A Riemannian structure is defined on
this space so that the gradient operator is well defined.  Next, the
main result is given with ensuing corollaries.  Finally, the results
of a numerical simulation are provided.

\subsection{Matrices with fixed singular values}

Recall the following standard mathematical notation and concepts.  Let
$\R^{n\times k}$ denote the set of all real $n\by k$ matrices. Let
$\O(n)$ and $\o(n)$ represent the real orthogonal group and its Lie
algebra of skew-symmetric matrices, respectively, such that for
$\Theta\in\O(n)$ and $\Omega\in\o(n)$, $\Theta^\T\Theta=I$ and
$\Omega+\Omega^\T=0$.  Both spaces have dimension~$n(n-1)/2$.  The
notation $\diag(\alpha_1,\ldots,\alpha_k)$ represents a $k\by k$
diagonal matrix whose diagonal elements are $\alpha_i$, and
$\diag_{n\times k}(\alpha_1,\ldots,\alpha_k)$ represents the $n\by k$
matrix $$\diag_{n\times k}(\alpha_1,\ldots,\alpha_k)
=\pmatrix{\diag(\alpha_1,\ldots,\alpha_k)\cr \hbox{\large0}\cr},$$
where, in this instance, $n\geq k$.  Let $D$ represent the discrete
subgroup of~$\O(k)$ consisting of matrices of the form
$\diag(\pm1,\ldots,\pm1)$.  Finally, let $\K(\sigma)$ denote the
manifold of real $n\by k$ matrices with the set of singular values
${\bmit\sigma}=\{\sigma_1,\ldots,\sigma_k\}$.  In this section it is
assumed that the singular values~$\sigma_i$ are distinct, and unless
stated otherwise, nonzero.

Let $K\in\K(\sigma)$ and assume, without loss of generality, that
$n\geq k$.  Then $K$ has the |SVD| \begin{equation} K=U\diag_{n\times
k}(\sigma_1,\ldots,\sigma_k)V^\T, \label{eq:svddef} \end{equation}
where $U\in\O(n)$, $V\in\O(k)$, and $\sigma_i\geq0$ for~$i=1$,
\dots,~$k$.  This decomposition is also expressible as
\begin{equation} \label{eq:svdrel} Kv_i=\sigma_iu_i
\qquad\hbox{or}\qquad K^\T u_i=\sigma_iv_i, \end{equation} where the
$\sigma_i$ are called the singular values of~$K$, and the $u_i$ and
$v_i$ are called the left and right singular vectors of~$K$,
respectively, for~$i=1$, \dots,~$k$.  If the singular values are
distinct, the left and right singular vectors are unique up to
multiplication of $u_i$ and $v_i$ by~$\pm1$.

\begin{remark}\label{rem:svdman}\ignorespaces The set $\K(\sigma)$ is
a differentiable manifold of dimension $nk-k$ if the~$\sigma_i$ are
distinct and nonzero.  This fact can be inferred from the existence of
a map~$p$ from~$\R^{n\times k}$ to the coefficients of the polynomials
of degree~$k$ over~$\R$ whose Jacobian has constant rank, viz., $$p(K)
=\det(\lambda I-K^\T\!K) -(\lambda-\sigma_1^2)
\ldots(\lambda-\sigma_k^2).$$ The differential of~$p$ at~$K$ is given
by $$dp_K(X) =-2\det(\lambda I-K^\T\!K) \tr(\lambda
I-K^\T\!K)^{-1}K^\T\!X.$$ This mapping has rank~$k$ for all
$K\in\K(\sigma)$; therefore the inverse image~$\K(\sigma)=p^{-1}(0)$
is a (compact) submanifold of~$\R^{n\times k}$ of dimension~$nk-k$.
It will be shown later that if $n>k$, then $\K(\sigma)$ is connected,
if $n=k$, then $\K(\sigma)$ has two connected components, and if $n=k$
and the elements of~$\K(\sigma)$ are restricted to be symmetric, then
$\K(\sigma)$ restricted to the symmetric matrices has $2^k$ connected
components.

A similar argument shows that if $\{\sigma_i\}$ has $r$ nonzero
distinct elements and $k-r$ zero elements, then ${\bmit
K}_{\{\sigma_1,\ldots,\sigma_r,0,\ldots,0\}}$ is a manifold of
dimension~$nr+kr-r^2-r$.  In particular, if $r=k-1$, then ${\bmit
K}_{\{\sigma_1,\ldots,\sigma_{k-1},0\}}$ is a manifold of
dimension~$nk-n$.  \end{remark}

The statement of the main result of this section contains statements
about the gradient of a certain function defined on~$\K(\sigma)$.  The
definition of the gradient on a manifold depends upon the choice of
Riemannian metric; therefore a metric must be chosen if the gradient
is to be well defined.  The approach of this section is standard:
$\K(\sigma)$ is identified with a suitable homogeneous space on which
a Riemannian metric is defined (see, e.g.,
\citeasnoun{KobayashiandNomizu}).

\begin{remark}\label{rem:svdhomspc}\ignorespaces The product group
$\O(n)\times\O(k)$ acts effectively on~$\K(\sigma)$ via the map
$\bigl((\theta,\vartheta),K\bigr) \mapsto\theta K\vartheta^\T$.
Clearly this action is transitive; therefore $\K(\sigma)$ is a
homogeneous space with the transformation group $\O(n)\times\O(k)$.
If the $\sigma_i$ are distinct and nonzero, then the isotropy group or
stabilizer of this action at the point $\diag_{n\times
k}(\sigma_1,\ldots,\sigma_k)\in\K(\sigma)$ is the closed subgroup
$\bigl\{\,\bigl(\diag(\Delta,\Psi),\Delta\bigr): \Psi\in\O(n-k),
\Delta\in D\,\bigr\}$, as can be verified from an elementary
calculation.  Note that this subgroup is the semidirect product
of~$\O(n-k)$ and $\Dg(D)\buildrel{\rm
def}\over=\bigl\{\,\bigl(\diag(\Delta,I),\Delta\bigr): \Delta\in
D\,\bigr\}$, the set theoretic diagonal of $\diag(D,I)\times D$;
therefore it will be represented by the notation $$\Dg(D)\O(n-k)
\buildrel{\rm def}\over=
\bigl\{\,\bigl(\diag(\Delta,\Psi),\Delta\bigr): \Psi\in\O(n-k),
\Delta\in D\,\bigr\}.$$ Thus $\K(\sigma)$ may be identified with the
homogeneous space $\svd$ of dimension $nk-k$.  Let $K\in\K(\sigma)$
have the |SVD| $K=U\diag_{n\times k}(\sigma_1,\ldots,\sigma_k)V^\T$.
It is straightforward to show that the map
$\psi\colon\K(\sigma)\to\svd$ defined by the action $\psi\colon
K\mapsto (U,V)\Dg(D)\O(n-k)$ is a bijection.  Because matrix
multiplication as an operation on~$\R^{n\times k}$ is smooth,
$\psi^{-1}$ is $C^\infty$; therefore $\psi$ is a
diffeomorphism.

A similar argument shows that if the set $\{\sigma_i\in\R\}$ has $r$
nonzero distinct elements, then ${\bmit
K}_{\{\sigma_1,\ldots,\sigma_r,0,\ldots,0\}}$ can be identified with
the homogeneous space $\bigl(\O(n)\times\O(k)\bigr)/
\Dg(D)\bigl(\O(n-r) \penalty1000\times\O(k-r)\bigr)$ of
dimension~$nr+kr-r^2-r$, where
$$\eqalign{\Dg(D)\bigl(\O(n-r)\times\O(k-r)\bigr)&\buildrel {\rm
def}\over =\bigl\{\,\bigl(\diag(\Delta,\Psi),
\diag(\Delta,\Upsilon)\bigr):\cr &\qquad \Psi\in\O(n-r),
\Upsilon\in\O(k-r), \Delta\in D\subset\O(r)\,\bigr\}.\cr}$$ In
particular, if $r=k-1$, then ${\bmit
K}_{\{\sigma_1,\ldots,\sigma_{k-1},0\}}$ can be identified with the
homogeneous space $\bigl(\O(n)\times\O(k)\bigr)/
\Dg(D)\bigl(\O(n-k+1)\times\O(1)\bigr)$ of dimension~$nk-n$.
\end{remark}

\begin{remark}\label{rem:svdreductive}\ignorespaces The homogeneous
space $\svd$ is reductive; i.e., there exists a linear subspace
$\k\times\o(k)$ of~$\o(n)\times\o(k)$ such that $$\o(n)\times\o(k)
=\o(n-k)\times0 +\k\times\o(k)\rlap{\quad\quad(direct sum)}$$ and
$\Ad_{\Dg(D)\O(n-k)} \bigl(\k\times\o(k)\bigr) \subset\k\times\o(k)$,
viz., $$\k =\left\{\left(\hbox{$a\atop
b$}\;\hbox{$-b^\T\atop0$}\right)\in\o(n)\right\}.$$ This is the
perpendicular subspace given by~Proposition~\ref{prop:cheegin} of
Chapter~\ref{chap:geom}.  Therefore there is a natural correspondence
between $\Ad_{\Dg(D)\O(n-k)}$-invariant nondegenerate symmetric
bilinear forms on~$\k\times\o(k)$ and $\O(n)\times\O(k)$-invariant
Riemannian metrics on $\svd$.  A general exposition of these ideas is
given by \citeasnoun[Chap.~10]{KobayashiandNomizu}.  \end{remark}

The object of these remarks is to establish the identification
$$\K(\sigma)\cong\svd$$ when the $\sigma_i$
are distinct and nonzero, where $\Dg(D)\O(n-k)$ is the closed subgroup
of~$\O(n)\times\O(k)$ defined in Remark~\ref{rem:svdhomspc}, and to
assert that a positive definite quadratic form on~$\k\times\o(k)$
defines a Riemannian metric on~$\K(\sigma)$, where $\k$ is
the linear subspace defined in Remark~\ref{rem:svdreductive}.

\begin{proposition}\label{prop:svdmetric}\ignorespaces The
nondegenerate symmetric bilinear form on\/ $\k\times\o(k)$ defined by
\begin{equation} g\bigl((\Gamma_1,\Phi_1),(\Gamma_2,\Phi_2)\bigr)
={n-2\over2}\tr\Gamma_1^\T\Gamma_2 +{k-2\over2}\tr\Phi_1^\T\Phi_2,
\label{eq:svdmetric} \end{equation} where\/ $\Phi_1,\Phi_2\in\o(k)$,
$\Gamma_1,\Gamma_2\in\k$, and $n\geq k\geq3$, defines an
$\O(n)\times\O(k)$-invariant Riemannian metric on\/ $\svd$.  If\/ $n$
or~$k$ equals~2, replacing the coefficients\/ $(n-2)$ or\/~$(k-2)$
by~unity, respectively, yields an $\O(n)\times\O(k)$-invariant
Riemannian metric on~$\svd$.  \end{proposition}

\begin{proof} The product space $\O(n)\times\O(k)$ is a compact
semisimple Lie group, $n,k\geq3$; therefore the Killing form
$\kf\bigl((\Gamma_1,\Phi_1), (\Gamma_2,\Phi_2)\bigr)
=(n-2)\tr\Gamma_1\Gamma_2 +(k-2)\tr\Phi_1\Phi_2$ of~$\o(n)\times\o(k)$
is strictly negative definite.  From \citeasnoun[Chap.~10,
Coroll.~3.2]{KobayashiandNomizu}, or \citeasnoun[Chap.~4,
Prop.~3.4]{Helgason}, it can be seen that there is a natural
correspondence between $\O(n)\times\O(k)$-invariant Riemannian metrics
on~$\svd$ and nondegenerate symmetric
bilinear forms on~$\k\times\o(k)$.  Therefore the form
$g=-\half\varphi$ restricted to~$\k\times\o(k)$ defines such a metric.
If~$n$ or~$k$ equals~2, the nondegenerate symmetric bilinear form
$(\Omega_1,\Omega_2)\mapsto\half\tr\Omega_1^\T\Omega_2$ on~$\o(2)$
defines an $\O(2)$-invariant Riemannian metric on~$\O(2)$.  Therefore
replacing the expressions $(n-2)$ or~$(k-2)$ by unity in
Equation~(\ref{eq:svdmetric}) yields an $\O(n)\times\O(k)$-invariant
Riemannian metric on~$\svd$.  \end{proof}

\begin{proposition}\label{prop:svdtan}\ignorespaces Let\/
$\Sigma\colon\R\to\K(\sigma)$ be a smoothly parameterized curve
in\/~$\K(\sigma)$. Then the tangent vector to the curve\/~$\Sigma$
at~$t$ is of the form \begin{equation} \dot\Sigma(t)
=\Sigma(t)\Phi-\Gamma\Sigma(t), \label{eq:svdtan} \end{equation}
where\/ $\Phi\in\o(k)$ and\/ $\Gamma\in\k$.  \end{proposition}

\begin{proof} Let $K\in\K(\sigma)$.  Then $\Sigma(t)=U^\T(t)KV(t)$ for
$U(t)\in\O(n)$ and $V(t)\in\O(k)$.  The perturbations $U(t)\to
Ue^{(t+\epsilon)\Gamma}$ and $V(t)\to Ve^{(t+\iota)\Phi}$ for
$U\in\O(n)$, $V\in\O(k)$ $\Gamma\in\k$, $\Phi\in\o(k)$, and real
$\epsilon$ and $\iota$, give rise to the tangent vector of
Equation~(\ref{eq:svdtan}) under the change of coordinates
$\Sigma(t)=U^\T(t)KV(t)$.  The elements of~$\Gamma$ and $\Phi$
parameterize the tangent plane of~$\K(\sigma)$ at~$\Sigma(t)$
completely; therefore this set of tangent vectors is complete.
\end{proof}

\subsection{Gradient flows}

Consider the extremization problem $$\max_{\Sigma\in\K(\sigma)}\tr
N^\T\Sigma,$$ where $N=\diag_{n\times k}(\nu_1,\ldots,\nu_k)$ and the
$\nu_i$ are real (cf.\ von~Neumann [1937] (1962)).  For the remainder
of this section, assume that $n\geq k\geq3$.  The following results
may be extended to include the cases where $n$ or~$k$ equals~2 by
replacing the expressions $(n-2)$ or~$(k-2)$ by unity, respectively.
The notation $\[\,{,}\,\]\colon\R^{m\times l}\times\R^{m\times
l}\to\o(m)$ defined for $m\geq3$ by the bilinear operation
$\[A,B\]=(AB^\T-BA^\T)/(m-2)$ is employed in the statement of the
following proposition.

\begin{proposition}\label{prop:svdflow}\ignorespaces Let\/
$\Sigma$, $K\in\K(\sigma)$.  The gradient ascent equation
on\/~$\K(\sigma)$ for the function\/ $\tr N^\T\Sigma$ with respect to
the Riemannian metric defined above is $$\makeatletter
\refstepcounter{equation}\let\@currentlabel=\theequation
\label{eq:svdflow} \dot\Sigma =\Sigma\[\Sigma^\T,N^\T\]
-\[\Sigma,N\]\Sigma; \qquad\Sigma(0)=K.\eqno(\theequation{\rm a})$$
Equivalently, let $U\in\O(n)$ and\/ $V\in\O(k)$. The gradient ascent
equations on~$\O(n)\times\O(k)$ for the function\/ $\tr N^\T U^\T\!KV$
with respect to the Riemannian metric defined above are
$$\diffeqalign{\dot U &=U\[U^\T\!KV,N\]; & U(0)=I,\cr \dot V
&=V\[V^\T\!K^\T U,N^\T\]; & V(0)=I.\cr}\eqno(\ref{eq:svdflow}{\rm
b})$$ Furthermore, if\/ $\{\sigma_i\}$ and\/ $\{\nu_i\}$ have distinct
elements, then with the exception of certain initial points contained
within a finite union of codimension\/~$1$ submanifolds
of\/~$\K(\sigma)\times \O(n)\times\O(k)$, the triple\/~$(\Sigma,U,V)$
converges exponentially to the singular value decomposition of~$K$ (up
to the signs of the singular values).  If\/~$\Sigma$ is nonsquare
or~nonsymmetric, then in the limit the moduli of the\/~$\pm\sigma_i$
and the\/~$\nu_i$ are similarly ordered.  If\/~$\Sigma$ is square and
symmetric, then in the limit the eigenvalues~$\lambda_i=\pm\sigma_i$
of\/~$\Sigma$ and the\/~$\nu_i$ are similarly ordered.
\end{proposition}

\begin{proof} Let $K$ have the |SVD| $K=U_1\diag_{n\times
k}(\sigma_1,\ldots,\sigma_k)V_1^\T$ where the singular values are
distinct and nonzero.  Denote the isotropy group at~$K$ by $H=
(U_1,V_1)\Dg(D)\O(n-k)\* (U_1^\T,V_1^\T)$ (n.b.\
Remark~\ref{rem:svdhomspc}).  The gradient of the function
$f\colon\svd\to\R$ at the point~$H(U,V)$ is uniquely defined by the
equality $$df_{H(U,V)}(\Gamma,\Phi) =g\bigl((\grad\f)_{H(U,V)},
(\Gamma,\Phi)\bigr).$$ For $f(H(U,V))=\tr N^\T U^\T\!KV$, it can be
seen that $$df_{H(U,V)}(\Gamma,\Phi) =\half\tr\Gamma^\T(\Sigma
N^\T-N\Sigma^\T) +\half\tr\Phi^\T(\Sigma^\T\!N-N^\T\Sigma),$$ where
the identities $\tr ABC=\tr BCA=\tr CAB$ and $\tr
A^\T\!B=\tr(A^\T\!B+B^\T\!A)/2$ are employed.  From the definition of
the Riemannian metric in Proposition~\ref{prop:svdmetric}, it is clear
that the gradient directions of~$\tr N^\T U^\T\!KV$ are
$$\diffeqalign{\grad\f&={1\over n-2}\bigl(\Sigma
N^\T-N\Sigma^\T\bigr)=\[\Sigma,N\]\in\k &\hbox{($\k$-component)},\cr
\grad\f &={1\over k-2}\bigl(\Sigma^\T\! N-N^\T\Sigma\bigr)
=\[\Sigma^\T,N^\T\]\in\o(k) &\hbox{($\o(k)$-component)}.\cr}$$ This
with Proposition~\ref{prop:svdtan} proves the first part.

Because the derivative $${d\over dt}\tr N^\T\Sigma
={(n-2)^2\over2(k-2)}\tr\[\Sigma,N\]^\T\[\Sigma,N\]
+{(k-2)^2\over2(n-2)}\tr\[\Sigma^\T,N^\T\]^\T\[\Sigma^\T,N^\T\]$$ is
nonnegative and $\tr N^\T\Sigma$ is bounded from above ($\K(\sigma)$
is a compact subset of~$\R^{n\times k}$), $\tr N^\T\Sigma$ has a limit
and its derivative approaches zero as $\[\Sigma,N\]$ and
$\[\Sigma^\T,N^\T\]$ approach zero. In the limit these become, for
$1\leq i,j\leq k$, $$\sigma_{ij}\nu_j-\nu_i\sigma_{ji}=0, \qquad
\sigma_{ji}\nu_j-\nu_i\sigma_{ij}=0$$ and, for~$1\leq i\leq k$,
$k<j\leq n$, $$\nu_i\sigma_{ji}=0,$$ where the $\sigma_{ij}$ are
elements of~$\Sigma$.  If the $\nu_i$ are distinct, these conditions
imply that $\sigma_{ij}=0$ for $i\neq j$.  Therefore the critical
points of Equations (\ref{eq:svdflow}a) and (\ref{eq:svdflow}b) occur
when the prescribed singular values are along the diagonal
of~$\Sigma$; i.e., $\Sigma=\diag_{n\times
k}(\pm\sigma_{\pi(1)},\ldots,\pm\sigma_{\pi(k)})$ for some
permutation~$\pi$ of the integers~$1$, \dots,~$k$.

Inspecting the second order terms of~$\tr N^\T\Sigma$ at a critical
point $\Sigma=\diag_{n\times
k}(\sigma_{\pi(1)},\penalty1000\ldots,\sigma_{\pi(k)})$ will show
which of these points is asymptotically stable.  Let $\Sigma$ be the
parameterized matrix $(Ue^{\epsilon\Gamma})^\T\!K(Ve^{\iota\Phi})$,
where $U^\T\!KV= \diag_{n\times
k}(\sigma_{\pi(1)},\ldots,\sigma_{\pi(k)})$.  The second order terms
of~$\tr N^\T\Sigma$ are
\begin{eqnarray} &\displaystyle
\matrix{\half\pmatrix{\epsilon&\iota\cr}\cr\cr} \pmatrix{\tr
N^\T\Gamma^{2T}\Sigma&\tr N^\T\Gamma^\T\Sigma\Phi\cr \tr
N^\T\Gamma^\T\Sigma\Phi&\tr N^\T\Sigma\Phi^2\cr}
\pmatrix{\epsilon\cr\iota\cr} \nonumber\\ &\displaystyle
\quad{}={-}\half\sum_{1\leq i<j\leq k}
\matrix{\pmatrix{\epsilon\gamma_{ij} &\iota\phi_{ij}\cr}\cr\cr}
\pmatrix{\nu_i\sigma_{\pi(i)} +\nu_j\sigma_{\pi(j)}
&-(\nu_i\sigma_{\pi(j)}+\nu_j\sigma_{\pi(i)})\cr
-(\nu_i\sigma_{\pi(j)}+\nu_j\sigma_{\pi(i)})
&\nu_i\sigma_{\pi(i)}+\nu_j\sigma_{\pi(j)}\cr}
\pmatrix{\epsilon\gamma_{ij}\cr\iota\phi_{ij}\cr} \nonumber\\
&\displaystyle \quad\hphantom{{}={}}{-}\half\sum_{\scriptstyle 1\leq
i\leq k\atop \scriptstyle k<j\leq n}
\nu_i\sigma_{\pi(i)}(\epsilon\gamma_{ij})^2, \label{eq:svd2d}
\end{eqnarray} where $\gamma_{ij}$ and $\phi_{ij}$ are the elements of
$\Gamma$ and $\Phi$, respectively.  Thus $\tr N^\T\Sigma$ is negative
definite if and only if this quadratic form is negative definite.
Three cases must be considered.

\indent {\it Case I.\enskip $\Sigma$ nonsquare}\par\nobreak

The quadratic form of Equation~(\ref{eq:svd2d}) is negative definite
if and only if the $2\by2$ matrices in the first sum are positive
definite and the coefficients~$\nu_i\sigma_{\pi(i)}$ of the second sum
are positive.  The matrices will be inspected first.  A $2\by2$
symmetric matrix is positive definite if and only if its
$(1,1)$~element and its determinant are positive.  In this case, these
conditions imply that $$\nu_i\sigma_{\pi(i)}+\nu_j\sigma_{\pi(j)}>0,
\qquad (\nu_i^2-\nu_j^2)(\sigma_{\pi(i)}^2-\sigma_{\pi(j)}^2)>0.$$ The
condition that the determinant be positive implies that the moduli
of~$\nu_i$ and $\pm\sigma_{\pi(i)}$ must be similarly ordered
and that the singular values must be distinct.  Given that the moduli
are similarly ordered, the condition that the $(1,1)$~element
$\nu_i\sigma_{\pi(i)}+\nu_j\sigma_{\pi(j)}$ be positive demands that
$$\sign\sigma_{\pi(i)}=\sign\nu_i \rlap{$\qquad i=1,\ldots,k-1$,}$$
because if $|\nu_i|>|\nu_j|$ (implying that
$|\sigma_{\pi(i)}|>|\sigma_{\pi(j)}|$) and
$\sign\sigma_{\pi(i)}=-\sign\nu_i$, then the $(1,1)$~element would be
negative.  This argument asserts nothing about the sign of the
smallest singular value, which, without loss of generality, may be
taken as~$\sigma_k$.  As stated previously, the
coefficients~$\nu_i\sigma_{\pi(i)}$ of the second sum of
Equation~(\ref{eq:svd2d}) must be positive.  Therefore
$$\sign\sigma_k=\sign\nu_{\pi^{-1}(k)}.$$

\indent {\it Case II.\enskip $\Sigma$ square and nonsymmetric}\par\nobreak

In this case the second sum of Equation~(\ref{eq:svd2d}) vanishes and
cannot be used to determine the sign of~$\sigma_k$, but the additional
structure of square matrices compensates for this loss.  Consider the
(square) decomposition $\Sigma=U^\T\!KV$, where $K$ is nonsymmetric and
$U,V\in\SO(n)$ (the special orthogonal group
$\SO(n)=\{\,\Theta\in\O(n):\det\Theta=1\,\}$).  Then $$\det\Sigma=\det
K$$ and the sign of~$\sigma_k$ is determined.

\indent {\it Case III.\enskip $\Sigma$ square and symmetric}\par\nobreak

When $\Sigma$ is square and symmetric, Equation~(\ref{eq:svdflow}a)
reduces to the matrix double bracket equation described by
\citeasnoun{Brockett:sort}; i.e., $\dot\Sigma=[\Sigma,[\Sigma,N]]$;
$\Sigma(0)=K=K^\T$ defined over real $k\by k$ symmetric matrices
with fixed eigenvalues (where the time parameter is scaled by~$k-2$).
Thus the flow of~$\Sigma$ on~$\K(\sigma)$ is isospectral and
$\Sigma(t)$ is symmetric for all~$t$.  The critical points of
Equation~(\ref{eq:svdflow}a) occur when the
eigenvalues~$\lambda_i=\pm\sigma_i$ of~$K$ are along the diagonal
of~$\Sigma$; i.e., $\Sigma =\diag(\lambda_{\pi(1)},\ldots
,\lambda_{\pi(n)})$ for some permutation~$\pi$ of the integers $1$,
\dots,~$k$.  A square symmetric matrix~$K$ in
Equation~(\ref{eq:svdflow}b) implies that $U\equiv V$; i.e., $\dot
V=V[V^\T\!KV,N]$; $V(0)=I$ or~$\diag(-1,1,\ldots,1)$ (where the time
parameter is scaled by~$k-2$).  Therefore $\gamma_{ij}=\phi_{ij}$ in
Equation~(\ref{eq:svd2d}), which reduces to the sum $$-\sum_{1\leq
i<j\leq k}(\nu_i-\nu_j)(\lambda_{\pi(i)}-\lambda_{\pi(j)})
(\epsilon\phi_{ij})^2.$$ This sum is negative definite if and only if
$\{\lambda_i\}$ and $\{\nu_i\}$ are similarly ordered.

If one of the singular values vanishes, the proof holds if the
homogeneous space $\svd$ is replaced by
$\bigl(\O(n)\times\O(k)\bigr)/\penalty1000
\Dg(D)\bigl(\O(n-k+1)\penalty1000\times\O(1)\bigr),$ and the linear
space~$\k$ is replaced by the linear space~$\k'$ defined by the
orthogonal decomposition $\o(n)=\diag(0,\o(n-k+1))+\k'$ (direct sum).
This completes the proof of the second part.  \end{proof}

\begin{remark} If $K$ or~$N$ in Proposition~\ref{prop:svdflow} has
repeated singular values, exponential stability, but not asymptotic
stability, is lost.
\end{remark}

\begin{corollary} Let the $\sigma_i$ and the $\nu_i$ be distinct and
nonzero.  The following hold: \par $(1)$\enskip Let
$\Sigma\in\K(\sigma)$ be nonsquare.  Then\/ $\K(\sigma)$ is connected
and Equation\/~$\rm(\ref{eq:svdflow}a)$ has\/ $2^kk!$ critical points,
of which one is a sink, one is a source, and the remainder are saddle
points.  Also, the set of critical points of
Equation\/~$\rm(\ref{eq:svdflow}b)$ is a submanifold
of~$\O(n)\times\O(k)$ of dimension\/~$(n-k)(n-k-1)/2$. \par
$(2)$\enskip Let $\Sigma\in\K(\sigma)$ be square and nonsymmetric.
Then $\K(\sigma)$ has two connected components corresponding to the
sign of\/~$\det\Sigma$.  On each connected component
Equation\/~$\rm(\ref{eq:svdflow}a)$ has\/ $2^{k-1}k!$ critical points,
of which one is a sink, one is a source, and the remainder are saddle
points.  Also, Equation\/~$\rm(\ref{eq:svdflow}b)$ has $2^{2k}k!$
critical points, of which\/ $2^{2k}$ are sinks, $2^{2k}$ are sources,
and the remainder are saddle points. \par $(3)$\enskip Let
$\Sigma\in\K(\sigma)$ be square and symmetric. Then
$\K(\sigma)\cap\{\,Q\in\R^{k\times k}: Q=Q^\T\,\}$ has $2^k$ connected
components corresponding to matrices with eigenvalues
$\{\pm\sigma_i\}$.  On each connected component
Equation\/~$\rm(\ref{eq:svdflow}a)$ has $k!$ critical points, of which
one is a sink, one is a source, and the remainder are saddle points.
Also, Equation\/~$\rm(\ref{eq:svdflow}b)$ has\/ $2^kk!$ critical
points, of which\/ $2^k$ are sinks, $2^k$ are sources, and the
remainder are saddle points.  \end{corollary}

\begin{proof} Without loss of generality, let $N=\diag_{n\times
k}(k,\ldots,1)$.  In the nonsquare case every trajectory with initial
point $K\in\K(\sigma)$ converges to the point $\diag_{n\times
k}(\sigma_1,\ldots,\sigma_k)$, except for a finite union of
codimension~1 submanifolds of~$\K(\sigma)$.  But the closure of this
set of initial points is $\K(\sigma)$; therefore $\K(\sigma)$ is path
connected, and, as seen in the proof of
Proposition~\ref{prop:svdflow}, Equation~(\ref{eq:svdflow}a) has
$2^kk!$ critical points.  Furthermore, for every critical point of
Equation~(\ref{eq:svdflow}a) there is a corresponding critical point
$(U,V)$ of Equation~(\ref{eq:svdflow}b).  But every point in the coset
$(U,V)\Dg(D)\O(n-k)$ is also a critical point of
Equation~(\ref{eq:svdflow}b).

In the square nonsymmetric case every trajectory with initial point
$K\in\K(\sigma)$ converges to the point
$\diag\bigl(\sigma_1,\ldots,(\sign\det K)\sigma_k\bigr)$, except for a
finite union of codimension~1 submanifolds of~$\K(\sigma)$.  The
closures of these sets of initial points with positive and negative
determinants are path connected and disjoint; therefore $\K(\sigma)$
has two connected components, and there are $2^{k-1}k!$ critical
points in each connected component.  Furthermore, for every critical
point of Equation~(\ref{eq:svdflow}a) there is a corresponding
critical point $(U,V)$ of Equation~(\ref{eq:svdflow}b).  But every
point in the coset $(U,V)\Dg(D)$ is also a critical point of
Equation~(\ref{eq:svdflow}b).

 In the square symmetric case every trajectory with initial point
$K\in\K(\sigma)\cap\{\,Q\in\R^{k\times k}: Q=Q^\T\,\}$ converges to
the point $\diag_{n\times k}(\lambda_1,\ldots,\lambda_k)$, where the
$\lambda_i=\pm\sigma_i$ are the ordered eigenvalues of~$K$, except for
a finite union of codimension~1 submanifolds of~$\K(\sigma)$.  The
closures of these isospectral sets of initial points are path
connected and disjoint; therefore $\K(\sigma)\cap\{\,Q\in\R^{k\times
k}: Q=Q^\T\,\}$ has $2^k$ connected components, and there are $k!$
critical points in each connected component.  Furthermore, for all
critical points $\diag(\lambda_{\pi(1)},\ldots,\lambda_{\pi(k)})$ of
Equation~(\ref{eq:svdflow}a) there is a corresponding critical
point~$V$ of Equation~(\ref{eq:svdflow}b).  But every point in the
coset $VD$ is also a critical point of Equation~(\ref{eq:svdflow}b).
\end{proof}

\begin{corollary} Let the $\sigma_i$ and the $\nu_i$ be distinct and
nonzero.  The function\/ $\tr N^\T\Sigma$ mapping\/ $\K(\sigma)$ to
the real line has\/ $2^kk!$ critical points, of which one is a global
minimum (and one is a local minimum if $n=k$), one is a global maximum
(and one is a local maximum if $n=k$), and the remainder are saddle
points.  Furthermore, if~$n>k$, the submanifold~$\Dg(D)\O(n-k)$
of~$\O(n)\times\O(k)$ is a nondegenerate critical manifold of the
function\/ $\tr N^\T U^\T\!KV$ mapping $\O(n)\times\O(k)$ to the real
line.  If~$n=k$ and $K$ is nonsymmetric, the function\/ $\tr
N^\T U^\T\!KV$ has\/ $2^{2k}k!$ critical points, of which\/ $2^k$ are
global minima, $2^k$ are local minima, $2^k$ are global maxima, $2^k$
are local maxima, and the remainder are saddle points.  If~$n=k$ and
$K$ is symmetric, the function\/ $\tr N^\T U^\T\!KV$ has\/ $2^kk!$
critical points, of which\/ $2^k$ are global minima, $2^k$ are global
maxima, and the remainder are saddle points.  \end{corollary}

\begin{corollary} Let $l$ be a positive integer.  The
matrix\/~$(\Sigma^\T\Sigma)^l$ evolves isospectrally on flows of
Equation\/~$\rm(\ref{eq:svdflow}a)$.
\end{corollary}

\begin{proof} The corollary follows from the fact that $${d\over
dt}(\Sigma^\T\Sigma)^l=\left[(\Sigma^\T\Sigma)^l,{k-2\over
n-2}\[\Sigma^\T,N^\T\]\right]$$ is in standard isospectral form.
\end{proof}

\begin{proposition}\label{prop:svdconvrate}\ignorespaces Let\/
$\{\sigma_i\}$ and\/ $\{\nu_i\}$ have distinct nonzero elements. \par
$(1)$\ Near the critical points\/ $\Sigma=\diag_{n\times
k}(\pm\sigma_{\pi(1)},\ldots,\penalty1000 \pm\sigma_{\pi(k)})$ the
off-diagonal elements $\sigma_{ij}$ of Equation\/
$\rm(\ref{eq:svdflow}a)$ converge exponentially with rates
$r^{(1)}_{ij}$ given by the eigenvalues of the matrix
$$R_1=\pmatrix{\displaystyle {\nu_i\sigma_i\over k-2}
+{\nu_j\sigma_j\over n-2}& \displaystyle -{\nu_j\sigma_i\over k-2}
-{\nu_i\sigma_j\over n-2}\cr \noalign{\bigskip} \displaystyle
-{\nu_j\sigma_i\over n-2} -{\nu_i\sigma_j\over k-2}& \displaystyle
{\nu_i\sigma_i\over n-2} +{\nu_j\sigma_j\over k-2}\cr}$$ for $1\le
i,j\le k$, and by $$r^{(1)}_{ij}={\nu_i\sigma_i\over n-2}$$ for
$k<i\le n$, $1\le j\le k$.  \par $(2)$\ Near the critical
points\/ $(U,V)$ such that $U^\T\!KV=\diag_{n\times
k}(\pm\sigma_{\pi(1)}, \ldots,\pm\sigma_{\pi(k)})$, the elements of
Equation\/ $\rm(\ref{eq:svdflow}b)$ converge exponentially with rates
$r^{(2)}_{ij}$ given by the eigenvalues of the matrix
$$R_2=\pmatrix{\displaystyle {\nu_i\sigma_i+\nu_j\sigma_j\over n-2}
&\displaystyle -{\nu_i\sigma_j+\nu_j\sigma_i\over n-2}\cr
\noalign{\bigskip} \displaystyle -{\nu_i\sigma_i+\nu_j\sigma_j\over
k-2} &\displaystyle {\nu_i\sigma_i+\nu_j\sigma_j\over k-2}\cr}$$ for
$1\le i,j\le k$, and by $$r^{(2)}_{ij}={\nu_i\sigma_i\over n-2}$$ for
$k<i\le n$, $1\le j\le k$. \par $(3)$\enskip For~all $i$ and $j$,
$r^{(1)}_{ij}=r^{(2)}_{ij}$.  \end{proposition}

\begin{proof} Let $\delta U=U\delta\Gamma$ and $\delta V=V\delta\Phi$
be first order perturbations of~$(U,V)\in\O(n)\times\O(k)$, i.e.,
$\delta\Gamma\in\so(n)$ and $\delta\Phi\in\so(k)$.  Then
$$\delta\Sigma=\Sigma(\delta\Phi)-(\delta\Gamma)\Sigma$$ is a first order
perturbation of~$\Sigma\in\K(\sigma)$.  Computing the first order
perturbation of Equation~(\ref{eq:svdflow}a) at the critical point
$\Sigma=\diag_{n\times k}(\pm\sigma_{\pi(1)},
\ldots,\pm\sigma_{\pi(k)})$, it is seen that $$\delta\dot\Sigma
=\Sigma\[(\delta\Sigma)^\T, N^\T\] -\[(\delta\Sigma),N\]\Sigma.$$ This
differential equation is equivalent to the set of differential
equations $$\biggl({\delta\dot\sigma_{ij}\atop
\delta\dot\sigma_{ji}}\biggr)= -R_1\biggl({\delta\sigma_{ij}\atop
\delta\sigma_{ji}}\biggr)$$ for $1\le i,j\le k$, and
$\delta\dot\sigma_{ij}=-r^{(1)}_{ij}\delta\sigma_{ij}$ for $k<i\le n$,
$1\le j\le k$.  This establishes the first part.

Computing the first order perturbation of Equation~(\ref{eq:svdflow}b)
at the critical point $(U,V)$ such that $U^\T\!KV=\diag_{n\times
k}(\pm\sigma_{\pi(1)}, \ldots,\pm\sigma_{\pi(k)})$, it is seen that
$$\delta\dot\Gamma =-\[(\delta\Gamma)\Sigma,N\]
+\[\Sigma(\delta\Phi),N\],\qquad \delta\dot\Phi
=-\[(\delta\Phi)\Sigma^\T,N^\T\]+\[\Sigma^\T(\delta\Gamma),N^\T\].$$
These differential equations are equivalent to the set of differential
equations $$\biggl({\delta\dot\gamma_{ij}\atop
\delta\dot\phi_{ij}}\biggr)= -R_2\biggl({\delta\gamma_{ij}\atop
\delta\phi_{ij}}\biggr)$$ for $1\le i,j\le k$,
$\delta\dot\gamma_{ij}=-r^{(2)}_{ij}\delta\gamma_{ij}$ for $k<i\le n$,
$1\le j\le k$, and $\delta\dot\gamma_{ij}=0$ for $k<i,j\le n$.  This
establishes the second part.

The final part follows immediately from the equalities $\tr R_1=\tr
R_2$ and $\det R_1=\det R_2$.
\end{proof}

\begin{remark} Equations\/ $\rm(\ref{eq:svdflow}a)$ and\/
$\rm(\ref{eq:svdflow}b)$ become $$\diffeqalign{\dot\Sigma &={1\over
n-2}\Sigma\Sigma^\T\!N -\left({1\over n-2}+{1\over k-2}\right)\Sigma
N^\T\Sigma +{1\over k-2}N\Sigma^\T\Sigma; &\Sigma(0)=K,\cr \dot U
&={1\over n-2}\bigl(KVN^\T-UNV^\T\!K^\T U\bigr); &U(0)=I,\cr \dot V
&={1\over k-2}\bigl(K^\T UN-VN^\T U^\T\!KV\bigr); &V(0)=I,\cr}$$ when
the notation $\[\,{,}\,\]$ is expanded and $n\geq k\geq3$.
\end{remark}

\subsection{Experimental results}

The system of Equation~(\ref{eq:svdflow}b) was simulated with a
Runge--Kutta algorithm with $K=\diag_{7\times 5}(1,2,3,\relax4,5)$,
$N=\diag_{7\times 5}(5,4,3,2,1)$, and the initial conditions
$$\def\quad{\hskip.5em}{\ixpt\pmatrix{\hfill-0.210&\hfill -0.091&\hfill
0.455&\hfill 0.668&\hfill -0.217&\hfill 0.490&\hfill 0.085\cr \hfill
0.495&\hfill 0.365&\hfill 0.469&\hfill 0.291&\hfill 0.183&\hfill
-0.413&\hfill -0.335\cr \hfill 0.191&\hfill 0.647&\hfill 0.058&\hfill
-0.237&\hfill -0.578&\hfill 0.154&\hfill 0.356\cr \hfill 0.288&\hfill
-0.285&\hfill 0.403&\hfill -0.539&\hfill -0.089&\hfill 0.461&\hfill
-0.404\cr \hfill -0.490&\hfill -0.022&\hfill 0.633&\hfill
-0.339&\hfill 0.130&\hfill -0.340&\hfill 0.333\cr \hfill -0.426&\hfill
0.598&\hfill -0.064&\hfill -0.088&\hfill 0.438&\hfill 0.364&\hfill
-0.353\cr \hfill -0.412&\hfill -0.005&\hfill -0.046&\hfill
0.017&\hfill -0.607&\hfill -0.325&\hfill -0.595\cr}}, \
{\ixpt\pmatrix{\hfill 0.679&\hfill 0.524&\hfill 0.091&\hfill
-0.438&\hfill 0.253\cr \hfill -0.521&\hfill 0.427&\hfill 0.406&\hfill
0.137&\hfill 0.602\cr \hfill 0.504&\hfill -0.108&\hfill 0.315&\hfill
0.788&\hfill 0.120\cr \hfill -0.032&\hfill -0.089&\hfill 0.839&\hfill
-0.255&\hfill -0.472\cr \hfill 0.113&\hfill -0.723&\hfill 0.156&\hfill
-0.322&\hfill 0.579\cr}}$$
representing $U(0)$ and $V(0)$ chosen at random using Gram-Schmidt
orthogonalization from $\O(7)$ and $\O(5)$, respectively.
Figure~\ref{fig:svdflow} illustrates the convergence of the diagonal
elements of~$\Sigma$ to the singular values $5$, $4$, $3$, $2$, $1$
of~$K$.  Figure~\ref{fig:svdconv} illustrates the rates of convergence
of a few off diagonal elements of~$\Sigma$, which are tabulated in
Table~\ref{tab:convrate} with the predicted convergence rates of
Proposition~\ref{prop:svdconvrate}.

\begin{table}[h]
\caption[Convergence rates of selected elements of $\Sigma$]
  {\label{tab:convrate}\ignorespaces
    Convergence rates of selected elements of\/ $\Sigma$}
\vbox{\tabskip=0pt \def\opentable{\noalign{\vskip4pt}}
\halign to\textwidth{\strut\quad$#$\hfil\tabskip=1em plus2em
  &\hfil$#$\hfil&\hfil$#$\quad&\hfil$#$\quad \tabskip=0pt\cr
\noalign{\hrule}\opentable
\omit Convergence Rate\hfil &\omit\hfil Exact Value\hfil &\omit 
  Approximate Value\hfil &\omit Measured Value*\hfil\cr
\opentable\noalign{\hrule}\opentable
r_{12}&(164-\sqrt{25681})/15&0.24980&0.2498\cr
r_{43}&(52-\sqrt{2329})/15&0.24935&0.2471\cr
r_{31}&(136-8\sqrt{229})/15&0.99587&0.969\hphantom0\cr
r_{24}&(80-4\sqrt{265})/15&0.99231&0.989\hphantom0\cr
r_{75}&1/5&\omit\hfil---\qquad&0.1999\cr \opentable\noalign{\hrule}\opentable
\multispan4\footnotesize* Based upon linear regression analysis.\hfil\cr}}
\end{table}

\makeatletter
\if@twoside    
\eject
\markboth{\uppercase{\ch@ptern@me\hfill\ifnum\c@secnumdepth>\m@ne
    \@chapabb\ {\numbersize\thechapter}.\fi}}{}
\fi
\makeatother

\begin{figure}
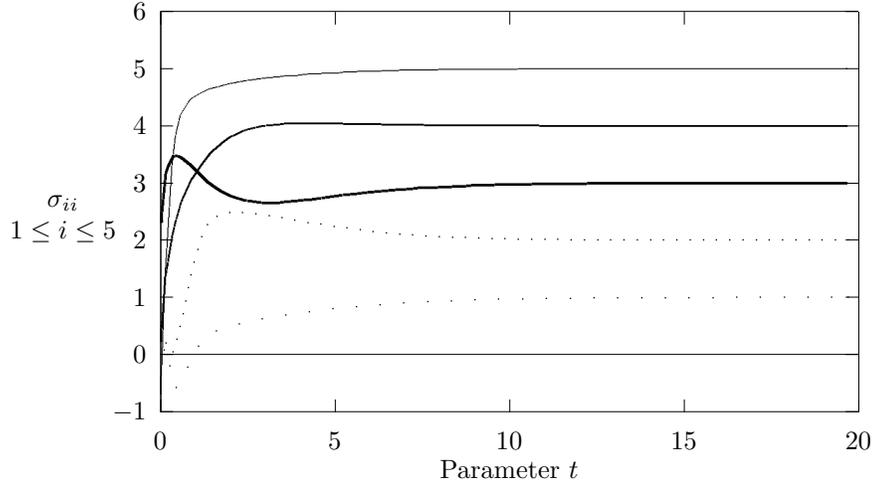

\begin{gnuplot}{10pt}
\input svd7x5flow
\end{gnuplot}
\caption[Gradient flow of singular value
decomposition]{\protect\smallbmit Gradient flow of the singular value
decomposition on ${\bmit K}_{\{1,2,3,4,5\}}\subset\R^{7\times 5}$.
The five diagonal elements of $\Sigma(t)$ satisfying
Eq.~(\ref{eq:svdflow}a) are shown.\label{fig:svdflow}}
\end{figure}

\begin{figure}
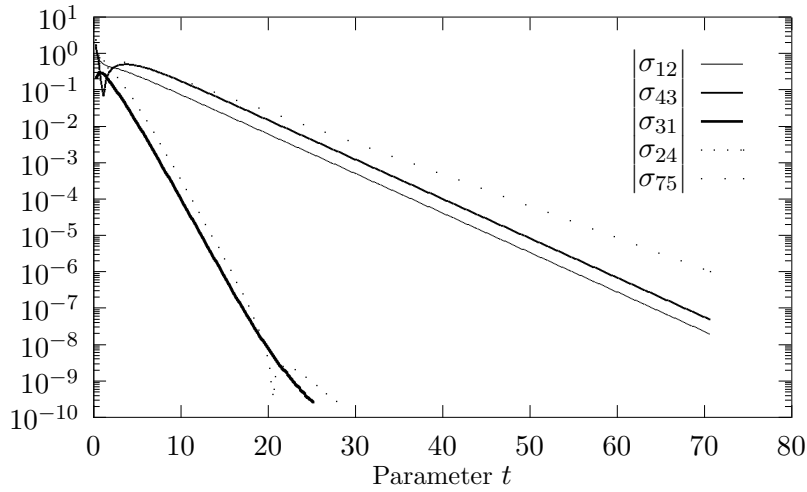

\begin{gnuplot}{10pt}
\input svd7x5offd
\end{gnuplot}
\caption[Off-diagonal convergence of {\tenrm SVD} gradient flow]
{Off-diagonal convergence of {\tenrm SVD} gradient flow on ${\bmit
K}_{\{1,2,3,4,5\}} \subset\R^{7\times 5}$. Selected off-diagonal
elements of~$\Sigma(t)$ satisfying Eq.~(\ref{eq:svdflow}a) are
shown.\label{fig:svdconv}}
\end{figure}

%% file: svd7x5flow.tex
\setlength{\unitlength}{0.240900pt}
\ifx\plotpoint\undefined\newsavebox{\plotpoint}\fi
\sbox{\plotpoint}{\rule[-0.175pt]{0.350pt}{0.350pt}}%
\begin{picture}(1424,900)(0,0)
\tenrm
\sbox{\plotpoint}{\rule[-0.175pt]{0.350pt}{0.350pt}}%
\put(264,248){\rule[-0.175pt]{264.026pt}{0.350pt}}
\put(264,158){\rule[-0.175pt]{0.350pt}{151.526pt}}
\put(264,158){\rule[-0.175pt]{4.818pt}{0.350pt}}
\put(242,158){\makebox(0,0)[r]{$-1$}}
\put(1340,158){\rule[-0.175pt]{4.818pt}{0.350pt}}
\put(264,248){\rule[-0.175pt]{4.818pt}{0.350pt}}
\put(242,248){\makebox(0,0)[r]{$0$}}
\put(1340,248){\rule[-0.175pt]{4.818pt}{0.350pt}}
\put(264,338){\rule[-0.175pt]{4.818pt}{0.350pt}}
\put(242,338){\makebox(0,0)[r]{$1$}}
\put(1340,338){\rule[-0.175pt]{4.818pt}{0.350pt}}
\put(264,428){\rule[-0.175pt]{4.818pt}{0.350pt}}
\put(242,428){\makebox(0,0)[r]{$2$}}
\put(1340,428){\rule[-0.175pt]{4.818pt}{0.350pt}}
\put(264,517){\rule[-0.175pt]{4.818pt}{0.350pt}}
\put(242,517){\makebox(0,0)[r]{$3$}}
\put(1340,517){\rule[-0.175pt]{4.818pt}{0.350pt}}
\put(264,607){\rule[-0.175pt]{4.818pt}{0.350pt}}
\put(242,607){\makebox(0,0)[r]{$4$}}
\put(1340,607){\rule[-0.175pt]{4.818pt}{0.350pt}}
\put(264,697){\rule[-0.175pt]{4.818pt}{0.350pt}}
\put(242,697){\makebox(0,0)[r]{$5$}}
\put(1340,697){\rule[-0.175pt]{4.818pt}{0.350pt}}
\put(264,787){\rule[-0.175pt]{4.818pt}{0.350pt}}
\put(242,787){\makebox(0,0)[r]{$6$}}
\put(1340,787){\rule[-0.175pt]{4.818pt}{0.350pt}}
\put(264,158){\rule[-0.175pt]{0.350pt}{4.818pt}}
\put(264,113){\makebox(0,0){$0$}}
\put(264,767){\rule[-0.175pt]{0.350pt}{4.818pt}}
\put(538,158){\rule[-0.175pt]{0.350pt}{4.818pt}}
\put(538,113){\makebox(0,0){$5$}}
\put(538,767){\rule[-0.175pt]{0.350pt}{4.818pt}}
\put(812,158){\rule[-0.175pt]{0.350pt}{4.818pt}}
\put(812,113){\makebox(0,0){$10$}}
\put(812,767){\rule[-0.175pt]{0.350pt}{4.818pt}}
\put(1086,158){\rule[-0.175pt]{0.350pt}{4.818pt}}
\put(1086,113){\makebox(0,0){$15$}}
\put(1086,767){\rule[-0.175pt]{0.350pt}{4.818pt}}
\put(1360,158){\rule[-0.175pt]{0.350pt}{4.818pt}}
\put(1360,113){\makebox(0,0){$20$}}
\put(1360,767){\rule[-0.175pt]{0.350pt}{4.818pt}}
\put(264,158){\rule[-0.175pt]{264.026pt}{0.350pt}}
\put(1360,158){\rule[-0.175pt]{0.350pt}{151.526pt}}
\put(264,787){\rule[-0.175pt]{264.026pt}{0.350pt}}
\put(23,472){\makebox(0,0)[l]{\shortstack{${\textstyle\strut\sigma_{ii}\atop\textstyle1\le i\le5}$}}}
\put(812,68){\makebox(0,0){Parameter $t$}}
\put(264,158){\rule[-0.175pt]{0.350pt}{151.526pt}}
\put(264,186){\usebox{\plotpoint}}
\put(264,186){\rule[-0.175pt]{0.350pt}{6.866pt}}
\put(265,214){\rule[-0.175pt]{0.350pt}{4.216pt}}
\put(266,232){\rule[-0.175pt]{0.350pt}{6.649pt}}
\put(267,259){\rule[-0.175pt]{0.350pt}{6.649pt}}
\put(268,287){\rule[-0.175pt]{0.350pt}{6.649pt}}
\put(269,314){\rule[-0.175pt]{0.350pt}{6.649pt}}
\put(270,342){\rule[-0.175pt]{0.350pt}{6.649pt}}
\put(271,370){\rule[-0.175pt]{0.350pt}{3.694pt}}
\put(272,385){\rule[-0.175pt]{0.350pt}{3.694pt}}
\put(273,400){\rule[-0.175pt]{0.350pt}{3.694pt}}
\put(274,416){\rule[-0.175pt]{0.350pt}{3.915pt}}
\put(275,432){\rule[-0.175pt]{0.350pt}{3.915pt}}
\put(276,448){\rule[-0.175pt]{0.350pt}{3.915pt}}
\put(277,464){\rule[-0.175pt]{0.350pt}{3.915pt}}
\put(278,481){\rule[-0.175pt]{0.350pt}{3.915pt}}
\put(279,497){\rule[-0.175pt]{0.350pt}{3.915pt}}
\put(280,513){\rule[-0.175pt]{0.350pt}{3.915pt}}
\put(281,529){\rule[-0.175pt]{0.350pt}{3.915pt}}
\put(282,546){\rule[-0.175pt]{0.350pt}{2.072pt}}
\put(283,554){\rule[-0.175pt]{0.350pt}{2.072pt}}
\put(284,563){\rule[-0.175pt]{0.350pt}{2.072pt}}
\put(285,571){\rule[-0.175pt]{0.350pt}{2.072pt}}
\put(286,580){\rule[-0.175pt]{0.350pt}{2.072pt}}
\put(287,588){\rule[-0.175pt]{0.350pt}{1.205pt}}
\put(288,594){\rule[-0.175pt]{0.350pt}{1.204pt}}
\put(289,599){\rule[-0.175pt]{0.350pt}{1.044pt}}
\put(290,603){\rule[-0.175pt]{0.350pt}{1.044pt}}
\put(291,607){\rule[-0.175pt]{0.350pt}{1.044pt}}
\put(292,611){\rule[-0.175pt]{0.350pt}{1.044pt}}
\put(293,616){\rule[-0.175pt]{0.350pt}{1.044pt}}
\put(294,620){\rule[-0.175pt]{0.350pt}{1.044pt}}
\put(295,624){\rule[-0.175pt]{0.350pt}{0.402pt}}
\put(296,626){\rule[-0.175pt]{0.350pt}{0.402pt}}
\put(297,628){\rule[-0.175pt]{0.350pt}{0.401pt}}
\put(298,630){\rule[-0.175pt]{0.350pt}{0.385pt}}
\put(299,631){\rule[-0.175pt]{0.350pt}{0.385pt}}
\put(300,633){\rule[-0.175pt]{0.350pt}{0.385pt}}
\put(301,634){\rule[-0.175pt]{0.350pt}{0.385pt}}
\put(302,636){\rule[-0.175pt]{0.350pt}{0.385pt}}
\put(303,637){\rule[-0.175pt]{0.350pt}{0.385pt}}
\put(304,639){\rule[-0.175pt]{0.350pt}{0.385pt}}
\put(305,641){\rule[-0.175pt]{0.350pt}{0.385pt}}
\put(306,642){\rule[-0.175pt]{0.350pt}{0.385pt}}
\put(307,644){\rule[-0.175pt]{0.350pt}{0.385pt}}
\put(308,645){\usebox{\plotpoint}}
\put(308,646){\usebox{\plotpoint}}
\put(309,647){\usebox{\plotpoint}}
\put(310,648){\usebox{\plotpoint}}
\put(311,649){\usebox{\plotpoint}}
\put(312,650){\rule[-0.175pt]{0.410pt}{0.350pt}}
\put(313,651){\rule[-0.175pt]{0.410pt}{0.350pt}}
\put(315,652){\rule[-0.175pt]{0.410pt}{0.350pt}}
\put(317,653){\rule[-0.175pt]{0.410pt}{0.350pt}}
\put(318,654){\rule[-0.175pt]{0.410pt}{0.350pt}}
\put(320,655){\rule[-0.175pt]{0.410pt}{0.350pt}}
\put(322,656){\rule[-0.175pt]{0.410pt}{0.350pt}}
\put(323,657){\rule[-0.175pt]{0.410pt}{0.350pt}}
\put(325,658){\rule[-0.175pt]{0.410pt}{0.350pt}}
\put(327,659){\rule[-0.175pt]{0.410pt}{0.350pt}}
\put(329,660){\rule[-0.175pt]{0.482pt}{0.350pt}}
\put(331,661){\rule[-0.175pt]{0.482pt}{0.350pt}}
\put(333,662){\rule[-0.175pt]{0.482pt}{0.350pt}}
\put(335,663){\rule[-0.175pt]{0.482pt}{0.350pt}}
\put(337,664){\rule[-0.175pt]{0.482pt}{0.350pt}}
\put(339,665){\rule[-0.175pt]{0.723pt}{0.350pt}}
\put(342,666){\rule[-0.175pt]{1.044pt}{0.350pt}}
\put(346,667){\rule[-0.175pt]{1.044pt}{0.350pt}}
\put(350,668){\rule[-0.175pt]{1.044pt}{0.350pt}}
\put(355,669){\rule[-0.175pt]{0.723pt}{0.350pt}}
\put(358,670){\rule[-0.175pt]{0.903pt}{0.350pt}}
\put(361,671){\rule[-0.175pt]{0.903pt}{0.350pt}}
\put(365,672){\rule[-0.175pt]{0.903pt}{0.350pt}}
\put(369,673){\rule[-0.175pt]{0.903pt}{0.350pt}}
\put(373,674){\rule[-0.175pt]{1.204pt}{0.350pt}}
\put(378,675){\rule[-0.175pt]{1.365pt}{0.350pt}}
\put(383,676){\rule[-0.175pt]{1.365pt}{0.350pt}}
\put(389,677){\rule[-0.175pt]{1.365pt}{0.350pt}}
\put(394,678){\rule[-0.175pt]{1.445pt}{0.350pt}}
\put(401,679){\rule[-0.175pt]{1.767pt}{0.350pt}}
\put(408,680){\rule[-0.175pt]{1.767pt}{0.350pt}}
\put(415,681){\rule[-0.175pt]{1.767pt}{0.350pt}}
\put(423,682){\rule[-0.175pt]{1.686pt}{0.350pt}}
\put(430,683){\rule[-0.175pt]{2.650pt}{0.350pt}}
\put(441,684){\rule[-0.175pt]{2.650pt}{0.350pt}}
\put(452,685){\rule[-0.175pt]{1.445pt}{0.350pt}}
\put(458,686){\rule[-0.175pt]{5.300pt}{0.350pt}}
\put(480,687){\rule[-0.175pt]{1.686pt}{0.350pt}}
\put(487,688){\rule[-0.175pt]{1.927pt}{0.350pt}}
\put(495,689){\rule[-0.175pt]{5.300pt}{0.350pt}}
\put(517,690){\rule[-0.175pt]{6.745pt}{0.350pt}}
\put(545,691){\rule[-0.175pt]{1.445pt}{0.350pt}}
\put(551,692){\rule[-0.175pt]{5.300pt}{0.350pt}}
\put(573,693){\rule[-0.175pt]{9.154pt}{0.350pt}}
\put(611,694){\rule[-0.175pt]{6.745pt}{0.350pt}}
\put(639,695){\rule[-0.175pt]{14.695pt}{0.350pt}}
\put(700,696){\rule[-0.175pt]{30.112pt}{0.350pt}}
\put(825,697){\rule[-0.175pt]{124.786pt}{0.350pt}}
\sbox{\plotpoint}{\rule[-0.350pt]{0.700pt}{0.700pt}}%
\put(264,269){\usebox{\plotpoint}}
\put(264,269){\rule[-0.350pt]{0.700pt}{5.059pt}}
\put(265,290){\rule[-0.350pt]{0.700pt}{2.891pt}}
\put(266,302){\rule[-0.350pt]{0.700pt}{3.373pt}}
\put(267,316){\rule[-0.350pt]{0.700pt}{3.373pt}}
\put(268,330){\rule[-0.350pt]{0.700pt}{3.373pt}}
\put(269,344){\rule[-0.350pt]{0.700pt}{3.373pt}}
\put(270,358){\rule[-0.350pt]{0.700pt}{3.373pt}}
\put(271,372){\rule[-0.350pt]{0.700pt}{1.365pt}}
\put(272,377){\rule[-0.350pt]{0.700pt}{1.365pt}}
\put(273,383){\rule[-0.350pt]{0.700pt}{1.365pt}}
\put(274,388){\rule[-0.350pt]{0.700pt}{1.445pt}}
\put(275,395){\rule[-0.350pt]{0.700pt}{1.445pt}}
\put(276,401){\rule[-0.350pt]{0.700pt}{1.445pt}}
\put(277,407){\rule[-0.350pt]{0.700pt}{1.445pt}}
\put(278,413){\rule[-0.350pt]{0.700pt}{1.445pt}}
\put(279,419){\rule[-0.350pt]{0.700pt}{1.445pt}}
\put(280,425){\rule[-0.350pt]{0.700pt}{1.445pt}}
\put(281,431){\rule[-0.350pt]{0.700pt}{1.445pt}}
\put(282,437){\rule[-0.350pt]{0.700pt}{1.060pt}}
\put(283,441){\rule[-0.350pt]{0.700pt}{1.060pt}}
\put(284,445){\rule[-0.350pt]{0.700pt}{1.060pt}}
\put(285,450){\rule[-0.350pt]{0.700pt}{1.060pt}}
\put(286,454){\rule[-0.350pt]{0.700pt}{1.060pt}}
\put(287,458){\rule[-0.350pt]{0.700pt}{0.723pt}}
\put(288,462){\rule[-0.350pt]{0.700pt}{0.723pt}}
\put(289,465){\rule[-0.350pt]{0.700pt}{0.883pt}}
\put(290,468){\rule[-0.350pt]{0.700pt}{0.883pt}}
\put(291,472){\rule[-0.350pt]{0.700pt}{0.883pt}}
\put(292,475){\rule[-0.350pt]{0.700pt}{0.883pt}}
\put(293,479){\rule[-0.350pt]{0.700pt}{0.883pt}}
\put(294,483){\rule[-0.350pt]{0.700pt}{0.883pt}}
\put(295,486){\usebox{\plotpoint}}
\put(296,489){\usebox{\plotpoint}}
\put(297,491){\usebox{\plotpoint}}
\put(298,494){\usebox{\plotpoint}}
\put(299,496){\usebox{\plotpoint}}
\put(300,498){\usebox{\plotpoint}}
\put(301,500){\usebox{\plotpoint}}
\put(302,502){\usebox{\plotpoint}}
\put(303,505){\usebox{\plotpoint}}
\put(304,507){\usebox{\plotpoint}}
\put(305,509){\usebox{\plotpoint}}
\put(306,511){\usebox{\plotpoint}}
\put(307,513){\usebox{\plotpoint}}
\put(308,516){\usebox{\plotpoint}}
\put(309,518){\usebox{\plotpoint}}
\put(310,520){\usebox{\plotpoint}}
\put(311,522){\usebox{\plotpoint}}
\put(312,524){\usebox{\plotpoint}}
\put(313,525){\usebox{\plotpoint}}
\put(314,526){\usebox{\plotpoint}}
\put(315,528){\usebox{\plotpoint}}
\put(316,529){\usebox{\plotpoint}}
\put(317,531){\usebox{\plotpoint}}
\put(318,532){\usebox{\plotpoint}}
\put(319,534){\usebox{\plotpoint}}
\put(320,535){\usebox{\plotpoint}}
\put(321,537){\usebox{\plotpoint}}
\put(322,538){\usebox{\plotpoint}}
\put(323,540){\usebox{\plotpoint}}
\put(324,541){\usebox{\plotpoint}}
\put(325,543){\usebox{\plotpoint}}
\put(326,544){\usebox{\plotpoint}}
\put(327,546){\usebox{\plotpoint}}
\put(328,547){\usebox{\plotpoint}}
\put(329,548){\usebox{\plotpoint}}
\put(330,550){\usebox{\plotpoint}}
\put(331,551){\usebox{\plotpoint}}
\put(332,552){\usebox{\plotpoint}}
\put(333,554){\usebox{\plotpoint}}
\put(334,555){\usebox{\plotpoint}}
\put(335,556){\usebox{\plotpoint}}
\put(336,558){\usebox{\plotpoint}}
\put(337,559){\usebox{\plotpoint}}
\put(338,560){\usebox{\plotpoint}}
\put(339,561){\usebox{\plotpoint}}
\put(339,562){\usebox{\plotpoint}}
\put(340,563){\usebox{\plotpoint}}
\put(341,564){\usebox{\plotpoint}}
\put(342,565){\usebox{\plotpoint}}
\put(343,566){\usebox{\plotpoint}}
\put(344,567){\usebox{\plotpoint}}
\put(345,568){\usebox{\plotpoint}}
\put(346,569){\usebox{\plotpoint}}
\put(347,570){\usebox{\plotpoint}}
\put(349,571){\usebox{\plotpoint}}
\put(350,572){\usebox{\plotpoint}}
\put(351,573){\usebox{\plotpoint}}
\put(352,574){\usebox{\plotpoint}}
\put(353,575){\usebox{\plotpoint}}
\put(355,576){\usebox{\plotpoint}}
\put(356,577){\usebox{\plotpoint}}
\put(357,578){\usebox{\plotpoint}}
\put(358,579){\usebox{\plotpoint}}
\put(359,580){\usebox{\plotpoint}}
\put(360,581){\usebox{\plotpoint}}
\put(362,582){\usebox{\plotpoint}}
\put(363,583){\usebox{\plotpoint}}
\put(364,584){\usebox{\plotpoint}}
\put(366,585){\usebox{\plotpoint}}
\put(367,586){\usebox{\plotpoint}}
\put(368,587){\usebox{\plotpoint}}
\put(370,588){\usebox{\plotpoint}}
\put(371,589){\usebox{\plotpoint}}
\put(373,590){\usebox{\plotpoint}}
\put(375,591){\usebox{\plotpoint}}
\put(378,592){\usebox{\plotpoint}}
\put(380,593){\usebox{\plotpoint}}
\put(382,594){\usebox{\plotpoint}}
\put(384,595){\usebox{\plotpoint}}
\put(386,596){\usebox{\plotpoint}}
\put(388,597){\usebox{\plotpoint}}
\put(390,598){\usebox{\plotpoint}}
\put(392,599){\usebox{\plotpoint}}
\put(395,600){\rule[-0.350pt]{0.723pt}{0.700pt}}
\put(398,601){\rule[-0.350pt]{0.723pt}{0.700pt}}
\put(401,602){\rule[-0.350pt]{1.060pt}{0.700pt}}
\put(405,603){\rule[-0.350pt]{1.060pt}{0.700pt}}
\put(409,604){\rule[-0.350pt]{1.060pt}{0.700pt}}
\put(414,605){\rule[-0.350pt]{1.060pt}{0.700pt}}
\put(418,606){\rule[-0.350pt]{1.060pt}{0.700pt}}
\put(422,607){\rule[-0.350pt]{1.686pt}{0.700pt}}
\put(430,608){\rule[-0.350pt]{2.650pt}{0.700pt}}
\put(441,609){\rule[-0.350pt]{2.650pt}{0.700pt}}
\put(452,610){\rule[-0.350pt]{1.445pt}{0.700pt}}
\put(458,611){\rule[-0.350pt]{22.404pt}{0.700pt}}
\put(551,610){\rule[-0.350pt]{14.454pt}{0.700pt}}
\put(611,609){\rule[-0.350pt]{16.140pt}{0.700pt}}
\put(678,608){\rule[-0.350pt]{43.603pt}{0.700pt}}
\put(859,607){\rule[-0.350pt]{116.596pt}{0.700pt}}
\sbox{\plotpoint}{\rule[-0.500pt]{1.000pt}{1.000pt}}%
\put(264,457){\usebox{\plotpoint}}
\put(264,457){\rule[-0.500pt]{1.000pt}{4.095pt}}
\put(265,474){\rule[-0.500pt]{1.000pt}{2.409pt}}
\put(266,484){\rule[-0.500pt]{1.000pt}{2.409pt}}
\put(267,494){\rule[-0.500pt]{1.000pt}{2.409pt}}
\put(268,504){\rule[-0.500pt]{1.000pt}{2.409pt}}
\put(269,514){\rule[-0.500pt]{1.000pt}{2.409pt}}
\put(270,524){\rule[-0.500pt]{1.000pt}{2.409pt}}
\put(271,534){\usebox{\plotpoint}}
\put(272,537){\usebox{\plotpoint}}
\put(273,540){\usebox{\plotpoint}}
\put(274,543){\usebox{\plotpoint}}
\put(275,545){\usebox{\plotpoint}}
\put(276,547){\usebox{\plotpoint}}
\put(277,549){\usebox{\plotpoint}}
\put(278,551){\usebox{\plotpoint}}
\put(279,553){\usebox{\plotpoint}}
\put(280,555){\usebox{\plotpoint}}
\put(281,557){\usebox{\plotpoint}}
\put(282,559){\rule[-0.500pt]{1.204pt}{1.000pt}}
\put(287,560){\usebox{\plotpoint}}
\put(291,559){\usebox{\plotpoint}}
\put(293,558){\usebox{\plotpoint}}
\put(295,557){\usebox{\plotpoint}}
\put(296,556){\usebox{\plotpoint}}
\put(298,555){\usebox{\plotpoint}}
\put(299,554){\usebox{\plotpoint}}
\put(300,553){\usebox{\plotpoint}}
\put(301,552){\usebox{\plotpoint}}
\put(303,551){\usebox{\plotpoint}}
\put(304,550){\usebox{\plotpoint}}
\put(305,549){\usebox{\plotpoint}}
\put(306,548){\usebox{\plotpoint}}
\put(308,547){\usebox{\plotpoint}}
\put(309,546){\usebox{\plotpoint}}
\put(310,545){\usebox{\plotpoint}}
\put(311,544){\usebox{\plotpoint}}
\put(312,543){\usebox{\plotpoint}}
\put(313,542){\usebox{\plotpoint}}
\put(314,541){\usebox{\plotpoint}}
\put(315,540){\usebox{\plotpoint}}
\put(316,539){\usebox{\plotpoint}}
\put(317,538){\usebox{\plotpoint}}
\put(318,537){\usebox{\plotpoint}}
\put(319,536){\usebox{\plotpoint}}
\put(320,535){\usebox{\plotpoint}}
\put(321,534){\usebox{\plotpoint}}
\put(322,533){\usebox{\plotpoint}}
\put(323,532){\usebox{\plotpoint}}
\put(324,531){\usebox{\plotpoint}}
\put(325,530){\usebox{\plotpoint}}
\put(326,529){\usebox{\plotpoint}}
\put(327,528){\usebox{\plotpoint}}
\put(328,527){\usebox{\plotpoint}}
\put(329,526){\usebox{\plotpoint}}
\put(330,525){\usebox{\plotpoint}}
\put(331,524){\usebox{\plotpoint}}
\put(332,523){\usebox{\plotpoint}}
\put(333,522){\usebox{\plotpoint}}
\put(334,521){\usebox{\plotpoint}}
\put(335,520){\usebox{\plotpoint}}
\put(336,519){\usebox{\plotpoint}}
\put(337,518){\usebox{\plotpoint}}
\put(339,517){\usebox{\plotpoint}}
\put(340,516){\usebox{\plotpoint}}
\put(342,515){\usebox{\plotpoint}}
\put(343,514){\usebox{\plotpoint}}
\put(344,513){\usebox{\plotpoint}}
\put(346,512){\usebox{\plotpoint}}
\put(347,511){\usebox{\plotpoint}}
\put(349,510){\usebox{\plotpoint}}
\put(350,509){\usebox{\plotpoint}}
\put(352,508){\usebox{\plotpoint}}
\put(353,507){\usebox{\plotpoint}}
\put(355,506){\usebox{\plotpoint}}
\put(356,505){\usebox{\plotpoint}}
\put(358,504){\usebox{\plotpoint}}
\put(360,503){\usebox{\plotpoint}}
\put(362,502){\usebox{\plotpoint}}
\put(364,501){\usebox{\plotpoint}}
\put(366,500){\usebox{\plotpoint}}
\put(368,499){\usebox{\plotpoint}}
\put(370,498){\usebox{\plotpoint}}
\put(372,497){\usebox{\plotpoint}}
\put(375,496){\usebox{\plotpoint}}
\put(378,495){\usebox{\plotpoint}}
\put(381,494){\usebox{\plotpoint}}
\put(384,493){\usebox{\plotpoint}}
\put(388,492){\usebox{\plotpoint}}
\put(391,491){\usebox{\plotpoint}}
\put(394,490){\rule[-0.500pt]{1.445pt}{1.000pt}}
\put(401,489){\rule[-0.500pt]{1.767pt}{1.000pt}}
\put(408,488){\rule[-0.500pt]{1.767pt}{1.000pt}}
\put(415,487){\rule[-0.500pt]{1.767pt}{1.000pt}}
\put(423,486){\rule[-0.500pt]{6.986pt}{1.000pt}}
\put(452,487){\rule[-0.500pt]{1.445pt}{1.000pt}}
\put(458,488){\rule[-0.500pt]{2.650pt}{1.000pt}}
\put(469,489){\rule[-0.500pt]{2.650pt}{1.000pt}}
\put(480,490){\rule[-0.500pt]{2.168pt}{1.000pt}}
\put(489,491){\rule[-0.500pt]{3.212pt}{1.000pt}}
\put(502,492){\rule[-0.500pt]{1.767pt}{1.000pt}}
\put(509,493){\rule[-0.500pt]{1.767pt}{1.000pt}}
\put(517,494){\rule[-0.500pt]{1.445pt}{1.000pt}}
\put(523,495){\rule[-0.500pt]{1.767pt}{1.000pt}}
\put(530,496){\rule[-0.500pt]{1.767pt}{1.000pt}}
\put(537,497){\rule[-0.500pt]{1.767pt}{1.000pt}}
\put(544,498){\rule[-0.500pt]{1.445pt}{1.000pt}}
\put(551,499){\rule[-0.500pt]{2.650pt}{1.000pt}}
\put(562,500){\rule[-0.500pt]{2.650pt}{1.000pt}}
\put(573,501){\rule[-0.500pt]{1.686pt}{1.000pt}}
\put(580,502){\rule[-0.500pt]{2.168pt}{1.000pt}}
\put(589,503){\rule[-0.500pt]{2.650pt}{1.000pt}}
\put(600,504){\rule[-0.500pt]{2.650pt}{1.000pt}}
\put(611,505){\rule[-0.500pt]{2.650pt}{1.000pt}}
\put(622,506){\rule[-0.500pt]{2.650pt}{1.000pt}}
\put(633,507){\rule[-0.500pt]{4.095pt}{1.000pt}}
\put(650,508){\rule[-0.500pt]{2.650pt}{1.000pt}}
\put(661,509){\rule[-0.500pt]{1.686pt}{1.000pt}}
\put(668,510){\rule[-0.500pt]{7.709pt}{1.000pt}}
\put(700,511){\rule[-0.500pt]{5.300pt}{1.000pt}}
\put(722,512){\rule[-0.500pt]{5.300pt}{1.000pt}}
\put(744,513){\rule[-0.500pt]{1.927pt}{1.000pt}}
\put(752,514){\rule[-0.500pt]{8.913pt}{1.000pt}}
\put(789,515){\rule[-0.500pt]{13.972pt}{1.000pt}}
\put(847,516){\rule[-0.500pt]{20.958pt}{1.000pt}}
\put(934,517){\rule[-0.500pt]{98.528pt}{1.000pt}}
\sbox{\plotpoint}{\rule[-0.250pt]{0.500pt}{0.500pt}}%
\put(264,277){\usebox{\plotpoint}}
\put(264,275){\usebox{\plotpoint}}
\put(269,255){\usebox{\plotpoint}}
\put(283,252){\usebox{\plotpoint}}
\put(289,272){\usebox{\plotpoint}}
\put(294,292){\usebox{\plotpoint}}
\put(299,312){\usebox{\plotpoint}}
\put(303,332){\usebox{\plotpoint}}
\put(307,353){\usebox{\plotpoint}}
\put(311,373){\usebox{\plotpoint}}
\put(316,393){\usebox{\plotpoint}}
\put(322,413){\usebox{\plotpoint}}
\put(328,433){\usebox{\plotpoint}}
\put(337,452){\usebox{\plotpoint}}
\put(352,465){\usebox{\plotpoint}}
\put(372,471){\usebox{\plotpoint}}
\put(392,471){\usebox{\plotpoint}}
\put(413,469){\usebox{\plotpoint}}
\put(434,466){\usebox{\plotpoint}}
\put(454,463){\usebox{\plotpoint}}
\put(474,459){\usebox{\plotpoint}}
\put(495,455){\usebox{\plotpoint}}
\put(515,453){\usebox{\plotpoint}}
\put(536,449){\usebox{\plotpoint}}
\put(556,447){\usebox{\plotpoint}}
\put(577,444){\usebox{\plotpoint}}
\put(597,441){\usebox{\plotpoint}}
\put(618,439){\usebox{\plotpoint}}
\put(639,437){\usebox{\plotpoint}}
\put(659,436){\usebox{\plotpoint}}
\put(680,434){\usebox{\plotpoint}}
\put(701,433){\usebox{\plotpoint}}
\put(722,432){\usebox{\plotpoint}}
\put(742,432){\usebox{\plotpoint}}
\put(763,431){\usebox{\plotpoint}}
\put(784,430){\usebox{\plotpoint}}
\put(804,430){\usebox{\plotpoint}}
\put(825,429){\usebox{\plotpoint}}
\put(846,429){\usebox{\plotpoint}}
\put(867,429){\usebox{\plotpoint}}
\put(887,429){\usebox{\plotpoint}}
\put(908,428){\usebox{\plotpoint}}
\put(929,428){\usebox{\plotpoint}}
\put(950,428){\usebox{\plotpoint}}
\put(970,428){\usebox{\plotpoint}}
\put(991,428){\usebox{\plotpoint}}
\put(1012,428){\usebox{\plotpoint}}
\put(1033,428){\usebox{\plotpoint}}
\put(1053,428){\usebox{\plotpoint}}
\put(1074,428){\usebox{\plotpoint}}
\put(1095,428){\usebox{\plotpoint}}
\put(1116,428){\usebox{\plotpoint}}
\put(1136,428){\usebox{\plotpoint}}
\put(1157,428){\usebox{\plotpoint}}
\put(1178,428){\usebox{\plotpoint}}
\put(1199,428){\usebox{\plotpoint}}
\put(1219,428){\usebox{\plotpoint}}
\put(1240,428){\usebox{\plotpoint}}
\put(1261,428){\usebox{\plotpoint}}
\put(1282,428){\usebox{\plotpoint}}
\put(1302,428){\usebox{\plotpoint}}
\put(1323,428){\usebox{\plotpoint}}
\put(1343,428){\usebox{\plotpoint}}
\put(264,293){\usebox{\plotpoint}}
\put(264,293){\usebox{\plotpoint}}
\put(272,266){\usebox{\plotpoint}}
\put(278,230){\usebox{\plotpoint}}
\put(288,196){\usebox{\plotpoint}}
\put(306,228){\usebox{\plotpoint}}
\put(323,261){\usebox{\plotpoint}}
\put(351,285){\usebox{\plotpoint}}
\put(387,296){\usebox{\plotpoint}}
\put(423,304){\usebox{\plotpoint}}
\put(460,310){\usebox{\plotpoint}}
\put(497,315){\usebox{\plotpoint}}
\put(534,320){\usebox{\plotpoint}}
\put(571,323){\usebox{\plotpoint}}
\put(608,326){\usebox{\plotpoint}}
\put(645,329){\usebox{\plotpoint}}
\put(683,331){\usebox{\plotpoint}}
\put(720,332){\usebox{\plotpoint}}
\put(757,334){\usebox{\plotpoint}}
\put(794,335){\usebox{\plotpoint}}
\put(832,335){\usebox{\plotpoint}}
\put(869,336){\usebox{\plotpoint}}
\put(907,336){\usebox{\plotpoint}}
\put(944,337){\usebox{\plotpoint}}
\put(981,337){\usebox{\plotpoint}}
\put(1019,337){\usebox{\plotpoint}}
\put(1056,337){\usebox{\plotpoint}}
\put(1093,337){\usebox{\plotpoint}}
\put(1131,337){\usebox{\plotpoint}}
\put(1168,338){\usebox{\plotpoint}}
\put(1205,338){\usebox{\plotpoint}}
\put(1243,338){\usebox{\plotpoint}}
\put(1280,338){\usebox{\plotpoint}}
\put(1317,338){\usebox{\plotpoint}}
\put(1343,338){\usebox{\plotpoint}}
\end{picture}

%% file: chap-orm.tex
\chapter{Optimization on Riemannian Manifolds}\label{chap:orm}

The preponderance of optimization techniques address problems posed on
Euclidean spaces.  Indeed, several fundamental algorithms have arisen
from the desire to compute the minimum of quadratic forms on Euclidean
space.  However, many optimization problems are posed on non-Euclidean
spaces.  For example, finding the largest eigenvalue of a symmetric
matrix may be posed as the maximization of the Rayleigh quotient
defined on the sphere.  Optimization problems subject to nonlinear
differentiable equality constraints on Euclidean space also lie within
this category.  Many optimization problems share with these examples
the structure of a differentiable manifold endowed with a Riemannian
metric.  This is the subject of this chapter: the extremization of
functions defined on Riemannian manifolds.

The minimization of functions on a Riemannian manifold is, at least
locally, equivalent to the smoothly constrained optimization problem
on a Euclidean space, because every $C^\infty$ Riemannian manifold can
be isometrically imbedded in some Euclidean
space~\cite[Vol.~5]{Spivak}.  However, the dimension of the Euclidean
space may be larger than the dimension of the manifold; practical and
aesthetic considerations suggest that one try to exploit the intrinsic
structure of the manifold.  Elements of this spirit may be found
throughout the field of numerical methods, such as the emphasis on
unitary (norm preserving) transformations in numerical linear
algebra~\cite{GVL}, or the use of feasible direction
methods~\cite{Fletcher,GillMurray,Sargent}.

An intrinsic approach leads one from the extrinsic idea of vector
addition to the exponential map and parallel translation, from
minimization along lines to minimization along geodesics, and from
partial differentiation to covariant differentiation.  The computation
of geodesics, parallel translation, and covariant derivatives can be
quite expensive. For an $n\hyphen$dimensional manifold, the
computation of geodesics and parallel translation requires the
solution of a system of $2n$ nonlinear and $n$ linear ordinary
differential equations.  Nevertheless, many optimization problems are
posed on manifolds that have an underlying structure that may be
exploited to greatly reduce the complexity of these computations.  For
example, on a real compact semisimple Lie group endowed with its
natural Riemannian metric, geodesics and parallel translation may be
computed via matrix exponentiation~\cite{Helgason}. Several algorithms
are available to perform this computation~\cite{GVL,nineteendubious}.
This structure may be found in the problems posed by
Brockett~\citeyear{Brockett:match,Brockett:sort,Brockett:grad}, Bloch
et~al.~\citeyear{BBR1,BBR2}, Smith~\citeyear{Me},
Faybusovich~\citeyear{Leonid}, Lagarias~\citeyear{Lagarias}, Chu
et~al.~\citeyear{Chu:sphere,Chu:grad}, Perkins
et~al.~\citeyear{balreal}, and Helmke~\citeyear{Uwe}.  This approach
is also applicable if the manifold can be identified with a symmetric
space or, excepting parallel translation, a reductive homogeneous
space~\cite{Nomizu,KobayashiandNomizu}.  Perhaps the simplest
nontrivial example is the sphere, where geodesics and parallel
translation can be computed at low cost with trigonometric functions
and vector addition.  If the reductive homogeneous space does not have
a symmetric space structure, the result of
Proposition~\ref{prop:ptredhom} of Chapter~\ref{chap:geom} can be used
to compute the parallel translation of arbitrary vectors along
geodesics.  Furthermore, Brown and Bartholomew-Biggs~\citeyear{Bcubed}
show that in some cases function minimization by following the
solution of a system of ordinary differential equations can be
implemented so as to make it competitive with conventional techniques.

The outline of the chapter is as follows. In Section~\ref{sec:prelim},
the optimization problem is posed and conventions to be held
throughout the chapter are established.  The method of steepest
descent on a Riemannian manifold is described in
Section~\ref{sec:steepdesc}.  To fix ideas, a proof of linear
convergence is given.  The examples of the Rayleigh quotient on the
sphere and the function $\lyapunov$ on the special orthogonal group
are presented.  In Section~\ref{sec:newton}, Newton's method on a
Riemannian manifold is derived.  As in Euclidean space, this algorithm
may be used to compute the extrema of differentiable functions.  It is
proved that this method converges quadratically.  The example of the
Rayleigh quotient is continued, and it is shown that Newton's method
applied to this function converges cubically, and is approximated by
the Rayleigh quotient iteration.  The example considering $\lyapunov$
is continued.  In a related example, it is shown that Newton's method
applied to the sum of the squares of the off-diagonal elements of a
symmetric matrix converges cubically.  This provides an example of a
cubically convergent Jacobi-like method.  The conjugate gradient
method is presented in Section~\ref{sec:conjgrad} with a proof of
superlinear convergence.  This technique is shown to provide an
effective algorithm for computing the extreme eigenvalues of a
symmetric matrix.  The conjugate gradient method is applied to the
function $\lyapunov$.

\section{Preliminaries}\label{sec:prelim}

This chapter is concerned with the following problem.

\begin{problem}\label{mainproblem}\ignorespaces Let $M$ be a complete
Riemannian manifold, and $f$ a $C^\infty$ function on~$M$.  Compute
$$\min_{p\in M}f(p)$$ and find the minimizing point~$p$.
\end{problem}

There are many well-known algorithms for solving this problem in the
case where $M$ is a Euclidean space.  This section generalizes several
of these algorithms to the case of complete Riemannian manifolds by
replacing the Euclidean notions of straight lines and ordinary
differentiation with geodesics and covariant differentiation.  These
concepts are reviewed in Chapter~\ref{chap:geom}.

Unless otherwise specified, all manifolds, vector fields, and
functions are assumed to be smooth.  When considering a function $f$
to be minimized, the assumption that $f$ is differentiable of
class~$C^\infty$ can be relaxed throughout the chapter, but $f$ must
be continuously differentiable at least beyond the derivatives that
appear.  As the results of this chapter are local ones, the assumption
that $M$ be complete may also be relaxed in certain instances.

We will use the the following definitions to compare the convergence
rates of various algorithms.

\begin{definition} Let $\{p_i\}$ be a Cauchy sequence in~$M$ that
converges to $\phat$.  (i)~The sequence $\{p_i\}$ is said to converge
(at least) {\it linearly\/} if there exists an integer $N$ and a
constant $\theta\in[0,1)$ such that $d(p_{i+1},\phat)\le\theta
d(p_i,\phat)$ for~all $i\ge N$.  (ii)~The sequence $\{p_i\}$ is said
to converge (at least) {\it quadratically\/} if there exists an
integer $N$ and a constant $\theta\ge0$ such that
$d(p_{i+1},\phat)\le\theta d^2(p_i,\phat)$ for~all $i\ge N$.
(iii)~The sequence $\{p_i\}$ is said to converge (at least) {\it
cubically\/} if there exists an integer $N$ and a constant
$\theta\ge0$ such that $d(p_{i+1},\phat)\le \theta d^3(p_i,\phat)$
for~all $i\ge{N}$. (iv)~The sequence $\{p_i\}$ is said to converge
{\it superlinearly\/} if it converges faster than any sequence that
converges linearly.
\end{definition}

\section{Steepest descent on Riemannian manifolds}\label{sec:steepdesc}

The method of steepest descent on a Riemannian manifold is
conceptually identical to the method of steepest descent on Euclidean
space.  Each iteration involves a gradient computation and
minimization along the geodesic determined by the gradient.
Fletcher~\citeyear{Fletcher},
Botsaris~\citeyear{Botsaris:grad,Botsaris:class,Botsaris:geod}, and
Luenberger~\citeyear{Luenberger} describe this algorithm in Euclidean
space.  Gill and Murray~\citeyear{GillMurray} and Sargent~\citeyear{Sargent}
apply this technique in the presence of constraints.  In this section
we restate the method of steepest descent described in the literature
and provide an alternative formalism that will be useful in the
development of Newton's method and the conjugate gradient method on
Riemannian manifolds.

\begin{algorithm}[The method of steepest
descent]\label{al:steepman}\ignorespaces Let $M$ be a Riemannian
manifold with Riemannian structure $g$ and Levi-Civita connection
$\nabla$, and let $f\in C^\infty(M)$.
\begin{steps}
\step[0.] Select $p_0\in M$, compute
     $G_0=-(\grad\f)_{p_0}$, and set $i=0$.
\step[1.] Compute $\lambda_i$ such that
     $$f(\exp_{p_i}\lambda_iG_i)\le f(\exp_{p_i}\lambda G_i)$$ for all
     $\lambda\ge0$.
\step[2.] Set $$\eqalign{p_{i+1} &=\exp_{p_i}\lambda_iG_i,\cr
     G_{i+1} &=-(\grad\f)_{p_{i+1}},\cr}$$ increment $i$, and go~to Step~1.
\end{steps}
\end{algorithm}

It is easy to verify that $\(G_{i+1},\tau G_i\)=0$, for $i\ge0$, where
$\tau$ is the parallelism with respect to the geodesic from $p_i$ to
$p_{i+1}$.  By assumption, the function $\lambda\mapsto f(\exp\lambda
G_i)$ is minimized at $\lambda_i$. Therefore, we have $0
={(d/dt)|_{t=0}}\penalty0{f(\exp(\lambda_i+t)G_i)} =df_{p_{i+1}}(\tau
G_i) =\((\grad\f)_{p_{i+1}},\tau G_i\)$.  Thus the method of steepest
descent on a Riemannian manifold has the same deficiency as its
counterpart on a Euclidean space, i.e., it makes a ninety degree turn
at every step.

The convergence of Algorithm~\ref{al:steepman} is linear.  To prove
this fact, we will make use of a standard theorem of the calculus,
expressed in differential geometric language.  The covariant
derivative $\covD X\f$ of~$f$ along $X$ is defined to be $X\f$.  For
$k=1$, $2$,~\dots, define $\covD{X}^k\f
=\covD{X}\circ\cdots\circ\covD{X}\f$ ($k$~times), and let
$\covD{X}^0\f=f$.

\begin{remark}[Taylor's formula]\label{rem:taylorf}\ignorespaces Let
$M$ be a manifold with an affine connection $\nabla$, $N_p$ a normal
neighborhood of~$p\in M$, the vector field $\Xtilde$ on~$N_p$ adapted
to~$X$ in~$T_p$, and $f$ a $C^\infty$ function on~$M$.  Then there
exists an $\epsilon>0$ such that for every $\lambda\in[0,\epsilon)$
\begin{equation}\label{eq:taylorf}
\eqalign{f(\exp_p \lambda X) &=f(p) +\lambda(\covD\Xtilde\f)(p)
+\cdots+{\lambda^{n-1}\over(n-1)!}(\covD\Xtilde^{n-1}\f)(p)\cr
&\quad{}+{\lambda^n\over(n-1)!}\int_0^1(1-t)^{n-1}
(\covD\Xtilde^n\f)(\exp_p t\lambda X)\,dt.\cr}
\end{equation}
\end{remark}

\begin{proof} Let $N_0$ be a star-shaped neighborhood of $0\in
T_p$ such that $N_p=\exp N_0$.  There exists $\epsilon>0$ such that
$\lambda X\in N_0$ for~all $\lambda\in[0,\epsilon)$.  The map
$\lambda\mapsto f(\exp\lambda X)$ is a real $C^\infty$ function on
$[0,\epsilon)$ with derivative $(\covD\Xtilde\f)(\exp\lambda X)$.  The
statement follows by repeated integration by parts.
\end{proof}

Note that if $M$ is an analytic manifold with an analytic affine
connection $\nabla$, the representation $$f(\exp_p\lambda X)
=\sum_{k=0}^\infty{\lambda^k\over k!}(\covD\Xtilde^k\f)(p)$$ is valid
for~all $X\in T_p$ and all $\lambda\in[0,\epsilon)$.
Helgason~\citeyear{Helgason} provides a proof.

The following special cases of Remark~\ref{rem:taylorf} will be
particularly useful.  When $n=2$, Equation~(\ref{eq:taylorf}) yields
\begin{equation} \label{eq:taylorf2}
\eqalign{f(\exp_p\lambda X) &=f(p) +\lambda(\covD\Xtilde\f)(p)
+\lambda^2\int_0^1(1-t) (\covD\Xtilde^2\f)(\exp_pt\lambda X)\,dt.\cr}
\end{equation} Furthermore, when $n=1$, Equation~(\ref{eq:taylorf})
applied to the function $\Xtilde\f=\covD\Xtilde\f$ yields
\begin{equation} \label{eq:dtaylorf1} (\Xtilde\f)(\exp_p\lambda X)
=(\Xtilde\f)(p) +\lambda\int_0^1 (\covD\Xtilde^2\f)(\exp_p t\lambda
X)\,dt. \end{equation}

The second order terms of~$f$ near a critical point are required for
the convergence proofs.  Consider the second covariant differential
$\nabla\nabla\f=\nabla^2\f$ of a smooth function $f\colon M\to\R$.  If
$(U,x^1,\ldots,x^n)$ is a coordinate chart on~$M$, then at~$p\in U$
this $(0,2)$ tensor takes the form
\begin{equation}\label{eq:d2f} (\nabla^2\f)_p =\sum_{i,j}
\biggl(\Bigl({\partial^2\f\over\partial x^i\partial x^j}\Bigr)_p
-\sum_k \Gamma_{ji}^k \Bigl({\partial\f\over\partial
x^k}\Bigr)_p\biggr)\, dx^i\otimes dx^j, \end{equation} where
$\Gamma_{ij}^k$ are the Christoffel symbols at~$p$.  If $\phat$ in~$U$
is a critical point of~$f\!$, then $(\partial\f/\partial x^k)_\phat=0$,
$k=1$,~\dots, $n$. Therefore $(\nabla^2\f)_\phat=(d^2\f)_\phat$, where
$(d^2\f)_\phat$ is the Hessian of~$f$ at the critical point $\phat$.
Furthermore, for $p\in M$, $X$, $Y\in T_p$, and $\Xtilde$ and
$\Ytilde$ vector fields adapted to $X$ and $Y$, respectively, on a
normal neighborhood $N_p$ of~$p$, we have $(\nabla^2\f)
(\Xtilde,\Ytilde) =\covD\Ytilde\covD\Xtilde\f$ on~$N_p$.  Therefore
the coefficient of the second term of the Taylor expansion of $f(\exp
tX)$ is $(\covD\Xtilde^2\f)_p =(\nabla^2\f)_p(X,X)$.  Note that the
bilinear form $(\nabla^2\f)_p$ on $T_p\times T_p$ is symmetric if and
only if $\nabla$ is symmetric, which true of the Levi-Civita
connection by definition.

\begin{theorem}\label{th:linearconv}\ignorespaces Let $M$ be a complete
Riemannian manifold with Riemannian structure $g$ and Levi-Civita
connection $\nabla$.  Let $f\in C^\infty(M)$ have a nondegenerate
critical point at $\phat$ such that the Hessian $(d^2\f)_\phat$ is
positive definite.  Let $p_i$ be a sequence of points in $M$
converging to~$\phat$ and $H_i\in T_{p_i}$ a sequence of tangent
vectors such that $$\eqaligncond{(i)& p_{i+1}
&=\exp_{p_i}\lambda_iH_i &for $i=0$, $1$,~\dots,\cr (ii)&
\(-(\grad\f)_{p_i},H_i\) &\ge c\,\|(\grad\f)_{p_i}\|\>\|H_i\| &for
$c\in(0,1]$,\cr}$$ where $\lambda_i$ is chosen such that
$f(\exp\lambda_iH_i)\le f(\exp\lambda H_i)$ for all $\lambda\ge0$.
Then there exists a constant $E$ and a $\theta\in[0,1)$ such that
for~all $i=0$, $1$,~\dots, $$d(p_i,\phat)\le E\theta^i.$$
\end{theorem}

\begin{proof} The proof is a generalization of the one given in
Polak~\citeyear[p.~242ff\/]{Polak} for the method of steepest descent on
Euclidean space.

The existence of a convergent sequence is guaranteed by the smoothness
of~$f$.  If $p_j=\phat$ for some integer $j$, the assertion becomes
trivial; assume otherwise.  By the smoothness of~$f\!$, there exists an
open neighborhood $U$ of~$\phat$ such that $(\nabla^2\f)_p$ is
positive definite for~all $p\in U$. Therefore, there exist constants
$k>0$ and $K\ge k>0$ such that for all $X\in T_p$ and all $p\in U$,
\begin{equation} \label{eq:d2fbounds} k\|X\|^2\le
(\nabla^2\f)_p(X,X)\le K\|X\|^2. \end{equation}

Define $X_i\in T_{\phat}$ by the relations $\exp X_i=p_i$, $i=0$, $1$,
\dots\spacefactor=3000\relax\space By assumption, $df_\phat=0$ and
from Equation~(\ref{eq:taylorf2}), we have \begin{equation}
\label{eq:taylorfsp} f(p_i)-f(\phat) =\int_0^1
(1-t)(\covD{\Xtilde_i}^2f)(\exp_\phat tX_i)\,dt.\end{equation}
Combining this equality with the inequalities of (\ref{eq:d2fbounds})
yields \begin{equation} \label{eq:fbounds} \half kd^2(p_i,\phat)\le
f(p_i)-f(\phat)\le \half Kd^2(p_i,\phat). \end{equation} Similarly, we
have by Equation~(\ref{eq:dtaylorf1}) $$(\Xtilde_if)(p_i) =\int_0^1
(\covD{\Xtilde_i}^2f)(\exp_\phat tX_i)\,dt.$$ Next, use
(\ref{eq:taylorfsp}) with Schwarz's inequality and the first
inequality of~(\ref{eq:fbounds}) to obtain
$$\eqalign{kd^2(p_i,\phat)=k\|X_i\|^2 &\le\int_0^1
(\covD{\Xtilde_i}^2f)(\exp_\phat tX_i)\,dt =(\Xtilde_if)(p_i)\cr
&=df_{p_i}\bigl((\Xtilde_i)_{p_i}\bigr) =df_{p_i}(\tau X_i)
=\((\grad\f)_{p_i},\tau X_i\)\cr &\le
\|(\grad\f)_{p_i}\|\>\|\tau X_i\| =\|(\grad\f)_{p_i}\|\>
d(p_i,\phat).\cr}$$ Therefore, \begin{equation} \label{eq:gradlb}
\|(\grad\f)_{p_i}\|\ge kd(p_i,\phat).\end{equation}

Define the function $\Diff\colon T_p\times\R\to\R$ by the equation
$\Diff(X,\lambda)=f(\exp_p\lambda X)-f(p)$. By
Equation~(\ref{eq:taylorf2}), the second order Taylor formula, we have
$$\Diff(H_i,\lambda) =\lambda(\Htilde_if)(p_i) +\half\lambda^2\int_0^1
(1-t)(\covD{\Htilde_i}^2f)(\exp_{p_i}\lambda H_i)\,dt.$$ Using
assumption~(ii) of the theorem along with (\ref{eq:d2fbounds}) we
establish for $\lambda\ge0$ \begin{equation}\label{eq:diffbounds}
\Diff(H_i,\lambda)\le -\lambda c\|(\grad\f)_{p_i}\|\>\|H_i\|
+\half\lambda^2K\|H_i\|^2. \end{equation}

We may now compute an upper bound for the rate of linear convergence
$\theta$.  By assumption~(i) of the theorem, $\lambda$ must be chosen
to minimize the right hand side of~(\ref{eq:diffbounds}). This
corresponds to choosing $\lambda=c\|(\grad\f)_{p_i}\|\big/K\|H_i\|$.
A computation reveals that $$\Diff(H_i,\lambda_i)\le
-{c^2\over2K}\|(\grad\f)_{p_i}\|^2.$$ Applying (\ref{eq:fbounds}) and
(\ref{eq:gradlb}) to this inequality and rearranging terms yields
\begin{equation} \label{eq:fconv} f(p_{i+1})-f(\phat)\le
\theta\bigl(f(p_i)-f(\phat)\bigr), \end{equation} where
$\theta=\bigl(1-(ck/K)^2\bigr)$.  By assumption, $c\in(0,1]$ and
$0<k\le K$, therefore $\theta\in[0,1)$.  (Note that Schwarz's
inequality bounds $c$ below unity.)  From (\ref{eq:fconv}) it is seen
that $\bigl(f(p_i)-f(\phat)\bigr)\le E\theta^i$, where
$E=\bigl(f(p_0)-f(\phat)\bigr)$.  From (\ref{eq:fbounds}) we conclude
that for $i=0$, $1$,~\dots, \begin{equation} \label{eq:linearconv}
d(p_i,\phat)\le \sqrt{2E\over k}\bigl(\sqrt\theta\,\bigr)^i.\tombstone
\end{equation}
\end{proof}

\begin{corollary} If Algorithm~\ref{al:steepman} converges to a local
minimum, it converges linearly.
\end{corollary}

The choice $H_i=-(\grad\f)_{p_i}$ yields $c=1$ in the second
assumption the Theorem~\ref{th:linearconv}, which establishes the
corollary.

\begin{example}[Rayleigh quotient on the
sphere]\label{eg:rayqgrad}\ignorespaces Let $S^{n-1}$ be the imbedded
sphere in~$\R^n$, i.e., $S^{n-1}=\{\,x\in\R^n:x^\T x=1\,\}$, where
$x^\T y$ denotes the standard inner product on~$\R^n$, which induces a
metric on~$S^{n-1}$.  Geodesics on the sphere are great circles and
parallel translation along geodesics is equivalent to rotating the
tangent plane along the great circle.  Let $x\in S^{n-1}$ and $h\in
T_x$ have unit length, and $v\in T_x$ be any tangent vector.  Then
$$\eqalign{\exp_x th &=x\cos t +h\sin t,\cr \tau h &=h\cos t -x\sin
t,\cr \tau v &=v-(h^\T v)\bigl(x\sin t +h(1-\cos t)\bigr),\cr}$$ where
$\tau$ is the parallelism along the geodesic $t\mapsto\exp th$.  Let
$Q$ be an $n\by n$ positive definite symmetric matrix with distinct
eigenvalues and define $\rho\colon S^{n-1}\to\R$ by $\rho(x)=x^\T Qx$.
A computation shows that \begin{equation}\label{eq:raygrad}
\half(\grad\rho)_x =Qx-\rho(x)x.  \end{equation} The function $\rho$
has a unique minimum and maximum point at the eigenvectors
corresponding to the smallest and largest eigenvalues of~$Q$,
respectively. Because $S^{n-1}$ is geodesically complete, the method
of steepest descent in the opposite direction of the gradient
converges to the eigenvector corresponding to the smallest eigenvalue
of~$Q$; likewise for the eigenvector corresponding to the largest
eigenvalue.  Chu~\citeyear{Chu:sphere} considers the continuous limit
of this problem.  A computation shows that $\rho(x)$ is maximized
along the geodesic $\exp_x th$ ($\|h\|=1$) when $a\cos2t-b\sin2t=0$,
where $a=2x^\T Qh$ and $b=\rho(x)-\rho(h)$.  Thus $\cos t$ and $\sin
t$ may be computed with simple algebraic functions of $a$ and~$b$
(which appear below in Algorithm~\ref{al:raysphere}).  The results of
a numerical experiment demonstrating the convergence of the method of
steepest descent applied to maximizing the Rayleigh quotient
on~$S^{20}$ are shown in Figure~\ref{fig:raysphere} on
page~\pageref{fig:raysphere}.  \end{example}

\def\temp{\citeyear{Brockett:sort,Brockett:grad}}
\begin{example}[Brockett~\temp]\label{eg:trHNgrad}\ignorespaces
Consider the function $f(\Theta)=\lyapunov$ on the special orthogonal
group $\SO(n)$, where $Q$ is a real symmetric matrix with distinct
eigenvalues and $N$ is a real diagonal matrix with distinct diagonal
elements. It will be convenient to identify tangent vectors in
$T_\Theta$ with tangent vectors in $T_I\cong\so(n)$, the tangent plane
at the identity, via left translation.  The gradient of~$f$ (with
respect to the negative Killing form of~$\so(n)$, scaled by~$1/(n-2)$)
at $\Theta\in\SO(n)$ is $\Theta[H,N]$, where
$H=\Ad_{\Theta^\T}(Q)=\Theta^\T Q\Theta$.  The group $\SO(n)$ acts on
the set of symmetric matrices by conjugation; the orbit of~$Q$ under
the action of~$\SO(n)$ is an isospectral submanifold of the symmetric
matrices.  We seek a $\Thetahat$ such that $f(\Thetahat)$ is
maximized.  This point corresponds to a diagonal matrix whose diagonal
entries are ordered similarly to those of~$N$.  A related example is
found in Smith~\citeyear{Me}, who considers the homogeneous space of
matrices with fixed singular values, and in Chu~\citeyear{Chu:grad}.

The Levi-Civita connection on~$\SO(n)$ is bi-invariant and invariant
with respect to inversion; therefore, geodesics and parallel
translation may be computed via matrix exponentiation of elements
in~$\so(n)$ and left (or right) translation~\cite[Chap.~2,
Ex.~6]{Helgason}.  The geodesic emanating from the identity in
$\SO(n)$ in direction $X\in\so(n)$ is given by the formula
$\exp_ItX=e^{Xt}$, where the right hand side denotes regular matrix
exponentiation.  The expense of geodesic minimization may be avoided
if instead one uses Brockett's estimate~\cite{Brockett:grad} for the
step size.  Given $\Omega\in\so(n)$, we wish to find $t>0$ such that
$\phi(t)=\tr\Ad_{e^{-\Omega t}}(H)N$ is minimized.  Differentiating
$\phi$ twice shows that $\phi'(t)=-\tr\Ad_{e^{-\Omega t}}(\ad_\Omega
H)N$ and $\phi''(t)=-\tr\Ad_{e^{-\Omega t}}(\ad_\Omega H)\ad_\Omega N$,
where $\ad_\Omega A=[\Omega,A]$.  Hence, $\phi'(0)=2\tr H\Omega N$
and, by Schwarz's inequality and the fact that $\Ad$ is an isometry,
$|\phi''(t)|\le \|\ad_\Omega H\|\;\|\ad_\Omega N\|$.  We conclude that
if $\phi'(0)>0$, then $\phi'$ is nonnegative on the interval
\begin{equation}\label{eq:Brockettest} 0\le t\le {2\tr H\Omega
N\over\|\ad_\Omega H\|\;\|\ad_\Omega N\|}, \end{equation} which
provides an estimate for the step size of Step~1 in
Algorithm~\ref{al:steepman}.  The results of a numerical experiment
demonstrating the convergence of the method of steepest descent
(ascent) in~$\SO(20)$ using this estimate are shown in
Figure~\ref{fig:SOnconv}.
\end{example}

\section{Newton's method on Riemannian manifolds}\label{sec:newton}

As in the optimization of functions on Euclidean space, quadratic
convergence can be obtained if the second order terms of the Taylor
expansion are used appropriately.  In this section we present Newton's
algorithm on Riemannian manifolds, prove that its convergence is
quadratic, and provide examples.  Whereas the convergence proof for
the method of steepest descent relies upon the Taylor expansion of the
function $f\!$, the convergence proof for Newton's method will rely upon
the Taylor expansion of the one-form $df$.  Note that Newton's method
has a counterpart in the theory of constrained optimization, as
described by, e.g., Fletcher~\citeyear{Fletcher},
Bertsekas~\citeyear{Bertsekas:newton,Bertsekas}, or
Dunn~\citeyear{Dunn:newton,Dunn:grad}.  The Newton method presented in
this section has only local convergence properties. There is a theory
of global Newton methods on Euclidean space and computational
complexity; see the work of Hirsch and Smale~\citeyear{HirschSmale},
Smale~\citeyear{Smale:fta,Smale:eff}, and Shub and
Smale~\citeyear{ShubSmale:I,ShubSmale:II}.

Let $M$ be an $n\hyphen$dimensional Riemannian manifold with Riemannian
structure $g$ and Levi-Civita connection $\nabla$, let $\mu$ be a
$C^\infty$ one-form on~$M$, and let $p$ in~$M$ be such that the
bilinear form $(\D\mu)_p\colon T_p\times T_p\to\R$ is nondegenerate.
Then, by abuse of notation, we have the pair of isomorphisms
$$T_p\mathrel{\adjarrow^{(\D\mu)_p}_
{(\D\mu)_p^{\rlap{$\scriptscriptstyle-1$}}}}T_p^*$$ with the forward
map defined by $X\mapsto (\covD X\mu)_p= (\D\mu)_p(\rdot,X)$, which is
nonsingular.  The notation $(\D\mu)_p$ will henceforth be used for
both the bilinear form defined by the covariant differential of~$\mu$
evaluated at~$p$ and the homomorphism from $T_p$ to~$T_p^*$ induced by
this bilinear form.  In case of an isomorphism, the inverse can be
used to compute a point in~$M$ where $\mu$ vanishes, if such a point
exists.  The case $\mu=df$ will be of particular interest, in which
case $\D\mu=\nabla^2\f$.  Before expounding on these ideas, we make
the following remarks.

\begin{remark}[The mean value theorem]\label{rem:mvt}\ignorespaces Let
$M$ be a manifold with affine connection~$\nabla$, $N_p$ a normal
neighborhood of~$p\in M$, the vector field $\Xtilde$ on~$N_p$ adapted
to~$X\in T_p$, $\mu$ a one-form on~$N_p$, and $\tau_\lambda$ the
parallelism with respect to $\exp tX$ for $t\in[0,\lambda]$.  Denote
the point $\exp\lambda X$ by~$p_\lambda$.  Then there exists an
$\epsilon>0$ such that for every $\lambda\in[0,\epsilon)$, there is an
$\alpha\in[0,\lambda]$ such that $$\tau_\lambda^{-1}\mu_{p_\lambda}
-\mu_p =\lambda(\covD\Xtilde\mu)_{p_\alpha}\circ\tau_\alpha.$$
\end{remark}

\begin{proof} As in the proof of Remark~\ref{rem:taylorf},
there exists an $\epsilon>0$ such that $\lambda X\in N_0$ for~all
$\lambda\in[0,\epsilon)$.  The map $\lambda\mapsto
(\tau_\lambda^{-1}\mu_{p_\lambda})(A)$, for any $A$ in~$T_p$, is a
$C^\infty$ function on~$[0,\epsilon)$ with derivative
$(d/dt)(\tau_t^{-1}\mu_{p_t})(A) =(d/dt)\mu_{p_t}(\tau_tA)
=\covD\Xtilde\bigl(\mu_{p_t}(\tau_tA)\bigr)
=(\covD\Xtilde\mu)_{p_t}(\tau_tA)+\mu_{p_t}\bigl(\covD\Xtilde(\tau_tA)\bigr)
=(\covD\Xtilde\mu)_{p_t}(\tau_tA)$.  The lemma follows from the mean
value theorem of real analysis.
\end{proof}

This remark can be generalized in the following way.

\begin{remark}[Taylor's theorem]\label{rem:taylort}\ignorespaces Let
$M$ be a manifold with affine connection~$\nabla$, $N_p$ a normal
neighborhood of~$p\in M$, the vector field $\Xtilde$ on~$N_p$ adapted
to~$X\in T_p$, $\mu$ a one-form on~$N_p$, and $\tau_\lambda$ the
parallelism with respect to $\exp tX$ for $t\in[0,\lambda]$.  Denote
the point $\exp\lambda X$ by~$p_\lambda$.  Then there exists an
$\epsilon>0$ such that for every $\lambda\in[0,\epsilon)$, there is an
$\alpha\in[0,\lambda]$ such that \begin{equation}\label{eq:taylort}
\tau_\lambda^{-1}\mu_{p_\lambda} =\mu_p +\lambda(\covD\Xtilde\mu)_p
+\cdots+{\lambda^{n-1}\over(n-1)!} (\covD\Xtilde^{n-1}\mu)_p
+{\lambda^n\over n!} (\covD\Xtilde^n\mu)_{p_\alpha}\circ\tau_\alpha.
\end{equation}
\end{remark}

The remark follows by applying Remark~\ref{rem:mvt} and the Taylor's
theorem of real analysis to the function $\lambda\mapsto
(\tau_\lambda^{-1}\mu_{p_\lambda})(A)$ for any $A$ in~$T_p$.

Remarks \ref{rem:mvt} and \ref{rem:taylort} can be generalized to
$C^\infty$ tensor fields, but we will only require
Remark~\ref{rem:taylort} for case $n=2$ to make the following
observation.

Let $\mu$ be a one-form on~$M$ such that for some $\phat$ in~$M$,
$\mu_\phat=0$.  Given any $p$ in a normal neighborhood of~$\phat$, we
wish to find $X$ in~$T_p$ such that $\exp_pX=\phat$.  Consider the
Taylor expansion of~$\mu$ about $p$, and let $\tau$ be the parallel
translation along the unique geodesic joining $p$ to~$\phat$. We have
by our assumption that $\mu$ vanishes at~$\phat$, and from
Equation~(\ref{eq:taylort}) for $n=2$, $$0=\tau^{-1}\mu_\phat
=\tau^{-1}\mu_{\exp_pX} =\mu_p +(\D\mu)_p(\rdot,X) +\hot$$ If the
bilinear form $(\D\mu)_p$ is nondegenerate, the tangent vector $X$ may
be approximated by discarding the higher order terms and solving the
resulting linear equation $$\mu_p +(\D\mu)_p(\rdot,X) =0$$ for~$X$,
which yields $$X=-(\D\mu)_p^{-1}\mu_p.$$ This approximation is the
basis of the following algorithm.

\begin{algorithm}[Newton's method]\label{al:newtonman}\ignorespaces Let
$M$ be a complete Riemannian manifold with Riemannian structure $g$
and Levi-Civita connection $\nabla$, and let $\mu$ be a $C^\infty$
one-form on~$M$.
\begin{steps}
\step[0.] Select $p_0\in M$ such that $(\D\mu)_{p_0}$ is
     nondegenerate, and set $i=0$.
\step[1.] Compute $$\eqalign{H_i
     &=-(\D\mu)_{p_i}^{-1}\mu_{p_i}\cr p_{i+1} &=\exp_{p_i}H_i,\cr}$$
     (assume that $(\D\mu)_{p_i}$ is nondegenerate), increment $i$, and
     repeat.
\end{steps}
\end{algorithm}

It can be shown that if $p_0$ is chosen suitably close (within the
so-called domain of attraction) to a point $\phat$ in~$M$ such
that $\mu_\phat=0$ and $(\D\mu)_\phat$ is nondegenerate, then
Algorithm~\ref{al:newtonman} converges quadratically to~$\phat$.  The
following theorem holds for general one-forms; we will consider the
case where $\mu$ is exact.

\begin{theorem}\label{th:newtonquad}\ignorespaces Let $f\in
C^\infty(M)$ have a nondegenerate critical point at $\phat$.  Then
there exists a neighborhood $U$ of~$\phat$ such that for any $p_0\in
U$, the iterates of Algorithm~\ref{al:newtonman} for $\mu=df$ are well
defined and converge quadratically to~$\phat$.
\end{theorem}

The proof of this theorem is a generalization of the corresponding
proof for Euclidean spaces, with an extra term containing the
Riemannian curvature tensor (which of course vanishes in the latter
case).

\begin{proof} If $p_j=\phat$ for some integer $j$, the assertion
becomes trivial; assume otherwise.  Define $X_i\in T_{p_i}$ by the
relations $\phat=\exp X_i$, $i=0$, $1$, \dots, so that
$d(p_i,\phat)=\|X_i\|$ (n.b.\ this convention is opposite that used in
the proof of Theorem~\ref{th:linearconv}).  Consider the geodesic
triangle with vertices $p_i$, $p_{i+1}$, and $\phat$, and sides $\exp
tX_i$ from $p_i$ to~$\phat$, $\exp tH_i$ from $p_i$ to~$p_{i+1}$, and
$\exp tX_{i+1}$ from $p_{i+1}$ to~$\phat$, for $t\in[0,1]$.  Let
$\tau$ be the parallelism with respect to the side $\exp tH_i$ between
$p_i$ and $p_{i+1}$.  There exists a unique tangent vector $\vrem_i$
in~$T_{p_i}$ defined by the equation \begin{equation}\label{eq:vrem}
X_i=H_i+\tau^{-1}X_{i+1}+\vrem_i \end{equation} ($\vrem_i$ may be
interpreted as the amount by which vector addition fails).  If we use
the definition $H_i=-(\nabla^2\f)_{p_i}^{-1}df_{p_i}$ of
Algorithm~\ref{al:newtonman}, apply the isomorphism
$(\nabla^2\f)_{p_i}\colon T_{p_i}\to T_{p_i}^*$ to both sides of
Equation~(\ref{eq:vrem}), we obtain the equation
\begin{equation}\label{eq:quadproof1} (\nabla^2\f)_{p_i}
(\tau^{-1}X_{i+1}) =df_{p_i} +(\nabla^2\f)_{p_i}X_i
-(\nabla^2\f)_{p_i}\vrem_i. \end{equation} By Taylor's theorem, there
exists an $\alpha\in[0,1]$ such that
\begin{equation}\label{eq:quadproof2} \tau_1^{-1}df_\phat =df_{p_i}
+(\covD{\Xtilde_i}df)_{p_i} +\half(\covD{\Xtilde_i}^2df)_{p_\alpha}
\circ\tau_\alpha, \end{equation} where $\tau_t$ is the parallel
translation from~$p_i$ to~$p_t=\exp tX_i$.  The trivial identities
$(\covD{\Xtilde_i}df)_{p_i} =(\nabla^2\f)_{p_i}X_i$ and
$(\covD{\Xtilde_i}^2df)_{p_\alpha} =(\nabla^3\f)_{p_\alpha}
(\tau_\alpha\rdot,\tau_\alpha X_i,\tau_\alpha X_i)$ will be used to
replace the last two terms on the right hand side of
Equation~(\ref{eq:quadproof2}).  Combining the assumption that $df_\phat=0$
with Equations (\ref{eq:quadproof1}) and (\ref{eq:quadproof2}), we obtain
\begin{equation}\label{eq:quadproof3} (\nabla^2\f)_{p_i}
(\tau^{-1}X_{i+1})= -\half(\covD{\Xtilde_i}^2df)_{p_\alpha}
\circ\tau_\alpha -(\nabla^2\f)_{p_i}\vrem_i. \end{equation}

By the smoothness of~$f$ and $g$, there exists an $\epsilon>0$ and
constants $\delta'$, $\delta''$, $\delta'''$, all greater than zero,
such that whenever $p$ is in the convex normal ball $\Ball$,
$$\eqaligncond{(i)& \|(\nabla^2\f)_p(\rdot,X)\|&\ge\delta'\|X\|
&for~all $X\in T_p$,\cr (ii)&
\|(\nabla^2\f)_p(\rdot,X)\|&\le\delta''\|X\| &for~all $X\in T_p$,\cr
(iii)& \|(\nabla^3\f)_p(\rdot,X,X)\|&\le\delta'''\|X\|^2 &for~all
$X\in T_p$,\cr}$$ where the induced norm on~$T_p^*$ is used in all
three cases.  Taking the norm of both sides of
Equation~(\ref{eq:quadproof3}), applying the triangle inequality to the
right hand side, and using the fact that parallel translation is an
isometry, we obtain the inequality
\begin{equation}\label{eq:quadproof4} \delta'd(p_{i+1},\phat)\le
\delta'''d^2(p_i,\phat) +\delta''\|\vrem_i\|. \end{equation}

The length of~$\vrem_i$ can be bounded by a cubic expression
in~$d(p_i,\phat)$ by considering the distance between the points
$\exp(H_i+\tau^{-1}X_{i+1})$ and $\exp X_{i+1}=\phat$.  Given $p\in
M$, $\epsilon>0$ small enough, let $a$, $v\in T_p$ be such that
$\|a\|+\|v\|\le\epsilon$, and let $\tau$ be the parallel translation
with respect to the geodesic from~$p$ to~$q=\exp_p a$.
Karcher~\citeyear[App.~C2.2]{Karcher} shows that \begin{equation}
\label{eq:Karcher} d\bigl(\exp_p(a+v), \exp_q(\tau v)\bigr)\le
\|a\|\cdot{\rm const.}\,(\max|K|)\cdot\epsilon^2,\end{equation} where
$K$ is the sectional curvature of~$M$ along any section in the tangent
plane at any point near~$p$.

There exists a constant $c>0$ such that $\|\vrem_i\|\le c\,
d\bigl(\phat, {\exp(H_i +\tau^{-1}X_{i+1})}\bigr)$.
By~(\ref{eq:Karcher}), we have $\|\vrem_i\|\le {\rm
const.}\,\|H_i\|\epsilon^2$.  Taking the norm of both sides of the
Taylor formula $df_{p_i} =-\int_0^1 (\covD{\Xtilde_i}df) (\exp
tX_i)\,dt$ and applying a standard integral inequality and
inequality~(ii) from above yields $\|df_{p_i}\|\le\delta''\|X_i\|$ so
that $\|H_i\|\le {\rm const.}\,\|X_i\|$.  Furthermore, we have the
triangle inequality $\|X_{i+1}\|\le \|X_i\|+\|H_i\|$, therefore
$\epsilon$ may be chosen such that $\|H_i\|+\|X_{i+1}\|\le\epsilon\le
{\rm const.}\,\|X_i\|$.  By~(\ref{eq:Karcher}) there exists
$\delta^{\rm iv}>0$ such that $\|\vrem_i\|\le\delta^{\rm
iv}d^3(p_i,\phat)$.
\end{proof}

\begin{corollary} If\/ $(\nabla^2\f)_\phat$ is positive (negative)
definite and Algorithm~\ref{al:newtonman} converges to~$\phat$, then
Algorithm~\ref{al:newtonman} converges quadratically to a local
minimum (maximum) of~$f$.
\end{corollary}

\begin{example}[Rayleigh quotient on the
sphere]\label{eg:newtonray}\ignorespaces Let $S^{n-1}$ and
$\rho(x)=x^\T Qx$ be as in Example~\ref{eg:rayqgrad}. It will be
convenient to work with the coordinates $x^1$,~\dots, $x^n$ of the
ambient space~$\R^n$, treat the tangent plane $T_xS^{n-1}$ as a vector
subspace of~$\R^n$, and make the identification $T_xS^{n-1}\cong
T_x^*S^{n-1}$ via the metric. In this coordinate system, geodesics on
the sphere obey the second order differential equation $\ddot
x^k+x^k=0$, $k=1$,~\dots, $n$. Thus the Christoffel symbols are given
by $\Gamma_{ij}^k=\delta_{ij}x^k$, where $\delta_{ij}$ is the
Kronecker delta.  The $ij$th component of the second covariant
differential of~$\rho$ at~$x$ in~$S^{n-1}$ is given by (cf.\
Equation~(\ref{eq:d2f})) $$\bigl((\Dsqr\rho)_x\bigr)_{ij} =2Q_{ij}
-\sum_{k,l}\delta_{ij}x^k\cdot 2Q_{kl}x^l =2\bigl(Q_{ij}
-\rho(x)\delta_{ij}\bigr),$$ or, written as matrices,
\begin{equation}\label{eq:d2ray} \half(\Dsqr\rho)_x=Q-\rho(x)I.
\end{equation} Let $u$ be a tangent vector in~$T_x S^{n-1}$.
A linear operator $A\colon\R^n\to\R^n$ defines a linear operator on
the tangent plane $T_xS^{n-1}$ for each $x$ in~$S^{n-1}$ such that
$$A\rdot u=Au-(x^\T\!Au)x =(I-xx^\T)Au.$$ If $A$ is invertible as an
endomorphism of the ambient space $\R^n$, the solution to the linear
equation $A\rdot u=v$ for $u$, $v$ in~$T_xS^{n-1}$ is
\begin{equation}\label{eq:tanginv} u=A^{-1}\left(v
-{(x^\T\!A^{-1}v)\over(x^\T\!A^{-1}x)}x\right).
\end{equation} For Newton's method, the direction
$H_i$ in~$T_xS^{n-1}$ is the solution of the equation
$$(\Dsqr\rho)_{x_i}\rdot H_i= -(\grad\rho)_{x_i}.$$ Combining Equations
(\ref{eq:raygrad}), (\ref{eq:d2ray}), and (\ref{eq:tanginv}), we
obtain $$H_i =-x_i +\alpha_i\bigl(Q-\rho(x_i)I\bigr)^{-1}x_i,$$ where
$\alpha_i=1\big/x_i^\T(Q-\rho(x_i)I)^{-1}x_i$.  This gives rise to the
following algorithm for computing eigenvectors of the symmetric
matrix~$Q$.
\end{example}

\begin{algorithm}[Newton-Rayleigh quotient
method]\label{al:newtonray}\ignorespaces Let $Q$ be a real symmetric
$n\by n$ matrix.
\begin{steps}
\step[0.] Select $x_0$ in~$\R^n$ such that $x_0^\T x_0=1$, and
     set $i=0$.
\step[1.] Compute $$y_i=\bigl(Q-\rho(x_i)I\bigr)^{-1}x_i$$ and
     set $\alpha_i=1\big/x_i^\T y_i$.
\step[2.] Compute $$\eqalign{H_i &=-x_i+\alpha_iy_i,\quad
     \theta_i=\|H_i\|,\cr
     x_{i+1} &=x_i\cos\theta_i +H_i\sin\theta_i/\theta_i,\cr}$$
     increment $i$, and go~to Step~1.
\end{steps}
\end{algorithm}

The quadratic convergence guaranteed by Theorem~\ref{th:newtonquad} is
in fact too conservative for Algorithm~\ref{al:newtonray}.  As
evidenced by Figure~\ref{fig:raysphere}, Algorithm~\ref{al:newtonray}
converges cubically.

\begin{proposition} If $\lambda$ is a distinct eigenvalue of the
symmetric matrix~$Q$, and Algorithm~\ref{al:newtonray} converges to
the corresponding eigenvector $\xhat$, then it converges cubically.
\end{proposition}

\begin{proof}[Proof 1] In the coordinates $x^1$,~\dots, $x^n$ of the
ambient space $\R^n$, the $ijk$th component of the third covariant
differential of~$\rho$ at~$\xhat$ is $-2\lambda \xhat^k\delta_{ij}$.
Let $X\in T_\xhat S^{n-1}$.  Then $(\nabla^3\rho)_\xhat(\rdot,X,X)=0$
and the second order terms on the right hand side of
Equation~(\ref{eq:quadproof3}) vanish at the critical point.  The
proposition follows from the smoothness of~$\rho$.
\end{proof}

\begin{proof}[Proof 2] The proof follows
Parlett's~\citeyear[p.~72ff\/]{Parlett} proof of cubic convergence for the
Rayleigh quotient iteration.  Assume that for~all $i$, $x_i\ne\xhat$,
and denote $\rho(x_i)$ by~$\rho_i$.  For~all $i$, there is an angle
$\psi_i$ and a unit length vector $u_i$ defined by the equation $x_i
=\xhat\cos\psi_i +u_i\sin\psi_i$, such that $\xhat^\T u_i=0$. By
Algorithm~\ref{al:newtonray} $$\eqalign{x_{i+1} &=\xhat\cos\psi_{i+1}
+u_{i+1}\sin\psi_{i+1} =x_i\cos\theta_i +H_i\sin\theta_i/\theta_i\cr
&=\xhat\biggl({\alpha_i\sin\theta_i\over (\lambda-\rho_i)\theta_i}
+\beta_i\biggr)\cos\psi_i +\biggl( {\alpha_i\sin\theta_i\over\theta_i}
(Q-\rho_iI)^{-1}u_i +\beta_iu_i\biggr)\sin\psi_i,\cr}$$ where
$\beta_i=\cos\theta_i-\sin\theta_i/\theta_i$.  Therefore,
\begin{equation}\label{eq:tanpsi} |\tan\psi_{i+1}| ={\Bigl\|
{\alpha_i\sin\theta_i\over\theta_i} (Q-\rho_iI)^{-1}u_i
+\beta_iu_i\Bigr\|\over \Bigl| {\alpha_i\sin\theta_i\over
(\lambda-\rho_i)\theta_i} +\beta_i\Bigr|} \cdot|\tan\psi_i|.
\end{equation} The following equalities and low order approximations
in terms of the small quantities $\lambda-\rho_i$, $\theta_i$, and
$\psi_i$ are straightforward to establish: ${\lambda-\rho_i}
={(\lambda-\rho(u_i))}
\discretionary{}{\the\textfont2\char2\thinspace}{}
\sin^2\psi_i$, $\theta_i^2=\cos^2\psi_i\sin^2\psi_i +\hot$, 
$\alpha_i={(\lambda-\rho_i)} +\hot$, and $\beta_i=-\theta_i^2/3
+\hot$\spacefactor=3000\relax\space Thus, the denominator of the large
fraction in Equation~(\ref{eq:tanpsi}) is of order unity and the numerator
is of order $\sin^2\psi_i$.  Therefore, we have $$|\psi_{i+1}|={\rm
const.}\,|\psi_i|^3 +\hot\tombstone$$
\end{proof}

\begin{remark} If Algorithm~\ref{al:newtonray} is simplified by
replacing Step~2 with
\begin{steps}
\step[$2'$.] Compute $$x_{i+1}=y_i\big/\|y_i\|,$$ increment
     $i$, and go~to Step~1.
\end{steps}
then we obtain the Rayleigh quotient iteration.  These two algorithms
differ by the method in which they use the vector
$y_i=(Q-\rho(x_i)I)^{-1}x_i$ to compute the next iterate on the
sphere.  Algorithm~\ref{al:newtonray} computes the point $H_i$
in~$T_{x_i}S^{n-1}$ where $y_i$ intersects this tangent plane, then
computes $x_{i+1}$ via the exponential map of this vector (which
``rolls'' the tangent vector $H_i$ onto the sphere).  The Rayleigh
quotient iteration computes the intersection of~$y_i$ with the sphere
itself and takes this intersection to be $x_{i+1}$. The latter
approach approximates Algorithm~\ref{al:newtonray} up to quadratic
terms when $x_i$ is close to an eigenvector.
Algorithm~\ref{al:newtonray} is more expensive to compute
than---though of the same order as---the Rayleigh quotient iteration;
thus, the |RQI| is seen to be an efficient approximation of Newton's
method.
\end{remark}

If the exponential map is replaced by the chart $v\in T_x\mapsto
(x+v)/\|x+v\|\in S^{n-1}$, Shub~\citeyear{Shub:ray} shows that a
corresponding version of Newton's method is equivalent to the |RQI|.

\begin{example}[The function $\lyapunov$]\label{eg:trHNnewton}\ignorespaces
Let $\Theta$, $Q$, $H=\Ad_{\Theta^\T}(Q)$, and $\Omega$ be as in
Example~\ref{eg:trHNgrad}.  The second covariant differential of
$f(\Theta)=\lyapunov$ may be computed either by polarization of the
second order term of $\tr\Ad_{e^{-\Omega t}}(H)N$, or by covariant
differentiation of the differential $df_\Theta
=-\tr[H,N]\Theta^\T(\rdot)$: $$(\nabla^2\f)_\Theta(\Theta X,\Theta Y)
=-\half\tr\bigl([H,\ad_X N]-[\ad_X H,N]\bigr)Y,$$ where $X$,
$Y\in\so(n)$.  To compute the direction $\Theta X\in T_\Theta$,
$X\in\so(n)$, for Newton's method, we must solve the equation
$(\nabla^2\f)_\Theta(\Theta\rdot,\Theta X)=df_\Theta$, which yields
the linear equation $$L_\Theta(X)\buildrel{\rm def}\over=[H,\ad_X
N]-[\ad_X H,N] =2[H,N].$$ The linear operator
$L_\Theta\colon\so(n)\to\so(n)$ is self-adjoint for~all $\Theta$ and,
in a neighborhood of the maximum, negative definite.  Therefore,
standard iterative techniques in the vector space $\so(n)$, such as
the classical conjugate gradient method, may be used to solve this
equation near the maximum.  The results of a numerical experiment
demonstrating the convergence of Newton's method in~$\SO(20)$ are
shown in Figure~\ref{fig:SOnconv}.  As can be seen, Newton's method
converged within round-off error in 2 iterations.
\end{example}

\begin{remark}\label{rem:trHNcub}\ignorespaces If Newton's method
applied to the function $f(\Theta)=\lyapunov$ converges to the point
$\Thetahat$ such that $\Ad_{\Thetahat^\T}(Q)=H_\infty=\alpha N$,
$\alpha\in\R$, then it converges cubically.
\end{remark}

\begin{proof} By covariant differentiation of~$\nabla^2\f\!$, the third
covariant differential of~$f$ at~$\Theta$ evaluated at the tangent
vectors $\Theta X$, $\Theta Y$, $\Theta Z\in T_\Theta$, $X$, $Y$,
$Z\in\so(n)$, is $$\eqalign{(\nabla^3\f)_\Theta(\Theta X,\Theta
Y,\Theta Z) =-\quarter\tr\bigl(&[\ad_Y\ad_ZH,N] -[\ad_Z\ad_YN,H]\cr
{}+[H,\ad_{\ad_YZ}N]-{}&[\ad_YH,\ad_ZN] +[\ad_YN,\ad_ZH]\bigr)
X.\cr}$$ If $H=\alpha N$, $\alpha\in\R$, then $(\nabla^3\f)
_\Theta(\rdot,\Theta X,\Theta X)=0$.  Therefore, the second order
terms on the right hand side of Equation~(\ref{eq:quadproof3}) vanish at
the critical point.  The remark follows from the smoothness of~$f$.
\end{proof}

This remark illuminates how rapid convergence of Newton's method
applied to the function $f$ can be achieved in some instances.  If
$E_{ij}\in\so(n)$ ($i<j$) is a matrix with entry $+1$ at element
$(i,j)$, $-1$ at element $(j,i)$, and zero elsewhere,
$X=\sum_{i<j}x^{ij}E_{ij}$, $H=\diag(h_{1},\ldots,h_{n})$, and
$N=\diag(\nu_1,\ldots,\nu_n)$, then
$$\eqalign{&(\nabla^3\f)_\Theta(\Theta E_{ij},\Theta X,\Theta X)={}\cr
&\qquad{-2}\sum_{k\neq i,j}x^{ik}x^{jk}\bigl((h_i\nu_j-h_j\nu_i)
+(h_j\nu_k-h_k\nu_j) +(h_k\nu_i-h_i\nu_k)\bigr).\cr}$$ If the $h_i$
are close to $\alpha \nu_i$, $\alpha\in\R$, for~all~$i$, then
$(\nabla^3\f)_\Theta(\rdot,\Theta X,\Theta X)$ may be small, yielding
a fast rate of quadratic convergence.

\begin{example}[Jacobi's method]\label{eg:jacobi}\ignorespaces Let
$\pi$ be the projection of a square matrix onto its diagonal, and let
$Q$ be as above.  Consider the maximization of the function
$f(\Theta)=\tr H\pi(H)$, $H=\Ad_{\Theta^\T}(Q)$, on the special
orthogonal group. This is equivalent to minimizing the sum of the
squares of the off-diagonal elements of~$H$ (Golub and
Van~Loan~\citeyear{GVL} derive the classical Jacobi method).  The gradient
of this function at~$\Theta$ is $2\Theta[H,\pi(H)]$~\cite{Chu:grad}.
By repeated covariant differentiation of~$f\!$, we find
$$\def\.{\kern-30pt}\eqalign{(\nabla\f)_I(X) &=-2\tr[H,\pi(H)]X,\cr
(\nabla^2\f)_I(X,Y) &=-\tr\bigl([H,\ad_X\pi(H)] -[\ad_XH,\pi(H)]
-2[H,\pi(\ad_XH)]\bigr)Y\cr (\nabla^3\f)_I(X,Y,Z)
&=-\half\tr\bigl([\ad_Y\ad_ZH,\pi(H)] -[\ad_Z\ad_Y\pi(H),H],\cr
&\.{}+[H,\ad_{\ad_YZ}\pi(H)]-[\ad_YH,\ad_Z\pi(H)]+[\ad_Y\pi(H),\ad_ZH]\cr
&\.{}+2[H,\pi(\ad_Y\ad_ZH)]+2[H,\pi(\ad_Z\ad_YH)]\cr
&\.{}+2[\ad_YH,\pi(\ad_ZH)]-2[H,\ad_Y\pi(\ad_ZH)]\cr
&\.{}+2[\ad_ZH,\pi(\ad_YH)]-2[H,\ad_Z\pi(\ad_YH)]\bigr)X,\cr}$$ where
$I$ is the identity matrix and $X$, $Y$, $Z\in\so(n)$.  It is easily
shown that if $[H,\pi(H)]=0$, i.e., if $H$ is diagonal, then
$(\nabla^3\f)_\Theta(\rdot,\Theta X,\Theta X)=0$ (n.b.\
$\pi(\ad_XH)=0$). Therefore, by the same argument as the proof of
Remark~\ref{rem:trHNcub}, Newton's method applied to the function~$\tr
H\pi(H)$ converges cubically.
\end{example}

\section{Conjugate gradient method on Riemannian
manifolds}\label{sec:conjgrad}

The method of steepest descent provides an optimization technique
which is relatively inexpensive per iteration, but converges
relatively slowly.  Each step requires the computation of a geodesic
and a gradient direction.  Newton's method provides a technique which
is more costly both in terms of computational complexity and memory
requirements, but converges relatively rapidly.  Each step requires
the computation of a geodesic, a gradient, a second covariant
differential, and its inverse.  In this section we describe the
conjugate gradient method, which has the dual advantages of
algorithmic simplicity and superlinear convergence.

Hestenes and Stiefel~\citeyear{HestenesStiefel} first used conjugate
gradient methods to compute the solutions of linear equations, or,
equivalently, to compute the minimum of a quadratic form on~$\R^n$.
This approach can be modified to yield effective algorithms to compute
the minima of nonquadratic functions on~$\R^n$.  In particular,
Fletcher and Reeves~\citeyear{FletcherReeves} and Polak and
Ribi\`ere~\cite{Polak} provide algorithms based upon the assumption
that the second order Taylor expansion of the function to be minimized
sufficiently approximates this function near the minimum.  In
addition, Davidon, Fletcher, and Reeves developed the variable metric
methods~\cite{Davidon,Fletcher,Polak}, but these will not be discussed
here.  One noteworthy feature of conjugate gradient algorithms
on~$\R^n$ is that when the function to be minimized is quadratic, they
compute its minimum in no more than $n$ iterations, i.e., they have
the property of quadratic termination.

The conjugate gradient method on Euclidean space is uncomplicated.
Given a function $f\colon\R^n\to\R$ with continuous second derivatives
and a local minimum at~$\xhat$, and an initial point $x_0\in\R^n$, the
algorithm is initialized by computing the (negative) gradient
direction $G_0=H_0=-(\grad\f)_{x_0}$.  The recursive part of the
algorithm involves (i)~a line minimization of~$f$ along the affine
space $x_i+tH_i$, $t\in\R$, where the minimum occurs at, say,
$t=\lambda_i$, (ii)~computation of the step
$x_{i+1}=x_i+\lambda_iH_i$, (iii)~computation of the (negative)
gradient $G_{i+1}=-(\grad\f)_{x_{i+1}}$, and (iv)~computation of the
next direction for line minimization, \begin{equation}\label{eq:EucCG}
H_{i+1}=G_{i+1}+\gamma_iH_i, \end{equation} where $\gamma_i$ is chosen
such that $H_i$ and $H_{i+1}$ conjugate with respect to the Hessian
matrix of~$f$ at~$\xhat$.  When $f$ is a quadratic form represented by
the symmetric positive definite matrix~$Q$, the conjugacy condition
becomes $H_i^\T QH_{i+1}=0$; therefore, $\gamma_i=-H_i^\T
QG_{i+1}/H_i^\T QH_i$.  It can be shown in this case that the sequence
of vectors $G_i$ are all mutually orthogonal and the sequence of
vectors $H_i$ are all mutually conjugate with respect to~$Q$.  Using
these facts, the computation of $\gamma_i$ may be simplified with the
observation that $\gamma_i =\|G_{i+1}\|^2/\|G_i\|^2$ (Fletcher-Reeves)
or $\gamma_i =(G_{i+1}-G_i)^\T G_{i+1}/\|G_i\|^2$ (Polak-Ribi\`ere).
When $f$ is not quadratic, it is assumed that its second order Taylor
expansion sufficiently approximates $f$ in a neighborhood of the
minimum, and the $\gamma_i$ are chosen so that $H_i$ and $H_{i+1}$ are
conjugate with respect to the matrix $(\partial^2\f/\partial
x^i\partial x^j)(x_{i+1})$ of second partial derivatives of~$f$
at~$x_{i+1}$.  It may be desirable to ``reset'' the algorithm by
setting $H_{i+1}=G_{i+1}$ every $r$th step (frequently, $r=n$) because
the conjugate gradient method does not, in general, converge in $n$
steps if the function $f$ is nonquadratic.  However, if $f$ is closely
approximated by a quadratic function, the reset strategy may be
expected to converge rapidly, whereas the unmodified algorithm may not
be.

Many of these ideas have straightforward generalizations in the
geometry of Riemannian manifolds; several of them have already
appeared.  We need only make the following definition.

\begin{definition} Given a tensor field $\omega$ of type~$(0,2)$ on~$M$
such that for $p$ in~$M$, $\omega_p\colon T_p\times T_p\to\R$ is a
symmetric bilinear form, the tangent vectors $X$ and $Y$ in~$T_p$ are
said to be {\it $\omega_p$-conjugate\/} or {\it conjugate with respect
to $\omega_p$} if $\omega_p(X,Y)=0$.
\end{definition}

An outline of the conjugate gradient method on Riemannian manifolds
may now be given.  Let $M$ be an $n\hyphen$dimensional Riemannian
manifold with Riemannian structure $g$ and Levi-Civita connection
$\nabla$, and let $f\in C^\infty(M)$ have a local minimum at~$\phat$.
As in the conjugate gradient method on Euclidean space, choose an
initial point $p_0$ in~$M$ and compute the (negative) gradient
directions $G_0=H_0=-(\grad\f)_{p_0}$ in~$T_{p_0}$.  The recursive
part of the algorithm involves minimizing $f$ along the geodesic
$t\mapsto\exp_{p_i}tH_i$, $t\in\R$, making a step along the geodesic
to the minimum point $p_{i+1}=\exp\lambda_iH_i$, computing
$G_{i+1}=-(\grad\f)_{p_{i+1}}$, and computing the next direction in
$T_{p_{i+1}}$ for geodesic minimization. This direction is given by
the formula \begin{equation}\label{eq:RieCG}
H_{i+1}=G_{i+1}+\gamma_i\tau H_i, \end{equation} where $\tau$ is the
parallel translation with respect to the geodesic step from~$p_i$
to~$p_{i+1}$, and $\gamma_i$ is chosen such that $\tau H_i$ and
$H_{i+1}$ are $(\nabla^2\f)_{p_{i+1}}$-conjugate, i.e.,
\begin{equation}\label{eq:gammadef} \gamma_i =-{(\nabla^2\f)_{p_{i+1}}
(\tau H_i,G_{i+1})\over (\nabla^2\f)_{p_{i+1}}(\tau H_i,\tau H_i)}.
\end{equation}

Equation~(\ref{eq:gammadef}) is, in general, expensive to use because the
second covariant differential of~$f$ appears.  However, we can use the
Taylor expansion of~$df$ about $p_{i+1}$ to compute an efficient
approximation of~$\gamma_i$.  By the fact that
$p_i=\exp_{p_{i+1}}(-\lambda_i\tau H_i)$ and by Equation~(\ref{eq:taylort}),
we have $$\tau df_{p_i} =\tau df_{\exp_{p_{i+1}}(-\lambda_i\tau H_i)}
=df_{p_{i+1}} -\lambda_i(\nabla^2\f)_{p_{i+1}}(\rdot,\tau H_i) +\hot$$
Therefore, the numerator of the right hand side of
Equation~(\ref{eq:gammadef}) multiplied by the step size $\lambda_i$ can be
approximated by the equation $$\eqalign{\lambda_i
(\nabla^2\f)_{p_{i+1}} (\tau H_i, G_{i+1}) &=df_{p_{i+1}}(G_{i+1})
-(\tau df_{p_i})(G_{i+1})\cr &=-\(G_{i+1} -\tau G_i, G_{i+1}\)\cr}$$
because, by definition, $G_i=-(\grad\f)_{p_i}$, $i=0$, $1$,~\dots, and
for any $X$ in~$T_{p_{i+1}}$, $(\tau df_{p_i})(X)
=df_{p_i}(\tau^{-1}X) =\((\grad\f)_{p_i}, \tau^{-1}X\)
=\(\tau(\grad\f)_{p_i}, X\)$. Similarly, the denominator of the right
hand side of Equation~(\ref{eq:gammadef}) multiplied by $\lambda_i$ can be
approximated by the equation $$\eqalign{\lambda_i
(\nabla^2\f)_{p_{i+1}} (\tau H_i,\tau H_i) &=df_{p_{i+1}}(\tau H_i)
-(\tau df_{p_i})(\tau H_i)\cr &=\(G_i,H_i\)\cr}$$ because
$\(G_{i+1},\tau H_i\)=0$ by the assumption that $f$ is minimized along
the geodesic $t\mapsto\exp tH_i$ at~$t=\lambda_i$.  Combining these
two approximations with Equation~(\ref{eq:gammadef}), we obtain a formula
for~$\gamma_i$ that is relatively inexpensive to compute:
\begin{equation}\label{eq:gammaeff} \gamma_i ={\(G_{i+1} -\tau
G_i,G_{i+1}\)\over \(G_i,H_i\)}. \end{equation} Of course, as the
connection $\nabla$ is compatible with the metric $g$, the denominator
of Equation~(\ref{eq:gammaeff}) may be replaced, if desired, by $\(\tau
G_i,\tau H_i\)$.

The conjugate gradient method may now be presented in full.

\begin{algorithm}[Conjugate gradient
method]\label{al:cgman}\ignorespaces Let $M$ be a complete Riemannian
manifold with Riemannian structure $g$ and Levi-Civita connection
$\nabla$, and let $f$ be a $C^\infty$ function on~$M$.
\begin{steps}
\step[0.] Select $p_0\in M$,
     compute $G_0=H_0=-(\grad\f)_{p_0}$, and set $i=0$.
\step[1.] Compute $\lambda_i$ such that
     $$f(\exp_{p_i}\lambda_iH_i)\le f(\exp_{p_i}\lambda H_i)$$ for all
     $\lambda\ge0$.
\step[2.] Set $p_{i+1}=\exp_{p_i}\lambda_iH_i$.
\step[3.] Set $$\eqalign{G_{i+1}
     &=-(\grad\f)_{p_{i+1}},\cr H_{i+1} &=G_{i+1} +\gamma_i\tau H_i,
     \qquad\gamma_i={\(G_{i+1}-\tau G_i,G_{i+1}\)\over
     \(G_i,H_i\)},\cr \noalign{\vskip2pt}}$$ where $\tau$ is the
     parallel translation with respect to the geodesic from~$p_i$
     to~$p_{i+1}$.  If $i\equiv n-1\ (\bmod\ n)$, set
     $H_{i+1}=G_{i+1}$.  Increment $i$, and go~to Step~1.
\end{steps}
\end{algorithm}

\begin{theorem}\label{th:cgsuper}\ignorespaces Let $f\in C^\infty(M)$
have a nondegenerate critical point at $\phat$ such that the Hessian
$(d^2\f)_\phat$ is positive definite.  Let $p_i$ be a sequence of
points in~$M$ generated by Algorithm~\ref{al:cgman} converging
to~$\phat$.  Then there exists a constant $\theta>0$ and an integer
$N$ such that for~all $i\ge N$,
$$d(p_{i+n},\phat)\le \theta d^2(p_i,\phat).$$
\end{theorem}

Note that linear convergence is already guaranteed by
Theorem~\ref{th:linearconv}.

\begin{proof} If $p_j=\phat$ for some integer $j$, the assertion
becomes trivial; assume otherwise.  Recall that if $X_1$,~\dots, $X_n$
is some basis for $T_\phat$, then the map $\exp_\phat(a^1X_1 +\cdots+
a^nX_n) \buildrel\nu\over\to (a^1,\ldots,a^n)$ defines a set of normal
coordinates at~$\phat$.  Let $N_\phat$ be a normal neighborhood
of~$\phat$ on which the normal coordinates $\nu=(x^1,\ldots,x^n)$ are
defined. Consider the map $\nupushf\buildrel {\scriptscriptstyle\rm
def}\over= f\circ\nu^{-1}\colon\R^n\to\R$.  By the smoothness of $f$
and $\exp$, $\nupushf$ has a critical point at $0\in\R^n$ such that
the Hessian matrix of~$\nupushf$ at~$0$ is positive definite.  Indeed,
by the fact that $(d\exp)_0=\id$, the $ij$th component of the Hessian
matrix of $\nupushf$ at~$0$ is given by $(d^2\f)_\phat(X_i,X_j)$.

Therefore, there exists a neighborhood $U$ of $0\in\R^n$, a constant
$\theta'>0$, and an integer $N$, such that for any initial point
$x_0\in U$, the conjugate gradient method on Euclidean space (with
resets) applied to the function $\nupushf$ yields a sequence of points
$x_i$ converging to~$0$ such that for~all $i\ge N$,
$$\|x_{i+n}\|\le\theta'\|x_i\|^2.$$ See Polak~\citeyear[p.~260ff\/]{Polak}
for a proof of this fact.  Let $x_0=\nu(p_0)$ in~$U$ be an initial
point.  Because $\exp$ is not an isometry, Algorithm~\ref{al:cgman}
yields a different sequence of points in~$\R^n$ than the classical
conjugate gradient method on $\R^n$ (upon equating points in a
neighborhood of~$\phat\in M$ with points in a neighborhood
of~$0\in\R^n$ via the normal coordinates).

Nevertheless, the amount by which $\exp$ fails to preserve inner
products can be quantified via the Gauss Lemma and Jacobi's equation;
see, e.g., Cheeger and Ebin~\citeyear{Cheegin}, or the appendices of
Karcher~\citeyear{Karcher}.  Let $t$ be small, and let $X\in T_\phat$ and
$Y\in T_{tX}(T_\phat)\cong T_\phat$ be orthonormal tangent vectors.
The amount by which the exponential map changes the length of tangent
vectors is approximated by the Taylor expansion $$\|d\exp(tY)\|^2 =t^2
-\third Kt^4 +\hot,$$ where $K$ is the sectional curvature of~$M$ along
the section in~$T_\phat$ spanned by $X$ and $Y$.  Therefore, near
$\phat$ Algorithm~\ref{al:cgman} differs from the conjugate gradient
method on~$\R^n$ applied to the function $\nupushf$ only by third
order and higher terms.  Thus both algorithms have the same rate of
convergence.  The theorem follows.
\end{proof}

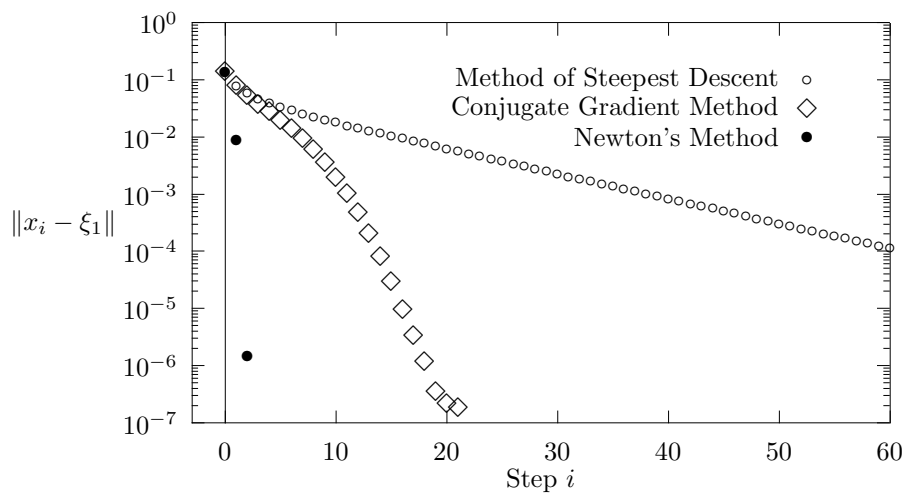
\begin{figure}
\begin{gnuplot}{20pt}
\let\p=\put
\let\r=\rule
\def\c{\circle{12}}
\def\D{\raisebox{-1.2pt}{\makebox(0,0){$\Diamond$}}}
\def\s{\circle*{18}}
\setlength{\unitlength}{0.240900pt}
\ifx\plotpoint\undefined\newsavebox{\plotpoint}\fi
\begin{picture}(1424,900)(0,0)
\tenrm
\sbox{\plotpoint}{\r[-0.175pt]{0.350pt}{0.350pt}}%
\p(316,158){\r[-0.175pt]{0.350pt}{151.526pt}}
\p(264,158){\r[-0.175pt]{4.818pt}{0.350pt}}
\p(242,158){\makebox(0,0)[r]{$10^{\gnulog 1e-07}$}}
\p(1340,158){\r[-0.175pt]{4.818pt}{0.350pt}}
\p(264,185){\r[-0.175pt]{2.409pt}{0.350pt}}
\p(1350,185){\r[-0.175pt]{2.409pt}{0.350pt}}
\p(264,201){\r[-0.175pt]{2.409pt}{0.350pt}}
\p(1350,201){\r[-0.175pt]{2.409pt}{0.350pt}}
\p(264,212){\r[-0.175pt]{2.409pt}{0.350pt}}
\p(1350,212){\r[-0.175pt]{2.409pt}{0.350pt}}
\p(264,221){\r[-0.175pt]{2.409pt}{0.350pt}}
\p(1350,221){\r[-0.175pt]{2.409pt}{0.350pt}}
\p(264,228){\r[-0.175pt]{2.409pt}{0.350pt}}
\p(1350,228){\r[-0.175pt]{2.409pt}{0.350pt}}
\p(264,234){\r[-0.175pt]{2.409pt}{0.350pt}}
\p(1350,234){\r[-0.175pt]{2.409pt}{0.350pt}}
\p(264,239){\r[-0.175pt]{2.409pt}{0.350pt}}
\p(1350,239){\r[-0.175pt]{2.409pt}{0.350pt}}
\p(264,244){\r[-0.175pt]{2.409pt}{0.350pt}}
\p(1350,244){\r[-0.175pt]{2.409pt}{0.350pt}}
\p(264,248){\r[-0.175pt]{4.818pt}{0.350pt}}
\p(242,248){\makebox(0,0)[r]{$10^{\gnulog 1e-06}$}}
\p(1340,248){\r[-0.175pt]{4.818pt}{0.350pt}}
\p(264,275){\r[-0.175pt]{2.409pt}{0.350pt}}
\p(1350,275){\r[-0.175pt]{2.409pt}{0.350pt}}
\p(264,291){\r[-0.175pt]{2.409pt}{0.350pt}}
\p(1350,291){\r[-0.175pt]{2.409pt}{0.350pt}}
\p(264,302){\r[-0.175pt]{2.409pt}{0.350pt}}
\p(1350,302){\r[-0.175pt]{2.409pt}{0.350pt}}
\p(264,311){\r[-0.175pt]{2.409pt}{0.350pt}}
\p(1350,311){\r[-0.175pt]{2.409pt}{0.350pt}}
\p(264,318){\r[-0.175pt]{2.409pt}{0.350pt}}
\p(1350,318){\r[-0.175pt]{2.409pt}{0.350pt}}
\p(264,324){\r[-0.175pt]{2.409pt}{0.350pt}}
\p(1350,324){\r[-0.175pt]{2.409pt}{0.350pt}}
\p(264,329){\r[-0.175pt]{2.409pt}{0.350pt}}
\p(1350,329){\r[-0.175pt]{2.409pt}{0.350pt}}
\p(264,334){\r[-0.175pt]{2.409pt}{0.350pt}}
\p(1350,334){\r[-0.175pt]{2.409pt}{0.350pt}}
\p(264,338){\r[-0.175pt]{4.818pt}{0.350pt}}
\p(242,338){\makebox(0,0)[r]{$10^{\gnulog 1e-05}$}}
\p(1340,338){\r[-0.175pt]{4.818pt}{0.350pt}}
\p(264,365){\r[-0.175pt]{2.409pt}{0.350pt}}
\p(1350,365){\r[-0.175pt]{2.409pt}{0.350pt}}
\p(264,381){\r[-0.175pt]{2.409pt}{0.350pt}}
\p(1350,381){\r[-0.175pt]{2.409pt}{0.350pt}}
\p(264,392){\r[-0.175pt]{2.409pt}{0.350pt}}
\p(1350,392){\r[-0.175pt]{2.409pt}{0.350pt}}
\p(264,401){\r[-0.175pt]{2.409pt}{0.350pt}}
\p(1350,401){\r[-0.175pt]{2.409pt}{0.350pt}}
\p(264,408){\r[-0.175pt]{2.409pt}{0.350pt}}
\p(1350,408){\r[-0.175pt]{2.409pt}{0.350pt}}
\p(264,414){\r[-0.175pt]{2.409pt}{0.350pt}}
\p(1350,414){\r[-0.175pt]{2.409pt}{0.350pt}}
\p(264,419){\r[-0.175pt]{2.409pt}{0.350pt}}
\p(1350,419){\r[-0.175pt]{2.409pt}{0.350pt}}
\p(264,423){\r[-0.175pt]{2.409pt}{0.350pt}}
\p(1350,423){\r[-0.175pt]{2.409pt}{0.350pt}}
\p(264,428){\r[-0.175pt]{4.818pt}{0.350pt}}
\p(242,428){\makebox(0,0)[r]{$10^{\gnulog 1e-04}$}}
\p(1340,428){\r[-0.175pt]{4.818pt}{0.350pt}}
\p(264,455){\r[-0.175pt]{2.409pt}{0.350pt}}
\p(1350,455){\r[-0.175pt]{2.409pt}{0.350pt}}
\p(264,470){\r[-0.175pt]{2.409pt}{0.350pt}}
\p(1350,470){\r[-0.175pt]{2.409pt}{0.350pt}}
\p(264,482){\r[-0.175pt]{2.409pt}{0.350pt}}
\p(1350,482){\r[-0.175pt]{2.409pt}{0.350pt}}
\p(264,490){\r[-0.175pt]{2.409pt}{0.350pt}}
\p(1350,490){\r[-0.175pt]{2.409pt}{0.350pt}}
\p(264,497){\r[-0.175pt]{2.409pt}{0.350pt}}
\p(1350,497){\r[-0.175pt]{2.409pt}{0.350pt}}
\p(264,504){\r[-0.175pt]{2.409pt}{0.350pt}}
\p(1350,504){\r[-0.175pt]{2.409pt}{0.350pt}}
\p(264,509){\r[-0.175pt]{2.409pt}{0.350pt}}
\p(1350,509){\r[-0.175pt]{2.409pt}{0.350pt}}
\p(264,513){\r[-0.175pt]{2.409pt}{0.350pt}}
\p(1350,513){\r[-0.175pt]{2.409pt}{0.350pt}}
\p(264,517){\r[-0.175pt]{4.818pt}{0.350pt}}
\p(242,517){\makebox(0,0)[r]{$10^{\gnulog 1e-03}$}}
\p(1340,517){\r[-0.175pt]{4.818pt}{0.350pt}}
\p(264,544){\r[-0.175pt]{2.409pt}{0.350pt}}
\p(1350,544){\r[-0.175pt]{2.409pt}{0.350pt}}
\p(264,560){\r[-0.175pt]{2.409pt}{0.350pt}}
\p(1350,560){\r[-0.175pt]{2.409pt}{0.350pt}}
\p(264,572){\r[-0.175pt]{2.409pt}{0.350pt}}
\p(1350,572){\r[-0.175pt]{2.409pt}{0.350pt}}
\p(264,580){\r[-0.175pt]{2.409pt}{0.350pt}}
\p(1350,580){\r[-0.175pt]{2.409pt}{0.350pt}}
\p(264,587){\r[-0.175pt]{2.409pt}{0.350pt}}
\p(1350,587){\r[-0.175pt]{2.409pt}{0.350pt}}
\p(264,593){\r[-0.175pt]{2.409pt}{0.350pt}}
\p(1350,593){\r[-0.175pt]{2.409pt}{0.350pt}}
\p(264,599){\r[-0.175pt]{2.409pt}{0.350pt}}
\p(1350,599){\r[-0.175pt]{2.409pt}{0.350pt}}
\p(264,603){\r[-0.175pt]{2.409pt}{0.350pt}}
\p(1350,603){\r[-0.175pt]{2.409pt}{0.350pt}}
\p(264,607){\r[-0.175pt]{4.818pt}{0.350pt}}
\p(242,607){\makebox(0,0)[r]{$10^{\gnulog 1e-02}$}}
\p(1340,607){\r[-0.175pt]{4.818pt}{0.350pt}}
\p(264,634){\r[-0.175pt]{2.409pt}{0.350pt}}
\p(1350,634){\r[-0.175pt]{2.409pt}{0.350pt}}
\p(264,650){\r[-0.175pt]{2.409pt}{0.350pt}}
\p(1350,650){\r[-0.175pt]{2.409pt}{0.350pt}}
\p(264,661){\r[-0.175pt]{2.409pt}{0.350pt}}
\p(1350,661){\r[-0.175pt]{2.409pt}{0.350pt}}
\p(264,670){\r[-0.175pt]{2.409pt}{0.350pt}}
\p(1350,670){\r[-0.175pt]{2.409pt}{0.350pt}}
\p(264,677){\r[-0.175pt]{2.409pt}{0.350pt}}
\p(1350,677){\r[-0.175pt]{2.409pt}{0.350pt}}
\p(264,683){\r[-0.175pt]{2.409pt}{0.350pt}}
\p(1350,683){\r[-0.175pt]{2.409pt}{0.350pt}}
\p(264,688){\r[-0.175pt]{2.409pt}{0.350pt}}
\p(1350,688){\r[-0.175pt]{2.409pt}{0.350pt}}
\p(264,693){\r[-0.175pt]{2.409pt}{0.350pt}}
\p(1350,693){\r[-0.175pt]{2.409pt}{0.350pt}}
\p(264,697){\r[-0.175pt]{4.818pt}{0.350pt}}
\p(242,697){\makebox(0,0)[r]{$10^{\gnulog 1e-01}$}}
\p(1340,697){\r[-0.175pt]{4.818pt}{0.350pt}}
\p(264,724){\r[-0.175pt]{2.409pt}{0.350pt}}
\p(1350,724){\r[-0.175pt]{2.409pt}{0.350pt}}
\p(264,740){\r[-0.175pt]{2.409pt}{0.350pt}}
\p(1350,740){\r[-0.175pt]{2.409pt}{0.350pt}}
\p(264,751){\r[-0.175pt]{2.409pt}{0.350pt}}
\p(1350,751){\r[-0.175pt]{2.409pt}{0.350pt}}
\p(264,760){\r[-0.175pt]{2.409pt}{0.350pt}}
\p(1350,760){\r[-0.175pt]{2.409pt}{0.350pt}}
\p(264,767){\r[-0.175pt]{2.409pt}{0.350pt}}
\p(1350,767){\r[-0.175pt]{2.409pt}{0.350pt}}
\p(264,773){\r[-0.175pt]{2.409pt}{0.350pt}}
\p(1350,773){\r[-0.175pt]{2.409pt}{0.350pt}}
\p(264,778){\r[-0.175pt]{2.409pt}{0.350pt}}
\p(1350,778){\r[-0.175pt]{2.409pt}{0.350pt}}
\p(264,783){\r[-0.175pt]{2.409pt}{0.350pt}}
\p(1350,783){\r[-0.175pt]{2.409pt}{0.350pt}}
\p(264,787){\r[-0.175pt]{4.818pt}{0.350pt}}
\p(242,787){\makebox(0,0)[r]{$10^{\gnulog 1e+00}$}}
\p(1340,787){\r[-0.175pt]{4.818pt}{0.350pt}}
\p(316,158){\r[-0.175pt]{0.350pt}{4.818pt}} \p(316,113){\makebox(0,0){$0$}}
\p(316,767){\r[-0.175pt]{0.350pt}{4.818pt}}
\p(490,158){\r[-0.175pt]{0.350pt}{4.818pt}} \p(490,113){\makebox(0,0){$10$}}
\p(490,767){\r[-0.175pt]{0.350pt}{4.818pt}}
\p(664,158){\r[-0.175pt]{0.350pt}{4.818pt}} \p(664,113){\makebox(0,0){$20$}}
\p(664,767){\r[-0.175pt]{0.350pt}{4.818pt}}
\p(838,158){\r[-0.175pt]{0.350pt}{4.818pt}} \p(838,113){\makebox(0,0){$30$}}
\p(838,767){\r[-0.175pt]{0.350pt}{4.818pt}}
\p(1012,158){\r[-0.175pt]{0.350pt}{4.818pt}} \p(1012,113){\makebox(0,0){$40$}}
\p(1012,767){\r[-0.175pt]{0.350pt}{4.818pt}}
\p(1186,158){\r[-0.175pt]{0.350pt}{4.818pt}} \p(1186,113){\makebox(0,0){$50$}}
\p(1186,767){\r[-0.175pt]{0.350pt}{4.818pt}}
\p(1360,158){\r[-0.175pt]{0.350pt}{4.818pt}} \p(1360,113){\makebox(0,0){$60$}}
\p(1360,767){\r[-0.175pt]{0.350pt}{4.818pt}}
\p(264,158){\r[-0.175pt]{264.026pt}{0.350pt}}
\p(1360,158){\r[-0.175pt]{0.350pt}{151.526pt}}
\p(264,787){\r[-0.175pt]{264.026pt}{0.350pt}}
\p(-21,472){\makebox(0,0)[l]{\shortstack{$\|x_i-\xi_1\|$}}}
\p(812,68){\makebox(0,0){Step $i$}}
\p(264,158){\r[-0.175pt]{0.350pt}{151.526pt}}
\p(1186,697){\makebox(0,0)[r]{Method of Steepest Descent}}
\p(1230,697){\c} \p(316,710){\c} \p(334,688){\c} \p(351,676){\c}
\p(368,667){\c} \p(386,660){\c} \p(403,654){\c} \p(421,649){\c}
\p(438,644){\c} \p(455,639){\c} \p(473,634){\c} \p(490,630){\c}
\p(508,625){\c} \p(525,621){\c} \p(542,617){\c} \p(560,613){\c}
\p(577,609){\c} \p(595,605){\c} \p(612,601){\c} \p(629,597){\c}
\p(647,593){\c} \p(664,589){\c} \p(682,585){\c} \p(699,581){\c}
\p(716,577){\c} \p(734,573){\c} \p(751,569){\c} \p(769,565){\c}
\p(786,561){\c} \p(803,557){\c} \p(821,553){\c} \p(838,549){\c}
\p(855,545){\c} \p(873,541){\c} \p(890,538){\c} \p(908,534){\c}
\p(925,530){\c} \p(942,526){\c} \p(960,522){\c} \p(977,518){\c}
\p(995,514){\c} \p(1012,510){\c} \p(1029,506){\c} \p(1047,502){\c}
\p(1064,499){\c} \p(1082,495){\c} \p(1099,491){\c} \p(1116,487){\c}
\p(1134,483){\c} \p(1151,479){\c} \p(1169,475){\c} \p(1186,471){\c}
\p(1203,467){\c} \p(1221,463){\c} \p(1238,459){\c} \p(1256,455){\c}
\p(1273,452){\c} \p(1290,448){\c} \p(1308,444){\c} \p(1325,440){\c}
\p(1343,436){\c} \p(1360,432){\c} \p(1186,652){\makebox(0,0)[r]{Conjugate
Gradient Method}} \p(1230,652){\D} \p(316,710){\D} \p(334,688){\D}
\p(351,672){\D} \p(368,659){\D} \p(386,647){\D} \p(403,634){\D}
\p(421,621){\D} \p(438,605){\D} \p(455,588){\D} \p(473,567){\D}
\p(490,544){\D} \p(508,518){\D} \p(525,488){\D} \p(542,456){\D}
\p(560,420){\D} \p(577,380){\D} \p(595,337){\D} \p(612,295){\D}
\p(629,254){\D} \p(647,207){\D} \p(664,188){\D} \p(682,183){\D}
\p(1186,607){\makebox(0,0)[r]{Newton's Method}} \p(1230,607){\s}
\p(316,710){\s} \p(334,603){\s} \p(351,263){\s}
\end{picture}
\end{gnuplot}
\caption[Convergence with the Rayleigh quotient on
$S^{20}$]{Maximization of the Rayleigh quotient $x^\T Qx$
on~$S^{20}\subset\R^{21}$, where $Q=\diag(21,\ldots,1)$.  The $i$th
iterate is $x_i$, and $\xi_1$ is the eigenvector corresponding to the
largest eigenvalue of~$Q$.  Algorithm~\protect\ref{al:newtonray} was
used for Newton's method and Algorithm~\protect\ref{al:raysphere} was
used for the conjugate gradient method.\label{fig:raysphere}}
\end{figure}

\begin{example}[Rayleigh quotient on the
sphere]\label{eg:rayqcg}\ignorespaces Applied to the Rayleigh quotient
on the sphere, the conjugate gradient method provides an efficient
technique to compute the eigenvectors corresponding to the largest or
smallest eigenvalue of a real symmetric matrix.  Let $S^{n-1}$ and
$\rho(x)=x^\T Qx$ be as in Examples \ref{eg:rayqgrad} and
\ref{eg:newtonray}.  From Algorithm~\ref{al:cgman}, we have the
following algorithm.

\begin{algorithm}[Conjugate gradient for the extreme
eigenvalue/eigenvector]\label{al:raysphere}\ignorespaces Let $Q$ be a
real symmetric $n\by n$ matrix.
\begin{steps}
\step[0.] Select $x_0$ in~$\R^n$ such that $x_0^\T x_0=1$, compute
    $G_0=H_0=(Q-\rho(x_0)I)x_0$, and set $i=0$.
\step[1.] Compute $c$, $s$, and $v=1-c=s^2/(1+c)$, such that
    $\rho(x_ic+h_is)$ is maximized, where $c^2+s^2=1$ and
    $h_i=H_i/\|H_i\|$.  This can be accomplished by geodesic
    minimization, or by the formulae $$\eqalign{c
    &=\bigl(\half(1+b/r)\bigr)^\half\cr s &=a/(2rc)\cr} \quad\hbox{if
    $b\geq0$,} \quad\hbox{or}\quad \eqalign{s
    &=\bigl(\half(1-b/r)\bigr)^\half\cr c &=a/(2rs)\cr} \quad\hbox{if
    $b\leq0$,}$$ where $a=2x_i^\T Qh_i$, $b=x_i^\T Qx_i -h_i^\T Qh_i$,
    and $r=\surd(a^2+b^2)$.
\step[2.] Set $$x_{i+1}=x_ic+h_is,\quad \tau H_i=H_ic -
    x_i\|H_i\|s,\quad \tau G_i=G_i -(h_i^\T G_i)(x_is+h_iv).$$
\step[3.] Set $$\eqalign{G_{i+1} &=\bigl(Q-\rho(x_{i+1})I\bigr)x_{i+1},\cr
    H_{i+1} &=G_{i+1} +\gamma_i\tau H_i,
    \qquad\gamma_i={(G_{i+1}-\tau G_i)^\T G_{i+1}\over G_i^\T
    H_i}.\cr\noalign{\vskip2pt}}$$ If $i\equiv n-1\ (\bmod\ n)$, set
    $H_{i+1}=G_{i+1}$.  Increment $i$, and go~to Step~1.
\end{steps}
\end{algorithm}

\begin{figure}[t]
\def\contours{\pdfximage width 4.5in {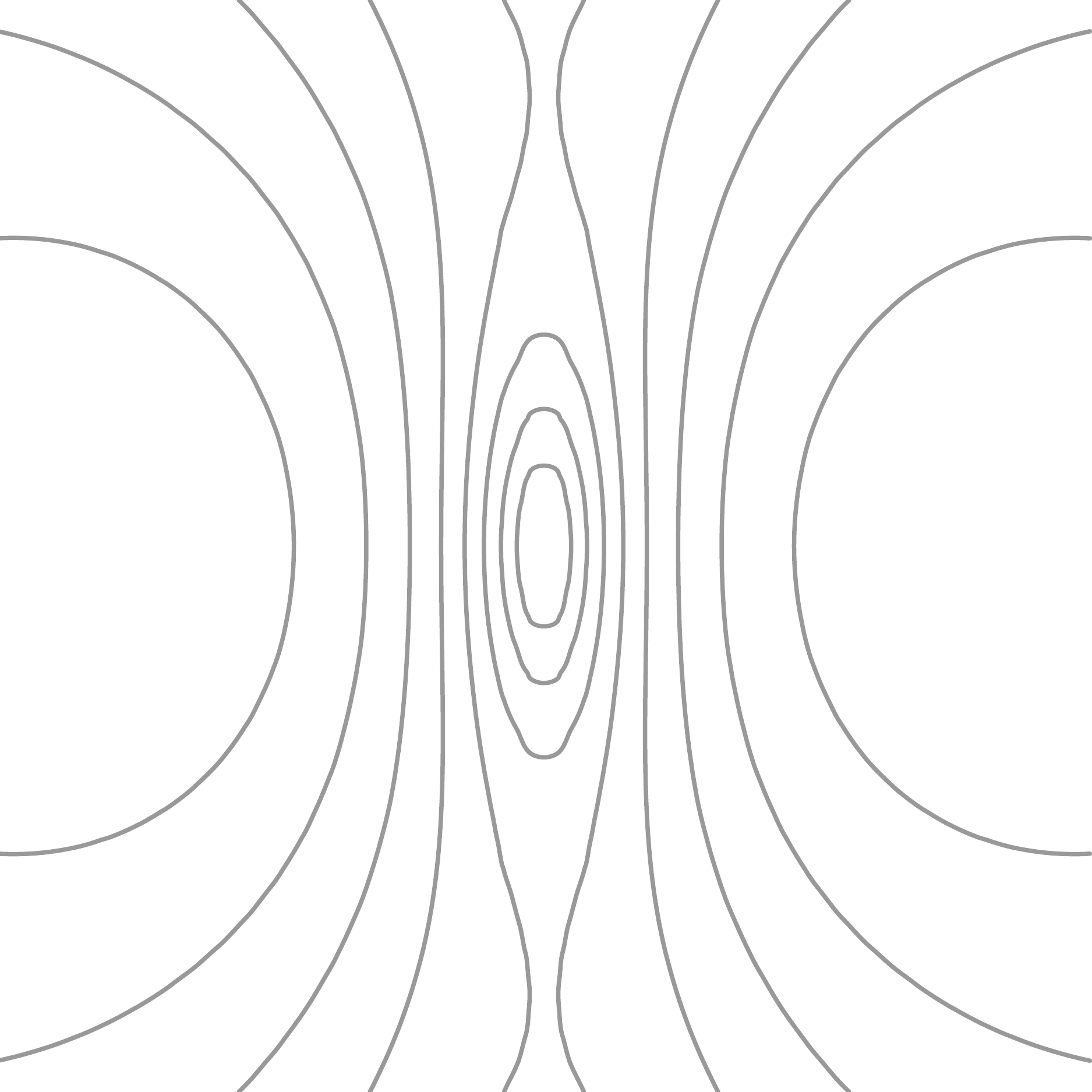}\pdfrefximage\pdflastximage}
\def\steepdesc{\pdfximage width 4.5in {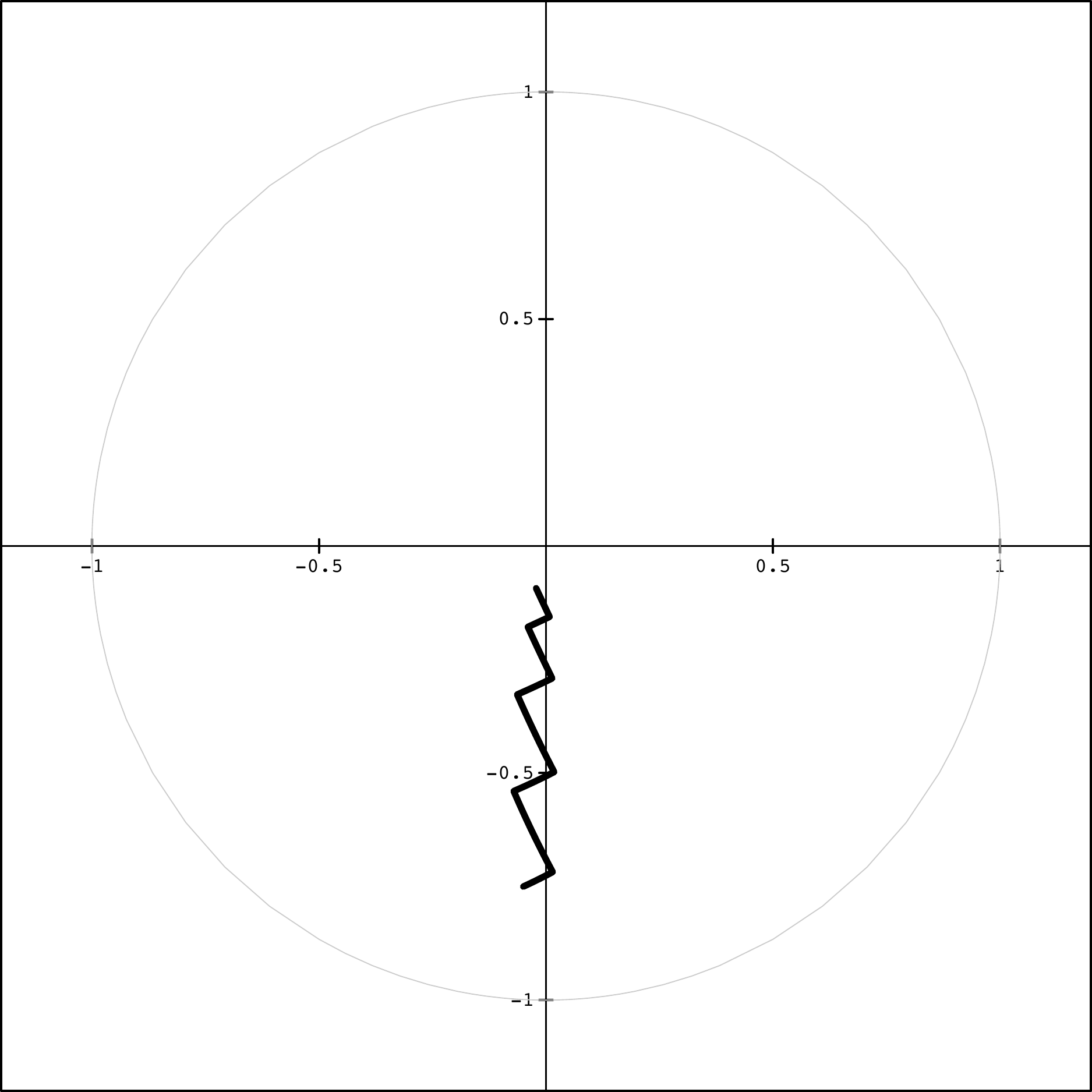}\pdfrefximage\pdflastximage}
\def\conjgrad{\pdfximage width 4.5in {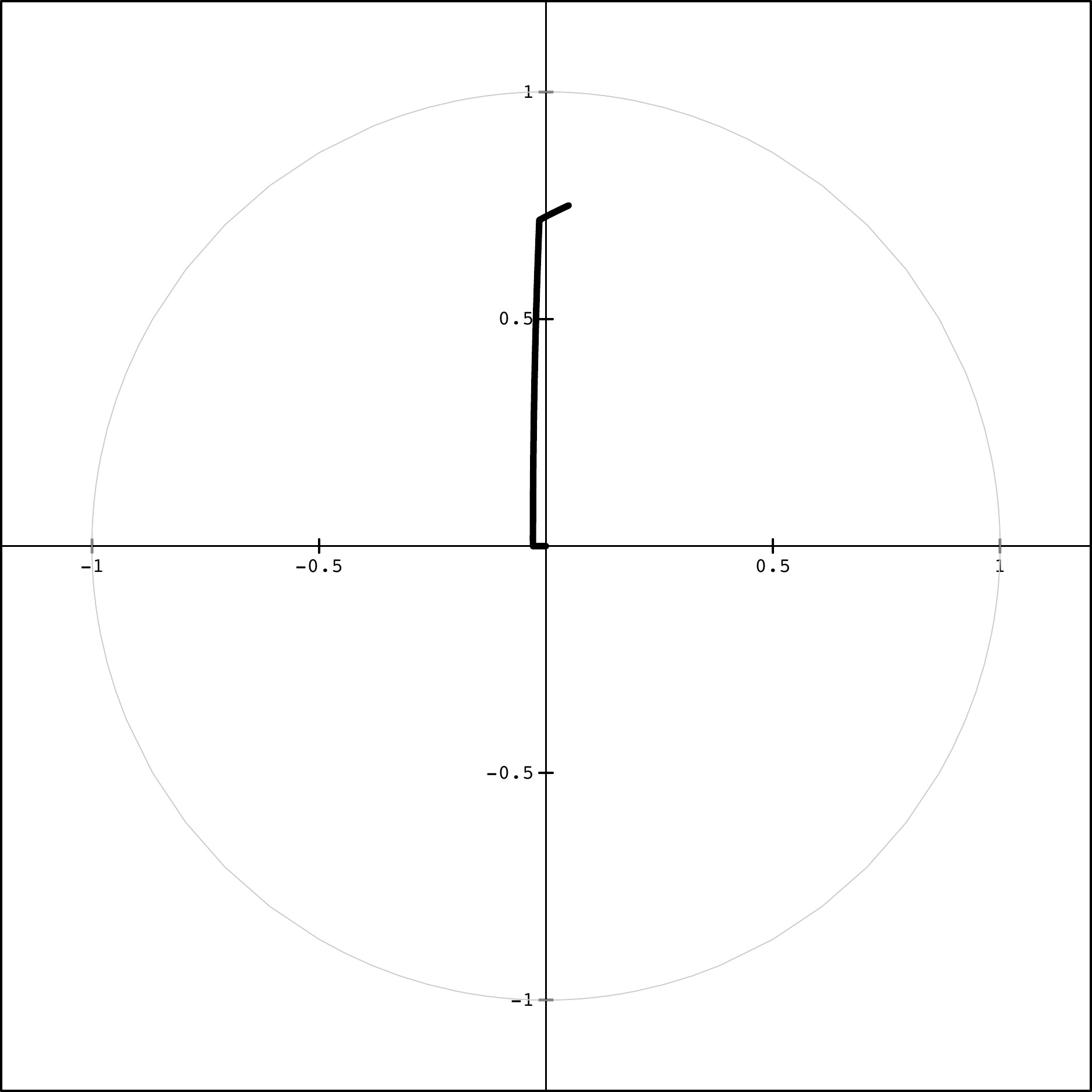}\pdfrefximage\pdflastximage}
\centerline{\rlap{\contours}
  \rlap{\steepdesc}\conjgrad}
\caption[Conjugate gradient iterates on the $2$-sphere]{Iterates of
the conjugate gradient method applied to the Rayleigh quotient on the
$2$-sphere.  The sphere is mapped sterographically onto~$\R^2$ with
the north pole at the origin and the equator represented by the thin
gray unit circle. Contours of the Rayleigh quotient $x^\T Qx$,
$Q=\diag(1,9,10)$, are represented by the dark gray curves. The
iterates of the Algorithm~\ref{al:raysphere} are connected by
geodesics shown in the upper black path. Note that this function has a
nonterminating Taylor series; therefore, the quadratic termination
property of the Euclidean conjugate gradient method is not seen.  The
iterates of the method of steepest descent are shown in the lower
black path.\looseness=-1\par \label{fig:2sphere}}
\end{figure}

The convergence rate of this algorithm to the eigenvector
corresponding to the largest eigenvalue of~$Q$ is given by
Theorem~\ref{th:cgsuper}.  This algorithm requires one matrix-vector
multiplication (relatively inexpensive when $Q$ is sparse), one
geodesic minimization or computation of~$\rho(h_i)$, and $10n$~flops
per iteration.  The results of a numerical experiment demonstrating
the convergence of Algorithm~\ref{al:raysphere} on~$S^{20}$ are shown
in Figure~\ref{fig:raysphere}.  A graphical illustration of the
conjugate gradient algorithm's performance on the $2$-sphere is shown
in Figure~\ref{fig:2sphere}. Stereographic projection is used to map
the sphere onto the plane.  There are maximum points at the north and
south poles, located at the center of the image and at infinity,
respectively. There are minimum points and saddle points antipodally
located along the equator, which is shown by the thin gray circle.
The light gray contours represent the level sets of the function $x^\T
Qx$ on~$S^2\subset\R^3$, where $Q=\diag(1,9,10)$.  The conjugate
gradient method was used to compute the sequence of points at the top
of the figure, and the method of steepest descent was used to compute
the sequence of points at the bottom.  Fuhrmann and
Liu~\citeyear{FuhrLiu} provide a conjugate gradient algorithm for the
Rayleigh quotient on the sphere that uses an azimuthal projection onto
tangent planes.
\end{example}

\begin{figure}
\begin{gnuplot}{20pt}
\let\p=\put
\let\r=\rule
\def\c{\circle{12}}
\def\D{\raisebox{-1.2pt}{\makebox(0,0){$\Diamond$}}}
\def\s{\circle*{18}}
\setlength{\unitlength}{0.240900pt}
\ifx\plotpoint\undefined\newsavebox{\plotpoint}\fi
\begin{picture}(1424,900)(0,0)
\tenrm
\sbox{\plotpoint}{\r[-0.175pt]{0.350pt}{0.350pt}}%
\p(333,158){\r[-0.175pt]{0.350pt}{151.526pt}}
\p(264,158){\r[-0.175pt]{4.818pt}{0.350pt}}
\p(242,158){\makebox(0,0)[r]{$10^{\gnulog 1e-05}$}}
\p(1340,158){\r[-0.175pt]{4.818pt}{0.350pt}}
\p(264,190){\r[-0.175pt]{2.409pt}{0.350pt}}
\p(1350,190){\r[-0.175pt]{2.409pt}{0.350pt}}
\p(264,208){\r[-0.175pt]{2.409pt}{0.350pt}}
\p(1350,208){\r[-0.175pt]{2.409pt}{0.350pt}}
\p(264,221){\r[-0.175pt]{2.409pt}{0.350pt}}
\p(1350,221){\r[-0.175pt]{2.409pt}{0.350pt}}
\p(264,231){\r[-0.175pt]{2.409pt}{0.350pt}}
\p(1350,231){\r[-0.175pt]{2.409pt}{0.350pt}}
\p(264,240){\r[-0.175pt]{2.409pt}{0.350pt}}
\p(1350,240){\r[-0.175pt]{2.409pt}{0.350pt}}
\p(264,247){\r[-0.175pt]{2.409pt}{0.350pt}}
\p(1350,247){\r[-0.175pt]{2.409pt}{0.350pt}}
\p(264,253){\r[-0.175pt]{2.409pt}{0.350pt}}
\p(1350,253){\r[-0.175pt]{2.409pt}{0.350pt}}
\p(264,258){\r[-0.175pt]{2.409pt}{0.350pt}}
\p(1350,258){\r[-0.175pt]{2.409pt}{0.350pt}}
\p(264,263){\r[-0.175pt]{4.818pt}{0.350pt}}
\p(242,263){\makebox(0,0)[r]{$10^{\gnulog 1e-04}$}}
\p(1340,263){\r[-0.175pt]{4.818pt}{0.350pt}}
\p(264,294){\r[-0.175pt]{2.409pt}{0.350pt}}
\p(1350,294){\r[-0.175pt]{2.409pt}{0.350pt}}
\p(264,313){\r[-0.175pt]{2.409pt}{0.350pt}}
\p(1350,313){\r[-0.175pt]{2.409pt}{0.350pt}}
\p(264,326){\r[-0.175pt]{2.409pt}{0.350pt}}
\p(1350,326){\r[-0.175pt]{2.409pt}{0.350pt}}
\p(264,336){\r[-0.175pt]{2.409pt}{0.350pt}}
\p(1350,336){\r[-0.175pt]{2.409pt}{0.350pt}}
\p(264,344){\r[-0.175pt]{2.409pt}{0.350pt}}
\p(1350,344){\r[-0.175pt]{2.409pt}{0.350pt}}
\p(264,351){\r[-0.175pt]{2.409pt}{0.350pt}}
\p(1350,351){\r[-0.175pt]{2.409pt}{0.350pt}}
\p(264,358){\r[-0.175pt]{2.409pt}{0.350pt}}
\p(1350,358){\r[-0.175pt]{2.409pt}{0.350pt}}
\p(264,363){\r[-0.175pt]{2.409pt}{0.350pt}}
\p(1350,363){\r[-0.175pt]{2.409pt}{0.350pt}}
\p(264,368){\r[-0.175pt]{4.818pt}{0.350pt}}
\p(242,368){\makebox(0,0)[r]{$10^{\gnulog 1e-03}$}}
\p(1340,368){\r[-0.175pt]{4.818pt}{0.350pt}}
\p(264,399){\r[-0.175pt]{2.409pt}{0.350pt}}
\p(1350,399){\r[-0.175pt]{2.409pt}{0.350pt}}
\p(264,418){\r[-0.175pt]{2.409pt}{0.350pt}}
\p(1350,418){\r[-0.175pt]{2.409pt}{0.350pt}}
\p(264,431){\r[-0.175pt]{2.409pt}{0.350pt}}
\p(1350,431){\r[-0.175pt]{2.409pt}{0.350pt}}
\p(264,441){\r[-0.175pt]{2.409pt}{0.350pt}}
\p(1350,441){\r[-0.175pt]{2.409pt}{0.350pt}}
\p(264,449){\r[-0.175pt]{2.409pt}{0.350pt}}
\p(1350,449){\r[-0.175pt]{2.409pt}{0.350pt}}
\p(264,456){\r[-0.175pt]{2.409pt}{0.350pt}}
\p(1350,456){\r[-0.175pt]{2.409pt}{0.350pt}}
\p(264,462){\r[-0.175pt]{2.409pt}{0.350pt}}
\p(1350,462){\r[-0.175pt]{2.409pt}{0.350pt}}
\p(264,468){\r[-0.175pt]{2.409pt}{0.350pt}}
\p(1350,468){\r[-0.175pt]{2.409pt}{0.350pt}}
\p(264,473){\r[-0.175pt]{4.818pt}{0.350pt}}
\p(242,473){\makebox(0,0)[r]{$10^{\gnulog 1e-02}$}}
\p(1340,473){\r[-0.175pt]{4.818pt}{0.350pt}}
\p(264,504){\r[-0.175pt]{2.409pt}{0.350pt}}
\p(1350,504){\r[-0.175pt]{2.409pt}{0.350pt}}
\p(264,523){\r[-0.175pt]{2.409pt}{0.350pt}}
\p(1350,523){\r[-0.175pt]{2.409pt}{0.350pt}}
\p(264,536){\r[-0.175pt]{2.409pt}{0.350pt}}
\p(1350,536){\r[-0.175pt]{2.409pt}{0.350pt}}
\p(264,546){\r[-0.175pt]{2.409pt}{0.350pt}}
\p(1350,546){\r[-0.175pt]{2.409pt}{0.350pt}}
\p(264,554){\r[-0.175pt]{2.409pt}{0.350pt}}
\p(1350,554){\r[-0.175pt]{2.409pt}{0.350pt}}
\p(264,561){\r[-0.175pt]{2.409pt}{0.350pt}}
\p(1350,561){\r[-0.175pt]{2.409pt}{0.350pt}}
\p(264,567){\r[-0.175pt]{2.409pt}{0.350pt}}
\p(1350,567){\r[-0.175pt]{2.409pt}{0.350pt}}
\p(264,573){\r[-0.175pt]{2.409pt}{0.350pt}}
\p(1350,573){\r[-0.175pt]{2.409pt}{0.350pt}}
\p(264,577){\r[-0.175pt]{4.818pt}{0.350pt}}
\p(242,577){\makebox(0,0)[r]{$10^{\gnulog 1e-01}$}}
\p(1340,577){\r[-0.175pt]{4.818pt}{0.350pt}}
\p(264,609){\r[-0.175pt]{2.409pt}{0.350pt}}
\p(1350,609){\r[-0.175pt]{2.409pt}{0.350pt}}
\p(264,627){\r[-0.175pt]{2.409pt}{0.350pt}}
\p(1350,627){\r[-0.175pt]{2.409pt}{0.350pt}}
\p(264,640){\r[-0.175pt]{2.409pt}{0.350pt}}
\p(1350,640){\r[-0.175pt]{2.409pt}{0.350pt}}
\p(264,651){\r[-0.175pt]{2.409pt}{0.350pt}}
\p(1350,651){\r[-0.175pt]{2.409pt}{0.350pt}}
\p(264,659){\r[-0.175pt]{2.409pt}{0.350pt}}
\p(1350,659){\r[-0.175pt]{2.409pt}{0.350pt}}
\p(264,666){\r[-0.175pt]{2.409pt}{0.350pt}}
\p(1350,666){\r[-0.175pt]{2.409pt}{0.350pt}}
\p(264,672){\r[-0.175pt]{2.409pt}{0.350pt}}
\p(1350,672){\r[-0.175pt]{2.409pt}{0.350pt}}
\p(264,677){\r[-0.175pt]{2.409pt}{0.350pt}}
\p(1350,677){\r[-0.175pt]{2.409pt}{0.350pt}}
\p(264,682){\r[-0.175pt]{4.818pt}{0.350pt}}
\p(242,682){\makebox(0,0)[r]{$10^{\gnulog 1e+00}$}}
\p(1340,682){\r[-0.175pt]{4.818pt}{0.350pt}}
\p(264,714){\r[-0.175pt]{2.409pt}{0.350pt}}
\p(1350,714){\r[-0.175pt]{2.409pt}{0.350pt}}
\p(264,732){\r[-0.175pt]{2.409pt}{0.350pt}}
\p(1350,732){\r[-0.175pt]{2.409pt}{0.350pt}}
\p(264,745){\r[-0.175pt]{2.409pt}{0.350pt}}
\p(1350,745){\r[-0.175pt]{2.409pt}{0.350pt}}
\p(264,755){\r[-0.175pt]{2.409pt}{0.350pt}}
\p(1350,755){\r[-0.175pt]{2.409pt}{0.350pt}}
\p(264,764){\r[-0.175pt]{2.409pt}{0.350pt}}
\p(1350,764){\r[-0.175pt]{2.409pt}{0.350pt}}
\p(264,771){\r[-0.175pt]{2.409pt}{0.350pt}}
\p(1350,771){\r[-0.175pt]{2.409pt}{0.350pt}}
\p(264,777){\r[-0.175pt]{2.409pt}{0.350pt}}
\p(1350,777){\r[-0.175pt]{2.409pt}{0.350pt}}
\p(264,782){\r[-0.175pt]{2.409pt}{0.350pt}}
\p(1350,782){\r[-0.175pt]{2.409pt}{0.350pt}}
\p(264,787){\r[-0.175pt]{4.818pt}{0.350pt}}
\p(242,787){\makebox(0,0)[r]{$10^{\gnulog 1e+01}$}}
\p(1340,787){\r[-0.175pt]{4.818pt}{0.350pt}}
\p(333,158){\r[-0.175pt]{0.350pt}{4.818pt}} \p(333,113){\makebox(0,0){$0$}}
\p(333,767){\r[-0.175pt]{0.350pt}{4.818pt}}
\p(470,158){\r[-0.175pt]{0.350pt}{4.818pt}} \p(470,113){\makebox(0,0){$20$}}
\p(470,767){\r[-0.175pt]{0.350pt}{4.818pt}}
\p(607,158){\r[-0.175pt]{0.350pt}{4.818pt}} \p(607,113){\makebox(0,0){$40$}}
\p(607,767){\r[-0.175pt]{0.350pt}{4.818pt}}
\p(744,158){\r[-0.175pt]{0.350pt}{4.818pt}} \p(744,113){\makebox(0,0){$60$}}
\p(744,767){\r[-0.175pt]{0.350pt}{4.818pt}}
\p(881,158){\r[-0.175pt]{0.350pt}{4.818pt}} \p(881,113){\makebox(0,0){$80$}}
\p(881,767){\r[-0.175pt]{0.350pt}{4.818pt}}
\p(1018,158){\r[-0.175pt]{0.350pt}{4.818pt}} \p(1018,113){\makebox(0,0){$100$}}
\p(1018,767){\r[-0.175pt]{0.350pt}{4.818pt}}
\p(1155,158){\r[-0.175pt]{0.350pt}{4.818pt}} \p(1155,113){\makebox(0,0){$120$}}
\p(1155,767){\r[-0.175pt]{0.350pt}{4.818pt}}
\p(1292,158){\r[-0.175pt]{0.350pt}{4.818pt}} \p(1292,113){\makebox(0,0){$140$}}
\p(1292,767){\r[-0.175pt]{0.350pt}{4.818pt}}
\p(264,158){\r[-0.175pt]{264.026pt}{0.350pt}}
\p(1360,158){\r[-0.175pt]{0.350pt}{151.526pt}}
\p(264,787){\r[-0.175pt]{264.026pt}{0.350pt}}
\p(-43,472){\makebox(0,0)[l]{\shortstack{$\|H_i-D_i\|$}}}
\p(812,68){\makebox(0,0){Step $i$}}
\p(264,158){\r[-0.175pt]{0.350pt}{151.526pt}}
\p(1257,504){\makebox(0,0)[r]{Method of Steepest Descent}}
\p(1301,504){\c} \p(333,712){\c} \p(339,698){\c} \p(346,693){\c}
\p(353,689){\c} \p(360,687){\c} \p(367,685){\c} \p(374,683){\c}
\p(380,682){\c} \p(387,680){\c} \p(394,679){\c} \p(401,678){\c}
\p(408,677){\c} \p(415,676){\c} \p(422,675){\c} \p(428,674){\c}
\p(435,673){\c} \p(442,672){\c} \p(449,671){\c} \p(456,670){\c}
\p(463,670){\c} \p(470,669){\c} \p(476,668){\c} \p(483,668){\c}
\p(490,667){\c} \p(497,666){\c} \p(504,666){\c} \p(511,665){\c}
\p(517,665){\c} \p(524,664){\c} \p(531,664){\c} \p(538,663){\c}
\p(545,663){\c} \p(552,662){\c} \p(559,662){\c} \p(565,661){\c}
\p(572,661){\c} \p(579,660){\c} \p(586,660){\c} \p(593,660){\c}
\p(600,659){\c} \p(607,659){\c} \p(613,658){\c} \p(620,658){\c}
\p(627,657){\c} \p(634,657){\c} \p(641,657){\c} \p(648,656){\c}
\p(654,656){\c} \p(661,656){\c} \p(668,655){\c} \p(675,655){\c}
\p(682,654){\c} \p(689,654){\c} \p(696,654){\c} \p(702,653){\c}
\p(709,653){\c} \p(716,653){\c} \p(723,652){\c} \p(730,652){\c}
\p(737,652){\c} \p(744,651){\c} \p(750,651){\c} \p(757,651){\c}
\p(764,650){\c} \p(771,650){\c} \p(778,650){\c} \p(785,649){\c}
\p(791,649){\c} \p(798,649){\c} \p(805,648){\c} \p(812,648){\c}
\p(819,648){\c} \p(826,647){\c} \p(833,647){\c} \p(839,647){\c}
\p(846,646){\c} \p(853,646){\c} \p(860,646){\c} \p(867,646){\c}
\p(874,645){\c} \p(881,645){\c} \p(887,645){\c} \p(894,644){\c}
\p(901,644){\c} \p(908,644){\c} \p(915,643){\c} \p(922,643){\c}
\p(928,643){\c} \p(935,643){\c} \p(942,642){\c} \p(949,642){\c}
\p(956,642){\c} \p(963,641){\c} \p(970,641){\c} \p(976,641){\c}
\p(983,641){\c} \p(990,640){\c} \p(997,640){\c} \p(1004,640){\c}
\p(1011,639){\c} \p(1018,639){\c} \p(1024,639){\c} \p(1031,639){\c}
\p(1038,638){\c} \p(1045,638){\c} \p(1052,638){\c} \p(1059,637){\c}
\p(1065,637){\c} \p(1072,637){\c} \p(1079,637){\c} \p(1086,636){\c}
\p(1093,636){\c} \p(1100,636){\c} \p(1107,635){\c} \p(1113,635){\c}
\p(1120,635){\c} \p(1127,635){\c} \p(1134,634){\c} \p(1141,634){\c}
\p(1148,634){\c} \p(1155,634){\c} \p(1161,633){\c} \p(1168,633){\c}
\p(1175,633){\c} \p(1182,632){\c} \p(1189,632){\c} \p(1196,632){\c}
\p(1202,632){\c} \p(1209,631){\c} \p(1216,631){\c} \p(1223,631){\c}
\p(1230,631){\c} \p(1237,630){\c} \p(1244,630){\c} \p(1250,630){\c}
\p(1257,629){\c} \p(1264,629){\c} \p(1271,629){\c} \p(1278,629){\c}
\p(1285,628){\c} \p(1292,628){\c} \p(1298,628){\c} \p(1305,628){\c}
\p(1312,627){\c} \p(1319,627){\c} \p(1326,627){\c} \p(1333,627){\c}
\p(1339,626){\c} \p(1346,626){\c} \p(1353,626){\c} \p(1360,625){\c}
\p(1257,459){\makebox(0,0)[r]{Conjugate Gradient Method}}
\p(1301,459){\D} \p(333,712){\D} \p(339,698){\D} \p(346,691){\D}
\p(353,685){\D} \p(360,680){\D} \p(367,675){\D} \p(374,671){\D}
\p(380,667){\D} \p(387,664){\D} \p(394,661){\D} \p(401,658){\D}
\p(408,654){\D} \p(415,649){\D} \p(422,646){\D} \p(428,643){\D}
\p(435,640){\D} \p(442,637){\D} \p(449,632){\D} \p(456,627){\D}
\p(463,624){\D} \p(470,618){\D} \p(476,615){\D} \p(483,606){\D}
\p(490,600){\D} \p(497,596){\D} \p(504,576){\D} \p(511,554){\D}
\p(517,546){\D} \p(524,541){\D} \p(531,536){\D} \p(538,533){\D}
\p(545,530){\D} \p(552,526){\D} \p(559,523){\D} \p(565,520){\D}
\p(572,517){\D} \p(579,512){\D} \p(586,508){\D} \p(593,504){\D}
\p(600,502){\D} \p(607,498){\D} \p(613,495){\D} \p(620,490){\D}
\p(627,484){\D} \p(634,479){\D} \p(641,474){\D} \p(648,469){\D}
\p(654,459){\D} \p(661,451){\D} \p(668,444){\D} \p(675,428){\D}
\p(682,410){\D} \p(689,403){\D} \p(696,399){\D} \p(702,396){\D}
\p(709,393){\D} \p(716,391){\D} \p(723,385){\D} \p(730,383){\D}
\p(737,379){\D} \p(744,376){\D} \p(750,372){\D} \p(757,368){\D}
\p(764,364){\D} \p(771,362){\D} \p(778,359){\D} \p(785,354){\D}
\p(791,351){\D} \p(798,347){\D} \p(805,342){\D} \p(812,335){\D}
\p(819,329){\D} \p(826,325){\D} \p(833,320){\D} \p(839,314){\D}
\p(846,304){\D} \p(853,299){\D} \p(860,292){\D} \p(867,287){\D}
\p(874,287){\D} \p(881,286){\D} \p(887,284){\D} \p(894,284){\D}
\p(901,286){\D} \p(908,286){\D} \p(1257,414){\makebox(0,0)[r]{Newton's
Method}} \p(1301,414){\s} \p(333,712){\s} \p(339,573){\s} \p(346,269){\s}
\p(353,244){\s} \p(360,246){\s} \p(367,240){\s} \p(374,240){\s}
\p(380,235){\s}
\end{picture}
\end{gnuplot}
\caption[Convergence with $\lyapunov$ on $\SO(20)$]{\protect\smallfrak
Maximization of $\lyapunov$ on $\SO(20)$ (dimension $\SO(20)=190$),
where $Q=\diag(20,\ldots,1)$ and $N=\diag(20,\ldots,1)$.  The $i$th
iterate is $H_i=\Theta_i^\T Q\Theta_i$, $D_i$ is the diagonal matrix
of eigenvalues of~$H_i$, $H_0$ is near $N$, and $\|\cdot\|$ is the
norm induced by the standard inner product on~$\gl(n)$.  Geodesics and
parallel translation were computed using the algorithm of Ward and
Gray~\protect\citeyear{WG:1,WG:2}; the step sizes for the method of
steepest descent and the conjugate gradient method were computed using
Brockett's
estimate~\protect\cite{Brockett:grad}.\parfillskip=0pt\label{fig:SOnconv}}
\end{figure}

\removelastskip
\vskip -10pt plus10pt
\vskip 0pt          

\begin{example}[The function
$\lyapunov$]\label{eg:trHNcg}\ignorespaces Let $\Theta$, $Q$, and $H$
be as in Examples \ref{eg:trHNgrad} and \ref{eg:trHNnewton}.  As
before, the natural Riemannian structure of~$\SO(n)$ is used, whereby
geodesics and parallel translation along geodesics are given by
Equations (\ref{eq:geodesicson}) and (\ref{eq:SOnpt}) of
Chapter~\ref{chap:geom}.  Brockett's estimate (n.b.\
Equation~(\ref{eq:Brockettest})) for the step size may be used in
Algorithm~\ref{al:cgman}.  The results of a numerical experiment
demonstrating the convergence of the conjugate gradient method
in~$\SO(20)$ are shown in Figure~\ref{fig:SOnconv}.  \end{example}

\makeatletter
\if@twoside    
\markboth{\uppercase{\ch@ptern@me\hfill\ifnum\c@secnumdepth>\m@ne
    \@chapabb\ {\numbersize\thechapter}.\fi}}{}
\fi
\makeatother

\eject 

%% file: chap-af.tex

\chapter{Application to Adaptive Filtering}\label{chap:af}

Principal component analysis and optimization methods are used to
solve a wide variety of engineering problems.  Optimization methods,
such as gradient following, are often used when the solution to a
given problem corresponds to the minimizing value of a real valued
function, such as a square error.  There are many terms for principal
component analysis---the eigenvalue problem in algebra, the
Karhunen-Lo\`eve expansion in stochastic processes, and factor
analysis in statistics---indicating the extent of its application.
Many applications use the fact that the best low rank approximation of
a symmetric or Hermitian linear mapping of a vector space onto itself
is given by the sum of outer products of eigenvectors corresponding to
the largest eigenvalues of the linear map.

In the case of linear systems modeling, a given state space model may
have an equivalent realization of lower dimension with identical
input\slash output characteristics.  Computing this lower dimensional
realization is called state space reduction, and the state space model
of smallest possible dimension is called a minimal realization.
\citeasnoun{Moore} uses the singular value decomposition of the
observability and controllability matrices of a specified
finite-dimensional state space model to derive a minimal realization.
The process of computing a state space model given its input\slash
output characteristics is called the identification problem.  This
problem is related to the field of adaptive control, where control
methods that use incomplete, inaccurate, or arbitrarily time-varying
models are considered.  \citeasnoun{MoonenDVV} use the singular value
decomposition of a block Hankel matrix constructed with measured
input\slash output data to identity linear systems.  On the other
hand, optimization methods for error minimization have long been used
for system identification and adaptive control
\cite{Lion,Astrom,CraigHS,SlotineLi,TosunogluTesar}, as well as
stochastic methods that use correlation data from input and output
measurements \cite{Akaike,Baram,KorenburgHunter}.

Furthermore, the computation of the dominant modes and buckling modes
of mechanical systems are important problems in mechanics.  These
problems may be expressed naturally either as infinite-dimensional
eigenvalue problems or as optimization problems on an infinite
dimensional Hilbert space.  Approximate solutions to these problems
may be obtained via finite element methods \cite{Hughes}, which rely
upon methods from numerical linear algebra discussed below, such as
Lanczos methods.  Projected conjugate gradient algorithms such as
Fried's \citeyear{Fried} algorithm have also been proposed.

In the past fifteen years, principle component techniques have become
increasingly important in the field of adaptive signal processing.
This is due primarily to the introduction of new methods for signal
parameter estimation which rely upon the signal's covariance
structure.  Notably, \citeasnoun{Schmidt} developed a signal subspace
algorithm called |MUSIC|, an acronym for multiple signal
classification, which from measurements taken from a completely
arbitrary sensor array provides accurate unbiased estimates of a
variety of signal parameters, such as number of signals, their
directions of arrival, their center frequency, and other parameters.
The central idea of~|MUSIC| is to exploit the sensor geometry and the
signal subspace determined by the data to compute the desired signal
parameters.  \citeasnoun{BienvenuKopp} demonstrate how related
techniques may be used when the background noise is nonisotropic.

With the |MUSIC| algorithm, the signal subspace is first computed from
the canonical eigenvalue decomposition of the data covariance matrix.
Then, knowledge of the array geometry is used to compute peaks of a
function defined on a parameter space.  This search is in general
computationally expensive. \citeasnoun{RoyKailath} have proposed an
algorithm which retains many advantages of the |MUSIC| algorithm with
a significantly reduced computational complexity.  This algorithm is
called |ESPRIT|, an acronym for estimation of signal parameters by
rotational invariant techniques.  It is important to note that the
rotational invariance refers to an intrinsic property of the algorithm
implied by a restriction on the sensor array; it does not refer to the
invariant methods discussed in Chapter~\ref{chap:geom}.  It is assumed
that the sensor array is comprised of a pair of subarrays that are
equivalent with respect to translation.  That is, there exists a
translation which maps one subarray into the other.  Except for this
restriction, the sensor array may be arbitrary.  This restriction
implies that the signal subspace of the array measurements is
invariant with respect to a certain complex rotation of the sensor
outputs.

The signal subspace methods used in the adaptive algorithms like
|MUSIC| and |ESPRIT| are especially important in the field of adaptive
signal processing.  In these contexts, the signal subspaces may be
thought to vary slowly with time, and it is desired to compute the
time varying eigenvalue decomposition of the covariance information.
Of course, one could use the symmetric |QR| algorithm at each time
step to obtain this decomposition; however, this is prohibitively
expensive, especially when only a few of the largest or smallest
eigenvalues are desired, and there is a wide choice of other
techniques available.  In their review, \citeasnoun{ComonGolub}
provide a thorough and descriptive list of many methods.  They are
careful to distinguish between methods that are of complexity
$O(nk^2)$ and complexity $O(n^2k)$, where $n$ is the dimension of the
total space and $k$ is the dimension of the signal subspace to be
tracked.

Several of the covariance matrix updating procedures rely upon
rank~one updates \cite{Owsley,Karhunen,Karasalo,Schreiber}.  There is
a well-known theory \cite{Wilkinson} of computing the updated
eigenvalue decomposition of a symmetric matrix updated by a rank~one
addition, and algorithms for this procedure are available
\cite{Bunchetal}.  However, this method requires knowledge of the full
eigenvalue decomposition to compute the rank~one updated
decomposition; the algorithm is $O(n^3)$ complexity, which is the same
order as the full |QR| algorithm, thus limiting its attractiveness.
If the covariance matrix has at most $k$ nonzero eigenvalues, then
this algorithm may be performed in $O(n^2k)$ steps \cite{Yu}. This
case holds approximately when the signal-to-noise ratio is high, and
when a ``forgetting'' factor is introduced into the covariance matrix
updates.

Other updating procedures are also important. For example, a rank~one
update of the covariance matrix corresponds to the addition of one
column to a data matrix.  The updated |QR| decomposition of the data
matrix is often desired.  \citeasnoun{GVL} provide several now
classical algorithms for this task. \citeasnoun{Rader} designed and
built a wafer scale integrated circuit utilizing on-chip |CORDIC|
transformations to compute the updated Cholesky factorization of a
data matrix.  \citeasnoun{MoonenVV} provide an updating method for the
singular value decomposition of the data matrix.
\citeasnoun{DemeureScharf} describe the use of updated Toeplitz
matrices in linear prediction theory.

Gradient-based algorithms are also widely used.  Some of the first
adaptive filtering algorithms, such as the |LMS| (least mean square)
and |SER| (sequential regression) algorithms \cite{WidrowStearns} are
gradient-based techniques.  These two algorithms provide a method to
compute a weighting vector for sensor outputs that provides the
minimal variance of the error between the weighted measurements and a
desired response.  These gradient techniques, as well as the ones
given by \citeasnoun{Owsley}, \citeasnoun{Larimore}, and
\citeasnoun{YangKaveh}, all have a fixed step length, which of course
affects their convergence rates.  Other gradient-based algorithms are
used to track the eigenvalue decomposition of a slowly varying
covariance matrix.  So called stochastic gradient methods
\cite{Larimore,Hu} are derived with the goal of maximizing the
Rayleigh quotient corresponding to the data covariance matrix.

The conjugate gradient method has been suggested by many researchers
as an appropriate tool for subspace tracking
\cite{BradFletch,Chenetal,FuhrLiu}, as well as for finite element
methods \cite{Fried}.  However, only Fuhrmann and Liu realized that
the formula $\gamma_i =\|G_{i+1}\|^2/\|G_i\|^2$ used to ensure
conjugate steps in the Euclidean case is not valid in the general case
of the constrained or Riemannian conjugate gradient method, as
discussed in Chapter~\ref{chap:orm}, Section~\ref{sec:conjgrad}.  They
provide a conjugate gradient algorithm on the sphere that depends upon
the choice of an azimuthal projection onto tangent planes.  This
algorithm is also distinguished from the others in that the steps are
constrained to the sphere, whereas the others take steps in the
ambient Euclidean space, then project onto the constraint surface.

In this chapter we present a new gradient-based algorithm for subspace
tracking that draws on the ideas developed in the preceding three
chapters.  As discussed in Chapter~\ref{chap:grad},
Section~\ref{sec:extremeig}, the eigenvectors corresponding to the
extreme eigenvalues of a symmetric matrix can be obtained by
maximizing the generalized Rayleigh quotient.  The Riemannian version
of the conjugate gradient method, Algorithm~\ref{al:cgman}, can be
implemented by an efficient $O(nk^2)$ algorithm by exploiting the
homogeneous space structure of the Stiefel manifold covered in
Chapter~\ref{chap:geom}, Section~\ref{sec:examples}.  The resulting
conjugate gradient algorithm can be modified so that it is useful in
the subspace tracking context described in the aforementioned
references.

\section{Adaptive estimation techniques}\label{sec:adapest}

In this section a general data model will be described that is used in
much of the literature on adaptive subspace tracking.  A discrete time
model is used, although this is not necessary; continuous models for
subspace tracking are possible \cite{Brockett:subspace}.  We imagine a
collection of~$m$ signals or states that span a subspace to be
identified.  To each signal or state there is associated a real value
at times $t=0$, $1$, \dots\spacefactor=3000\relax\space Many
applications require phase information and therefore use complex
numbers, but for simplicity we consider only the real case; the
complex version of this treatment and the algorithms to be presented
are obvious generalizations. Denote the $i$th signal or state ($1\le
i\le m$) by~$s^i$, whose value at time~$t$ is written as~$s^i(t)$
or~$s^i_t$. Hereafter we shall simply refer to states, although either
signals and states may be used. Thus the states can be viewed as a
vector $s$ with components $s^i$ in the $m\hyphen$dimensional affine space
$\R^m$; the designation of quiescent values for the states makes this
a vector space, which we shall endow with the standard metric.  The
vector $s$ is called the state vector.

A measurement model for the state vector is now provided.  It is
assumed that there are $n$ sensors whose outputs are denoted by the
real numbers $x^1$, \dots,~$x^n$, or simply by the data vector
$x\in\R^n$. The data vector at time~$t$ is given by the equation
$$x_t=As_t+w_t,$$ where $A$ is an $n\by m$ matrix, possibly
parameterized, and $w_t$ is a Gaussian independent random sequence.

\subsection{The stationary case}

Some simplifying assumptions about the state vector $s$ will be made.
It is assumed that $s_t$ is a wide-sense stationary random sequence
that is ergodic in the mean and ergodic in covariance, i.e., $$E[s_0]
=\lim_{T\to\infty} {1\over2T} \sum_{t=-T}^Ts_t \qquad\hbox{and}\qquad
\lim_{T\to\infty} E[s_0^{\vphantom{\T}}s_0^\T] = {1\over2T}
\sum_{t=-T}^T s_t^{\vphantom{\T}}s_t^\T.$$ Furthermore, it is assumed
for simplicity that $E[s_t]=0$.  Then the covariance matrix
$R_{xx}=E[x_t^{\vphantom{\T}}x_t^\T]$ of~$x$ is given by
$$R_{xx}=AR_{ss}A^\T +R_{ww},$$ where $R_{ss}$ and $R_{ww}$ are the
covariance matrices of~$s$ and $w$, respectively.  The goal is to
estimate the principal invariant subspaces of~$R_{xx}$.  Several of
the covariance estimation techniques mentioned above use an averaging
approach to compute an estimate of~$R_{xx}$.  For example, the
estimate $$\hat R_{xx}= {1\over
T}\sum_{t=0}^{T-1}x_t^{\vphantom{\T}}x_t^\T,$$ which is easily
implemented as a sequence of rank~one updates, is often used.
\citeasnoun{Karasalo} provides an algorithm for estimating the
covariance matrix of a signal which requires fewer computations that
this averaging technique.  Standard iterative techniques such as those
mentioned above may be used to compute the principal invariant
subspaces of~$\hat R_{xx}$.

\subsection{The nonstationary case}

If the sequence $s_t$ is nonstationary but has second order statistics
that vary slowly with respect to some practical time scale, then many
applications require estimates of the principal invariant subspaces of
the covariance matrix $R_{xx}$ at any given time.  This is known as
the tracking problem.  One common approach is to form an $n\by l$ data
matrix $X$ from a moving window of the data vectors. I.e., the $j$th
column of~$X$ is the data vector $x_{t+j}$, where $t+1$ is time of the
first sample in the moving window and $t+l$ is the last. Typically $l$
is greater than~$n$.  The estimate of~$R_{xx}$ at time $t+l$ is
\begin{equation} \label{eq:hatRxx} \hat R_{xx} ={1\over l}XX^\T
={1\over l} \sum_{\tau=t+1}^{t+l}x_\tau^{\vphantom{\T}} x_\tau^\T.
\end{equation} Other approach include the so-called fading memory
estimate given by the equations $$\eqalign{P_{t+1} &=\hat R_{xx}(t)
+x_{t+1}^{\vphantom{\T}}x_{t+1}^\T,\cr \hat R_{xx}(t+1)
&=P_{t+1}\big/\|P_{t+1}\|,\cr}$$ where $\|\cdot\|$ is the Frobenius
norm, or the equation $$\hat R_{xx}(t+1) =\alpha_t\hat R_{xx}(t)
+\beta_tx_{t+1}^{\vphantom{\T}}x_{t+1}^\T,$$ where $\alpha$ and
$\beta$ are real-valued time sequences.

\subsection{Numerical considerations}

On a finite precision machine, there is a loss in accuracy that comes
with squaring the data and using the estimated covariance matrix $\hat
R_{xx}$ explicitly.  It is therefore recommended that the data matrix
$X$ be used directly.  By Equation~(\ref{eq:hatRxx}), the eigenvectors
of~$\hat R_{xx}$ correspond to the left singular vectors of~$X$.  To
reduce the computational effort involved in the iterative eigenvalue
algorithms, the matrix $X$ is often decomposed at each time step into
the |QR|~decomposition $X=LQ$, where $L$ is an $n\by l$ lower
triangular matrix and $Q$ is an $l\by l$ orthogonal matrix. Because
only the left singular vectors of~$X$ are desired, the orthogonal
matrix~$Q$ is not required, which allows for a reduction of the
computational effort.  However, there must be a method for updating
the |QR| decomposition of~$X$ at each time step.

\section{Conjugate gradient method for largest
eigenvalues}\label{sec:cg-genray}

Computing the extreme eigenvalues and associated eigenvectors of a
symmetric matrix is an important problem in general, and specifically
in subspace tracking.  Perhaps the best known and most widely used
algorithm for this task is the Lanczos algorithm, which may be derived
by maximizing the Rayleigh quotient \cite{GVL}.  The convergence
properties of the unmodified Lanczos method on a finite-precision
machine are poor, however, because there is an increasing loss of
orthogonality among the Lanczos vectors as the algorithm proceeds and
Ritz pairs converge.  Several modifications have been proposed which
yield a successful algorithm, such as complete reorthogonalization,
which is prohibitively expensive, selective reorthogonalization
\cite{ParlettScott}, block Lanczos methods, and $s$-step Lanczos
methods \cite{CullumWill}.  The latter methods are an iterative
version of the block Lanczos method for computing the largest
eigenvalues.  Of necessity these algorithms are more costly than the
unmodified Lanczos algorithm.  \citeasnoun{CullumWill} provide a
detailed analysis of practical Lanczos methods as well as a thorough
bibliography. Xu and Kailath \citeyear{XuKailath:flanc,XuKailath:fsd}
provide fast Lanczos methods whose speed depends upon a special
structure of the covariance matrix.

Given a symmetric $n\by n$ matrix $A$, \citeasnoun{CullumBIT}
considers the optimization problem $$\max (\tr X^\T\!AX -\tr
Y^\T\!AY)$$ over all $n\by k$ matrices $X$ and all $n\by l$ matrices
$Y$ ($k+l\le n$) such that $X^\T X=I$ and $Y^\T Y=I$, i.e.,
$X\in\stiefel(n,k)$ and $Y\in \stiefel(n,l)$.  In her paper it is
noted that an $(s+1)$-step Lanczos method generates eigenvector
estimates that are as least as good as an $s$-step constrained
conjugate gradient algorithm.  However, the conjugate gradient
algorithm presented there is linearly convergent and does not exploit
the natural Riemannian structure of the manifold as does
Algorithm~\ref{al:cgman} of Chapter~\ref{chap:orm}. See also
\citeasnoun[pp.~217ff\/]{CullumWill}.  \citeasnoun{Parlettetal} also
use the Lanczos method for computing the largest eigenvalue of a
symmetric matrix.  Alternatively, \citeasnoun{O'Learyetal} propose an
algorithm for computing the dominant eigenvalue of a positive definite
matrix, which is based upon the power method.  This algorithm is
useful for rough approximation of the spectral radius of a positive
definite matrix. A different point of view is offered by
\citeasnoun{Overton}, who considers an eigenvalue optimization problem
on a set of parameterized symmetric matrices.

\subsection{Generalized Rayleigh quotient}

Let $\stiefel(n,k)$ be the compact Stiefel manifold of $n\by k$
matrices ($k\le n$) with orthonormal columns.  Recall from
Chapter~\ref{chap:grad}, Section~\ref{sec:extremeig} that given an
$n\by n$ symmetric matrix~$A$ and a $k\by k$ diagonal matrix $N$, the
generalized Rayleigh quotient is the function
$\rho\colon\stiefel(n,k)\to\R$ defined by $p\mapsto \tr p^\T\!ApN$.
As described in Corollary~\ref{cor:raygencrit},
Chapter~\ref{chap:grad}, if the extreme eigenvalues of~$A$ and the
diagonal elements of~$N$ are distinct, then this function has $2^k$
maxima where the corresponding eigenvectors of~$A$ comprise the
columns of the maximum points, modulo $k$ choices of sign.  Let us
assume that our application requires the eigenvectors of a data
covariance matrix corresponding to the largest eigenvalues, so that
the diagonal elements of~$N$ are all positive.

As discussed in Chapter~\ref{chap:geom}, Section~\ref{sec:examples},
the Stiefel manifold can be identified with the reductive homogeneous
space $\O(n)/\O(n-k)$.  Let $G=\O(n)$, $M=\stiefel(n,k)$,
$o=\bigl({I\atop0}\bigr)$ the origin in~$M$, and
$H=\O(n-k)$ the isotropy group at~$o$.  Denote the Lie algebras of~$G$
and $H$ by~$\g$ and $\h$, respectively, and let $\pi\colon G\to M$ be
the projection $g\mapsto g\cdot o$.  The tangent plane of~$M$ at~$o$
can be identified with the vector subspace $\m=\h^\perp$ of~$\g$,
where orthogonality is with respect to the canonical invariant metric
on~$G/H$.

Let $g\in G$ be a coset representative of~$p\in M$, i.e., $p=g\cdot
o$.  Then the tangent plane $T_pM$ can be identified with the vector
subspace $\Ad_g(\m)$ of~$\g$.  The choice of coset representative is
not unique, so neither is this subspace.  Given $x\in\m$,
$x_p=\Ad_g(x)\in\Ad_g(\m)$, the unique geodesic through $o\in M$ in
the direction corresponding to $x_p\in\Ad_g(\m)$ is given
by~$e^{x_pt}\cdot p=ge^{xt}\cdot o$, where $e^{xt}$ denotes matrix
exponentiation.  As shown in the proof of
Proposition~\ref{prop:genraygrad}, the first order term
of~$\rho(ge^{xt}\cdot o)$ can be used to compute the gradient of the
generalized Rayleigh quotient at~$p\in M$.  Given the coset
representative $g$ of~$p$, we have
{\makeatletter\refstepcounter{equation}\let\@currentlabel=\theequation
\label{eq:raygengrad}}$$\eqalignno{\Ad_{g^{-1}}\cdot(\grad\rho)_p
&=[g^\T\!Ag,oNo^\T] &(\theequation)\cr &=g^\T\!ApNo^\T-oNp^\T\!Ag.
&(\theequation')\cr}$$ From a computational standpoint,
Equation~$(\ref{eq:raygengrad}')$ is preferable to
Equation~$(\ref{eq:raygengrad})$ because it can be computed with $k$
matrix vector multiplications, whereas Equation~(\ref{eq:raygengrad})
requires $n$ matrix-vector multiplications.

Similarly, by Equation~(\ref{eq:taylorf}), Chapter~\ref{chap:orm}, the
second order term of~$\rho(ge^{xt}\cdot o)$ can be used to compute the
second covariant differential of~$\rho$ at~$p$ evaluated at~$(X,X)$,
where $X$ is the tangent vector in~$T_pM$ corresponding to~$x\in\m$.
Because the second covariant differential at~$p$ is a symmetric
bilinear form on~$T_pM$, polarization of~$(\Dsqr\rho)_p(X,X)$ may be
used to obtain \begin{equation} \label{eq:d2rho}
(\Dsqr\rho)_p(X,Y)=\tr\bigl(p^\T\!Ag(xy+yx)oN -2o^\T
xg^\T\!AgyoN\bigr), \end{equation} where $Y\in T_pM$ corresponds to
$y\in\m$.

Both Equations $(\ref{eq:raygengrad}')$ and $(\ref{eq:d2rho})$ will be
used to perform the Riemannian version of the conjugate gradient
method of the generalized Rayleigh quotient given in
Algorithm~\ref{al:cgman}.

\subsection{The choice of coset representatives}

Given $p$ in~$M=\stiefel(n,k)$, a coset representative $g$
in~$G=\O(n)$ must be computed to exploit the underlying structure of
the homogeneous space using the methods described above.  In the case
of the Stiefel manifold, a coset representative of~$p$ is simply any
$n\by n$ orthogonal matrix whose first $k$ columns are the $k$ columns
of~$p$, as easily seen by examining the equality $p=g\cdot o$.  The
choice of coset representative is completely arbitrary, thus it is
desirable to choose a representative that is least expensive in terms
of both computational effort and storage requirements.  For example,
the element $g$ in~$G$ could be computed by performing the
Gram-Schmidt orthogonalization process, yielding a real $n\by n$
orthogonal matrix.  This procedure requires $O(n^2k)$ operations and
$O(n^2)$ storage, which are relatively expensive.

The |QR| decomposition, however, satisfies our requirements for low
cost.  Recall that for any $n\by k$ matrix $F$ ($k\le n$), there
exists an $n\by n$ orthogonal matrix $Q$ and an $n\by k$ upper
triangular matrix $R$ such that $$F=QR.$$ There is an efficient
algorithm, called Householder orthogonalization, for computing the
|QR| decomposition of~$F$ employing Householder reflections.
Specifically, we have $$P_k\ldots P_1F=R,$$ where the $P_i$, $i=1$,
\dots,~$k$ are Householder reflections of the form
$$P_i=I-{1\over\beta_i}\nu_i^{\vphantom{\T}}\nu_i^\T,$$
$\nu_i\in\R^n$, and $\beta_i=2/\nu_i^\T\nu_i$. This algorithm requires
$k^2(n-k/3)+O(nk)$ operations and requires only $kn$ storage units
because the orthogonal matrix $Q$ may be stored as a sequence of
vectors used for the Householder reflections---the so-called factored
form.  See \citeasnoun{GVL} for details and explanations of these
facts.

\begin{remark} Let $F$ be an $n\by k$ matrix ($n\le k$) with
orthonormal columns.  Then the |QR| decomposition of~$F$ yields an
upper triangular matrix $R$ whose off-diagonal elements vanish and
diagonal elements are $\pm1$.
\end{remark}

Therefore, the |QR| decomposition provides an inexpensive method of
computing a coset representative of any point $p$ in~$\stiefel(n,k)$.
Specifically, let $p\in\stiefel(n,k)$ have the |QR| decomposition
$p=QR$, $Q^\T=P_k\ldots P_1$, and partition $R$ as $R=
\bigl(\!{R_1\atop0}\!\bigr)$, where $R_1$ is a $k\by k$ upper
triangular matrix.  Then the coset representative $g$ of~$p$ is given
by $$g=Q\cdot\diag(R_1,I).$$

As discussed above, the choice of a coset representative provides an
identification of the tangent plane $T_pM$ with the vector
subspace~$\m$.  The conjugate gradient algorithm computes a sequence
of points $p_i$ in~$M$, all of which necessarily have different coset
representatives, as well as a sequence of tangent vectors $H_i\in
T_{p_i}M$ which are compared by parallel translation.  Thus it will be
necessary to compute how the change in the coset representative of a
point changes the elements in~$\m$ corresponding to tangent vectors at
a point.  Let $g_1$ and $g_2$ be coset representative of the point $p$
in~$M$, and let $X$ be a tangent vector in~$T_pM$.  The elements $g_1$
and $g_2$ in~$G$ define elements $x_1$ and $x_2$ in~$\m$ by the
equation $$X=\Ad_{g_1}(x_1)=\Ad_{g_2}(x_2).$$ Given $x_1$, we wish to
compute $x_2$ efficiently.  By assumption, there exists an $h\in H$
such that $g_2=g_1h$.  Then $$\eqalign{x_2
&=(\Ad_{g_2^{-1}}\circ\Ad_{g_1})(x_1)\cr &=\Ad_{g_2^{-1}g_1}(x)\cr
&=\Ad_{h^{-1}}(x).\cr}$$ The vector subspace $\m$ is
$\Ad_H$-invariant; therefore, $x_2$ may be computed by conjugating
$x_1$ by~$g_1$, then by~$g_2^{-1}$.

Any element $x$ in~$\m$ and $h$ in~$H$ can be partitioned as
$$\diffeqalign{x &=\left({a\atop b}\>{-b^\T\atop0}\right) &\hbox{$a$
in $\so(k)$ and $b$ $(n-k)\by k$ arbitrary,}\cr h &=\left({I\atop
0}\>{0\atop h'}\right) &\hbox{$h'$ in $\O(n-k)$.}\cr}$$ It is easy to
see that if $x_2=\Ad_h(x_1)$, then $a_2=a_1$ and $b_2=h'b_1$.  Thus
elements $x\in\m$, i.e., $n\by n$ matrices of the form given above,
may be stored as $n\by k$ matrices of the form $$x=\left({a\atop
b}\right),$$ where $a$ is a $k\by k$ skew-symmetric matrix and $b$ is
an $(n-k)\by k$ matrix.

\subsection{Geodesic computation}

As discussed previously, the unique geodesic through $p=g\cdot o$
in~$M$ in direction $X\in T_pM$ is given by the formula $$\exp_p
tX=ge^{xt}\cdot o,$$ where $x\in\m$ corresponds to $X\in T_pM$ via
$\Ad_g$.  Thus geodesics in~$M=\stiefel(n,k)$ may be computed with
matrix exponentiation. The problem of computing the accurate matrix
exponential of a general matrix in~$\gl(n)$ is difficult
\cite{nineteendubious}. However, there are stable, accurate, and
efficient algorithms for computing the matrix exponential of symmetric
and skew-symmetric matrices that exploit the canonical symmetric or
skew-symmetric decompositions (Golub \& Van~Loan 1983; Ward \& Gray
1978a, 1978b).  Furthermore, elements in~$\m$ have a special block
structure that may be exploited to substantially reduce the required
computational effort.

For the remainder of this section, make the stronger assumption on the
dimension of~$\stiefel(n,k)$ that $2k\le n$.  Let $x=\bigl({a\atop
b}\>{-b^\T\atop0}\bigr)$ be an element in~$\m$, and let the $(n-k)\by
k$ matrix $b$ have the |QR| decomposition $b=QR$, where $Q$ is an
orthogonal matrix in~$\O(n-k)$ and $R=\bigl(\!{R_1\atop0}\!\bigr)$
such that $R_1$ is a $k\by k$ upper triangular matrix.  Then the
following equality holds: $$\left({I\atop0}\>{0\atop Q^\T}\right)
\left({a\atop b}\>{-b^\T\atop0}\right) \left({I\atop0}\>{0\atop
Q}\right)= \pmatrix{a&-R_1^\T&0\cr R_1&0&0\cr 0&0&0\cr}.$$ Thus,
matrix exponentiation of the $n\by n$ skew-symmetric matrix $x$ may be
obtained by exponentiating the $2k\by 2k$ skew-symmetric matrix
$$x'=\left({a\atop R_1}\>{-R_1^\T\atop0}\right).$$ Computing the
canonical decomposition of an $n\by n$ skew-symmetric matrix requires
about $8n^3+O(n^2)$ operations \cite{WG:1}.  In the case of computing
the canonical decomposition of elements in~$\m$, this is reduced to
$8(2k)^3+O(k^2)$ operations, plus the cost of $k^2(n-4k/3)+O(nk)$
operations to perform the |QR| decomposition of~$b$.

Let $p$ in~$\stiefel(n,k)$ have the |QR|~decomposition $p=\Psi D$,
where $\Psi^\T=(P_k\ldots P_1)\in\O(n)$ and $D$ is upper triangular
such that its top $k\by k$ block $D_1$ is of the form
$D_1=\diag(\pm1,\ldots,\pm1)$.  Given $x\in\m$, let $x$ be partitioned
as above such that $b$ has the |QR|~decomposition $b=QR$, where
$Q\in\O(n-k)$ and $R$ is upper triangular with top $k\by k$ block
$R_1$.  Let $x'$ be the $2k\by 2k$ reduced skew-symmetric matrix
obtained from~$x$ by the method described in the previous paragraph.
Let the $2k\by 2k$ matrix $x'$ have the canonical skew-symmetric
decomposition $$x'=\vartheta s\vartheta^\T,$$ where
$\vartheta\in\O(2k)$ and $s$ is of the form $$\def\quad{\hbox{$\>$}}
\let\s=\sigma \def\ddots{\mathinner{\mkern1mu\raise7pt\vbox{\kern7pt
\hbox{.}}\mkern1mu \raise4pt\hbox{.}\mkern1mu
\raise1pt\hbox{.}\mkern1mu}} s=\pmatrix{0&\s_1\cr-\s_1&0\cr
&&0&\s_2\cr&&-\s_2&0\cr &&&&\ddots\cr &&&&&0&\s_k\cr
&&&&&-\s_k&0\cr}.$$ Then the geodesic $t\mapsto \exp_p tX=ge^{xt}\cdot
o$ may be computed as follows: \begin{eqnarray} ge^{xt}\cdot o &=&
P_1\ldots P_k \nonumber\\ &&\quad{}\cdot \left({D_1\atop0}\>{0\atop
I}\right) \left({I\atop0}\>{0\atop Q}\right)
\left({\vartheta\atop0}\>{0\atop I}\right)
\left({e^{st}\atop0}\>{0\atop I}\right)
\left({\vartheta^\T\atop0}\>{0\atop I}\right) \left({I\atop0}\>{0\atop
Q^\T}\right) \left({I\atop0}\right). \label{eq:geodcomp}
\end{eqnarray} Note well that these matrices are not partitioned
conformably, and that $$\left({I\atop0}\>{0\atop Q^\T}\right)
\left({I\atop0}\right)= \left({I\atop0}\right).$$ These steps may all
be performed with $O(nk)$ storage, and the computational requirements
are summarized in Table~\ref{tab:geodesiccost}.  One particularly
appealing feature of the geodesic computation of
Equation~(\ref{eq:geodcomp}) is that within the accuracy of this
computation, orthogonality of the columns of~$p_i$ is maintained
for~all $i$.  Thus it is never necessary to reorthogonalize the
columns of~$p_i$ as in the Lanczos algorithm.

\begin{table}[t]
\caption[Computational requirements of geodesic computation]
  {\label{tab:geodesiccost}\ignorespaces
    Computational requirements of geodesic computation}
\centerline{\vbox{\tabskip=0pt \def\opentable{\noalign{\vskip4pt}}
\halign to.75\textwidth{\strut#\hfil\tabskip=1em plus1fil
  &\hfil$#$\hfil\tabskip=0pt\cr
\noalign{\hrule}\opentable
\omit Procedure\hfil&\omit\hfil Cost\hfil\cr
\opentable\noalign{\hrule}\opentable
|QR|~decomposition of~$p$&k^2(n-k/3)+O(nk)\cr
|QR|~decomposition of~$x$&k^2(n-4k/3)+O(nk)\cr
Canonical decomposition of~$x'$&64k^3+O(k^2)\cr
$\diag(I,Q^\T)\cdot o$&0\cr
$\diag(\vartheta^\T,I)\cdot o$&0\cr
$\diag(e^{st},I)\cdot{}$&4k^2\cr
$\diag(\vartheta,I)\cdot{}$&4k^3\cr
$\diag(I,Q)\cdot{}$&k^2(2n-3k)+O(nk)\cr
$g\cdot{}$&k^2(2n-k)+O(nk)\cr
\opentable \noalign{\nointerlineskip}
\omit\hfil&\omit\hrulefill\cr \opentable
Total&6nk^2+62\third k^3+O(nk)\cr
\opentable\noalign{\hrule}}}}
\end{table}

\subsection{Step direction computation}

Let $p_i\in M$, $i\ge0$, be the iterates generated by
Algorithm~\ref{al:cgman} applied to the generalized Rayleigh quotient.
The successive direction for geodesic minimization at each iterate
$p_{i+1}\in M$ is given by the equation
\begin{equation} \label{eq:cgsteprev} H_{i+1}=G_{i+1}+\gamma_i\tau
H_i, \qquad\gamma_i ={\(G_{i+1}-\tau G_i,G_{i+1}\)\over \(G_i,H_i\)},
\end{equation} where $G_i$ the the gradient of the function at the
point $p_i$, and $\tau$ is the parallelism with respect to the
geodesic from~$p_i$ to~$p_{i+1}$.  Let $g_i$, $i\ge0$, be the coset
representative of~$p_i$ chosen to be the |QR| decomposition of~$p_i$
as described above, let $h_i\in\m$ correspond to~$H_i\in T_{p_i}M$
via~$\Ad_{g_i}$, and let $\lambda_i$ be the step length along this
curve such that $p_{i+1}=\exp_{p_i}\lambda_iH_i$.  The computation
of~$\tau H_i$ is straightforward because this this is simply the
direction of the curve $t\mapsto\exp_{p_i}tH_i$ at~$p_{i+1}$, i.e.,
$$\tau H_i ={d\over dt}\Big|_{t=\lambda_i} g_ie^{h_it}\cdot o.$$ Thus
the element $h_i$ in $\m$ corresponding to~$H_i$ in~$T_{p_i}M$
via~$\Ad_{g_i}$ is the same as the element in~$\m$ corresponding
to~$\tau H_i$ in~$T_{p_{i+1}}M$ via~$\Ad_{(g_ie^{h_i\lambda_i})}$.
However, the coset representative $g_{i+1}$ chosen for the point
$p_{i+1}$ is in general not equal to the coset representative
$g_ie^{h_i\lambda_i}$ of~$p_{i+1}$, so the element $h_i$ must be
transformed as \begin{equation}\label{eq:parthi} h_i\mapsto
(\Ad_{g_{i+1}^{-1}}\circ \Ad_{g_ie^{h_i\lambda_i}})(h_i).
\end{equation} This ensures that $h_i$ is represented in the basis
of~$T_{p_{i+1}}M$ implied by the conventions previously established.
Equation~(\ref{eq:parthi}) is thus the only computation necessary to
compute a representation of~$\tau H_i$ in~$\m$ with respect to the
coset representative $g_{i+1}$.

As discussed at the end of Section~\ref{sec:liegphomsp},
Chapter~\ref{chap:geom}, computing the parallel translation of an
arbitrary tangent vector along a geodesic requires the solution of the
set of structured $\half k(k-1)+(n-k)k$ linear differential equations
given in Equation~(\ref{eq:stiefelpt}), Chapter~\ref{chap:geom}.  In
the cases $k\ne1$ or~$n$, the solution to these differential equations
cannot be expressed as the exponential of an $n\by n$ matrix.
Therefore, it appears to be impractical to use parallel translation to
compute $\gamma_i$ of Equation~(\ref{eq:cgsteprev}).

Instead, we fall back upon the demand that subsequent directions be
conjugate with respect to the second covariant differential of the
function at a point, and use the formula \begin{equation}
\label{eq:gammadefrev} \gamma_i =-{(\nabla^2\f)_{p_{i+1}} (\tau
H_i,G_{i+1})\over (\nabla^2\f)_{p_{i+1}}(\tau H_i,\tau H_i)}.
\end{equation} This avoids the computation of~$\tau G_i$, which is
used in Equation~(\ref{eq:cgsteprev}), but introduces computation
given by Equation~(\ref{eq:d2rho}), which requires $O(nk^2)$ operations
plus $2k$ matrix-vector multiplications.  The cost of computing
$\gamma_i$ by Equation~(\ref{eq:gammadefrev}) is summarized in
Table~\ref{tab:2dcovd}.  The cost of changing the coset representative
using Equation~(\ref{eq:parthi}) is summarized in
Table~\ref{tab:changecoset}.

\begin{table}[t]
\caption[Computational requirements of $(\Dsqr\rho)_p(X,Y)$ computation]
  {\label{tab:2dcovd}\ignorespaces
    Computational requirements of\/ $(\Dsqr\rho)_p(X,Y)$ computation}
\centerline{\vbox{\tabskip=0pt \def\opentable{\noalign{\vskip4pt}}
\halign to.75\textwidth{\strut#\hfil\tabskip=1em plus1fil
  &\hfil$#$\hfil\tabskip=0pt\cr
\noalign{\hrule}\opentable
\omit Procedure\hfil&\omit\hfil Cost\hfil\cr
\opentable\noalign{\hrule}\opentable
$(xy+yx)$&k^2(3n-2k)\cr
$g\cdot{}$\quad(thrice)&3k^2(2n-3k)+O(nk)\cr
$A\cdot{}$\quad(twice)&\hbox{$2k$ mat-vec*}\cr
$\tr p^\T qN$&nk\cr
\opentable \noalign{\nointerlineskip}
\omit\hfil&\omit\hrulefill\cr \opentable
Total&9nk^2-11k^3+2k\,\hbox{mat-vec}+O(nk)\cr
\opentable\noalign{\hrule}\opentable
\multispan2\footnotesize* Represents one matrix-vector
  multiplication.\hfil\cr}}}
\end{table}

\begin{table}
\caption[Computational requirements of changing coset representative\null]
  {\label{tab:changecoset}\ignorespaces
    Computational requirements of changing coset representation}
\centerline{\vbox{\tabskip=0pt \def\opentable{\noalign{\vskip4pt}}
\halign to.75\textwidth{\strut#\hfil\tabskip=1em plus1fil
  &\hfil$#$\hfil\tabskip=0pt\cr
\noalign{\hrule}\opentable
\omit Procedure\hfil&\omit\hfil Cost\hfil\cr
\opentable\noalign{\hrule}\opentable
|QR|~decomposition of~$p_i$&\hbox{none*}\cr
|QR|~decomposition of~$p_{i+1}$&\hbox{none*}\cr
|QR|~decomposition of~$x_i$&\hbox{none*}\cr
Canonical decomposition of~$x_i'$&\hbox{none*}\cr
$\diag(I,Q^\T)\cdot x_i$&k^2(2n-3k)+O(nk)\cr
$\diag(\vartheta^\T,I)\cdot{}$&4k^3\cr
$\diag(e^{st},I)\cdot{}$&4k^2\cr
$\diag(\vartheta,I)\cdot{}$&4k^3\cr
$\diag(I,Q)\cdot{}$&k^2(2n-3k)+O(nk)\cr
$g_i\cdot{}$&k^2(2n-k)+O(nk)\cr
$g_{i+1}^{-1}\cdot{}$&k^2(2n-k)+O(nk)\cr
\opentable \noalign{\nointerlineskip}
\omit\hfil&\omit\hrulefill\cr \opentable
Total&8nk^2+O(nk)\cr
\opentable\noalign{\hrule}\opentable
\multispan2\footnotesize* Assumed to be pre-computed in the geodesic
  computation.\hfil\cr}}}
\end{table}

\subsection{The stepsize}

Once the direction for geodesic minimization $H_i$ is computed, a
stepsize $\lambda_i$ must be computed such that $$\exp\lambda_iH_i\le
\exp\lambda H_i \qquad\hbox{for all $\lambda\ge0$.}$$ In the case
$k=1$ ($\stiefel(n,1)=S^{n-1}$), Algorithm~\ref{al:raysphere},
Chapter~\ref{chap:orm}, provides an explicit formula for the stepsize
(which requires one matrix vector multiplication and a few $O(n)$
inner products).  In the case $k=n$ ($\stiefel(n,n)=\O(n)$),
\citeasnoun{Brockett:grad} provides an estimate of the stepsize, which
is covered in Example~\ref{eg:trHNgrad}, Chapter~\ref{chap:orm}.
Consider this approach in the general context $1\le k\le n$.  Given
$p\in M$, $g\in G$ a coset representative of~$p$, and $x\in\m$, we
wish to compute $t>0$ such that the function $t\mapsto \phi(t)
=\rho(p_t) =\tr p_t^\T\!Ap_tN$ is minimized, where $p_t=ge^{xt}\cdot
o$.  Differentiating $\phi$ twice shows that $$\eqalign{\phi'(t)
&=-\tr\Ad_{(ge^{xt})^\T}\bigl([\Ad_gx,A]\bigr)oNo^\T,\cr \phi''(t)
&=-\tr\Ad_{(ge^{xt})^\T}\bigl([\Ad_gx,A]\bigr) [x,oNo^\T].\cr}$$ Hence we
have $\phi'(0)=2\tr p^\T\!AgxoN$, which may be computed with $nk^2$~flops
if the matrix $Ap$ is known.  By Schwarz's inequality and the fact
that $\Ad$ is an isometry, we have $$|\phi''(t)|\le
\bigl\|[\Ad_gx,A]\bigr\|\;\bigl\|[x,oNo^\T]\bigr\|.$$ The term
$\bigl\|[x,oNo^\T]\bigr\|$ is easily computed, but there is no
efficient, i.e., $O(nk^2)$, method to compute the term
$\bigl\|[\Ad_gx,A]\bigr\|$.

However, there are several line minimization algorithms from classical
optimization theory that may be employed in this context.  In general,
there is a tradeoff between the cost of the line search algorithm and
its accuracy; good algorithms allow the user to specify accuracy
requirements.  The Wolfe-Powell line search algorithm \cite{Fletcher}
is one such algorithm.  It is guaranteed to converge under mild
assumptions, and allows the user to specify bounds on the error of the
approximate stepsize to the desired stepsize.  Near the minimum of the
function, an approximate stepsize may be computed via a Taylor
expansion about zero: $$\phi(t)=\phi(0) +t(\covD{X}\rho)_p
+{t^2\over2}(\covD{X}^2\rho)_p +\hot$$ Truncating this expansion at
the third order terms and solving the resulting quadratic optimization
problem yields the approximation \begin{equation} \label{eq:stepsize}
\argmax \phi(t)= -{(\covD{X}\rho)_p\over (\covD{X}^2\rho)_p}.
\end{equation} Some of the information used in the computation
of~$\gamma_i$ described above may be used to compute this choice of
stepsize.  In practice, this choice of stepsize may be used as a trial
stepsize for the Wolfe-Powell or similar line searching algorithm.  As
the conjugate gradient algorithm converges, it will yield increasingly
better approximations of the desired stepsize, and the iterations
required in the line searching algorithm may be greatly reduced.

\subsection{The sorting problem}

One interesting feature of this type of optimization algorithm,
discovered by \citeasnoun{Brockett:sort}, is its ability to sort lists
of numbers.  However, from the viewpoint of the tracking application,
this property slows the algorithm's convergence because the sequence
of points $p_i$ may pass near one of the many saddle points where the
columns of~$p_i$ are approximately eigenvectors.  A practical
algorithm would impose convergence near these saddle points because
the eigenvectors may be sorted inexpensively with an $O(k\log k)$
algorithm such as heap~sort.  In the algorithm used in the next
section, the diagonal elements of~$N$ are sorted similarly to the
diagonal elements of~$p^\T\!Ap$.  Whenever a resorting of~$N$ occurs,
the conjugate gradient algorithm is reset so that its next direction
is simply the gradient direction of~$\tr p^\T\!ApN$, where the
diagonal of~$N$ is a sorted version of the original.  Conversely, the
columns of the matrix~$p$ may be re-sorted so that the diagonal
of~$p^\T\!Ap$ is ordered similarly to the diagonal of~$N$.  This
latter procedure is accomplished efficiently if $p$ is represented in
the computer as an array of pointers to vectors.

\begin{figure}[t]
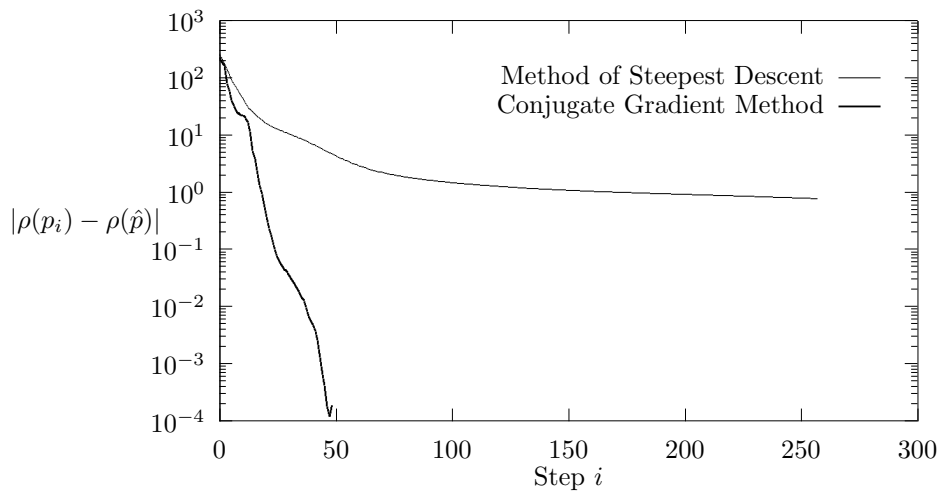

\begin{gnuplot}{20pt}
\input conv-V100,3
\end{gnuplot}
\caption[Convergence with the generalized Rayleigh quotient on
  $\stiefel(100,3)$]{Maximization of the generalized Rayleigh quotient
  $\tr p^\T\!ApN$ on~$\stiefel(100,3)$ (dimension
  $\stiefel(100,3)=294$) with both the method of steepest descent and
  the conjugate gradient algorithm of Section~\ref{sec:cg-genray}.
  Here $A=\diag(100,\dots,1)$ and $N=\diag(3,2,1)$. The $i$th iterate
  is $p_i$ and the maximum point is
  $\phat$.\label{fig:raygen}}
\end{figure}

\subsection{Experimental results}

Algorithm~\ref{al:cgman}, Chapter~\ref{chap:orm}, was applied to the
generalized Rayleigh quotient defined on the manifold
$\stiefel(100,3)$ with $A=\diag(100,99,\ldots,1)$, $N=\diag(3,2,1)$,
and $p_0$ chosen at random from $\stiefel(100,3)$ using Gram-Schmidt
orthogonalization.  The results are shown in Figure~\ref{fig:raygen}
along with the results of the method of steepest descent applied to
the generalized Rayleigh quotient.  Figure~\ref{fig:eigraygen} shows
the convergence of the estimated eigenvalues of the matrix~$A$.  As
can be seen in Figure~\ref{fig:raygen}, the algorithm converged to
machine accuracy in about 50 steps. Figure~\ref{fig:eigraygen} shows
that good estimates of the largest three eigenvalues are obtained in
less than 25 steps.  Instead of the formula for $\gamma_i$ specified
by this algorithm, which relies upon parallel translation of the
previous gradient direction, $\gamma_i$ was computed using
Equation~(\ref{eq:gammadefrev}) in conjunction with
Equation~(\ref{eq:d2rho}).  The stepsize was chosen with a modified
version of the the Wolfe-Powell line minimization algorithm described
by \citeasnoun{Fletcher} with $\rho=0.01$ (cf.\ p.~30 of~Fletcher),
$\sigma=0.1$ (ibid., 83), $\tau_1=9.0$, $\tau_2=0.1$, and $\tau_3=0.5$
(ibid., 34--36).  The initial test stepsize was computed using
Equation~(\ref{eq:stepsize}).  The diagonal elements of~$N$ were
sorted similarly to the diagonal elements of~$p_i^\T\!Ap_i$, $i\ge0$,
and the conjugate gradient algorithm was reset to the gradient
direction every time sorting took place.  The algorithm was also
programmed to reset every $r$ steps with $r=\mathop{\rm
dimension}\stiefel(100,3) =294$; however, as the results of
Figure~\ref{fig:raygen} show, the algorithm converged to machine
accuracy long before the latter type of reset would be used.  The
algorithm of Ward and Gray \citeyear{WG:1,WG:2} was used to compute
the canonical decomposition of the skew-symmetric matrix $x'$.

\begin{figure}[t]
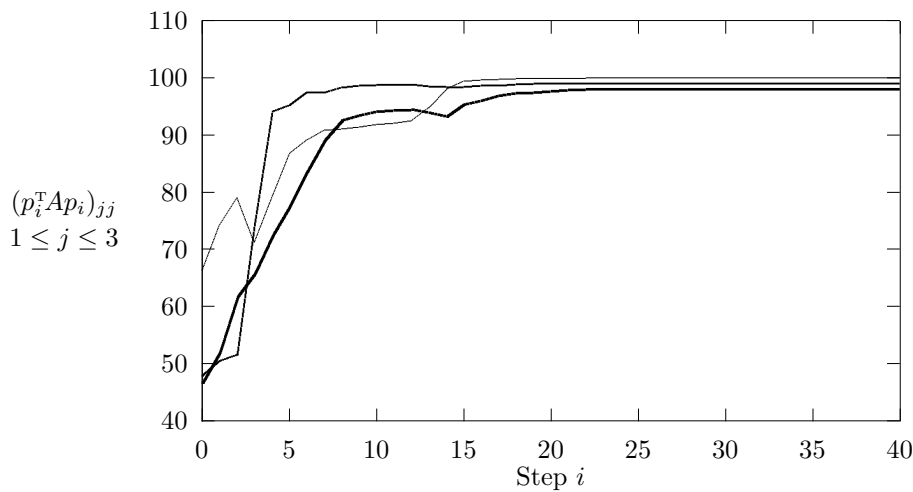

\begin{gnuplot}{20pt}
\input eig-V100,3
\end{gnuplot}
\caption[Convergence of the diagonal elements of
$p_i^\T\!Ap_i$]{Convergence of the diagonal elements of $p_i^\T\!Ap_i$
when the the conjugate gradient algorithm is applied to the
generalized Rayleigh quotient on~$\stiefel(100,3)$. Here the $i$th
iterate is $p_i$, and $A=\diag(100,\dots,1)$ and $N=\diag(3,2,1)$.
\label{fig:eigraygen}}
\end{figure}

\subsection{Largest left singular values}

Let $X$ be an $n\by l$ matrix with $n\le l$.  The matrix $X$ may be
thought of as a data matrix whose principal invariant subspaces are
desired, i.e., we wish to compute the eigenvectors corresponding to
the largest eigenvalues of $R=XX^\T$, or, equivalently, the left
singular vectors corresponding to the largest singular values of~$X$.
As explained at the end of Section~\ref{sec:adapest}, it is desirable
to work directly with the data matrix, or with the square root $L$
of~$R$, i.e., $R=LL^\T$.  This can be obtained from the
|QR|~decomposition $X=LQ$, where $L$ is a $n\by l$ lower triangular
matrix and $Q$ is a $l\by l$ orthogonal matrix.

The conjugate gradient algorithms presented in this section may be
modified to compute the largest singular vectors of~$X$.  Computations
of the form $p^\T\!Rq$, where $R$ is a symmetric matrix and $p$ and
$q$ are arbitrary $n\by k$ matrices, must be replaced with the
computation $(L^\T p)^\T(L^\T q)$, and computations of the form $Rp$
must be replaced with $L(L^\T p)$.  While not as bad as explicitly
computing $R=XX^\T$, these methods do involve squaring the data.

It is worthwhile to ask if this may be avoided.  Instead of optimizing
the generalized Rayleigh quotient to obtain the largest left singular
vectors, consider the function $\sigma\colon\stiefel(n,k)\to\R$
defined by the following steps.  Let $p\in\stiefel(n,k)$, $A$ an
arbitrary $n\by n$ matrix, and $N$ a real $k\by k$ diagonal matrix.
\begin{steps}
\step[1.] Compute $B=A^\T p$.
\step[2.] Compute the |QR|~decomposition of~$B=:QR$, where $Q$ is an
$n\by n$ orthogonal matrix and $R$ is an $n\by k$ upper triangular
matrix whose upper $k\by k$ block $R_1$ has positive real diagonal
entries ordered similarly to the diagonal of~$N$.
\step[3.] Set $\sigma(p)=\tr R_1N$.
\end{steps}
This approach avoids the data squaring problem.  Using the techniques
of Chapter~\ref{chap:grad}, it is straightforward to show that the
critical points of~$\sigma$ correspond to points~$p$ whose columns are
left singular vectors of~$A$.  The function $\sigma$ is maximized when
the corresponding singular values are similarly ordered to the
diagonal elements of~$N$.

However, computing a formula for the gradient and second covariant
differential of~$\sigma$ is difficult.  Indeed, when $R_1$ is
singular, this function is not differentiable on~$\stiefel(n,k)$.  To
compute the gradient of~$\sigma\colon\stiefel(n,k)\to\R$, the first
order perturbation of~$\sigma$ with respect to its argument must be
computed.  To do this, the first order perturbations of an arbitrary
|QR|~decomposition $B_t=Q_tR_t$, where $B_t$ is an $n\by k$ matrix
parameterized by~$t$, must be computed.  By assumption $B_0=Q_0R_0$
and $$\diffeqalign{B_t &=B_0+tB_0Y+\cdots &\hbox{$Y$ arbitrary $n\by
k$,}\cr Q_t &=Q_0(I+t\Omega+\cdots) &\hbox{$\Omega$ in $\so(n)$,}\cr
R_t &=R_0+t\Psi+\cdots &\hbox{$\Psi$ arbitrary $n\by k$.}\cr}$$ The
first order terms of $B_t=Q_tR_t$ may be written as $$R_0Y=\Omega
R_0+\Psi.$$ For the application we have in mind, $Y$ is a tangent
vector of the Stiefel manifold (by Step~1).  To fix ideas, we shall
consider the case $k=n=2$, and set $$Y
=\left({0\atop-y}\>{y\atop0}\right),\quad R_0
=\left({\alpha\atop0}\>{\beta\atop\gamma}\right), \quad \Omega
=\left({0\atop-\omega}\>{\omega\atop0}\right), \quad\hbox{and}\quad
\Psi =\left({\psi_1\atop 0}\>{\psi_2\atop\psi_3}\right).$$ Solving for
$\Omega$ and $\Psi$, we find $$\Omega={\gamma y\over\alpha}
\left({0\atop-1}\>{1\atop0}\right) \quad\hbox{and}\quad \Psi
={y\over\alpha} \left({-\alpha\gamma\atop
0}\>{\alpha^2-\gamma^2\atop\beta\gamma}\right).$$ There does not
appear to be an efficient $O(nk^2)$ algorithm for computing the
gradient of~$\sigma$ in general.

We can use Equation~(\ref{eq:svdrel}) of Chapter~\ref{chap:grad} to
define a more tractible function for optimization.  Given an arbitrary
$n\by n$ matrix $A$, let $\sigma'\colon \stiefel(n,k)\to\R$ be defined
by the following steps.
\begin{steps}
\step[1.] Compute $B=A^\T p$.
\step[2.] Compute the $n\by k$ matrix $q$ defined by the equation
  $B=:qD$ such that the columns of~$q$ have unit length and $D$ is a
  $k\by k$ diagonal matrix.
\step[3.] Set $\sigma'(p)=\tr q^\T\!A^\T pN$.
\end{steps}

This approach also avoids the data squaring problem.  It can be shown
that the critical points of~$\sigma'$ correspond to matrices
$p\in\stiefel(n,k)$ whose columns are the left singular vectors
corresponding to the $k$ largest singular values of~$A$.  The
differential and gradient of~$\sigma'$ are straightforward to compute.
Let $\zeta\colon \R^{n\times k}\to\R^{n\times k}$ be the projection
defined by setting the diagonal elements of an $n\by k$ matrix to
zero. Let $g$ be a coset representative of~$p$, and let $x\in\m$
correspond to~$X\in T_pM$.  Then $$d\sigma'_p(X)=\tr \bigl(\zeta(A^\T
gxo)^\T\!Ap +q^\T\!Agxo\bigr)N.$$ By the fact that
$\tr\zeta(a)^\T b=\tr a^\T\zeta(b)$ for $a$, $b\in\R^{n\times k}$, it
is seen that the vector $v\in\m$ corresponding to~$(\grad\sigma')_p$
is given by $$v=\bigl(oNq^\T\!A^\T g -q^\T\!A^\T\zeta(A^\T
pN)o^\T\bigr)_\m.$$ The second covariant differential of~$\sigma'$ may
be computed similarly, yielding the formulas necessary to implement a
conjugate gradient algorithm on~$\stiefel(n,k)$ yielding the left
singular vectors corresponding to the largest singular values of~$A$.

It is also possible to compute the corresponding right singular
vectors simultaneously.  Consider the function $\sigma''\colon
\stiefel(n,k)\times\stiefel(n,k)\to\R$ defined by $$\sigma''(p,q)=\tr
p^\T\!AqN.$$ The critical points of~$\sigma''$ correspond to matrices
$p$ and $q\in\stiefel(n,k)$ whose columns are left and right singular
vectors of~$A$, respectively.  Optimization algorithms developed in
this section may be generalized and applied to this function.

\section{Conjugate gradient method for subspace tracking}

Gradient-based algorithms are very appealing for tracking applications
because of their ability to move in the best direction to minimize
error. In the idealized scenario, the algorithm yields a sequence of
points that are at or near a minimum point.  When the minimum point
changes, it is assumed to change slowly or continuously so that the
gradient algorithm does not have far to go to follow the time varying
solution.

In their review of subspace tracking algorithms,
\citeasnoun{ComonGolub} provide computer simulations of the behavior
of a variety of algorithms tracking a step change in the signal
subspace.  Specifically, they track the principal subspaces of the
signal $$x_t=\cases{s^1_te_1 +s^2_te_2 &if $0\le t\le T$;\cr s^1_te_3
+s^2_te_4 &if $t>T$,\cr}$$ where $s^1_t$ and $s^2_t$ are wide-sense
stationary random sequences and $e_1$, $e_2$, $e_3$, and $e_4$ are the
first four standard basis elements of~$\R^{10}$.  To isolate the
tracking problem from the problem of covariance matrix estimation, we
choose a slightly different approach here.

Instead of changing the data sample $x$, and updating its covariance
matrix, we shall simply allow the symmetric matrix $A$ to change
arbitrarily over time, i.e., $A_t$ is an $n\by n$ symmetric matrix for
each $t=0$, $1$, \dots, and the goal shall be to track the largest $k$
eigenvalues of~$A_t$ and their associated eigenvectors.
Algorithm~\ref{al:cgman} of Chapter~\ref{chap:orm} may be modified as
follows so that one conjugate gradient step is performed at every time
step. Of course, more than one conjugate gradient step per time step
may be performed.

\begin{algorithm}[Conjugate gradient subspace
tracking]\label{al:cgtrack}\ignorespaces Let $A_i$ be a symmetric
matrix for $i=0$, $1$, \dots, and denote the generalized Rayleigh
quotient with respect to~$A_i$ by~$p\mapsto \rho(p)=\tr p^\T\!A_ipN$.
Select $p_0\in\stiefel(n,k)$ and set $i=0$.
\begin{steps}
\step[1.] Compute $$G_i=(\grad\rho)_{p_i}$$ via
  Equation~$(\ref{eq:raygengrad}')$.
\step[2.] If $i\equiv0\ (\bmod\ \dim\stiefel(n,k))$, then set
  $H_i=G_i$. If the diagonal elements of~$p_i^\T\!A_ip_i$ are not ordered
  similarly to those of~$N$, then re-sort the diagonal of~$N$, set
  $H_i=G_i$, and restart the step count.  Otherwise, set
  $$H_i=G_i+\gamma_{i-1}\tau H_{i-1},$$ where $\gamma_i$ is given by
  Equation~(\ref{eq:gammadefrev}).
\step[3.] Compute $\lambda_i$ such that
  $$\rho(\exp_{p_i}\lambda_iH_i)\le \rho(\exp_{p_i}\lambda H_i)$$
  for~all $\lambda>0$. Use Equation~(\ref{eq:stepsize}) for an initial
  guess of the stepsize for the Wolfe-Powell line search.
\step[4.] Set $p_{i+1}=\exp_{p_i} \lambda_i H_i$, increment $i$, and
  go~to Step~1.
\end{steps}
\end{algorithm}

\subsection{Experimental results}

Algorithm~\ref{al:cgtrack} with $n=100$ and $k=4$ was applied to the
time varying matrix \begin{equation} \label{eq:firsttest}
A_i=\cases{\diag(100,99,\ldots,1) &if $0\le i\le 40$;\cr \Theta_1\cdot
\diag(100,99,\ldots,1)\cdot \Theta_1^\T &if $i>40$,\cr} \end{equation}
where $$\Theta_1=\pmatrix{\cos135^\circ&\sin135^\circ&0\cr
-\sin135^\circ&\cos135^\circ&0\cr 0&0&I\cr}.$$ I.e., the invariant
subspace associated with the largest two eigenvalues of~$A_i$ is
rotated by~$135^\circ$ at time~$t=40$.  A slightly modified version of
the algorithm was also tested, whereby the conjugate gradient
algorithm was reset at $t=40$.  That is, at this time Step~2 was
replaced with
\begin{steps}
\step[$2'$.] Set $H_i=G_i$ and restart the step count.
\end{steps}

The values of $|\rho(p_i)-\rho(\phat)|$, where $\phat$ is the
minimizing value of~$\rho$, resulting from these two experiments are
shown in Figure~\ref{fig:track1raygen}.  As may be seen, both
algorithms track the step in the matrix $A$; however, the reset
algorithm, which ``forgets'' the directional information prior to
$t=40$, has better performance.  Thus in a practical subspace tracking
algorithm it may be desirable to reset the algorithm if there is a
large jump in the value of $|\rho(p_i)-\rho(\phat)|$.  The diagonal
elements of the matrix~$p_i^\T\!A_ip_i$ resulting from
Algorithm~\ref{al:cgtrack} (no~reset) are shown in
Figure~\ref{fig:trk1eigraygen}.  As may be seen, good estimates for
the largest eigenvalues of~$A_i$ are obtained in about 5~iterations
beyond the step at $t=40$.  This compares favorably to the subspace
tracking algorithms tested by \citeasnoun{ComonGolub}, where the
fastest convergence of about 20~iterations is obtained by the
Lanczos algorithm.  It is important to note however, that the two
experiments are different in several important ways, making a direct
comparison difficult.  The experiment of Comon and Golub incorporated
a covariance matrix estimation technique, whereas our matrix $A_i$
changes instantaneously.  Also, Comon and Golub implicitly use the
space $\stiefel(10,2)$, whose dimension is much smaller than that of
the space $\stiefel(100,3)$ which we have selected.

In the previous experiment, the principal invariant subspace was
unchanged and the corresponding eigenvalues were unchanged by the
rotation $\Theta_1$.  To test the algorithm's response to a step
change in the orientation of the principal invariant subspace along
with a step change in its corresponding eigenvalues, the algorithm was
applied to the time varying matrix \begin{equation}
\label{eq:secondtest} A_i=\cases{\diag(100,99,98, 97,96,95, \ldots,1)
&if $0\le i\le 40$;\cr \Theta_2\cdot \diag(100,99,98, 101,102,103,94,
\ldots,1)\cdot \Theta_2^\T &if $i>40$,\cr} \end{equation} where
$\Theta_2=R_{14}(135^\circ)\cdot R_{25}(135^\circ)\cdot
R_{36}(135^\circ)$, and $R_{ij}(\theta)$ is rotation by $\theta$ of the
plane spanned by the vectors $e_i$ and $e_j$.
Figure~\ref{fig:track2raygen} shows the value of
$|\rho(p_i)-\rho(\phat)|$ and Figure~\ref{fig:trk2eigraygen} shows the
estimated eigenvalues.

Finally, we wish to determine the algorithm's performance when
principal invariant subspace changes in one step to a mutually
orthogonal subspace of itself.  This is important because the
generalized Rayleigh quotient has many ($2^k\,{}_nP_k$) critical
points, most of which are saddle points.  If the algorithm has
converged exactly to a minimum point, and a step change is then
introduced which makes this point a saddle point, an exact
implementation of the conjugate gradient algorithm could not adapt to
this change because the gradient is zero at the saddle point.
However, numerical inaccuracies on a finite-precision machine
eventually drive the iterates from the saddle point to the minimum
point.  The algorithm was applied to the time varying matrix
\begin{equation} \label{eq:thirdtest} A_i=\cases{\diag(100,99,98,
97,96,95, \ldots,1) &if $0\le i\le 40$;\cr \diag(97,96,95,100,99,98,
94, \ldots,1) &if $i>40$.\cr} \end{equation}
Figure~\ref{fig:track3raygen} shows the value of
$|\rho(p_i)-\rho(\phat)|$ and Figure~\ref{fig:trk3eigraygen} shows the
estimated eigenvalues.  As predicted, the iterates initially stay near
the old minimum point, which has become a saddle point.  After about
fifteen iterations, numerical inaccuracies drive the iterates away
from the saddle point to the new minimum point.

\vfil\eject  

\begin{figure}[p]
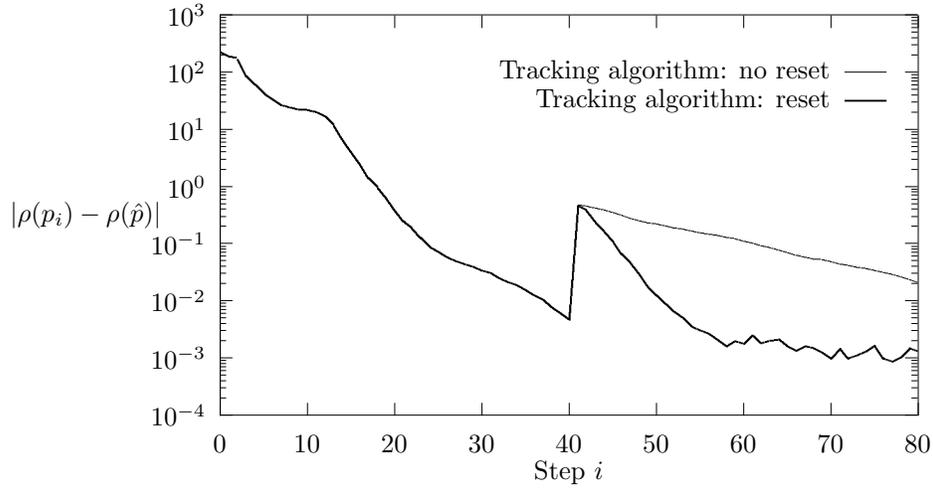

\begin{gnuplot}{20pt}
\input track1-V100,3
\end{gnuplot}
\caption[Convergence of the tracking algorithm: First
test]{Convergence of Algorithm~\ref{al:cgtrack} on~$\stiefel(100,3)$
applied to the matrix $A_i$ (from Eq.~(\ref{eq:firsttest})), which has
a step at $i=40$. The thin line represents values generated with no
reset at $i=40$, and the thicker line represents values generated when
the conjugate gradient algorithm is reset
at~$i=40$.\parfillskip=0pt\par \label{fig:track1raygen}}
\end{figure}

\begin{figure}[p]
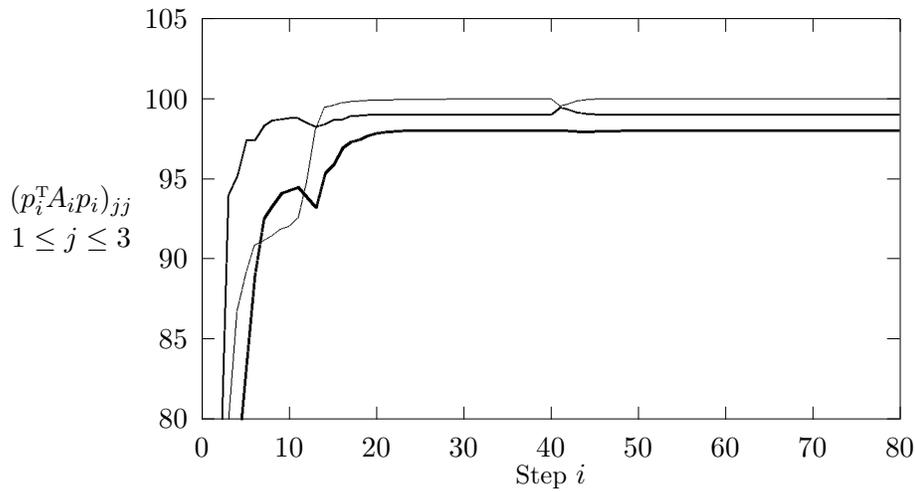

\begin{gnuplot}{20pt}
\input trk1reseig-V100,3
\end{gnuplot}
\caption[Tracking estimated eigenvalues: First test]{The diagonal
elements of $p_i^\T\!A_ip_i$ generated by Algorithm~\ref{al:cgtrack}
on~$\stiefel(100,3)$, where $A_i$ is define by
Eq.~(\ref{eq:firsttest}).  The conjugate gradient algorithm was reset
at $i=40$.\label{fig:trk1eigraygen}}
\end{figure}

\begin{figure}[p]
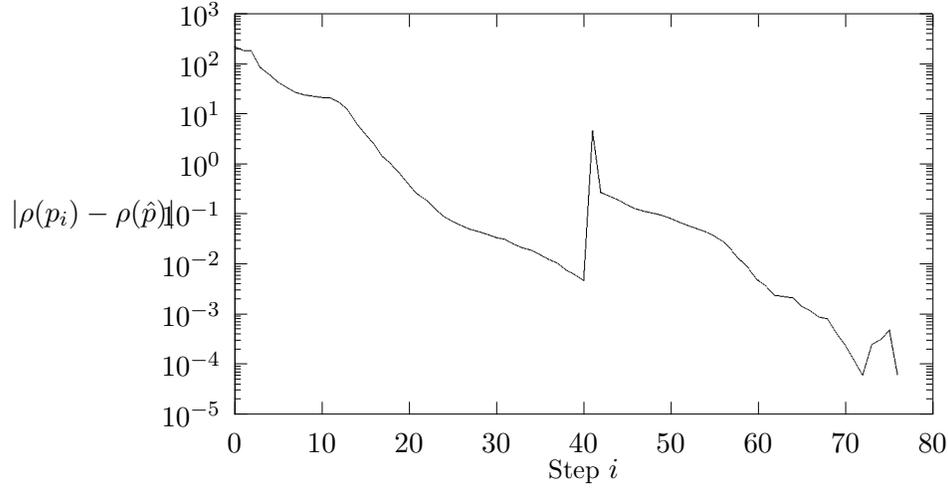

\begin{gnuplot}{20pt}
\input track2-V100,3
\end{gnuplot}
\caption[Convergence of the tracking algorithm: Second
test]{Convergence of Algorithm~\ref{al:cgtrack} on~$\stiefel(100,3)$
applied to the matrix $A_i$ (from Eq.~(\ref{eq:secondtest})), which
has a step at $i=40$.\label{fig:track2raygen}}
\end{figure}

\begin{figure}[p]
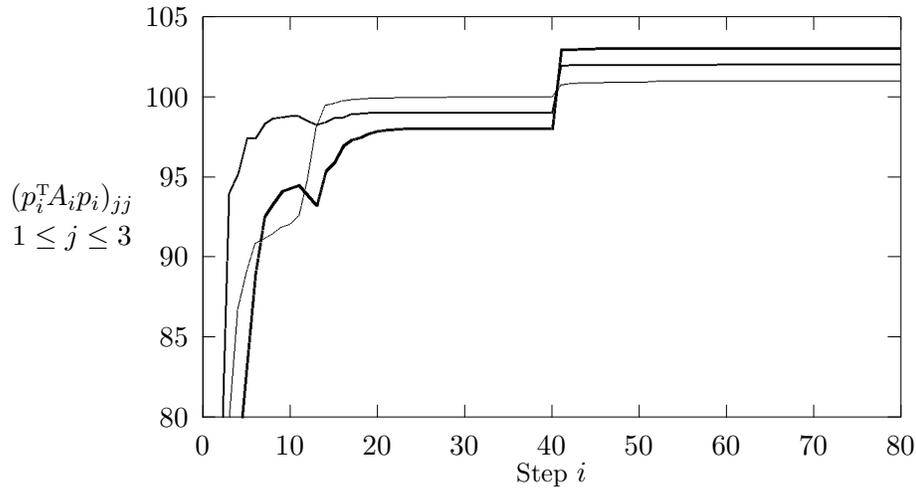

\begin{gnuplot}{20pt}
\input trk2eig-V100,3
\end{gnuplot}
\caption[Tracking estimated eigenvalues: Second test]{The diagonal
elements of $p_i^\T\!A_ip_i$ generated by Algorithm~\ref{al:cgtrack}
on~$\stiefel(100,3)$, where where $A_i$ is define by
Eq.~(\ref{eq:secondtest}).\label{fig:trk2eigraygen}}
\end{figure}

\makeatletter
\if@twoside    
\markboth{\uppercase{\ch@ptern@me\hfill
    \ifnum\c@secnumdepth>\m@ne\@chapabb\ 
      {\numbersize\thechapter}.\fi}}{\ifnum\c@secnumdepth>\m@ne
    \@sectionabb\thinspace{\numbersize\thesection}.\fi\hfill
  \uppercase{Conjugate gradient method for subspace tracking}}
\fi
\makeatother

\begin{figure}[p]
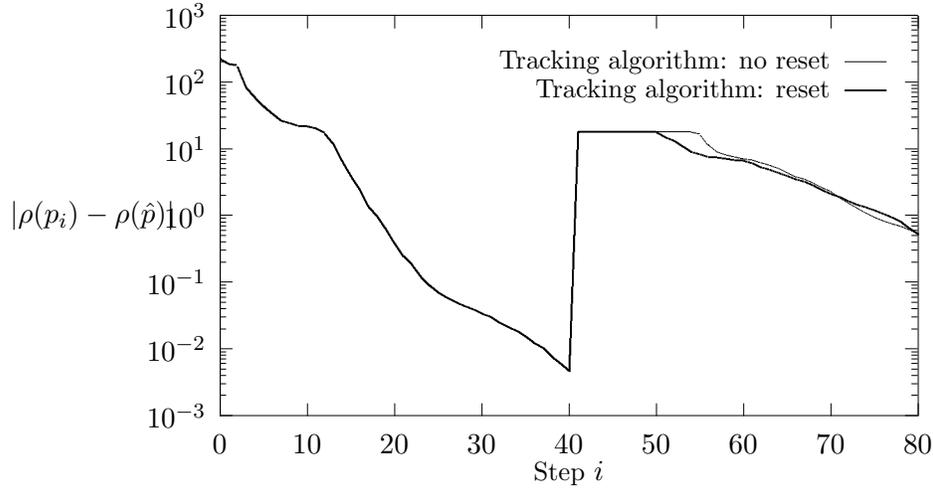

\begin{gnuplot}{20pt}
\input track3-V100,3
\end{gnuplot}
\caption[Convergence of the tracking algorithm: Third
test]{Convergence of Algorithm~\ref{al:cgtrack} on~$\stiefel(100,3)$
applied to the matrix $A_i$ (from Eq.~(\ref{eq:thirdtest})), which has
a step at $i=40$. The results show how the algorithm behaves when a
maximum point becomes a saddle point.\label{fig:track3raygen}}
\end{figure}

\begin{figure}[p]
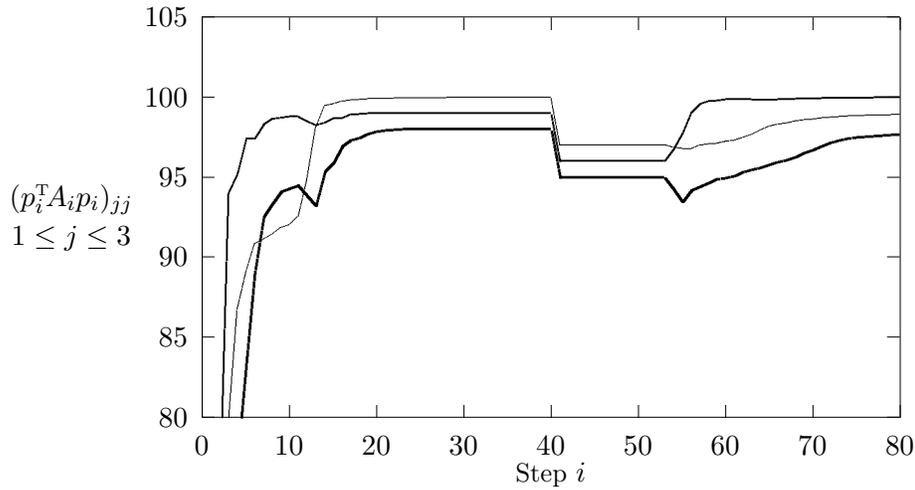

\begin{gnuplot}{20pt}
\input trk3eig-V100,3
\end{gnuplot}
\caption[Tracking estimated eigenvalues: Third test]{The diagonal
elements of $p_i^\T\!A_ip_i$ generated by Algorithm~\ref{al:cgtrack}
on~$\stiefel(100,3)$, where $A_i$ is define by
Eq.~(\ref{eq:thirdtest}). The results show how the eigenvalues
of~$A_i$ are tracked when a maximum point becomes a saddle point.
\label{fig:trk3eigraygen}}
\end{figure}

%% file: conv-V100,3.tex
\setlength{\unitlength}{0.240900pt}
\ifx\plotpoint\undefined\newsavebox{\plotpoint}\fi
\sbox{\plotpoint}{\rule[-0.175pt]{0.350pt}{0.350pt}}%
\begin{picture}(1424,900)(0,0)
\tenrm
\sbox{\plotpoint}{\rule[-0.175pt]{0.350pt}{0.350pt}}%
\put(264,158){\rule[-0.175pt]{0.350pt}{151.526pt}}
\put(264,158){\rule[-0.175pt]{4.818pt}{0.350pt}}
\put(242,158){\makebox(0,0)[r]{$10^{\gnulog 1e-04}$}}
\put(1340,158){\rule[-0.175pt]{4.818pt}{0.350pt}}
\put(264,185){\rule[-0.175pt]{2.409pt}{0.350pt}}
\put(1350,185){\rule[-0.175pt]{2.409pt}{0.350pt}}
\put(264,201){\rule[-0.175pt]{2.409pt}{0.350pt}}
\put(1350,201){\rule[-0.175pt]{2.409pt}{0.350pt}}
\put(264,212){\rule[-0.175pt]{2.409pt}{0.350pt}}
\put(1350,212){\rule[-0.175pt]{2.409pt}{0.350pt}}
\put(264,221){\rule[-0.175pt]{2.409pt}{0.350pt}}
\put(1350,221){\rule[-0.175pt]{2.409pt}{0.350pt}}
\put(264,228){\rule[-0.175pt]{2.409pt}{0.350pt}}
\put(1350,228){\rule[-0.175pt]{2.409pt}{0.350pt}}
\put(264,234){\rule[-0.175pt]{2.409pt}{0.350pt}}
\put(1350,234){\rule[-0.175pt]{2.409pt}{0.350pt}}
\put(264,239){\rule[-0.175pt]{2.409pt}{0.350pt}}
\put(1350,239){\rule[-0.175pt]{2.409pt}{0.350pt}}
\put(264,244){\rule[-0.175pt]{2.409pt}{0.350pt}}
\put(1350,244){\rule[-0.175pt]{2.409pt}{0.350pt}}
\put(264,248){\rule[-0.175pt]{4.818pt}{0.350pt}}
\put(242,248){\makebox(0,0)[r]{$10^{\gnulog 1e-03}$}}
\put(1340,248){\rule[-0.175pt]{4.818pt}{0.350pt}}
\put(264,275){\rule[-0.175pt]{2.409pt}{0.350pt}}
\put(1350,275){\rule[-0.175pt]{2.409pt}{0.350pt}}
\put(264,291){\rule[-0.175pt]{2.409pt}{0.350pt}}
\put(1350,291){\rule[-0.175pt]{2.409pt}{0.350pt}}
\put(264,302){\rule[-0.175pt]{2.409pt}{0.350pt}}
\put(1350,302){\rule[-0.175pt]{2.409pt}{0.350pt}}
\put(264,311){\rule[-0.175pt]{2.409pt}{0.350pt}}
\put(1350,311){\rule[-0.175pt]{2.409pt}{0.350pt}}
\put(264,318){\rule[-0.175pt]{2.409pt}{0.350pt}}
\put(1350,318){\rule[-0.175pt]{2.409pt}{0.350pt}}
\put(264,324){\rule[-0.175pt]{2.409pt}{0.350pt}}
\put(1350,324){\rule[-0.175pt]{2.409pt}{0.350pt}}
\put(264,329){\rule[-0.175pt]{2.409pt}{0.350pt}}
\put(1350,329){\rule[-0.175pt]{2.409pt}{0.350pt}}
\put(264,334){\rule[-0.175pt]{2.409pt}{0.350pt}}
\put(1350,334){\rule[-0.175pt]{2.409pt}{0.350pt}}
\put(264,338){\rule[-0.175pt]{4.818pt}{0.350pt}}
\put(242,338){\makebox(0,0)[r]{$10^{\gnulog 1e-02}$}}
\put(1340,338){\rule[-0.175pt]{4.818pt}{0.350pt}}
\put(264,365){\rule[-0.175pt]{2.409pt}{0.350pt}}
\put(1350,365){\rule[-0.175pt]{2.409pt}{0.350pt}}
\put(264,381){\rule[-0.175pt]{2.409pt}{0.350pt}}
\put(1350,381){\rule[-0.175pt]{2.409pt}{0.350pt}}
\put(264,392){\rule[-0.175pt]{2.409pt}{0.350pt}}
\put(1350,392){\rule[-0.175pt]{2.409pt}{0.350pt}}
\put(264,401){\rule[-0.175pt]{2.409pt}{0.350pt}}
\put(1350,401){\rule[-0.175pt]{2.409pt}{0.350pt}}
\put(264,408){\rule[-0.175pt]{2.409pt}{0.350pt}}
\put(1350,408){\rule[-0.175pt]{2.409pt}{0.350pt}}
\put(264,414){\rule[-0.175pt]{2.409pt}{0.350pt}}
\put(1350,414){\rule[-0.175pt]{2.409pt}{0.350pt}}
\put(264,419){\rule[-0.175pt]{2.409pt}{0.350pt}}
\put(1350,419){\rule[-0.175pt]{2.409pt}{0.350pt}}
\put(264,423){\rule[-0.175pt]{2.409pt}{0.350pt}}
\put(1350,423){\rule[-0.175pt]{2.409pt}{0.350pt}}
\put(264,428){\rule[-0.175pt]{4.818pt}{0.350pt}}
\put(242,428){\makebox(0,0)[r]{$10^{\gnulog 1e-01}$}}
\put(1340,428){\rule[-0.175pt]{4.818pt}{0.350pt}}
\put(264,455){\rule[-0.175pt]{2.409pt}{0.350pt}}
\put(1350,455){\rule[-0.175pt]{2.409pt}{0.350pt}}
\put(264,470){\rule[-0.175pt]{2.409pt}{0.350pt}}
\put(1350,470){\rule[-0.175pt]{2.409pt}{0.350pt}}
\put(264,482){\rule[-0.175pt]{2.409pt}{0.350pt}}
\put(1350,482){\rule[-0.175pt]{2.409pt}{0.350pt}}
\put(264,490){\rule[-0.175pt]{2.409pt}{0.350pt}}
\put(1350,490){\rule[-0.175pt]{2.409pt}{0.350pt}}
\put(264,497){\rule[-0.175pt]{2.409pt}{0.350pt}}
\put(1350,497){\rule[-0.175pt]{2.409pt}{0.350pt}}
\put(264,504){\rule[-0.175pt]{2.409pt}{0.350pt}}
\put(1350,504){\rule[-0.175pt]{2.409pt}{0.350pt}}
\put(264,509){\rule[-0.175pt]{2.409pt}{0.350pt}}
\put(1350,509){\rule[-0.175pt]{2.409pt}{0.350pt}}
\put(264,513){\rule[-0.175pt]{2.409pt}{0.350pt}}
\put(1350,513){\rule[-0.175pt]{2.409pt}{0.350pt}}
\put(264,517){\rule[-0.175pt]{4.818pt}{0.350pt}}
\put(242,517){\makebox(0,0)[r]{$10^{\gnulog 1e+00}$}}
\put(1340,517){\rule[-0.175pt]{4.818pt}{0.350pt}}
\put(264,544){\rule[-0.175pt]{2.409pt}{0.350pt}}
\put(1350,544){\rule[-0.175pt]{2.409pt}{0.350pt}}
\put(264,560){\rule[-0.175pt]{2.409pt}{0.350pt}}
\put(1350,560){\rule[-0.175pt]{2.409pt}{0.350pt}}
\put(264,572){\rule[-0.175pt]{2.409pt}{0.350pt}}
\put(1350,572){\rule[-0.175pt]{2.409pt}{0.350pt}}
\put(264,580){\rule[-0.175pt]{2.409pt}{0.350pt}}
\put(1350,580){\rule[-0.175pt]{2.409pt}{0.350pt}}
\put(264,587){\rule[-0.175pt]{2.409pt}{0.350pt}}
\put(1350,587){\rule[-0.175pt]{2.409pt}{0.350pt}}
\put(264,593){\rule[-0.175pt]{2.409pt}{0.350pt}}
\put(1350,593){\rule[-0.175pt]{2.409pt}{0.350pt}}
\put(264,599){\rule[-0.175pt]{2.409pt}{0.350pt}}
\put(1350,599){\rule[-0.175pt]{2.409pt}{0.350pt}}
\put(264,603){\rule[-0.175pt]{2.409pt}{0.350pt}}
\put(1350,603){\rule[-0.175pt]{2.409pt}{0.350pt}}
\put(264,607){\rule[-0.175pt]{4.818pt}{0.350pt}}
\put(242,607){\makebox(0,0)[r]{$10^{\gnulog 1e+01}$}}
\put(1340,607){\rule[-0.175pt]{4.818pt}{0.350pt}}
\put(264,634){\rule[-0.175pt]{2.409pt}{0.350pt}}
\put(1350,634){\rule[-0.175pt]{2.409pt}{0.350pt}}
\put(264,650){\rule[-0.175pt]{2.409pt}{0.350pt}}
\put(1350,650){\rule[-0.175pt]{2.409pt}{0.350pt}}
\put(264,661){\rule[-0.175pt]{2.409pt}{0.350pt}}
\put(1350,661){\rule[-0.175pt]{2.409pt}{0.350pt}}
\put(264,670){\rule[-0.175pt]{2.409pt}{0.350pt}}
\put(1350,670){\rule[-0.175pt]{2.409pt}{0.350pt}}
\put(264,677){\rule[-0.175pt]{2.409pt}{0.350pt}}
\put(1350,677){\rule[-0.175pt]{2.409pt}{0.350pt}}
\put(264,683){\rule[-0.175pt]{2.409pt}{0.350pt}}
\put(1350,683){\rule[-0.175pt]{2.409pt}{0.350pt}}
\put(264,688){\rule[-0.175pt]{2.409pt}{0.350pt}}
\put(1350,688){\rule[-0.175pt]{2.409pt}{0.350pt}}
\put(264,693){\rule[-0.175pt]{2.409pt}{0.350pt}}
\put(1350,693){\rule[-0.175pt]{2.409pt}{0.350pt}}
\put(264,697){\rule[-0.175pt]{4.818pt}{0.350pt}}
\put(242,697){\makebox(0,0)[r]{$10^{\gnulog 1e+02}$}}
\put(1340,697){\rule[-0.175pt]{4.818pt}{0.350pt}}
\put(264,724){\rule[-0.175pt]{2.409pt}{0.350pt}}
\put(1350,724){\rule[-0.175pt]{2.409pt}{0.350pt}}
\put(264,740){\rule[-0.175pt]{2.409pt}{0.350pt}}
\put(1350,740){\rule[-0.175pt]{2.409pt}{0.350pt}}
\put(264,751){\rule[-0.175pt]{2.409pt}{0.350pt}}
\put(1350,751){\rule[-0.175pt]{2.409pt}{0.350pt}}
\put(264,760){\rule[-0.175pt]{2.409pt}{0.350pt}}
\put(1350,760){\rule[-0.175pt]{2.409pt}{0.350pt}}
\put(264,767){\rule[-0.175pt]{2.409pt}{0.350pt}}
\put(1350,767){\rule[-0.175pt]{2.409pt}{0.350pt}}
\put(264,773){\rule[-0.175pt]{2.409pt}{0.350pt}}
\put(1350,773){\rule[-0.175pt]{2.409pt}{0.350pt}}
\put(264,778){\rule[-0.175pt]{2.409pt}{0.350pt}}
\put(1350,778){\rule[-0.175pt]{2.409pt}{0.350pt}}
\put(264,783){\rule[-0.175pt]{2.409pt}{0.350pt}}
\put(1350,783){\rule[-0.175pt]{2.409pt}{0.350pt}}
\put(264,787){\rule[-0.175pt]{4.818pt}{0.350pt}}
\put(242,787){\makebox(0,0)[r]{$10^{\gnulog 1e+03}$}}
\put(1340,787){\rule[-0.175pt]{4.818pt}{0.350pt}}
\put(264,158){\rule[-0.175pt]{0.350pt}{4.818pt}}
\put(264,113){\makebox(0,0){$0$}}
\put(264,767){\rule[-0.175pt]{0.350pt}{4.818pt}}
\put(447,158){\rule[-0.175pt]{0.350pt}{4.818pt}}
\put(447,113){\makebox(0,0){$50$}}
\put(447,767){\rule[-0.175pt]{0.350pt}{4.818pt}}
\put(629,158){\rule[-0.175pt]{0.350pt}{4.818pt}}
\put(629,113){\makebox(0,0){$100$}}
\put(629,767){\rule[-0.175pt]{0.350pt}{4.818pt}}
\put(812,158){\rule[-0.175pt]{0.350pt}{4.818pt}}
\put(812,113){\makebox(0,0){$150$}}
\put(812,767){\rule[-0.175pt]{0.350pt}{4.818pt}}
\put(995,158){\rule[-0.175pt]{0.350pt}{4.818pt}}
\put(995,113){\makebox(0,0){$200$}}
\put(995,767){\rule[-0.175pt]{0.350pt}{4.818pt}}
\put(1177,158){\rule[-0.175pt]{0.350pt}{4.818pt}}
\put(1177,113){\makebox(0,0){$250$}}
\put(1177,767){\rule[-0.175pt]{0.350pt}{4.818pt}}
\put(1360,158){\rule[-0.175pt]{0.350pt}{4.818pt}}
\put(1360,113){\makebox(0,0){$300$}}
\put(1360,767){\rule[-0.175pt]{0.350pt}{4.818pt}}
\put(264,158){\rule[-0.175pt]{264.026pt}{0.350pt}}
\put(1360,158){\rule[-0.175pt]{0.350pt}{151.526pt}}
\put(264,787){\rule[-0.175pt]{264.026pt}{0.350pt}}
\put(-65,472){\makebox(0,0)[l]{\shortstack{$|\rho(p_i)-\rho(\phat)|$}}}
\put(812,68){\makebox(0,0){Step $i$}}
\put(264,158){\rule[-0.175pt]{0.350pt}{151.526pt}}
\put(1214,697){\makebox(0,0)[r]{Method of Steepest Descent}}
\put(1236,697){\rule[-0.175pt]{15.899pt}{0.350pt}}
\put(264,734){\usebox{\plotpoint}}
\put(264,732){\rule[-0.175pt]{0.350pt}{0.361pt}}
\put(265,731){\rule[-0.175pt]{0.350pt}{0.361pt}}
\put(266,729){\rule[-0.175pt]{0.350pt}{0.361pt}}
\put(267,728){\rule[-0.175pt]{0.350pt}{0.361pt}}
\put(268,725){\rule[-0.175pt]{0.350pt}{0.562pt}}
\put(269,723){\rule[-0.175pt]{0.350pt}{0.562pt}}
\put(270,721){\rule[-0.175pt]{0.350pt}{0.562pt}}
\put(271,719){\rule[-0.175pt]{0.350pt}{0.482pt}}
\put(272,717){\rule[-0.175pt]{0.350pt}{0.482pt}}
\put(273,715){\rule[-0.175pt]{0.350pt}{0.482pt}}
\put(274,713){\rule[-0.175pt]{0.350pt}{0.482pt}}
\put(275,710){\rule[-0.175pt]{0.350pt}{0.542pt}}
\put(276,708){\rule[-0.175pt]{0.350pt}{0.542pt}}
\put(277,706){\rule[-0.175pt]{0.350pt}{0.542pt}}
\put(278,704){\rule[-0.175pt]{0.350pt}{0.542pt}}
\put(279,701){\rule[-0.175pt]{0.350pt}{0.642pt}}
\put(280,698){\rule[-0.175pt]{0.350pt}{0.642pt}}
\put(281,696){\rule[-0.175pt]{0.350pt}{0.642pt}}
\put(282,694){\rule[-0.175pt]{0.350pt}{0.422pt}}
\put(283,692){\rule[-0.175pt]{0.350pt}{0.422pt}}
\put(284,690){\rule[-0.175pt]{0.350pt}{0.422pt}}
\put(285,689){\rule[-0.175pt]{0.350pt}{0.422pt}}
\put(286,687){\rule[-0.175pt]{0.350pt}{0.361pt}}
\put(287,686){\rule[-0.175pt]{0.350pt}{0.361pt}}
\put(288,684){\rule[-0.175pt]{0.350pt}{0.361pt}}
\put(289,683){\rule[-0.175pt]{0.350pt}{0.361pt}}
\put(290,681){\rule[-0.175pt]{0.350pt}{0.402pt}}
\put(291,679){\rule[-0.175pt]{0.350pt}{0.402pt}}
\put(292,678){\rule[-0.175pt]{0.350pt}{0.401pt}}
\put(293,676){\rule[-0.175pt]{0.350pt}{0.422pt}}
\put(294,674){\rule[-0.175pt]{0.350pt}{0.422pt}}
\put(295,672){\rule[-0.175pt]{0.350pt}{0.422pt}}
\put(296,671){\rule[-0.175pt]{0.350pt}{0.422pt}}
\put(297,669){\rule[-0.175pt]{0.350pt}{0.361pt}}
\put(298,668){\rule[-0.175pt]{0.350pt}{0.361pt}}
\put(299,666){\rule[-0.175pt]{0.350pt}{0.361pt}}
\put(300,665){\rule[-0.175pt]{0.350pt}{0.361pt}}
\put(301,663){\rule[-0.175pt]{0.350pt}{0.482pt}}
\put(302,661){\rule[-0.175pt]{0.350pt}{0.482pt}}
\put(303,659){\rule[-0.175pt]{0.350pt}{0.482pt}}
\put(304,657){\rule[-0.175pt]{0.350pt}{0.361pt}}
\put(305,656){\rule[-0.175pt]{0.350pt}{0.361pt}}
\put(306,654){\rule[-0.175pt]{0.350pt}{0.361pt}}
\put(307,653){\rule[-0.175pt]{0.350pt}{0.361pt}}
\put(308,651){\rule[-0.175pt]{0.350pt}{0.402pt}}
\put(309,649){\rule[-0.175pt]{0.350pt}{0.402pt}}
\put(310,648){\rule[-0.175pt]{0.350pt}{0.401pt}}
\put(311,648){\usebox{\plotpoint}}
\put(312,647){\usebox{\plotpoint}}
\put(313,646){\usebox{\plotpoint}}
\put(314,645){\usebox{\plotpoint}}
\put(315,644){\usebox{\plotpoint}}
\put(316,643){\usebox{\plotpoint}}
\put(317,642){\usebox{\plotpoint}}
\put(318,641){\usebox{\plotpoint}}
\put(319,640){\usebox{\plotpoint}}
\put(320,639){\usebox{\plotpoint}}
\put(321,638){\usebox{\plotpoint}}
\put(322,637){\usebox{\plotpoint}}
\put(323,636){\usebox{\plotpoint}}
\put(324,635){\usebox{\plotpoint}}
\put(326,634){\usebox{\plotpoint}}
\put(327,633){\usebox{\plotpoint}}
\put(328,632){\usebox{\plotpoint}}
\put(330,631){\usebox{\plotpoint}}
\put(331,630){\usebox{\plotpoint}}
\put(332,629){\usebox{\plotpoint}}
\put(333,628){\usebox{\plotpoint}}
\put(334,627){\usebox{\plotpoint}}
\put(335,626){\usebox{\plotpoint}}
\put(337,625){\rule[-0.175pt]{0.482pt}{0.350pt}}
\put(339,624){\rule[-0.175pt]{0.482pt}{0.350pt}}
\put(341,623){\rule[-0.175pt]{0.361pt}{0.350pt}}
\put(342,622){\rule[-0.175pt]{0.361pt}{0.350pt}}
\put(344,621){\rule[-0.175pt]{0.482pt}{0.350pt}}
\put(346,620){\rule[-0.175pt]{0.482pt}{0.350pt}}
\put(348,619){\rule[-0.175pt]{0.482pt}{0.350pt}}
\put(350,618){\rule[-0.175pt]{0.482pt}{0.350pt}}
\put(352,617){\rule[-0.175pt]{0.723pt}{0.350pt}}
\put(355,616){\rule[-0.175pt]{0.482pt}{0.350pt}}
\put(357,615){\rule[-0.175pt]{0.482pt}{0.350pt}}
\put(359,614){\rule[-0.175pt]{0.964pt}{0.350pt}}
\put(363,613){\rule[-0.175pt]{0.361pt}{0.350pt}}
\put(364,612){\rule[-0.175pt]{0.361pt}{0.350pt}}
\put(366,611){\rule[-0.175pt]{0.964pt}{0.350pt}}
\put(370,610){\rule[-0.175pt]{0.482pt}{0.350pt}}
\put(372,609){\rule[-0.175pt]{0.482pt}{0.350pt}}
\put(374,608){\rule[-0.175pt]{0.723pt}{0.350pt}}
\put(377,607){\rule[-0.175pt]{0.482pt}{0.350pt}}
\put(379,606){\rule[-0.175pt]{0.482pt}{0.350pt}}
\put(381,605){\rule[-0.175pt]{0.964pt}{0.350pt}}
\put(385,604){\rule[-0.175pt]{0.361pt}{0.350pt}}
\put(386,603){\rule[-0.175pt]{0.361pt}{0.350pt}}
\put(388,602){\rule[-0.175pt]{0.964pt}{0.350pt}}
\put(392,601){\rule[-0.175pt]{0.482pt}{0.350pt}}
\put(394,600){\rule[-0.175pt]{0.482pt}{0.350pt}}
\put(396,599){\rule[-0.175pt]{0.723pt}{0.350pt}}
\put(399,598){\rule[-0.175pt]{0.482pt}{0.350pt}}
\put(401,597){\rule[-0.175pt]{0.482pt}{0.350pt}}
\put(403,596){\rule[-0.175pt]{0.361pt}{0.350pt}}
\put(404,595){\rule[-0.175pt]{0.361pt}{0.350pt}}
\put(406,594){\rule[-0.175pt]{0.482pt}{0.350pt}}
\put(408,593){\rule[-0.175pt]{0.482pt}{0.350pt}}
\put(410,592){\rule[-0.175pt]{0.964pt}{0.350pt}}
\put(414,591){\rule[-0.175pt]{0.361pt}{0.350pt}}
\put(415,590){\rule[-0.175pt]{0.361pt}{0.350pt}}
\put(417,589){\rule[-0.175pt]{0.482pt}{0.350pt}}
\put(419,588){\rule[-0.175pt]{0.482pt}{0.350pt}}
\put(421,587){\rule[-0.175pt]{0.482pt}{0.350pt}}
\put(423,586){\rule[-0.175pt]{0.482pt}{0.350pt}}
\put(425,585){\rule[-0.175pt]{0.361pt}{0.350pt}}
\put(426,584){\rule[-0.175pt]{0.361pt}{0.350pt}}
\put(428,583){\rule[-0.175pt]{0.482pt}{0.350pt}}
\put(430,582){\rule[-0.175pt]{0.482pt}{0.350pt}}
\put(432,581){\rule[-0.175pt]{0.482pt}{0.350pt}}
\put(434,580){\rule[-0.175pt]{0.482pt}{0.350pt}}
\put(436,579){\rule[-0.175pt]{0.723pt}{0.350pt}}
\put(439,578){\rule[-0.175pt]{0.482pt}{0.350pt}}
\put(441,577){\rule[-0.175pt]{0.482pt}{0.350pt}}
\put(443,576){\rule[-0.175pt]{0.482pt}{0.350pt}}
\put(445,575){\rule[-0.175pt]{0.482pt}{0.350pt}}
\put(447,574){\rule[-0.175pt]{0.361pt}{0.350pt}}
\put(448,573){\rule[-0.175pt]{0.361pt}{0.350pt}}
\put(450,572){\rule[-0.175pt]{0.482pt}{0.350pt}}
\put(452,571){\rule[-0.175pt]{0.482pt}{0.350pt}}
\put(454,570){\rule[-0.175pt]{0.964pt}{0.350pt}}
\put(458,569){\rule[-0.175pt]{0.361pt}{0.350pt}}
\put(459,568){\rule[-0.175pt]{0.361pt}{0.350pt}}
\put(461,567){\rule[-0.175pt]{0.964pt}{0.350pt}}
\put(465,566){\rule[-0.175pt]{0.482pt}{0.350pt}}
\put(467,565){\rule[-0.175pt]{0.482pt}{0.350pt}}
\put(469,564){\rule[-0.175pt]{0.361pt}{0.350pt}}
\put(470,563){\rule[-0.175pt]{0.361pt}{0.350pt}}
\put(472,562){\rule[-0.175pt]{0.964pt}{0.350pt}}
\put(476,561){\rule[-0.175pt]{0.964pt}{0.350pt}}
\put(480,560){\rule[-0.175pt]{0.361pt}{0.350pt}}
\put(481,559){\rule[-0.175pt]{0.361pt}{0.350pt}}
\put(483,558){\rule[-0.175pt]{0.964pt}{0.350pt}}
\put(487,557){\rule[-0.175pt]{0.964pt}{0.350pt}}
\put(491,556){\rule[-0.175pt]{0.723pt}{0.350pt}}
\put(494,555){\rule[-0.175pt]{0.482pt}{0.350pt}}
\put(496,554){\rule[-0.175pt]{0.482pt}{0.350pt}}
\put(498,553){\rule[-0.175pt]{0.723pt}{0.350pt}}
\put(501,552){\rule[-0.175pt]{0.964pt}{0.350pt}}
\put(505,551){\rule[-0.175pt]{0.964pt}{0.350pt}}
\put(509,550){\rule[-0.175pt]{0.723pt}{0.350pt}}
\put(512,549){\rule[-0.175pt]{1.927pt}{0.350pt}}
\put(520,548){\rule[-0.175pt]{0.723pt}{0.350pt}}
\put(523,547){\rule[-0.175pt]{0.964pt}{0.350pt}}
\put(527,546){\rule[-0.175pt]{0.964pt}{0.350pt}}
\put(531,545){\rule[-0.175pt]{1.686pt}{0.350pt}}
\put(538,544){\rule[-0.175pt]{0.964pt}{0.350pt}}
\put(542,543){\rule[-0.175pt]{1.686pt}{0.350pt}}
\put(549,542){\rule[-0.175pt]{0.964pt}{0.350pt}}
\put(553,541){\rule[-0.175pt]{1.686pt}{0.350pt}}
\put(560,540){\rule[-0.175pt]{1.686pt}{0.350pt}}
\put(567,539){\rule[-0.175pt]{1.927pt}{0.350pt}}
\put(575,538){\rule[-0.175pt]{1.686pt}{0.350pt}}
\put(582,537){\rule[-0.175pt]{1.686pt}{0.350pt}}
\put(589,536){\rule[-0.175pt]{2.650pt}{0.350pt}}
\put(600,535){\rule[-0.175pt]{1.686pt}{0.350pt}}
\put(607,534){\rule[-0.175pt]{2.650pt}{0.350pt}}
\put(618,533){\rule[-0.175pt]{1.927pt}{0.350pt}}
\put(626,532){\rule[-0.175pt]{2.650pt}{0.350pt}}
\put(637,531){\rule[-0.175pt]{2.650pt}{0.350pt}}
\put(648,530){\rule[-0.175pt]{3.373pt}{0.350pt}}
\put(662,529){\rule[-0.175pt]{2.650pt}{0.350pt}}
\put(673,528){\rule[-0.175pt]{3.613pt}{0.350pt}}
\put(688,527){\rule[-0.175pt]{3.373pt}{0.350pt}}
\put(702,526){\rule[-0.175pt]{3.613pt}{0.350pt}}
\put(717,525){\rule[-0.175pt]{3.613pt}{0.350pt}}
\put(732,524){\rule[-0.175pt]{3.373pt}{0.350pt}}
\put(746,523){\rule[-0.175pt]{4.577pt}{0.350pt}}
\put(765,522){\rule[-0.175pt]{5.059pt}{0.350pt}}
\put(786,521){\rule[-0.175pt]{5.300pt}{0.350pt}}
\put(808,520){\rule[-0.175pt]{5.300pt}{0.350pt}}
\put(830,519){\rule[-0.175pt]{5.300pt}{0.350pt}}
\put(852,518){\rule[-0.175pt]{6.986pt}{0.350pt}}
\put(881,517){\rule[-0.175pt]{7.227pt}{0.350pt}}
\put(911,516){\rule[-0.175pt]{6.986pt}{0.350pt}}
\put(940,515){\rule[-0.175pt]{7.950pt}{0.350pt}}
\put(973,514){\rule[-0.175pt]{7.950pt}{0.350pt}}
\put(1006,513){\rule[-0.175pt]{7.950pt}{0.350pt}}
\put(1039,512){\rule[-0.175pt]{7.709pt}{0.350pt}}
\put(1071,511){\rule[-0.175pt]{7.950pt}{0.350pt}}
\put(1104,510){\rule[-0.175pt]{6.986pt}{0.350pt}}
\put(1133,509){\rule[-0.175pt]{7.227pt}{0.350pt}}
\put(1163,508){\rule[-0.175pt]{6.986pt}{0.350pt}}
\put(1192,507){\rule[-0.175pt]{2.650pt}{0.350pt}}
\sbox{\plotpoint}{\rule[-0.350pt]{0.700pt}{0.700pt}}%
\put(1214,652){\makebox(0,0)[r]{Conjugate Gradient Method}}
\put(1236,652){\rule[-0.350pt]{15.899pt}{0.700pt}}
\put(264,728){\usebox{\plotpoint}}
\put(264,726){\usebox{\plotpoint}}
\put(265,724){\usebox{\plotpoint}}
\put(266,722){\usebox{\plotpoint}}
\put(267,721){\usebox{\plotpoint}}
\put(268,719){\usebox{\plotpoint}}
\put(269,718){\usebox{\plotpoint}}
\put(270,717){\usebox{\plotpoint}}
\put(271,710){\rule[-0.350pt]{0.700pt}{1.626pt}}
\put(272,703){\rule[-0.350pt]{0.700pt}{1.626pt}}
\put(273,696){\rule[-0.350pt]{0.700pt}{1.626pt}}
\put(274,690){\rule[-0.350pt]{0.700pt}{1.626pt}}
\put(275,686){\rule[-0.350pt]{0.700pt}{0.783pt}}
\put(276,683){\rule[-0.350pt]{0.700pt}{0.783pt}}
\put(277,680){\rule[-0.350pt]{0.700pt}{0.783pt}}
\put(278,677){\rule[-0.350pt]{0.700pt}{0.783pt}}
\put(279,672){\rule[-0.350pt]{0.700pt}{1.124pt}}
\put(280,667){\rule[-0.350pt]{0.700pt}{1.124pt}}
\put(281,663){\rule[-0.350pt]{0.700pt}{1.124pt}}
\put(282,660){\usebox{\plotpoint}}
\put(283,658){\usebox{\plotpoint}}
\put(284,656){\usebox{\plotpoint}}
\put(285,654){\usebox{\plotpoint}}
\put(286,651){\usebox{\plotpoint}}
\put(287,649){\usebox{\plotpoint}}
\put(288,647){\usebox{\plotpoint}}
\put(289,645){\usebox{\plotpoint}}
\put(290,643){\usebox{\plotpoint}}
\put(291,642){\usebox{\plotpoint}}
\put(292,641){\usebox{\plotpoint}}
\put(293,641){\usebox{\plotpoint}}
\put(293,641){\usebox{\plotpoint}}
\put(294,640){\usebox{\plotpoint}}
\put(295,639){\usebox{\plotpoint}}
\put(297,638){\rule[-0.350pt]{0.964pt}{0.700pt}}
\put(301,637){\usebox{\plotpoint}}
\put(302,636){\usebox{\plotpoint}}
\put(303,635){\usebox{\plotpoint}}
\put(304,632){\usebox{\plotpoint}}
\put(305,631){\usebox{\plotpoint}}
\put(306,629){\usebox{\plotpoint}}
\put(307,628){\usebox{\plotpoint}}
\put(308,623){\rule[-0.350pt]{0.700pt}{1.124pt}}
\put(309,618){\rule[-0.350pt]{0.700pt}{1.124pt}}
\put(310,614){\rule[-0.350pt]{0.700pt}{1.124pt}}
\put(311,606){\rule[-0.350pt]{0.700pt}{1.807pt}}
\put(312,599){\rule[-0.350pt]{0.700pt}{1.807pt}}
\put(313,591){\rule[-0.350pt]{0.700pt}{1.807pt}}
\put(314,584){\rule[-0.350pt]{0.700pt}{1.807pt}}
\put(315,580){\rule[-0.350pt]{0.700pt}{0.783pt}}
\put(316,577){\rule[-0.350pt]{0.700pt}{0.783pt}}
\put(317,574){\rule[-0.350pt]{0.700pt}{0.783pt}}
\put(318,571){\rule[-0.350pt]{0.700pt}{0.783pt}}
\put(319,565){\rule[-0.350pt]{0.700pt}{1.445pt}}
\put(320,559){\rule[-0.350pt]{0.700pt}{1.445pt}}
\put(321,553){\rule[-0.350pt]{0.700pt}{1.445pt}}
\put(322,547){\rule[-0.350pt]{0.700pt}{1.385pt}}
\put(323,541){\rule[-0.350pt]{0.700pt}{1.385pt}}
\put(324,535){\rule[-0.350pt]{0.700pt}{1.385pt}}
\put(325,530){\rule[-0.350pt]{0.700pt}{1.385pt}}
\put(326,526){\rule[-0.350pt]{0.700pt}{0.843pt}}
\put(327,523){\rule[-0.350pt]{0.700pt}{0.843pt}}
\put(328,519){\rule[-0.350pt]{0.700pt}{0.843pt}}
\put(329,516){\rule[-0.350pt]{0.700pt}{0.843pt}}
\put(330,510){\rule[-0.350pt]{0.700pt}{1.285pt}}
\put(331,505){\rule[-0.350pt]{0.700pt}{1.285pt}}
\put(332,500){\rule[-0.350pt]{0.700pt}{1.285pt}}
\put(333,495){\rule[-0.350pt]{0.700pt}{1.204pt}}
\put(334,490){\rule[-0.350pt]{0.700pt}{1.204pt}}
\put(335,485){\rule[-0.350pt]{0.700pt}{1.204pt}}
\put(336,480){\rule[-0.350pt]{0.700pt}{1.204pt}}
\put(337,475){\rule[-0.350pt]{0.700pt}{1.084pt}}
\put(338,471){\rule[-0.350pt]{0.700pt}{1.084pt}}
\put(339,466){\rule[-0.350pt]{0.700pt}{1.084pt}}
\put(340,462){\rule[-0.350pt]{0.700pt}{1.084pt}}
\put(341,458){\rule[-0.350pt]{0.700pt}{0.803pt}}
\put(342,455){\rule[-0.350pt]{0.700pt}{0.803pt}}
\put(343,452){\rule[-0.350pt]{0.700pt}{0.803pt}}
\put(344,448){\rule[-0.350pt]{0.700pt}{0.964pt}}
\put(345,444){\rule[-0.350pt]{0.700pt}{0.964pt}}
\put(346,440){\rule[-0.350pt]{0.700pt}{0.964pt}}
\put(347,436){\rule[-0.350pt]{0.700pt}{0.964pt}}
\put(348,432){\rule[-0.350pt]{0.700pt}{0.843pt}}
\put(349,429){\rule[-0.350pt]{0.700pt}{0.843pt}}
\put(350,425){\rule[-0.350pt]{0.700pt}{0.843pt}}
\put(351,422){\rule[-0.350pt]{0.700pt}{0.843pt}}
\put(352,419){\usebox{\plotpoint}}
\put(353,416){\usebox{\plotpoint}}
\put(354,414){\usebox{\plotpoint}}
\put(355,412){\usebox{\plotpoint}}
\put(356,410){\usebox{\plotpoint}}
\put(357,408){\usebox{\plotpoint}}
\put(358,406){\usebox{\plotpoint}}
\put(359,404){\usebox{\plotpoint}}
\put(360,403){\usebox{\plotpoint}}
\put(361,401){\usebox{\plotpoint}}
\put(362,400){\usebox{\plotpoint}}
\put(363,398){\usebox{\plotpoint}}
\put(364,396){\usebox{\plotpoint}}
\put(365,395){\usebox{\plotpoint}}
\put(366,395){\usebox{\plotpoint}}
\put(366,395){\usebox{\plotpoint}}
\put(367,394){\usebox{\plotpoint}}
\put(368,393){\usebox{\plotpoint}}
\put(369,392){\usebox{\plotpoint}}
\put(370,389){\usebox{\plotpoint}}
\put(371,388){\usebox{\plotpoint}}
\put(372,386){\usebox{\plotpoint}}
\put(373,385){\usebox{\plotpoint}}
\put(374,383){\usebox{\plotpoint}}
\put(375,381){\usebox{\plotpoint}}
\put(376,380){\usebox{\plotpoint}}
\put(377,378){\usebox{\plotpoint}}
\put(378,377){\usebox{\plotpoint}}
\put(379,375){\usebox{\plotpoint}}
\put(380,374){\usebox{\plotpoint}}
\put(381,372){\usebox{\plotpoint}}
\put(382,370){\usebox{\plotpoint}}
\put(383,368){\usebox{\plotpoint}}
\put(384,367){\usebox{\plotpoint}}
\put(385,365){\usebox{\plotpoint}}
\put(386,363){\usebox{\plotpoint}}
\put(387,362){\usebox{\plotpoint}}
\put(388,359){\usebox{\plotpoint}}
\put(389,357){\usebox{\plotpoint}}
\put(390,355){\usebox{\plotpoint}}
\put(391,353){\usebox{\plotpoint}}
\put(392,351){\usebox{\plotpoint}}
\put(393,350){\usebox{\plotpoint}}
\put(394,349){\usebox{\plotpoint}}
\put(395,348){\usebox{\plotpoint}}
\put(396,344){\rule[-0.350pt]{0.700pt}{0.803pt}}
\put(397,341){\rule[-0.350pt]{0.700pt}{0.803pt}}
\put(398,338){\rule[-0.350pt]{0.700pt}{0.803pt}}
\put(399,334){\rule[-0.350pt]{0.700pt}{0.843pt}}
\put(400,331){\rule[-0.350pt]{0.700pt}{0.843pt}}
\put(401,327){\rule[-0.350pt]{0.700pt}{0.843pt}}
\put(402,324){\rule[-0.350pt]{0.700pt}{0.843pt}}
\put(403,321){\usebox{\plotpoint}}
\put(404,319){\usebox{\plotpoint}}
\put(405,317){\usebox{\plotpoint}}
\put(406,315){\usebox{\plotpoint}}
\put(407,313){\usebox{\plotpoint}}
\put(408,311){\usebox{\plotpoint}}
\put(409,310){\usebox{\plotpoint}}
\put(410,307){\usebox{\plotpoint}}
\put(411,305){\usebox{\plotpoint}}
\put(412,302){\usebox{\plotpoint}}
\put(413,300){\usebox{\plotpoint}}
\put(414,295){\rule[-0.350pt]{0.700pt}{1.124pt}}
\put(415,290){\rule[-0.350pt]{0.700pt}{1.124pt}}
\put(416,286){\rule[-0.350pt]{0.700pt}{1.124pt}}
\put(417,279){\rule[-0.350pt]{0.700pt}{1.626pt}}
\put(418,272){\rule[-0.350pt]{0.700pt}{1.626pt}}
\put(419,265){\rule[-0.350pt]{0.700pt}{1.626pt}}
\put(420,259){\rule[-0.350pt]{0.700pt}{1.626pt}}
\put(421,252){\rule[-0.350pt]{0.700pt}{1.626pt}}
\put(422,245){\rule[-0.350pt]{0.700pt}{1.626pt}}
\put(423,238){\rule[-0.350pt]{0.700pt}{1.626pt}}
\put(424,232){\rule[-0.350pt]{0.700pt}{1.626pt}}
\put(425,226){\rule[-0.350pt]{0.700pt}{1.365pt}}
\put(426,220){\rule[-0.350pt]{0.700pt}{1.365pt}}
\put(427,215){\rule[-0.350pt]{0.700pt}{1.365pt}}
\put(428,206){\rule[-0.350pt]{0.700pt}{1.987pt}}
\put(429,198){\rule[-0.350pt]{0.700pt}{1.987pt}}
\put(430,190){\rule[-0.350pt]{0.700pt}{1.987pt}}
\put(431,182){\rule[-0.350pt]{0.700pt}{1.987pt}}
\put(432,178){\rule[-0.350pt]{0.700pt}{0.964pt}}
\put(433,174){\rule[-0.350pt]{0.700pt}{0.964pt}}
\put(434,170){\rule[-0.350pt]{0.700pt}{0.964pt}}
\put(435,166){\rule[-0.350pt]{0.700pt}{0.964pt}}
\put(436,166){\rule[-0.350pt]{0.700pt}{1.285pt}}
\put(437,171){\rule[-0.350pt]{0.700pt}{1.285pt}}
\put(438,176){\rule[-0.350pt]{0.700pt}{1.285pt}}
\put(439,181){\usebox{\plotpoint}}
\end{picture}

%% file: eig-V100,3.tex
\setlength{\unitlength}{0.240900pt}
\ifx\plotpoint\undefined\newsavebox{\plotpoint}\fi
\sbox{\plotpoint}{\rule[-0.175pt]{0.350pt}{0.350pt}}%
\begin{picture}(1424,900)(0,0)
\tenrm
\sbox{\plotpoint}{\rule[-0.175pt]{0.350pt}{0.350pt}}%
\put(264,158){\rule[-0.175pt]{0.350pt}{151.526pt}}
\put(264,158){\rule[-0.175pt]{4.818pt}{0.350pt}}
\put(242,158){\makebox(0,0)[r]{$40$}}
\put(1340,158){\rule[-0.175pt]{4.818pt}{0.350pt}}
\put(264,248){\rule[-0.175pt]{4.818pt}{0.350pt}}
\put(242,248){\makebox(0,0)[r]{$50$}}
\put(1340,248){\rule[-0.175pt]{4.818pt}{0.350pt}}
\put(264,338){\rule[-0.175pt]{4.818pt}{0.350pt}}
\put(242,338){\makebox(0,0)[r]{$60$}}
\put(1340,338){\rule[-0.175pt]{4.818pt}{0.350pt}}
\put(264,428){\rule[-0.175pt]{4.818pt}{0.350pt}}
\put(242,428){\makebox(0,0)[r]{$70$}}
\put(1340,428){\rule[-0.175pt]{4.818pt}{0.350pt}}
\put(264,517){\rule[-0.175pt]{4.818pt}{0.350pt}}
\put(242,517){\makebox(0,0)[r]{$80$}}
\put(1340,517){\rule[-0.175pt]{4.818pt}{0.350pt}}
\put(264,607){\rule[-0.175pt]{4.818pt}{0.350pt}}
\put(242,607){\makebox(0,0)[r]{$90$}}
\put(1340,607){\rule[-0.175pt]{4.818pt}{0.350pt}}
\put(264,697){\rule[-0.175pt]{4.818pt}{0.350pt}}
\put(242,697){\makebox(0,0)[r]{$100$}}
\put(1340,697){\rule[-0.175pt]{4.818pt}{0.350pt}}
\put(264,787){\rule[-0.175pt]{4.818pt}{0.350pt}}
\put(242,787){\makebox(0,0)[r]{$110$}}
\put(1340,787){\rule[-0.175pt]{4.818pt}{0.350pt}}
\put(264,158){\rule[-0.175pt]{0.350pt}{4.818pt}}
\put(264,113){\makebox(0,0){$0$}}
\put(264,767){\rule[-0.175pt]{0.350pt}{4.818pt}}
\put(401,158){\rule[-0.175pt]{0.350pt}{4.818pt}}
\put(401,113){\makebox(0,0){$5$}}
\put(401,767){\rule[-0.175pt]{0.350pt}{4.818pt}}
\put(538,158){\rule[-0.175pt]{0.350pt}{4.818pt}}
\put(538,113){\makebox(0,0){$10$}}
\put(538,767){\rule[-0.175pt]{0.350pt}{4.818pt}}
\put(675,158){\rule[-0.175pt]{0.350pt}{4.818pt}}
\put(675,113){\makebox(0,0){$15$}}
\put(675,767){\rule[-0.175pt]{0.350pt}{4.818pt}}
\put(812,158){\rule[-0.175pt]{0.350pt}{4.818pt}}
\put(812,113){\makebox(0,0){$20$}}
\put(812,767){\rule[-0.175pt]{0.350pt}{4.818pt}}
\put(949,158){\rule[-0.175pt]{0.350pt}{4.818pt}}
\put(949,113){\makebox(0,0){$25$}}
\put(949,767){\rule[-0.175pt]{0.350pt}{4.818pt}}
\put(1086,158){\rule[-0.175pt]{0.350pt}{4.818pt}}
\put(1086,113){\makebox(0,0){$30$}}
\put(1086,767){\rule[-0.175pt]{0.350pt}{4.818pt}}
\put(1223,158){\rule[-0.175pt]{0.350pt}{4.818pt}}
\put(1223,113){\makebox(0,0){$35$}}
\put(1223,767){\rule[-0.175pt]{0.350pt}{4.818pt}}
\put(1360,158){\rule[-0.175pt]{0.350pt}{4.818pt}}
\put(1360,113){\makebox(0,0){$40$}}
\put(1360,767){\rule[-0.175pt]{0.350pt}{4.818pt}}
\put(264,158){\rule[-0.175pt]{264.026pt}{0.350pt}}
\put(1360,158){\rule[-0.175pt]{0.350pt}{151.526pt}}
\put(264,787){\rule[-0.175pt]{264.026pt}{0.350pt}}
\put(-43,472){\makebox(0,0)[l]{\shortstack{${\displaystyle{(p_i^\T\!Ap_i)_{jj}\atop1\le j\le3}}$}}}
\put(812,68){\makebox(0,0){Step $i$}}
\put(264,158){\rule[-0.175pt]{0.350pt}{151.526pt}}
\put(264,395){\usebox{\plotpoint}}
\put(264,395){\rule[-0.175pt]{0.350pt}{0.633pt}}
\put(265,397){\rule[-0.175pt]{0.350pt}{0.633pt}}
\put(266,400){\rule[-0.175pt]{0.350pt}{0.633pt}}
\put(267,402){\rule[-0.175pt]{0.350pt}{0.633pt}}
\put(268,405){\rule[-0.175pt]{0.350pt}{0.633pt}}
\put(269,408){\rule[-0.175pt]{0.350pt}{0.633pt}}
\put(270,410){\rule[-0.175pt]{0.350pt}{0.633pt}}
\put(271,413){\rule[-0.175pt]{0.350pt}{0.633pt}}
\put(272,416){\rule[-0.175pt]{0.350pt}{0.633pt}}
\put(273,418){\rule[-0.175pt]{0.350pt}{0.633pt}}
\put(274,421){\rule[-0.175pt]{0.350pt}{0.633pt}}
\put(275,423){\rule[-0.175pt]{0.350pt}{0.633pt}}
\put(276,426){\rule[-0.175pt]{0.350pt}{0.633pt}}
\put(277,429){\rule[-0.175pt]{0.350pt}{0.633pt}}
\put(278,431){\rule[-0.175pt]{0.350pt}{0.633pt}}
\put(279,434){\rule[-0.175pt]{0.350pt}{0.633pt}}
\put(280,437){\rule[-0.175pt]{0.350pt}{0.633pt}}
\put(281,439){\rule[-0.175pt]{0.350pt}{0.633pt}}
\put(282,442){\rule[-0.175pt]{0.350pt}{0.633pt}}
\put(283,444){\rule[-0.175pt]{0.350pt}{0.633pt}}
\put(284,447){\rule[-0.175pt]{0.350pt}{0.633pt}}
\put(285,450){\rule[-0.175pt]{0.350pt}{0.633pt}}
\put(286,452){\rule[-0.175pt]{0.350pt}{0.633pt}}
\put(287,455){\rule[-0.175pt]{0.350pt}{0.633pt}}
\put(288,458){\rule[-0.175pt]{0.350pt}{0.633pt}}
\put(289,460){\rule[-0.175pt]{0.350pt}{0.633pt}}
\put(290,463){\rule[-0.175pt]{0.350pt}{0.633pt}}
\put(291,466){\rule[-0.175pt]{0.350pt}{0.379pt}}
\put(292,467){\rule[-0.175pt]{0.350pt}{0.379pt}}
\put(293,469){\rule[-0.175pt]{0.350pt}{0.379pt}}
\put(294,470){\rule[-0.175pt]{0.350pt}{0.379pt}}
\put(295,472){\rule[-0.175pt]{0.350pt}{0.379pt}}
\put(296,473){\rule[-0.175pt]{0.350pt}{0.379pt}}
\put(297,475){\rule[-0.175pt]{0.350pt}{0.379pt}}
\put(298,477){\rule[-0.175pt]{0.350pt}{0.379pt}}
\put(299,478){\rule[-0.175pt]{0.350pt}{0.379pt}}
\put(300,480){\rule[-0.175pt]{0.350pt}{0.379pt}}
\put(301,481){\rule[-0.175pt]{0.350pt}{0.379pt}}
\put(302,483){\rule[-0.175pt]{0.350pt}{0.379pt}}
\put(303,484){\rule[-0.175pt]{0.350pt}{0.379pt}}
\put(304,486){\rule[-0.175pt]{0.350pt}{0.379pt}}
\put(305,488){\rule[-0.175pt]{0.350pt}{0.379pt}}
\put(306,489){\rule[-0.175pt]{0.350pt}{0.379pt}}
\put(307,491){\rule[-0.175pt]{0.350pt}{0.379pt}}
\put(308,492){\rule[-0.175pt]{0.350pt}{0.379pt}}
\put(309,494){\rule[-0.175pt]{0.350pt}{0.379pt}}
\put(310,495){\rule[-0.175pt]{0.350pt}{0.379pt}}
\put(311,497){\rule[-0.175pt]{0.350pt}{0.379pt}}
\put(312,499){\rule[-0.175pt]{0.350pt}{0.379pt}}
\put(313,500){\rule[-0.175pt]{0.350pt}{0.379pt}}
\put(314,502){\rule[-0.175pt]{0.350pt}{0.379pt}}
\put(315,503){\rule[-0.175pt]{0.350pt}{0.379pt}}
\put(316,505){\rule[-0.175pt]{0.350pt}{0.379pt}}
\put(317,506){\rule[-0.175pt]{0.350pt}{0.379pt}}
\put(318,508){\rule[-0.175pt]{0.350pt}{0.378pt}}
\put(319,507){\rule[-0.175pt]{0.350pt}{0.633pt}}
\put(320,504){\rule[-0.175pt]{0.350pt}{0.633pt}}
\put(321,502){\rule[-0.175pt]{0.350pt}{0.633pt}}
\put(322,499){\rule[-0.175pt]{0.350pt}{0.633pt}}
\put(323,496){\rule[-0.175pt]{0.350pt}{0.633pt}}
\put(324,494){\rule[-0.175pt]{0.350pt}{0.633pt}}
\put(325,491){\rule[-0.175pt]{0.350pt}{0.633pt}}
\put(326,488){\rule[-0.175pt]{0.350pt}{0.633pt}}
\put(327,486){\rule[-0.175pt]{0.350pt}{0.633pt}}
\put(328,483){\rule[-0.175pt]{0.350pt}{0.633pt}}
\put(329,481){\rule[-0.175pt]{0.350pt}{0.633pt}}
\put(330,478){\rule[-0.175pt]{0.350pt}{0.633pt}}
\put(331,475){\rule[-0.175pt]{0.350pt}{0.633pt}}
\put(332,473){\rule[-0.175pt]{0.350pt}{0.633pt}}
\put(333,470){\rule[-0.175pt]{0.350pt}{0.633pt}}
\put(334,467){\rule[-0.175pt]{0.350pt}{0.633pt}}
\put(335,465){\rule[-0.175pt]{0.350pt}{0.633pt}}
\put(336,462){\rule[-0.175pt]{0.350pt}{0.633pt}}
\put(337,460){\rule[-0.175pt]{0.350pt}{0.633pt}}
\put(338,457){\rule[-0.175pt]{0.350pt}{0.633pt}}
\put(339,454){\rule[-0.175pt]{0.350pt}{0.633pt}}
\put(340,452){\rule[-0.175pt]{0.350pt}{0.633pt}}
\put(341,449){\rule[-0.175pt]{0.350pt}{0.633pt}}
\put(342,446){\rule[-0.175pt]{0.350pt}{0.633pt}}
\put(343,444){\rule[-0.175pt]{0.350pt}{0.633pt}}
\put(344,441){\rule[-0.175pt]{0.350pt}{0.633pt}}
\put(345,439){\rule[-0.175pt]{0.350pt}{0.633pt}}
\put(346,439){\rule[-0.175pt]{0.350pt}{0.619pt}}
\put(347,441){\rule[-0.175pt]{0.350pt}{0.619pt}}
\put(348,444){\rule[-0.175pt]{0.350pt}{0.619pt}}
\put(349,446){\rule[-0.175pt]{0.350pt}{0.619pt}}
\put(350,449){\rule[-0.175pt]{0.350pt}{0.619pt}}
\put(351,451){\rule[-0.175pt]{0.350pt}{0.619pt}}
\put(352,454){\rule[-0.175pt]{0.350pt}{0.619pt}}
\put(353,457){\rule[-0.175pt]{0.350pt}{0.619pt}}
\put(354,459){\rule[-0.175pt]{0.350pt}{0.619pt}}
\put(355,462){\rule[-0.175pt]{0.350pt}{0.619pt}}
\put(356,464){\rule[-0.175pt]{0.350pt}{0.619pt}}
\put(357,467){\rule[-0.175pt]{0.350pt}{0.619pt}}
\put(358,469){\rule[-0.175pt]{0.350pt}{0.619pt}}
\put(359,472){\rule[-0.175pt]{0.350pt}{0.619pt}}
\put(360,475){\rule[-0.175pt]{0.350pt}{0.619pt}}
\put(361,477){\rule[-0.175pt]{0.350pt}{0.619pt}}
\put(362,480){\rule[-0.175pt]{0.350pt}{0.619pt}}
\put(363,482){\rule[-0.175pt]{0.350pt}{0.619pt}}
\put(364,485){\rule[-0.175pt]{0.350pt}{0.619pt}}
\put(365,487){\rule[-0.175pt]{0.350pt}{0.619pt}}
\put(366,490){\rule[-0.175pt]{0.350pt}{0.619pt}}
\put(367,493){\rule[-0.175pt]{0.350pt}{0.619pt}}
\put(368,495){\rule[-0.175pt]{0.350pt}{0.619pt}}
\put(369,498){\rule[-0.175pt]{0.350pt}{0.619pt}}
\put(370,500){\rule[-0.175pt]{0.350pt}{0.619pt}}
\put(371,503){\rule[-0.175pt]{0.350pt}{0.619pt}}
\put(372,505){\rule[-0.175pt]{0.350pt}{0.619pt}}
\put(373,508){\rule[-0.175pt]{0.350pt}{0.619pt}}
\put(374,511){\rule[-0.175pt]{0.350pt}{0.598pt}}
\put(375,513){\rule[-0.175pt]{0.350pt}{0.598pt}}
\put(376,515){\rule[-0.175pt]{0.350pt}{0.598pt}}
\put(377,518){\rule[-0.175pt]{0.350pt}{0.598pt}}
\put(378,520){\rule[-0.175pt]{0.350pt}{0.598pt}}
\put(379,523){\rule[-0.175pt]{0.350pt}{0.598pt}}
\put(380,525){\rule[-0.175pt]{0.350pt}{0.598pt}}
\put(381,528){\rule[-0.175pt]{0.350pt}{0.598pt}}
\put(382,530){\rule[-0.175pt]{0.350pt}{0.598pt}}
\put(383,533){\rule[-0.175pt]{0.350pt}{0.598pt}}
\put(384,535){\rule[-0.175pt]{0.350pt}{0.598pt}}
\put(385,538){\rule[-0.175pt]{0.350pt}{0.598pt}}
\put(386,540){\rule[-0.175pt]{0.350pt}{0.598pt}}
\put(387,543){\rule[-0.175pt]{0.350pt}{0.598pt}}
\put(388,545){\rule[-0.175pt]{0.350pt}{0.598pt}}
\put(389,548){\rule[-0.175pt]{0.350pt}{0.598pt}}
\put(390,550){\rule[-0.175pt]{0.350pt}{0.598pt}}
\put(391,553){\rule[-0.175pt]{0.350pt}{0.598pt}}
\put(392,555){\rule[-0.175pt]{0.350pt}{0.598pt}}
\put(393,558){\rule[-0.175pt]{0.350pt}{0.598pt}}
\put(394,560){\rule[-0.175pt]{0.350pt}{0.598pt}}
\put(395,563){\rule[-0.175pt]{0.350pt}{0.598pt}}
\put(396,565){\rule[-0.175pt]{0.350pt}{0.598pt}}
\put(397,568){\rule[-0.175pt]{0.350pt}{0.598pt}}
\put(398,570){\rule[-0.175pt]{0.350pt}{0.598pt}}
\put(399,573){\rule[-0.175pt]{0.350pt}{0.598pt}}
\put(400,575){\rule[-0.175pt]{0.350pt}{0.598pt}}
\put(401,578){\usebox{\plotpoint}}
\put(402,579){\usebox{\plotpoint}}
\put(403,580){\usebox{\plotpoint}}
\put(404,581){\usebox{\plotpoint}}
\put(406,582){\usebox{\plotpoint}}
\put(407,583){\usebox{\plotpoint}}
\put(408,584){\usebox{\plotpoint}}
\put(409,585){\usebox{\plotpoint}}
\put(411,586){\usebox{\plotpoint}}
\put(412,587){\usebox{\plotpoint}}
\put(413,588){\usebox{\plotpoint}}
\put(415,589){\usebox{\plotpoint}}
\put(416,590){\usebox{\plotpoint}}
\put(417,591){\usebox{\plotpoint}}
\put(418,592){\usebox{\plotpoint}}
\put(420,593){\usebox{\plotpoint}}
\put(421,594){\usebox{\plotpoint}}
\put(422,595){\usebox{\plotpoint}}
\put(424,596){\usebox{\plotpoint}}
\put(425,597){\usebox{\plotpoint}}
\put(426,598){\usebox{\plotpoint}}
\put(427,599){\rule[-0.175pt]{0.422pt}{0.350pt}}
\put(429,600){\rule[-0.175pt]{0.422pt}{0.350pt}}
\put(431,601){\rule[-0.175pt]{0.422pt}{0.350pt}}
\put(433,602){\rule[-0.175pt]{0.422pt}{0.350pt}}
\put(435,603){\rule[-0.175pt]{0.422pt}{0.350pt}}
\put(436,604){\rule[-0.175pt]{0.422pt}{0.350pt}}
\put(438,605){\rule[-0.175pt]{0.422pt}{0.350pt}}
\put(440,606){\rule[-0.175pt]{0.422pt}{0.350pt}}
\put(442,607){\rule[-0.175pt]{0.422pt}{0.350pt}}
\put(443,608){\rule[-0.175pt]{0.422pt}{0.350pt}}
\put(445,609){\rule[-0.175pt]{0.422pt}{0.350pt}}
\put(447,610){\rule[-0.175pt]{0.422pt}{0.350pt}}
\put(449,611){\rule[-0.175pt]{0.422pt}{0.350pt}}
\put(450,612){\rule[-0.175pt]{0.422pt}{0.350pt}}
\put(452,613){\rule[-0.175pt]{0.422pt}{0.350pt}}
\put(454,614){\rule[-0.175pt]{0.422pt}{0.350pt}}
\put(456,615){\rule[-0.175pt]{3.252pt}{0.350pt}}
\put(469,616){\rule[-0.175pt]{3.252pt}{0.350pt}}
\put(483,617){\rule[-0.175pt]{2.248pt}{0.350pt}}
\put(492,618){\rule[-0.175pt]{2.248pt}{0.350pt}}
\put(501,619){\rule[-0.175pt]{2.248pt}{0.350pt}}
\put(511,620){\rule[-0.175pt]{1.626pt}{0.350pt}}
\put(517,621){\rule[-0.175pt]{1.626pt}{0.350pt}}
\put(524,622){\rule[-0.175pt]{1.626pt}{0.350pt}}
\put(531,623){\rule[-0.175pt]{1.626pt}{0.350pt}}
\put(538,624){\rule[-0.175pt]{3.252pt}{0.350pt}}
\put(551,625){\rule[-0.175pt]{3.252pt}{0.350pt}}
\put(565,626){\rule[-0.175pt]{1.686pt}{0.350pt}}
\put(572,627){\rule[-0.175pt]{1.686pt}{0.350pt}}
\put(579,628){\rule[-0.175pt]{1.686pt}{0.350pt}}
\put(586,629){\rule[-0.175pt]{1.686pt}{0.350pt}}
\put(593,630){\usebox{\plotpoint}}
\put(594,631){\usebox{\plotpoint}}
\put(595,632){\usebox{\plotpoint}}
\put(597,633){\usebox{\plotpoint}}
\put(598,634){\usebox{\plotpoint}}
\put(599,635){\usebox{\plotpoint}}
\put(601,636){\usebox{\plotpoint}}
\put(602,637){\usebox{\plotpoint}}
\put(603,638){\usebox{\plotpoint}}
\put(605,639){\usebox{\plotpoint}}
\put(606,640){\usebox{\plotpoint}}
\put(607,641){\usebox{\plotpoint}}
\put(609,642){\usebox{\plotpoint}}
\put(610,643){\usebox{\plotpoint}}
\put(611,644){\usebox{\plotpoint}}
\put(613,645){\usebox{\plotpoint}}
\put(614,646){\usebox{\plotpoint}}
\put(615,647){\usebox{\plotpoint}}
\put(617,648){\usebox{\plotpoint}}
\put(618,649){\usebox{\plotpoint}}
\put(619,650){\usebox{\plotpoint}}
\put(620,650){\usebox{\plotpoint}}
\put(621,651){\usebox{\plotpoint}}
\put(622,652){\usebox{\plotpoint}}
\put(623,653){\usebox{\plotpoint}}
\put(624,654){\usebox{\plotpoint}}
\put(625,655){\usebox{\plotpoint}}
\put(626,656){\usebox{\plotpoint}}
\put(627,657){\usebox{\plotpoint}}
\put(628,658){\usebox{\plotpoint}}
\put(629,659){\usebox{\plotpoint}}
\put(630,660){\usebox{\plotpoint}}
\put(631,661){\usebox{\plotpoint}}
\put(632,662){\usebox{\plotpoint}}
\put(633,663){\usebox{\plotpoint}}
\put(634,664){\usebox{\plotpoint}}
\put(635,666){\usebox{\plotpoint}}
\put(636,667){\usebox{\plotpoint}}
\put(637,668){\usebox{\plotpoint}}
\put(638,669){\usebox{\plotpoint}}
\put(639,670){\usebox{\plotpoint}}
\put(640,671){\usebox{\plotpoint}}
\put(641,672){\usebox{\plotpoint}}
\put(642,673){\usebox{\plotpoint}}
\put(643,674){\usebox{\plotpoint}}
\put(644,675){\usebox{\plotpoint}}
\put(645,676){\usebox{\plotpoint}}
\put(646,677){\usebox{\plotpoint}}
\put(647,678){\usebox{\plotpoint}}
\put(648,679){\usebox{\plotpoint}}
\put(648,680){\rule[-0.175pt]{0.542pt}{0.350pt}}
\put(650,681){\rule[-0.175pt]{0.542pt}{0.350pt}}
\put(652,682){\rule[-0.175pt]{0.542pt}{0.350pt}}
\put(654,683){\rule[-0.175pt]{0.542pt}{0.350pt}}
\put(657,684){\rule[-0.175pt]{0.542pt}{0.350pt}}
\put(659,685){\rule[-0.175pt]{0.542pt}{0.350pt}}
\put(661,686){\rule[-0.175pt]{0.542pt}{0.350pt}}
\put(663,687){\rule[-0.175pt]{0.542pt}{0.350pt}}
\put(666,688){\rule[-0.175pt]{0.542pt}{0.350pt}}
\put(668,689){\rule[-0.175pt]{0.542pt}{0.350pt}}
\put(670,690){\rule[-0.175pt]{0.542pt}{0.350pt}}
\put(672,691){\rule[-0.175pt]{0.542pt}{0.350pt}}
\put(675,692){\rule[-0.175pt]{3.252pt}{0.350pt}}
\put(688,693){\rule[-0.175pt]{3.252pt}{0.350pt}}
\put(702,694){\rule[-0.175pt]{6.745pt}{0.350pt}}
\put(730,695){\rule[-0.175pt]{6.504pt}{0.350pt}}
\put(757,696){\rule[-0.175pt]{26.499pt}{0.350pt}}
\put(867,697){\rule[-0.175pt]{118.764pt}{0.350pt}}
\sbox{\plotpoint}{\rule[-0.350pt]{0.700pt}{0.700pt}}%
\put(264,229){\usebox{\plotpoint}}
\put(264,229){\usebox{\plotpoint}}
\put(265,230){\usebox{\plotpoint}}
\put(266,231){\usebox{\plotpoint}}
\put(267,232){\usebox{\plotpoint}}
\put(268,233){\usebox{\plotpoint}}
\put(269,234){\usebox{\plotpoint}}
\put(271,235){\usebox{\plotpoint}}
\put(272,236){\usebox{\plotpoint}}
\put(273,237){\usebox{\plotpoint}}
\put(274,238){\usebox{\plotpoint}}
\put(275,239){\usebox{\plotpoint}}
\put(276,240){\usebox{\plotpoint}}
\put(278,241){\usebox{\plotpoint}}
\put(279,242){\usebox{\plotpoint}}
\put(280,243){\usebox{\plotpoint}}
\put(281,244){\usebox{\plotpoint}}
\put(282,245){\usebox{\plotpoint}}
\put(283,246){\usebox{\plotpoint}}
\put(285,247){\usebox{\plotpoint}}
\put(286,248){\usebox{\plotpoint}}
\put(287,249){\usebox{\plotpoint}}
\put(288,250){\usebox{\plotpoint}}
\put(289,251){\usebox{\plotpoint}}
\put(291,252){\usebox{\plotpoint}}
\put(293,253){\usebox{\plotpoint}}
\put(296,254){\usebox{\plotpoint}}
\put(299,255){\usebox{\plotpoint}}
\put(302,256){\usebox{\plotpoint}}
\put(304,257){\usebox{\plotpoint}}
\put(307,258){\usebox{\plotpoint}}
\put(310,259){\usebox{\plotpoint}}
\put(313,260){\usebox{\plotpoint}}
\put(316,261){\usebox{\plotpoint}}
\put(318,262){\usebox{\plotpoint}}
\put(319,262){\rule[-0.350pt]{0.700pt}{1.811pt}}
\put(320,269){\rule[-0.350pt]{0.700pt}{1.811pt}}
\put(321,277){\rule[-0.350pt]{0.700pt}{1.811pt}}
\put(322,284){\rule[-0.350pt]{0.700pt}{1.811pt}}
\put(323,292){\rule[-0.350pt]{0.700pt}{1.811pt}}
\put(324,299){\rule[-0.350pt]{0.700pt}{1.811pt}}
\put(325,307){\rule[-0.350pt]{0.700pt}{1.811pt}}
\put(326,314){\rule[-0.350pt]{0.700pt}{1.811pt}}
\put(327,322){\rule[-0.350pt]{0.700pt}{1.811pt}}
\put(328,329){\rule[-0.350pt]{0.700pt}{1.811pt}}
\put(329,337){\rule[-0.350pt]{0.700pt}{1.811pt}}
\put(330,344){\rule[-0.350pt]{0.700pt}{1.811pt}}
\put(331,352){\rule[-0.350pt]{0.700pt}{1.811pt}}
\put(332,359){\rule[-0.350pt]{0.700pt}{1.811pt}}
\put(333,367){\rule[-0.350pt]{0.700pt}{1.811pt}}
\put(334,374){\rule[-0.350pt]{0.700pt}{1.811pt}}
\put(335,382){\rule[-0.350pt]{0.700pt}{1.811pt}}
\put(336,389){\rule[-0.350pt]{0.700pt}{1.811pt}}
\put(337,397){\rule[-0.350pt]{0.700pt}{1.811pt}}
\put(338,404){\rule[-0.350pt]{0.700pt}{1.811pt}}
\put(339,412){\rule[-0.350pt]{0.700pt}{1.811pt}}
\put(340,419){\rule[-0.350pt]{0.700pt}{1.811pt}}
\put(341,427){\rule[-0.350pt]{0.700pt}{1.811pt}}
\put(342,434){\rule[-0.350pt]{0.700pt}{1.811pt}}
\put(343,442){\rule[-0.350pt]{0.700pt}{1.811pt}}
\put(344,449){\rule[-0.350pt]{0.700pt}{1.811pt}}
\put(345,457){\rule[-0.350pt]{0.700pt}{1.811pt}}
\put(346,465){\rule[-0.350pt]{0.700pt}{1.540pt}}
\put(347,471){\rule[-0.350pt]{0.700pt}{1.540pt}}
\put(348,477){\rule[-0.350pt]{0.700pt}{1.540pt}}
\put(349,484){\rule[-0.350pt]{0.700pt}{1.540pt}}
\put(350,490){\rule[-0.350pt]{0.700pt}{1.540pt}}
\put(351,496){\rule[-0.350pt]{0.700pt}{1.540pt}}
\put(352,503){\rule[-0.350pt]{0.700pt}{1.540pt}}
\put(353,509){\rule[-0.350pt]{0.700pt}{1.540pt}}
\put(354,516){\rule[-0.350pt]{0.700pt}{1.540pt}}
\put(355,522){\rule[-0.350pt]{0.700pt}{1.540pt}}
\put(356,528){\rule[-0.350pt]{0.700pt}{1.540pt}}
\put(357,535){\rule[-0.350pt]{0.700pt}{1.540pt}}
\put(358,541){\rule[-0.350pt]{0.700pt}{1.540pt}}
\put(359,548){\rule[-0.350pt]{0.700pt}{1.540pt}}
\put(360,554){\rule[-0.350pt]{0.700pt}{1.540pt}}
\put(361,560){\rule[-0.350pt]{0.700pt}{1.540pt}}
\put(362,567){\rule[-0.350pt]{0.700pt}{1.540pt}}
\put(363,573){\rule[-0.350pt]{0.700pt}{1.540pt}}
\put(364,580){\rule[-0.350pt]{0.700pt}{1.540pt}}
\put(365,586){\rule[-0.350pt]{0.700pt}{1.540pt}}
\put(366,592){\rule[-0.350pt]{0.700pt}{1.540pt}}
\put(367,599){\rule[-0.350pt]{0.700pt}{1.540pt}}
\put(368,605){\rule[-0.350pt]{0.700pt}{1.540pt}}
\put(369,612){\rule[-0.350pt]{0.700pt}{1.540pt}}
\put(370,618){\rule[-0.350pt]{0.700pt}{1.540pt}}
\put(371,624){\rule[-0.350pt]{0.700pt}{1.540pt}}
\put(372,631){\rule[-0.350pt]{0.700pt}{1.540pt}}
\put(373,637){\rule[-0.350pt]{0.700pt}{1.540pt}}
\put(374,644){\usebox{\plotpoint}}
\put(376,645){\usebox{\plotpoint}}
\put(379,646){\usebox{\plotpoint}}
\put(382,647){\usebox{\plotpoint}}
\put(384,648){\usebox{\plotpoint}}
\put(387,649){\usebox{\plotpoint}}
\put(390,650){\usebox{\plotpoint}}
\put(392,651){\usebox{\plotpoint}}
\put(395,652){\usebox{\plotpoint}}
\put(398,653){\usebox{\plotpoint}}
\put(401,654){\usebox{\plotpoint}}
\put(402,655){\usebox{\plotpoint}}
\put(403,656){\usebox{\plotpoint}}
\put(405,657){\usebox{\plotpoint}}
\put(406,658){\usebox{\plotpoint}}
\put(407,659){\usebox{\plotpoint}}
\put(409,660){\usebox{\plotpoint}}
\put(410,661){\usebox{\plotpoint}}
\put(411,662){\usebox{\plotpoint}}
\put(413,663){\usebox{\plotpoint}}
\put(414,664){\usebox{\plotpoint}}
\put(415,665){\usebox{\plotpoint}}
\put(417,666){\usebox{\plotpoint}}
\put(418,667){\usebox{\plotpoint}}
\put(419,668){\usebox{\plotpoint}}
\put(421,669){\usebox{\plotpoint}}
\put(422,670){\usebox{\plotpoint}}
\put(423,671){\usebox{\plotpoint}}
\put(425,672){\usebox{\plotpoint}}
\put(426,673){\usebox{\plotpoint}}
\put(428,674){\rule[-0.350pt]{7.558pt}{0.700pt}}
\put(459,675){\rule[-0.350pt]{0.813pt}{0.700pt}}
\put(462,676){\rule[-0.350pt]{0.813pt}{0.700pt}}
\put(466,677){\rule[-0.350pt]{0.813pt}{0.700pt}}
\put(469,678){\rule[-0.350pt]{0.813pt}{0.700pt}}
\put(472,679){\rule[-0.350pt]{0.813pt}{0.700pt}}
\put(476,680){\rule[-0.350pt]{0.813pt}{0.700pt}}
\put(479,681){\rule[-0.350pt]{0.813pt}{0.700pt}}
\put(483,682){\rule[-0.350pt]{2.248pt}{0.700pt}}
\put(492,683){\rule[-0.350pt]{2.248pt}{0.700pt}}
\put(501,684){\rule[-0.350pt]{2.248pt}{0.700pt}}
\put(511,685){\rule[-0.350pt]{6.504pt}{0.700pt}}
\put(538,686){\rule[-0.350pt]{15.418pt}{0.700pt}}
\put(602,685){\rule[-0.350pt]{2.168pt}{0.700pt}}
\put(611,684){\rule[-0.350pt]{2.168pt}{0.700pt}}
\put(620,683){\rule[-0.350pt]{6.745pt}{0.700pt}}
\put(648,682){\rule[-0.350pt]{6.504pt}{0.700pt}}
\put(675,683){\rule[-0.350pt]{3.252pt}{0.700pt}}
\put(688,684){\rule[-0.350pt]{3.252pt}{0.700pt}}
\put(702,685){\rule[-0.350pt]{9.997pt}{0.700pt}}
\put(743,686){\rule[-0.350pt]{3.252pt}{0.700pt}}
\put(757,687){\rule[-0.350pt]{6.745pt}{0.700pt}}
\put(785,688){\rule[-0.350pt]{138.517pt}{0.700pt}}
\sbox{\plotpoint}{\rule[-0.500pt]{1.000pt}{1.000pt}}%
\put(264,218){\usebox{\plotpoint}}
\put(264,218){\usebox{\plotpoint}}
\put(265,219){\usebox{\plotpoint}}
\put(266,221){\usebox{\plotpoint}}
\put(267,223){\usebox{\plotpoint}}
\put(268,224){\usebox{\plotpoint}}
\put(269,226){\usebox{\plotpoint}}
\put(270,228){\usebox{\plotpoint}}
\put(271,230){\usebox{\plotpoint}}
\put(272,231){\usebox{\plotpoint}}
\put(273,233){\usebox{\plotpoint}}
\put(274,235){\usebox{\plotpoint}}
\put(275,237){\usebox{\plotpoint}}
\put(276,238){\usebox{\plotpoint}}
\put(277,240){\usebox{\plotpoint}}
\put(278,242){\usebox{\plotpoint}}
\put(279,244){\usebox{\plotpoint}}
\put(280,245){\usebox{\plotpoint}}
\put(281,247){\usebox{\plotpoint}}
\put(282,249){\usebox{\plotpoint}}
\put(283,251){\usebox{\plotpoint}}
\put(284,252){\usebox{\plotpoint}}
\put(285,254){\usebox{\plotpoint}}
\put(286,256){\usebox{\plotpoint}}
\put(287,258){\usebox{\plotpoint}}
\put(288,259){\usebox{\plotpoint}}
\put(289,261){\usebox{\plotpoint}}
\put(290,263){\usebox{\plotpoint}}
\put(291,265){\usebox{\plotpoint}}
\put(292,268){\usebox{\plotpoint}}
\put(293,271){\usebox{\plotpoint}}
\put(294,274){\usebox{\plotpoint}}
\put(295,277){\usebox{\plotpoint}}
\put(296,280){\usebox{\plotpoint}}
\put(297,283){\usebox{\plotpoint}}
\put(298,286){\usebox{\plotpoint}}
\put(299,290){\usebox{\plotpoint}}
\put(300,293){\usebox{\plotpoint}}
\put(301,296){\usebox{\plotpoint}}
\put(302,299){\usebox{\plotpoint}}
\put(303,302){\usebox{\plotpoint}}
\put(304,305){\usebox{\plotpoint}}
\put(305,308){\usebox{\plotpoint}}
\put(306,312){\usebox{\plotpoint}}
\put(307,315){\usebox{\plotpoint}}
\put(308,318){\usebox{\plotpoint}}
\put(309,321){\usebox{\plotpoint}}
\put(310,324){\usebox{\plotpoint}}
\put(311,327){\usebox{\plotpoint}}
\put(312,330){\usebox{\plotpoint}}
\put(313,334){\usebox{\plotpoint}}
\put(314,337){\usebox{\plotpoint}}
\put(315,340){\usebox{\plotpoint}}
\put(316,343){\usebox{\plotpoint}}
\put(317,346){\usebox{\plotpoint}}
\put(318,349){\usebox{\plotpoint}}
\put(319,352){\usebox{\plotpoint}}
\put(320,354){\usebox{\plotpoint}}
\put(321,355){\usebox{\plotpoint}}
\put(322,357){\usebox{\plotpoint}}
\put(323,358){\usebox{\plotpoint}}
\put(324,359){\usebox{\plotpoint}}
\put(325,361){\usebox{\plotpoint}}
\put(326,362){\usebox{\plotpoint}}
\put(327,363){\usebox{\plotpoint}}
\put(328,365){\usebox{\plotpoint}}
\put(329,366){\usebox{\plotpoint}}
\put(330,367){\usebox{\plotpoint}}
\put(331,369){\usebox{\plotpoint}}
\put(332,370){\usebox{\plotpoint}}
\put(333,371){\usebox{\plotpoint}}
\put(334,373){\usebox{\plotpoint}}
\put(335,374){\usebox{\plotpoint}}
\put(336,375){\usebox{\plotpoint}}
\put(337,377){\usebox{\plotpoint}}
\put(338,378){\usebox{\plotpoint}}
\put(339,379){\usebox{\plotpoint}}
\put(340,381){\usebox{\plotpoint}}
\put(341,382){\usebox{\plotpoint}}
\put(342,383){\usebox{\plotpoint}}
\put(343,385){\usebox{\plotpoint}}
\put(344,386){\usebox{\plotpoint}}
\put(345,387){\usebox{\plotpoint}}
\put(346,389){\usebox{\plotpoint}}
\put(347,391){\usebox{\plotpoint}}
\put(348,393){\usebox{\plotpoint}}
\put(349,395){\usebox{\plotpoint}}
\put(350,397){\usebox{\plotpoint}}
\put(351,399){\usebox{\plotpoint}}
\put(352,401){\usebox{\plotpoint}}
\put(353,403){\usebox{\plotpoint}}
\put(354,406){\usebox{\plotpoint}}
\put(355,408){\usebox{\plotpoint}}
\put(356,410){\usebox{\plotpoint}}
\put(357,412){\usebox{\plotpoint}}
\put(358,414){\usebox{\plotpoint}}
\put(359,416){\usebox{\plotpoint}}
\put(360,418){\usebox{\plotpoint}}
\put(361,421){\usebox{\plotpoint}}
\put(362,423){\usebox{\plotpoint}}
\put(363,425){\usebox{\plotpoint}}
\put(364,427){\usebox{\plotpoint}}
\put(365,429){\usebox{\plotpoint}}
\put(366,431){\usebox{\plotpoint}}
\put(367,433){\usebox{\plotpoint}}
\put(368,436){\usebox{\plotpoint}}
\put(369,438){\usebox{\plotpoint}}
\put(370,440){\usebox{\plotpoint}}
\put(371,442){\usebox{\plotpoint}}
\put(372,444){\usebox{\plotpoint}}
\put(373,446){\usebox{\plotpoint}}
\put(374,448){\usebox{\plotpoint}}
\put(375,450){\usebox{\plotpoint}}
\put(376,452){\usebox{\plotpoint}}
\put(377,454){\usebox{\plotpoint}}
\put(378,455){\usebox{\plotpoint}}
\put(379,457){\usebox{\plotpoint}}
\put(380,459){\usebox{\plotpoint}}
\put(381,460){\usebox{\plotpoint}}
\put(382,462){\usebox{\plotpoint}}
\put(383,464){\usebox{\plotpoint}}
\put(384,466){\usebox{\plotpoint}}
\put(385,467){\usebox{\plotpoint}}
\put(386,469){\usebox{\plotpoint}}
\put(387,471){\usebox{\plotpoint}}
\put(388,472){\usebox{\plotpoint}}
\put(389,474){\usebox{\plotpoint}}
\put(390,476){\usebox{\plotpoint}}
\put(391,477){\usebox{\plotpoint}}
\put(392,479){\usebox{\plotpoint}}
\put(393,481){\usebox{\plotpoint}}
\put(394,483){\usebox{\plotpoint}}
\put(395,484){\usebox{\plotpoint}}
\put(396,486){\usebox{\plotpoint}}
\put(397,488){\usebox{\plotpoint}}
\put(398,489){\usebox{\plotpoint}}
\put(399,491){\usebox{\plotpoint}}
\put(400,493){\usebox{\plotpoint}}
\put(401,495){\usebox{\plotpoint}}
\put(402,497){\usebox{\plotpoint}}
\put(403,499){\usebox{\plotpoint}}
\put(404,501){\usebox{\plotpoint}}
\put(405,503){\usebox{\plotpoint}}
\put(406,505){\usebox{\plotpoint}}
\put(407,507){\usebox{\plotpoint}}
\put(408,509){\usebox{\plotpoint}}
\put(409,511){\usebox{\plotpoint}}
\put(410,513){\usebox{\plotpoint}}
\put(411,515){\usebox{\plotpoint}}
\put(412,517){\usebox{\plotpoint}}
\put(413,519){\usebox{\plotpoint}}
\put(414,521){\usebox{\plotpoint}}
\put(415,523){\usebox{\plotpoint}}
\put(416,525){\usebox{\plotpoint}}
\put(417,527){\usebox{\plotpoint}}
\put(418,529){\usebox{\plotpoint}}
\put(419,531){\usebox{\plotpoint}}
\put(420,533){\usebox{\plotpoint}}
\put(421,535){\usebox{\plotpoint}}
\put(422,537){\usebox{\plotpoint}}
\put(423,539){\usebox{\plotpoint}}
\put(424,541){\usebox{\plotpoint}}
\put(425,543){\usebox{\plotpoint}}
\put(426,545){\usebox{\plotpoint}}
\put(427,547){\usebox{\plotpoint}}
\put(428,549){\usebox{\plotpoint}}
\put(429,550){\usebox{\plotpoint}}
\put(430,552){\usebox{\plotpoint}}
\put(431,554){\usebox{\plotpoint}}
\put(432,556){\usebox{\plotpoint}}
\put(433,557){\usebox{\plotpoint}}
\put(434,559){\usebox{\plotpoint}}
\put(435,561){\usebox{\plotpoint}}
\put(436,563){\usebox{\plotpoint}}
\put(437,565){\usebox{\plotpoint}}
\put(438,566){\usebox{\plotpoint}}
\put(439,568){\usebox{\plotpoint}}
\put(440,570){\usebox{\plotpoint}}
\put(441,572){\usebox{\plotpoint}}
\put(442,573){\usebox{\plotpoint}}
\put(443,575){\usebox{\plotpoint}}
\put(444,577){\usebox{\plotpoint}}
\put(445,579){\usebox{\plotpoint}}
\put(446,581){\usebox{\plotpoint}}
\put(447,582){\usebox{\plotpoint}}
\put(448,584){\usebox{\plotpoint}}
\put(449,586){\usebox{\plotpoint}}
\put(450,588){\usebox{\plotpoint}}
\put(451,590){\usebox{\plotpoint}}
\put(452,591){\usebox{\plotpoint}}
\put(453,593){\usebox{\plotpoint}}
\put(454,595){\usebox{\plotpoint}}
\put(455,597){\usebox{\plotpoint}}
\put(456,598){\usebox{\plotpoint}}
\put(457,600){\usebox{\plotpoint}}
\put(458,601){\usebox{\plotpoint}}
\put(459,602){\usebox{\plotpoint}}
\put(460,603){\usebox{\plotpoint}}
\put(461,604){\usebox{\plotpoint}}
\put(462,605){\usebox{\plotpoint}}
\put(463,607){\usebox{\plotpoint}}
\put(464,608){\usebox{\plotpoint}}
\put(465,609){\usebox{\plotpoint}}
\put(466,610){\usebox{\plotpoint}}
\put(467,611){\usebox{\plotpoint}}
\put(468,612){\usebox{\plotpoint}}
\put(469,613){\usebox{\plotpoint}}
\put(470,615){\usebox{\plotpoint}}
\put(471,616){\usebox{\plotpoint}}
\put(472,617){\usebox{\plotpoint}}
\put(473,618){\usebox{\plotpoint}}
\put(474,619){\usebox{\plotpoint}}
\put(475,620){\usebox{\plotpoint}}
\put(476,621){\usebox{\plotpoint}}
\put(477,623){\usebox{\plotpoint}}
\put(478,624){\usebox{\plotpoint}}
\put(479,625){\usebox{\plotpoint}}
\put(480,626){\usebox{\plotpoint}}
\put(481,627){\usebox{\plotpoint}}
\put(482,628){\usebox{\plotpoint}}
\put(483,629){\usebox{\plotpoint}}
\put(483,630){\usebox{\plotpoint}}
\put(486,631){\usebox{\plotpoint}}
\put(490,632){\usebox{\plotpoint}}
\put(493,633){\usebox{\plotpoint}}
\put(497,634){\usebox{\plotpoint}}
\put(500,635){\usebox{\plotpoint}}
\put(504,636){\usebox{\plotpoint}}
\put(507,637){\usebox{\plotpoint}}
\put(511,638){\rule[-0.500pt]{1.084pt}{1.000pt}}
\put(515,639){\rule[-0.500pt]{1.084pt}{1.000pt}}
\put(520,640){\rule[-0.500pt]{1.084pt}{1.000pt}}
\put(524,641){\rule[-0.500pt]{1.084pt}{1.000pt}}
\put(529,642){\rule[-0.500pt]{1.084pt}{1.000pt}}
\put(533,643){\rule[-0.500pt]{1.084pt}{1.000pt}}
\put(538,644){\rule[-0.500pt]{3.252pt}{1.000pt}}
\put(551,645){\rule[-0.500pt]{3.252pt}{1.000pt}}
\put(565,646){\rule[-0.500pt]{6.745pt}{1.000pt}}
\put(593,647){\rule[-0.500pt]{1.301pt}{1.000pt}}
\put(598,646){\rule[-0.500pt]{1.301pt}{1.000pt}}
\put(603,645){\rule[-0.500pt]{1.301pt}{1.000pt}}
\put(609,644){\rule[-0.500pt]{1.301pt}{1.000pt}}
\put(614,643){\rule[-0.500pt]{1.301pt}{1.000pt}}
\put(620,642){\rule[-0.500pt]{1.124pt}{1.000pt}}
\put(624,641){\rule[-0.500pt]{1.124pt}{1.000pt}}
\put(629,640){\rule[-0.500pt]{1.124pt}{1.000pt}}
\put(634,639){\rule[-0.500pt]{1.124pt}{1.000pt}}
\put(638,638){\rule[-0.500pt]{1.124pt}{1.000pt}}
\put(643,637){\rule[-0.500pt]{1.124pt}{1.000pt}}
\put(648,636){\usebox{\plotpoint}}
\put(649,637){\usebox{\plotpoint}}
\put(650,638){\usebox{\plotpoint}}
\put(652,639){\usebox{\plotpoint}}
\put(653,640){\usebox{\plotpoint}}
\put(655,641){\usebox{\plotpoint}}
\put(656,642){\usebox{\plotpoint}}
\put(657,643){\usebox{\plotpoint}}
\put(659,644){\usebox{\plotpoint}}
\put(660,645){\usebox{\plotpoint}}
\put(662,646){\usebox{\plotpoint}}
\put(663,647){\usebox{\plotpoint}}
\put(665,648){\usebox{\plotpoint}}
\put(666,649){\usebox{\plotpoint}}
\put(667,650){\usebox{\plotpoint}}
\put(669,651){\usebox{\plotpoint}}
\put(670,652){\usebox{\plotpoint}}
\put(672,653){\usebox{\plotpoint}}
\put(673,654){\usebox{\plotpoint}}
\put(675,655){\rule[-0.500pt]{1.084pt}{1.000pt}}
\put(679,656){\rule[-0.500pt]{1.084pt}{1.000pt}}
\put(684,657){\rule[-0.500pt]{1.084pt}{1.000pt}}
\put(688,658){\rule[-0.500pt]{1.084pt}{1.000pt}}
\put(693,659){\rule[-0.500pt]{1.084pt}{1.000pt}}
\put(697,660){\rule[-0.500pt]{1.084pt}{1.000pt}}
\put(702,661){\usebox{\plotpoint}}
\put(705,662){\usebox{\plotpoint}}
\put(709,663){\usebox{\plotpoint}}
\put(712,664){\usebox{\plotpoint}}
\put(716,665){\usebox{\plotpoint}}
\put(719,666){\usebox{\plotpoint}}
\put(723,667){\usebox{\plotpoint}}
\put(726,668){\usebox{\plotpoint}}
\put(730,669){\rule[-0.500pt]{1.626pt}{1.000pt}}
\put(736,670){\rule[-0.500pt]{1.626pt}{1.000pt}}
\put(743,671){\rule[-0.500pt]{1.626pt}{1.000pt}}
\put(750,672){\rule[-0.500pt]{1.626pt}{1.000pt}}
\put(757,673){\rule[-0.500pt]{6.745pt}{1.000pt}}
\put(785,674){\rule[-0.500pt]{3.252pt}{1.000pt}}
\put(798,675){\rule[-0.500pt]{3.252pt}{1.000pt}}
\put(812,676){\rule[-0.500pt]{3.252pt}{1.000pt}}
\put(825,677){\rule[-0.500pt]{3.252pt}{1.000pt}}
\put(839,678){\rule[-0.500pt]{6.745pt}{1.000pt}}
\put(867,679){\rule[-0.500pt]{118.764pt}{1.000pt}}
\end{picture}

%% file: trk1reseig-V100,3.tex
\setlength{\unitlength}{0.240900pt}
\ifx\plotpoint\undefined\newsavebox{\plotpoint}\fi
\sbox{\plotpoint}{\rule[-0.175pt]{0.350pt}{0.350pt}}%
\begin{picture}(1424,900)(0,0)
\tenrm
\sbox{\plotpoint}{\rule[-0.175pt]{0.350pt}{0.350pt}}%
\put(264,158){\rule[-0.175pt]{0.350pt}{151.526pt}}
\put(264,158){\rule[-0.175pt]{4.818pt}{0.350pt}}
\put(242,158){\makebox(0,0)[r]{$80$}}
\put(1340,158){\rule[-0.175pt]{4.818pt}{0.350pt}}
\put(264,284){\rule[-0.175pt]{4.818pt}{0.350pt}}
\put(242,284){\makebox(0,0)[r]{$85$}}
\put(1340,284){\rule[-0.175pt]{4.818pt}{0.350pt}}
\put(264,410){\rule[-0.175pt]{4.818pt}{0.350pt}}
\put(242,410){\makebox(0,0)[r]{$90$}}
\put(1340,410){\rule[-0.175pt]{4.818pt}{0.350pt}}
\put(264,535){\rule[-0.175pt]{4.818pt}{0.350pt}}
\put(242,535){\makebox(0,0)[r]{$95$}}
\put(1340,535){\rule[-0.175pt]{4.818pt}{0.350pt}}
\put(264,661){\rule[-0.175pt]{4.818pt}{0.350pt}}
\put(242,661){\makebox(0,0)[r]{$100$}}
\put(1340,661){\rule[-0.175pt]{4.818pt}{0.350pt}}
\put(264,787){\rule[-0.175pt]{4.818pt}{0.350pt}}
\put(242,787){\makebox(0,0)[r]{$105$}}
\put(1340,787){\rule[-0.175pt]{4.818pt}{0.350pt}}
\put(264,158){\rule[-0.175pt]{0.350pt}{4.818pt}}
\put(264,113){\makebox(0,0){$0$}}
\put(264,767){\rule[-0.175pt]{0.350pt}{4.818pt}}
\put(401,158){\rule[-0.175pt]{0.350pt}{4.818pt}}
\put(401,113){\makebox(0,0){$10$}}
\put(401,767){\rule[-0.175pt]{0.350pt}{4.818pt}}
\put(538,158){\rule[-0.175pt]{0.350pt}{4.818pt}}
\put(538,113){\makebox(0,0){$20$}}
\put(538,767){\rule[-0.175pt]{0.350pt}{4.818pt}}
\put(675,158){\rule[-0.175pt]{0.350pt}{4.818pt}}
\put(675,113){\makebox(0,0){$30$}}
\put(675,767){\rule[-0.175pt]{0.350pt}{4.818pt}}
\put(812,158){\rule[-0.175pt]{0.350pt}{4.818pt}}
\put(812,113){\makebox(0,0){$40$}}
\put(812,767){\rule[-0.175pt]{0.350pt}{4.818pt}}
\put(949,158){\rule[-0.175pt]{0.350pt}{4.818pt}}
\put(949,113){\makebox(0,0){$50$}}
\put(949,767){\rule[-0.175pt]{0.350pt}{4.818pt}}
\put(1086,158){\rule[-0.175pt]{0.350pt}{4.818pt}}
\put(1086,113){\makebox(0,0){$60$}}
\put(1086,767){\rule[-0.175pt]{0.350pt}{4.818pt}}
\put(1223,158){\rule[-0.175pt]{0.350pt}{4.818pt}}
\put(1223,113){\makebox(0,0){$70$}}
\put(1223,767){\rule[-0.175pt]{0.350pt}{4.818pt}}
\put(1360,158){\rule[-0.175pt]{0.350pt}{4.818pt}}
\put(1360,113){\makebox(0,0){$80$}}
\put(1360,767){\rule[-0.175pt]{0.350pt}{4.818pt}}
\put(264,158){\rule[-0.175pt]{264.026pt}{0.350pt}}
\put(1360,158){\rule[-0.175pt]{0.350pt}{151.526pt}}
\put(264,787){\rule[-0.175pt]{264.026pt}{0.350pt}}
\put(-43,472){\makebox(0,0)[l]{\shortstack{${\displaystyle{(p_i^\T\!A_ip_i)_{jj}\atop1\le j\le3}}$}}}
\put(812,68){\makebox(0,0){Step $i$}}
\put(264,158){\rule[-0.175pt]{0.350pt}{151.526pt}}
\put(306,158){\rule[-0.175pt]{0.350pt}{3.150pt}}
\put(307,171){\rule[-0.175pt]{0.350pt}{3.150pt}}
\put(308,184){\rule[-0.175pt]{0.350pt}{3.150pt}}
\put(309,197){\rule[-0.175pt]{0.350pt}{3.150pt}}
\put(310,210){\rule[-0.175pt]{0.350pt}{3.150pt}}
\put(311,223){\rule[-0.175pt]{0.350pt}{3.150pt}}
\put(312,236){\rule[-0.175pt]{0.350pt}{3.150pt}}
\put(313,249){\rule[-0.175pt]{0.350pt}{3.150pt}}
\put(314,262){\rule[-0.175pt]{0.350pt}{3.150pt}}
\put(315,275){\rule[-0.175pt]{0.350pt}{3.150pt}}
\put(316,288){\rule[-0.175pt]{0.350pt}{3.150pt}}
\put(317,301){\rule[-0.175pt]{0.350pt}{3.150pt}}
\put(318,314){\rule[-0.175pt]{0.350pt}{3.150pt}}
\put(319,328){\rule[-0.175pt]{0.350pt}{1.032pt}}
\put(320,332){\rule[-0.175pt]{0.350pt}{1.032pt}}
\put(321,336){\rule[-0.175pt]{0.350pt}{1.032pt}}
\put(322,340){\rule[-0.175pt]{0.350pt}{1.032pt}}
\put(323,345){\rule[-0.175pt]{0.350pt}{1.032pt}}
\put(324,349){\rule[-0.175pt]{0.350pt}{1.032pt}}
\put(325,353){\rule[-0.175pt]{0.350pt}{1.032pt}}
\put(326,357){\rule[-0.175pt]{0.350pt}{1.032pt}}
\put(327,362){\rule[-0.175pt]{0.350pt}{1.032pt}}
\put(328,366){\rule[-0.175pt]{0.350pt}{1.032pt}}
\put(329,370){\rule[-0.175pt]{0.350pt}{1.032pt}}
\put(330,375){\rule[-0.175pt]{0.350pt}{1.032pt}}
\put(331,379){\rule[-0.175pt]{0.350pt}{1.032pt}}
\put(332,383){\rule[-0.175pt]{0.350pt}{1.032pt}}
\put(333,387){\rule[-0.175pt]{0.350pt}{0.797pt}}
\put(334,391){\rule[-0.175pt]{0.350pt}{0.797pt}}
\put(335,394){\rule[-0.175pt]{0.350pt}{0.797pt}}
\put(336,397){\rule[-0.175pt]{0.350pt}{0.797pt}}
\put(337,401){\rule[-0.175pt]{0.350pt}{0.797pt}}
\put(338,404){\rule[-0.175pt]{0.350pt}{0.797pt}}
\put(339,407){\rule[-0.175pt]{0.350pt}{0.797pt}}
\put(340,411){\rule[-0.175pt]{0.350pt}{0.797pt}}
\put(341,414){\rule[-0.175pt]{0.350pt}{0.797pt}}
\put(342,417){\rule[-0.175pt]{0.350pt}{0.797pt}}
\put(343,421){\rule[-0.175pt]{0.350pt}{0.797pt}}
\put(344,424){\rule[-0.175pt]{0.350pt}{0.797pt}}
\put(345,427){\rule[-0.175pt]{0.350pt}{0.797pt}}
\put(346,430){\usebox{\plotpoint}}
\put(346,431){\rule[-0.175pt]{0.482pt}{0.350pt}}
\put(348,432){\rule[-0.175pt]{0.482pt}{0.350pt}}
\put(350,433){\rule[-0.175pt]{0.482pt}{0.350pt}}
\put(352,434){\rule[-0.175pt]{0.482pt}{0.350pt}}
\put(354,435){\rule[-0.175pt]{0.482pt}{0.350pt}}
\put(356,436){\rule[-0.175pt]{0.482pt}{0.350pt}}
\put(358,437){\rule[-0.175pt]{0.482pt}{0.350pt}}
\put(360,438){\rule[-0.175pt]{0.422pt}{0.350pt}}
\put(361,439){\rule[-0.175pt]{0.422pt}{0.350pt}}
\put(363,440){\rule[-0.175pt]{0.422pt}{0.350pt}}
\put(365,441){\rule[-0.175pt]{0.422pt}{0.350pt}}
\put(367,442){\rule[-0.175pt]{0.422pt}{0.350pt}}
\put(368,443){\rule[-0.175pt]{0.422pt}{0.350pt}}
\put(370,444){\rule[-0.175pt]{0.422pt}{0.350pt}}
\put(372,445){\rule[-0.175pt]{0.422pt}{0.350pt}}
\put(374,446){\usebox{\plotpoint}}
\put(375,447){\usebox{\plotpoint}}
\put(376,448){\usebox{\plotpoint}}
\put(377,449){\usebox{\plotpoint}}
\put(379,450){\usebox{\plotpoint}}
\put(380,451){\usebox{\plotpoint}}
\put(381,452){\usebox{\plotpoint}}
\put(383,453){\usebox{\plotpoint}}
\put(384,454){\usebox{\plotpoint}}
\put(385,455){\usebox{\plotpoint}}
\put(386,456){\rule[-0.175pt]{0.675pt}{0.350pt}}
\put(389,457){\rule[-0.175pt]{0.675pt}{0.350pt}}
\put(392,458){\rule[-0.175pt]{0.675pt}{0.350pt}}
\put(395,459){\rule[-0.175pt]{0.675pt}{0.350pt}}
\put(398,460){\rule[-0.175pt]{0.675pt}{0.350pt}}
\put(400,461){\usebox{\plotpoint}}
\put(402,462){\usebox{\plotpoint}}
\put(403,463){\usebox{\plotpoint}}
\put(404,464){\usebox{\plotpoint}}
\put(405,465){\usebox{\plotpoint}}
\put(406,466){\usebox{\plotpoint}}
\put(407,467){\usebox{\plotpoint}}
\put(408,468){\usebox{\plotpoint}}
\put(409,469){\usebox{\plotpoint}}
\put(410,470){\usebox{\plotpoint}}
\put(411,471){\usebox{\plotpoint}}
\put(412,472){\usebox{\plotpoint}}
\put(413,473){\usebox{\plotpoint}}
\put(415,474){\rule[-0.175pt]{0.350pt}{1.038pt}}
\put(416,478){\rule[-0.175pt]{0.350pt}{1.038pt}}
\put(417,482){\rule[-0.175pt]{0.350pt}{1.038pt}}
\put(418,486){\rule[-0.175pt]{0.350pt}{1.038pt}}
\put(419,491){\rule[-0.175pt]{0.350pt}{1.038pt}}
\put(420,495){\rule[-0.175pt]{0.350pt}{1.038pt}}
\put(421,499){\rule[-0.175pt]{0.350pt}{1.038pt}}
\put(422,504){\rule[-0.175pt]{0.350pt}{1.038pt}}
\put(423,508){\rule[-0.175pt]{0.350pt}{1.038pt}}
\put(424,512){\rule[-0.175pt]{0.350pt}{1.038pt}}
\put(425,517){\rule[-0.175pt]{0.350pt}{1.038pt}}
\put(426,521){\rule[-0.175pt]{0.350pt}{1.038pt}}
\put(427,525){\rule[-0.175pt]{0.350pt}{1.038pt}}
\put(428,529){\rule[-0.175pt]{0.350pt}{1.428pt}}
\put(429,535){\rule[-0.175pt]{0.350pt}{1.428pt}}
\put(430,541){\rule[-0.175pt]{0.350pt}{1.428pt}}
\put(431,547){\rule[-0.175pt]{0.350pt}{1.428pt}}
\put(432,553){\rule[-0.175pt]{0.350pt}{1.428pt}}
\put(433,559){\rule[-0.175pt]{0.350pt}{1.428pt}}
\put(434,565){\rule[-0.175pt]{0.350pt}{1.428pt}}
\put(435,571){\rule[-0.175pt]{0.350pt}{1.428pt}}
\put(436,577){\rule[-0.175pt]{0.350pt}{1.428pt}}
\put(437,583){\rule[-0.175pt]{0.350pt}{1.428pt}}
\put(438,589){\rule[-0.175pt]{0.350pt}{1.428pt}}
\put(439,595){\rule[-0.175pt]{0.350pt}{1.428pt}}
\put(440,601){\rule[-0.175pt]{0.350pt}{1.428pt}}
\put(441,607){\rule[-0.175pt]{0.350pt}{1.428pt}}
\put(442,613){\rule[-0.175pt]{0.350pt}{0.602pt}}
\put(443,615){\rule[-0.175pt]{0.350pt}{0.602pt}}
\put(444,618){\rule[-0.175pt]{0.350pt}{0.602pt}}
\put(445,620){\rule[-0.175pt]{0.350pt}{0.602pt}}
\put(446,623){\rule[-0.175pt]{0.350pt}{0.602pt}}
\put(447,625){\rule[-0.175pt]{0.350pt}{0.602pt}}
\put(448,628){\rule[-0.175pt]{0.350pt}{0.602pt}}
\put(449,630){\rule[-0.175pt]{0.350pt}{0.602pt}}
\put(450,633){\rule[-0.175pt]{0.350pt}{0.602pt}}
\put(451,635){\rule[-0.175pt]{0.350pt}{0.602pt}}
\put(452,638){\rule[-0.175pt]{0.350pt}{0.602pt}}
\put(453,640){\rule[-0.175pt]{0.350pt}{0.602pt}}
\put(454,643){\rule[-0.175pt]{0.350pt}{0.602pt}}
\put(455,645){\rule[-0.175pt]{0.350pt}{0.602pt}}
\put(456,648){\rule[-0.175pt]{1.124pt}{0.350pt}}
\put(460,649){\rule[-0.175pt]{1.124pt}{0.350pt}}
\put(465,650){\rule[-0.175pt]{1.124pt}{0.350pt}}
\put(469,651){\rule[-0.175pt]{0.783pt}{0.350pt}}
\put(473,652){\rule[-0.175pt]{0.783pt}{0.350pt}}
\put(476,653){\rule[-0.175pt]{0.783pt}{0.350pt}}
\put(479,654){\rule[-0.175pt]{0.783pt}{0.350pt}}
\put(483,655){\rule[-0.175pt]{1.686pt}{0.350pt}}
\put(490,656){\rule[-0.175pt]{1.686pt}{0.350pt}}
\put(497,657){\rule[-0.175pt]{3.373pt}{0.350pt}}
\put(511,658){\rule[-0.175pt]{3.132pt}{0.350pt}}
\put(524,659){\rule[-0.175pt]{9.877pt}{0.350pt}}
\put(565,660){\rule[-0.175pt]{19.995pt}{0.350pt}}
\put(648,661){\rule[-0.175pt]{39.814pt}{0.350pt}}
\put(813,660){\usebox{\plotpoint}}
\put(814,659){\usebox{\plotpoint}}
\put(815,658){\usebox{\plotpoint}}
\put(817,657){\usebox{\plotpoint}}
\put(818,656){\usebox{\plotpoint}}
\put(819,655){\usebox{\plotpoint}}
\put(820,654){\usebox{\plotpoint}}
\put(822,653){\usebox{\plotpoint}}
\put(823,652){\usebox{\plotpoint}}
\put(824,651){\usebox{\plotpoint}}
\put(825,650){\rule[-0.175pt]{1.044pt}{0.350pt}}
\put(830,651){\rule[-0.175pt]{1.044pt}{0.350pt}}
\put(834,652){\rule[-0.175pt]{1.044pt}{0.350pt}}
\put(838,653){\rule[-0.175pt]{0.675pt}{0.350pt}}
\put(841,654){\rule[-0.175pt]{0.675pt}{0.350pt}}
\put(844,655){\rule[-0.175pt]{0.675pt}{0.350pt}}
\put(847,656){\rule[-0.175pt]{0.675pt}{0.350pt}}
\put(850,657){\rule[-0.175pt]{0.675pt}{0.350pt}}
\put(852,658){\rule[-0.175pt]{1.686pt}{0.350pt}}
\put(860,659){\rule[-0.175pt]{1.686pt}{0.350pt}}
\put(867,660){\rule[-0.175pt]{3.373pt}{0.350pt}}
\put(881,661){\rule[-0.175pt]{115.391pt}{0.350pt}}
\sbox{\plotpoint}{\rule[-0.350pt]{0.700pt}{0.700pt}}%
\put(295,158){\rule[-0.350pt]{0.700pt}{8.504pt}}
\put(296,193){\rule[-0.350pt]{0.700pt}{8.504pt}}
\put(297,228){\rule[-0.350pt]{0.700pt}{8.504pt}}
\put(298,263){\rule[-0.350pt]{0.700pt}{8.504pt}}
\put(299,299){\rule[-0.350pt]{0.700pt}{8.504pt}}
\put(300,334){\rule[-0.350pt]{0.700pt}{8.504pt}}
\put(301,369){\rule[-0.350pt]{0.700pt}{8.504pt}}
\put(302,405){\rule[-0.350pt]{0.700pt}{8.504pt}}
\put(303,440){\rule[-0.350pt]{0.700pt}{8.504pt}}
\put(304,475){\rule[-0.350pt]{0.700pt}{8.504pt}}
\put(305,510){\usebox{\plotpoint}}
\put(306,513){\usebox{\plotpoint}}
\put(307,515){\usebox{\plotpoint}}
\put(308,517){\usebox{\plotpoint}}
\put(309,519){\usebox{\plotpoint}}
\put(310,521){\usebox{\plotpoint}}
\put(311,523){\usebox{\plotpoint}}
\put(312,526){\usebox{\plotpoint}}
\put(313,528){\usebox{\plotpoint}}
\put(314,530){\usebox{\plotpoint}}
\put(315,532){\usebox{\plotpoint}}
\put(316,534){\usebox{\plotpoint}}
\put(317,536){\usebox{\plotpoint}}
\put(318,538){\usebox{\plotpoint}}
\put(319,541){\rule[-0.350pt]{0.700pt}{0.946pt}}
\put(320,544){\rule[-0.350pt]{0.700pt}{0.946pt}}
\put(321,548){\rule[-0.350pt]{0.700pt}{0.946pt}}
\put(322,552){\rule[-0.350pt]{0.700pt}{0.946pt}}
\put(323,556){\rule[-0.350pt]{0.700pt}{0.946pt}}
\put(324,560){\rule[-0.350pt]{0.700pt}{0.946pt}}
\put(325,564){\rule[-0.350pt]{0.700pt}{0.946pt}}
\put(326,568){\rule[-0.350pt]{0.700pt}{0.946pt}}
\put(327,572){\rule[-0.350pt]{0.700pt}{0.946pt}}
\put(328,576){\rule[-0.350pt]{0.700pt}{0.946pt}}
\put(329,580){\rule[-0.350pt]{0.700pt}{0.946pt}}
\put(330,584){\rule[-0.350pt]{0.700pt}{0.946pt}}
\put(331,588){\rule[-0.350pt]{0.700pt}{0.946pt}}
\put(332,592){\rule[-0.350pt]{0.700pt}{0.946pt}}
\put(333,596){\rule[-0.350pt]{3.132pt}{0.700pt}}
\put(346,596){\usebox{\plotpoint}}
\put(347,597){\usebox{\plotpoint}}
\put(348,599){\usebox{\plotpoint}}
\put(349,600){\usebox{\plotpoint}}
\put(350,602){\usebox{\plotpoint}}
\put(351,603){\usebox{\plotpoint}}
\put(352,605){\usebox{\plotpoint}}
\put(353,606){\usebox{\plotpoint}}
\put(354,608){\usebox{\plotpoint}}
\put(355,610){\usebox{\plotpoint}}
\put(356,611){\usebox{\plotpoint}}
\put(357,613){\usebox{\plotpoint}}
\put(358,614){\usebox{\plotpoint}}
\put(359,616){\usebox{\plotpoint}}
\put(360,617){\usebox{\plotpoint}}
\put(360,618){\usebox{\plotpoint}}
\put(361,619){\usebox{\plotpoint}}
\put(363,620){\usebox{\plotpoint}}
\put(364,621){\usebox{\plotpoint}}
\put(366,622){\usebox{\plotpoint}}
\put(367,623){\usebox{\plotpoint}}
\put(369,624){\usebox{\plotpoint}}
\put(370,625){\usebox{\plotpoint}}
\put(372,626){\usebox{\plotpoint}}
\put(373,627){\rule[-0.350pt]{1.566pt}{0.700pt}}
\put(380,628){\rule[-0.350pt]{1.566pt}{0.700pt}}
\put(387,629){\rule[-0.350pt]{1.686pt}{0.700pt}}
\put(394,630){\rule[-0.350pt]{1.686pt}{0.700pt}}
\put(401,631){\rule[-0.350pt]{3.373pt}{0.700pt}}
\put(415,630){\usebox{\plotpoint}}
\put(416,629){\usebox{\plotpoint}}
\put(418,628){\usebox{\plotpoint}}
\put(420,627){\usebox{\plotpoint}}
\put(422,626){\usebox{\plotpoint}}
\put(424,625){\usebox{\plotpoint}}
\put(426,624){\usebox{\plotpoint}}
\put(428,623){\usebox{\plotpoint}}
\put(430,622){\usebox{\plotpoint}}
\put(432,621){\usebox{\plotpoint}}
\put(435,620){\usebox{\plotpoint}}
\put(437,619){\usebox{\plotpoint}}
\put(439,618){\usebox{\plotpoint}}
\put(442,617){\rule[-0.350pt]{0.843pt}{0.700pt}}
\put(445,618){\rule[-0.350pt]{0.843pt}{0.700pt}}
\put(449,619){\rule[-0.350pt]{0.843pt}{0.700pt}}
\put(452,620){\rule[-0.350pt]{0.843pt}{0.700pt}}
\put(456,621){\usebox{\plotpoint}}
\put(458,622){\usebox{\plotpoint}}
\put(460,623){\usebox{\plotpoint}}
\put(462,624){\usebox{\plotpoint}}
\put(464,625){\usebox{\plotpoint}}
\put(466,626){\usebox{\plotpoint}}
\put(468,627){\usebox{\plotpoint}}
\put(470,628){\rule[-0.350pt]{3.694pt}{0.700pt}}
\put(485,629){\usebox{\plotpoint}}
\put(487,630){\usebox{\plotpoint}}
\put(490,631){\usebox{\plotpoint}}
\put(492,632){\usebox{\plotpoint}}
\put(494,633){\usebox{\plotpoint}}
\put(497,634){\rule[-0.350pt]{3.373pt}{0.700pt}}
\put(511,635){\rule[-0.350pt]{3.132pt}{0.700pt}}
\put(524,636){\rule[-0.350pt]{69.686pt}{0.700pt}}
\put(813,637){\usebox{\plotpoint}}
\put(814,638){\usebox{\plotpoint}}
\put(815,639){\usebox{\plotpoint}}
\put(817,640){\usebox{\plotpoint}}
\put(818,641){\usebox{\plotpoint}}
\put(819,642){\usebox{\plotpoint}}
\put(820,643){\usebox{\plotpoint}}
\put(822,644){\usebox{\plotpoint}}
\put(823,645){\usebox{\plotpoint}}
\put(824,646){\usebox{\plotpoint}}
\put(825,647){\rule[-0.350pt]{1.044pt}{0.700pt}}
\put(830,646){\rule[-0.350pt]{1.044pt}{0.700pt}}
\put(834,645){\rule[-0.350pt]{1.044pt}{0.700pt}}
\put(838,644){\usebox{\plotpoint}}
\put(841,643){\usebox{\plotpoint}}
\put(844,642){\usebox{\plotpoint}}
\put(847,641){\usebox{\plotpoint}}
\put(850,640){\usebox{\plotpoint}}
\put(852,639){\rule[-0.350pt]{1.686pt}{0.700pt}}
\put(860,638){\rule[-0.350pt]{1.686pt}{0.700pt}}
\put(867,637){\rule[-0.350pt]{3.373pt}{0.700pt}}
\put(881,636){\rule[-0.350pt]{115.391pt}{0.700pt}}
\sbox{\plotpoint}{\rule[-0.500pt]{1.000pt}{1.000pt}}%
\put(325,158){\rule[-0.500pt]{1.000pt}{2.620pt}}
\put(326,168){\rule[-0.500pt]{1.000pt}{2.620pt}}
\put(327,179){\rule[-0.500pt]{1.000pt}{2.620pt}}
\put(328,190){\rule[-0.500pt]{1.000pt}{2.620pt}}
\put(329,201){\rule[-0.500pt]{1.000pt}{2.620pt}}
\put(330,212){\rule[-0.500pt]{1.000pt}{2.620pt}}
\put(331,223){\rule[-0.500pt]{1.000pt}{2.620pt}}
\put(332,234){\rule[-0.500pt]{1.000pt}{2.620pt}}
\put(333,245){\rule[-0.500pt]{1.000pt}{2.594pt}}
\put(334,255){\rule[-0.500pt]{1.000pt}{2.594pt}}
\put(335,266){\rule[-0.500pt]{1.000pt}{2.594pt}}
\put(336,277){\rule[-0.500pt]{1.000pt}{2.594pt}}
\put(337,288){\rule[-0.500pt]{1.000pt}{2.594pt}}
\put(338,298){\rule[-0.500pt]{1.000pt}{2.594pt}}
\put(339,309){\rule[-0.500pt]{1.000pt}{2.594pt}}
\put(340,320){\rule[-0.500pt]{1.000pt}{2.594pt}}
\put(341,331){\rule[-0.500pt]{1.000pt}{2.594pt}}
\put(342,341){\rule[-0.500pt]{1.000pt}{2.594pt}}
\put(343,352){\rule[-0.500pt]{1.000pt}{2.594pt}}
\put(344,363){\rule[-0.500pt]{1.000pt}{2.594pt}}
\put(345,374){\rule[-0.500pt]{1.000pt}{2.594pt}}
\put(346,384){\rule[-0.500pt]{1.000pt}{1.514pt}}
\put(347,391){\rule[-0.500pt]{1.000pt}{1.514pt}}
\put(348,397){\rule[-0.500pt]{1.000pt}{1.514pt}}
\put(349,403){\rule[-0.500pt]{1.000pt}{1.514pt}}
\put(350,410){\rule[-0.500pt]{1.000pt}{1.514pt}}
\put(351,416){\rule[-0.500pt]{1.000pt}{1.514pt}}
\put(352,422){\rule[-0.500pt]{1.000pt}{1.514pt}}
\put(353,428){\rule[-0.500pt]{1.000pt}{1.514pt}}
\put(354,435){\rule[-0.500pt]{1.000pt}{1.514pt}}
\put(355,441){\rule[-0.500pt]{1.000pt}{1.514pt}}
\put(356,447){\rule[-0.500pt]{1.000pt}{1.514pt}}
\put(357,454){\rule[-0.500pt]{1.000pt}{1.514pt}}
\put(358,460){\rule[-0.500pt]{1.000pt}{1.514pt}}
\put(359,466){\rule[-0.500pt]{1.000pt}{1.514pt}}
\put(360,472){\usebox{\plotpoint}}
\put(361,474){\usebox{\plotpoint}}
\put(362,476){\usebox{\plotpoint}}
\put(363,477){\usebox{\plotpoint}}
\put(364,479){\usebox{\plotpoint}}
\put(365,480){\usebox{\plotpoint}}
\put(366,482){\usebox{\plotpoint}}
\put(367,483){\usebox{\plotpoint}}
\put(368,485){\usebox{\plotpoint}}
\put(369,486){\usebox{\plotpoint}}
\put(370,488){\usebox{\plotpoint}}
\put(371,489){\usebox{\plotpoint}}
\put(372,491){\usebox{\plotpoint}}
\put(373,492){\usebox{\plotpoint}}
\put(374,494){\usebox{\plotpoint}}
\put(375,495){\usebox{\plotpoint}}
\put(376,496){\usebox{\plotpoint}}
\put(377,498){\usebox{\plotpoint}}
\put(378,499){\usebox{\plotpoint}}
\put(379,500){\usebox{\plotpoint}}
\put(380,502){\usebox{\plotpoint}}
\put(381,503){\usebox{\plotpoint}}
\put(382,505){\usebox{\plotpoint}}
\put(383,506){\usebox{\plotpoint}}
\put(384,507){\usebox{\plotpoint}}
\put(385,509){\usebox{\plotpoint}}
\put(386,510){\usebox{\plotpoint}}
\put(387,511){\usebox{\plotpoint}}
\put(387,512){\usebox{\plotpoint}}
\put(389,513){\usebox{\plotpoint}}
\put(392,514){\usebox{\plotpoint}}
\put(395,515){\usebox{\plotpoint}}
\put(398,516){\usebox{\plotpoint}}
\put(400,517){\usebox{\plotpoint}}
\put(403,518){\usebox{\plotpoint}}
\put(406,519){\usebox{\plotpoint}}
\put(409,520){\usebox{\plotpoint}}
\put(412,521){\usebox{\plotpoint}}
\put(414,522){\usebox{\plotpoint}}
\put(415,520){\usebox{\plotpoint}}
\put(416,519){\usebox{\plotpoint}}
\put(417,518){\usebox{\plotpoint}}
\put(418,517){\usebox{\plotpoint}}
\put(419,516){\usebox{\plotpoint}}
\put(420,515){\usebox{\plotpoint}}
\put(421,513){\usebox{\plotpoint}}
\put(422,512){\usebox{\plotpoint}}
\put(423,511){\usebox{\plotpoint}}
\put(424,510){\usebox{\plotpoint}}
\put(425,509){\usebox{\plotpoint}}
\put(426,508){\usebox{\plotpoint}}
\put(427,507){\usebox{\plotpoint}}
\put(428,505){\usebox{\plotpoint}}
\put(429,504){\usebox{\plotpoint}}
\put(430,503){\usebox{\plotpoint}}
\put(431,502){\usebox{\plotpoint}}
\put(432,501){\usebox{\plotpoint}}
\put(433,500){\usebox{\plotpoint}}
\put(434,499){\usebox{\plotpoint}}
\put(435,497){\usebox{\plotpoint}}
\put(436,496){\usebox{\plotpoint}}
\put(437,495){\usebox{\plotpoint}}
\put(438,494){\usebox{\plotpoint}}
\put(439,493){\usebox{\plotpoint}}
\put(440,492){\usebox{\plotpoint}}
\put(441,491){\usebox{\plotpoint}}
\put(442,491){\usebox{\plotpoint}}
\put(442,491){\usebox{\plotpoint}}
\put(443,494){\usebox{\plotpoint}}
\put(444,498){\usebox{\plotpoint}}
\put(445,502){\usebox{\plotpoint}}
\put(446,506){\usebox{\plotpoint}}
\put(447,509){\usebox{\plotpoint}}
\put(448,513){\usebox{\plotpoint}}
\put(449,517){\usebox{\plotpoint}}
\put(450,521){\usebox{\plotpoint}}
\put(451,525){\usebox{\plotpoint}}
\put(452,528){\usebox{\plotpoint}}
\put(453,532){\usebox{\plotpoint}}
\put(454,536){\usebox{\plotpoint}}
\put(455,540){\usebox{\plotpoint}}
\put(456,543){\usebox{\plotpoint}}
\put(457,545){\usebox{\plotpoint}}
\put(458,546){\usebox{\plotpoint}}
\put(459,547){\usebox{\plotpoint}}
\put(460,548){\usebox{\plotpoint}}
\put(461,549){\usebox{\plotpoint}}
\put(462,550){\usebox{\plotpoint}}
\put(463,551){\usebox{\plotpoint}}
\put(464,552){\usebox{\plotpoint}}
\put(465,553){\usebox{\plotpoint}}
\put(466,554){\usebox{\plotpoint}}
\put(467,555){\usebox{\plotpoint}}
\put(468,556){\usebox{\plotpoint}}
\put(469,557){\usebox{\plotpoint}}
\put(470,558){\usebox{\plotpoint}}
\put(471,560){\usebox{\plotpoint}}
\put(472,562){\usebox{\plotpoint}}
\put(473,564){\usebox{\plotpoint}}
\put(474,566){\usebox{\plotpoint}}
\put(475,568){\usebox{\plotpoint}}
\put(476,570){\usebox{\plotpoint}}
\put(477,571){\usebox{\plotpoint}}
\put(478,573){\usebox{\plotpoint}}
\put(479,575){\usebox{\plotpoint}}
\put(480,577){\usebox{\plotpoint}}
\put(481,579){\usebox{\plotpoint}}
\put(482,581){\usebox{\plotpoint}}
\put(483,582){\usebox{\plotpoint}}
\put(483,583){\usebox{\plotpoint}}
\put(484,584){\usebox{\plotpoint}}
\put(485,585){\usebox{\plotpoint}}
\put(487,586){\usebox{\plotpoint}}
\put(488,587){\usebox{\plotpoint}}
\put(489,588){\usebox{\plotpoint}}
\put(491,589){\usebox{\plotpoint}}
\put(492,590){\usebox{\plotpoint}}
\put(494,591){\usebox{\plotpoint}}
\put(495,592){\usebox{\plotpoint}}
\put(496,593){\usebox{\plotpoint}}
\put(500,594){\usebox{\plotpoint}}
\put(504,595){\usebox{\plotpoint}}
\put(507,596){\usebox{\plotpoint}}
\put(511,597){\usebox{\plotpoint}}
\put(513,598){\usebox{\plotpoint}}
\put(515,599){\usebox{\plotpoint}}
\put(517,600){\usebox{\plotpoint}}
\put(519,601){\usebox{\plotpoint}}
\put(521,602){\usebox{\plotpoint}}
\put(524,603){\usebox{\plotpoint}}
\put(527,604){\usebox{\plotpoint}}
\put(531,605){\usebox{\plotpoint}}
\put(534,606){\usebox{\plotpoint}}
\put(538,607){\rule[-0.500pt]{1.686pt}{1.000pt}}
\put(545,608){\rule[-0.500pt]{1.686pt}{1.000pt}}
\put(552,609){\rule[-0.500pt]{3.132pt}{1.000pt}}
\put(565,610){\rule[-0.500pt]{3.373pt}{1.000pt}}
\put(579,611){\rule[-0.500pt]{62.634pt}{1.000pt}}
\put(839,610){\rule[-0.500pt]{3.373pt}{1.000pt}}
\put(853,609){\rule[-0.500pt]{6.745pt}{1.000pt}}
\put(881,610){\rule[-0.500pt]{9.877pt}{1.000pt}}
\put(922,611){\rule[-0.500pt]{105.514pt}{1.000pt}}
\end{picture}

%% file: track2-V100,3.tex
\setlength{\unitlength}{0.240900pt}
\ifx\plotpoint\undefined\newsavebox{\plotpoint}\fi
\sbox{\plotpoint}{\rule[-0.175pt]{0.350pt}{0.350pt}}%
\begin{picture}(1424,900)(0,0)
\tenrm
\sbox{\plotpoint}{\rule[-0.175pt]{0.350pt}{0.350pt}}%
\put(264,158){\rule[-0.175pt]{0.350pt}{151.526pt}}
\put(264,158){\rule[-0.175pt]{4.818pt}{0.350pt}}
\put(242,158){\makebox(0,0)[r]{$10^{\gnulog 1e-05}$}}
\put(1340,158){\rule[-0.175pt]{4.818pt}{0.350pt}}
\put(264,182){\rule[-0.175pt]{2.409pt}{0.350pt}}
\put(1350,182){\rule[-0.175pt]{2.409pt}{0.350pt}}
\put(264,196){\rule[-0.175pt]{2.409pt}{0.350pt}}
\put(1350,196){\rule[-0.175pt]{2.409pt}{0.350pt}}
\put(264,205){\rule[-0.175pt]{2.409pt}{0.350pt}}
\put(1350,205){\rule[-0.175pt]{2.409pt}{0.350pt}}
\put(264,213){\rule[-0.175pt]{2.409pt}{0.350pt}}
\put(1350,213){\rule[-0.175pt]{2.409pt}{0.350pt}}
\put(264,219){\rule[-0.175pt]{2.409pt}{0.350pt}}
\put(1350,219){\rule[-0.175pt]{2.409pt}{0.350pt}}
\put(264,224){\rule[-0.175pt]{2.409pt}{0.350pt}}
\put(1350,224){\rule[-0.175pt]{2.409pt}{0.350pt}}
\put(264,229){\rule[-0.175pt]{2.409pt}{0.350pt}}
\put(1350,229){\rule[-0.175pt]{2.409pt}{0.350pt}}
\put(264,233){\rule[-0.175pt]{2.409pt}{0.350pt}}
\put(1350,233){\rule[-0.175pt]{2.409pt}{0.350pt}}
\put(264,237){\rule[-0.175pt]{4.818pt}{0.350pt}}
\put(242,237){\makebox(0,0)[r]{$10^{\gnulog 1e-04}$}}
\put(1340,237){\rule[-0.175pt]{4.818pt}{0.350pt}}
\put(264,260){\rule[-0.175pt]{2.409pt}{0.350pt}}
\put(1350,260){\rule[-0.175pt]{2.409pt}{0.350pt}}
\put(264,274){\rule[-0.175pt]{2.409pt}{0.350pt}}
\put(1350,274){\rule[-0.175pt]{2.409pt}{0.350pt}}
\put(264,284){\rule[-0.175pt]{2.409pt}{0.350pt}}
\put(1350,284){\rule[-0.175pt]{2.409pt}{0.350pt}}
\put(264,292){\rule[-0.175pt]{2.409pt}{0.350pt}}
\put(1350,292){\rule[-0.175pt]{2.409pt}{0.350pt}}
\put(264,298){\rule[-0.175pt]{2.409pt}{0.350pt}}
\put(1350,298){\rule[-0.175pt]{2.409pt}{0.350pt}}
\put(264,303){\rule[-0.175pt]{2.409pt}{0.350pt}}
\put(1350,303){\rule[-0.175pt]{2.409pt}{0.350pt}}
\put(264,308){\rule[-0.175pt]{2.409pt}{0.350pt}}
\put(1350,308){\rule[-0.175pt]{2.409pt}{0.350pt}}
\put(264,312){\rule[-0.175pt]{2.409pt}{0.350pt}}
\put(1350,312){\rule[-0.175pt]{2.409pt}{0.350pt}}
\put(264,315){\rule[-0.175pt]{4.818pt}{0.350pt}}
\put(242,315){\makebox(0,0)[r]{$10^{\gnulog 1e-03}$}}
\put(1340,315){\rule[-0.175pt]{4.818pt}{0.350pt}}
\put(264,339){\rule[-0.175pt]{2.409pt}{0.350pt}}
\put(1350,339){\rule[-0.175pt]{2.409pt}{0.350pt}}
\put(264,353){\rule[-0.175pt]{2.409pt}{0.350pt}}
\put(1350,353){\rule[-0.175pt]{2.409pt}{0.350pt}}
\put(264,363){\rule[-0.175pt]{2.409pt}{0.350pt}}
\put(1350,363){\rule[-0.175pt]{2.409pt}{0.350pt}}
\put(264,370){\rule[-0.175pt]{2.409pt}{0.350pt}}
\put(1350,370){\rule[-0.175pt]{2.409pt}{0.350pt}}
\put(264,376){\rule[-0.175pt]{2.409pt}{0.350pt}}
\put(1350,376){\rule[-0.175pt]{2.409pt}{0.350pt}}
\put(264,382){\rule[-0.175pt]{2.409pt}{0.350pt}}
\put(1350,382){\rule[-0.175pt]{2.409pt}{0.350pt}}
\put(264,386){\rule[-0.175pt]{2.409pt}{0.350pt}}
\put(1350,386){\rule[-0.175pt]{2.409pt}{0.350pt}}
\put(264,390){\rule[-0.175pt]{2.409pt}{0.350pt}}
\put(1350,390){\rule[-0.175pt]{2.409pt}{0.350pt}}
\put(264,394){\rule[-0.175pt]{4.818pt}{0.350pt}}
\put(242,394){\makebox(0,0)[r]{$10^{\gnulog 1e-02}$}}
\put(1340,394){\rule[-0.175pt]{4.818pt}{0.350pt}}
\put(264,418){\rule[-0.175pt]{2.409pt}{0.350pt}}
\put(1350,418){\rule[-0.175pt]{2.409pt}{0.350pt}}
\put(264,431){\rule[-0.175pt]{2.409pt}{0.350pt}}
\put(1350,431){\rule[-0.175pt]{2.409pt}{0.350pt}}
\put(264,441){\rule[-0.175pt]{2.409pt}{0.350pt}}
\put(1350,441){\rule[-0.175pt]{2.409pt}{0.350pt}}
\put(264,449){\rule[-0.175pt]{2.409pt}{0.350pt}}
\put(1350,449){\rule[-0.175pt]{2.409pt}{0.350pt}}
\put(264,455){\rule[-0.175pt]{2.409pt}{0.350pt}}
\put(1350,455){\rule[-0.175pt]{2.409pt}{0.350pt}}
\put(264,460){\rule[-0.175pt]{2.409pt}{0.350pt}}
\put(1350,460){\rule[-0.175pt]{2.409pt}{0.350pt}}
\put(264,465){\rule[-0.175pt]{2.409pt}{0.350pt}}
\put(1350,465){\rule[-0.175pt]{2.409pt}{0.350pt}}
\put(264,469){\rule[-0.175pt]{2.409pt}{0.350pt}}
\put(1350,469){\rule[-0.175pt]{2.409pt}{0.350pt}}
\put(264,473){\rule[-0.175pt]{4.818pt}{0.350pt}}
\put(242,473){\makebox(0,0)[r]{$10^{\gnulog 1e-01}$}}
\put(1340,473){\rule[-0.175pt]{4.818pt}{0.350pt}}
\put(264,496){\rule[-0.175pt]{2.409pt}{0.350pt}}
\put(1350,496){\rule[-0.175pt]{2.409pt}{0.350pt}}
\put(264,510){\rule[-0.175pt]{2.409pt}{0.350pt}}
\put(1350,510){\rule[-0.175pt]{2.409pt}{0.350pt}}
\put(264,520){\rule[-0.175pt]{2.409pt}{0.350pt}}
\put(1350,520){\rule[-0.175pt]{2.409pt}{0.350pt}}
\put(264,527){\rule[-0.175pt]{2.409pt}{0.350pt}}
\put(1350,527){\rule[-0.175pt]{2.409pt}{0.350pt}}
\put(264,534){\rule[-0.175pt]{2.409pt}{0.350pt}}
\put(1350,534){\rule[-0.175pt]{2.409pt}{0.350pt}}
\put(264,539){\rule[-0.175pt]{2.409pt}{0.350pt}}
\put(1350,539){\rule[-0.175pt]{2.409pt}{0.350pt}}
\put(264,544){\rule[-0.175pt]{2.409pt}{0.350pt}}
\put(1350,544){\rule[-0.175pt]{2.409pt}{0.350pt}}
\put(264,548){\rule[-0.175pt]{2.409pt}{0.350pt}}
\put(1350,548){\rule[-0.175pt]{2.409pt}{0.350pt}}
\put(264,551){\rule[-0.175pt]{4.818pt}{0.350pt}}
\put(242,551){\makebox(0,0)[r]{$10^{\gnulog 1e+00}$}}
\put(1340,551){\rule[-0.175pt]{4.818pt}{0.350pt}}
\put(264,575){\rule[-0.175pt]{2.409pt}{0.350pt}}
\put(1350,575){\rule[-0.175pt]{2.409pt}{0.350pt}}
\put(264,589){\rule[-0.175pt]{2.409pt}{0.350pt}}
\put(1350,589){\rule[-0.175pt]{2.409pt}{0.350pt}}
\put(264,598){\rule[-0.175pt]{2.409pt}{0.350pt}}
\put(1350,598){\rule[-0.175pt]{2.409pt}{0.350pt}}
\put(264,606){\rule[-0.175pt]{2.409pt}{0.350pt}}
\put(1350,606){\rule[-0.175pt]{2.409pt}{0.350pt}}
\put(264,612){\rule[-0.175pt]{2.409pt}{0.350pt}}
\put(1350,612){\rule[-0.175pt]{2.409pt}{0.350pt}}
\put(264,618){\rule[-0.175pt]{2.409pt}{0.350pt}}
\put(1350,618){\rule[-0.175pt]{2.409pt}{0.350pt}}
\put(264,622){\rule[-0.175pt]{2.409pt}{0.350pt}}
\put(1350,622){\rule[-0.175pt]{2.409pt}{0.350pt}}
\put(264,626){\rule[-0.175pt]{2.409pt}{0.350pt}}
\put(1350,626){\rule[-0.175pt]{2.409pt}{0.350pt}}
\put(264,630){\rule[-0.175pt]{4.818pt}{0.350pt}}
\put(242,630){\makebox(0,0)[r]{$10^{\gnulog 1e+01}$}}
\put(1340,630){\rule[-0.175pt]{4.818pt}{0.350pt}}
\put(264,653){\rule[-0.175pt]{2.409pt}{0.350pt}}
\put(1350,653){\rule[-0.175pt]{2.409pt}{0.350pt}}
\put(264,667){\rule[-0.175pt]{2.409pt}{0.350pt}}
\put(1350,667){\rule[-0.175pt]{2.409pt}{0.350pt}}
\put(264,677){\rule[-0.175pt]{2.409pt}{0.350pt}}
\put(1350,677){\rule[-0.175pt]{2.409pt}{0.350pt}}
\put(264,685){\rule[-0.175pt]{2.409pt}{0.350pt}}
\put(1350,685){\rule[-0.175pt]{2.409pt}{0.350pt}}
\put(264,691){\rule[-0.175pt]{2.409pt}{0.350pt}}
\put(1350,691){\rule[-0.175pt]{2.409pt}{0.350pt}}
\put(264,696){\rule[-0.175pt]{2.409pt}{0.350pt}}
\put(1350,696){\rule[-0.175pt]{2.409pt}{0.350pt}}
\put(264,701){\rule[-0.175pt]{2.409pt}{0.350pt}}
\put(1350,701){\rule[-0.175pt]{2.409pt}{0.350pt}}
\put(264,705){\rule[-0.175pt]{2.409pt}{0.350pt}}
\put(1350,705){\rule[-0.175pt]{2.409pt}{0.350pt}}
\put(264,708){\rule[-0.175pt]{4.818pt}{0.350pt}}
\put(242,708){\makebox(0,0)[r]{$10^{\gnulog 1e+02}$}}
\put(1340,708){\rule[-0.175pt]{4.818pt}{0.350pt}}
\put(264,732){\rule[-0.175pt]{2.409pt}{0.350pt}}
\put(1350,732){\rule[-0.175pt]{2.409pt}{0.350pt}}
\put(264,746){\rule[-0.175pt]{2.409pt}{0.350pt}}
\put(1350,746){\rule[-0.175pt]{2.409pt}{0.350pt}}
\put(264,756){\rule[-0.175pt]{2.409pt}{0.350pt}}
\put(1350,756){\rule[-0.175pt]{2.409pt}{0.350pt}}
\put(264,763){\rule[-0.175pt]{2.409pt}{0.350pt}}
\put(1350,763){\rule[-0.175pt]{2.409pt}{0.350pt}}
\put(264,770){\rule[-0.175pt]{2.409pt}{0.350pt}}
\put(1350,770){\rule[-0.175pt]{2.409pt}{0.350pt}}
\put(264,775){\rule[-0.175pt]{2.409pt}{0.350pt}}
\put(1350,775){\rule[-0.175pt]{2.409pt}{0.350pt}}
\put(264,779){\rule[-0.175pt]{2.409pt}{0.350pt}}
\put(1350,779){\rule[-0.175pt]{2.409pt}{0.350pt}}
\put(264,783){\rule[-0.175pt]{2.409pt}{0.350pt}}
\put(1350,783){\rule[-0.175pt]{2.409pt}{0.350pt}}
\put(264,787){\rule[-0.175pt]{4.818pt}{0.350pt}}
\put(242,787){\makebox(0,0)[r]{$10^{\gnulog 1e+03}$}}
\put(1340,787){\rule[-0.175pt]{4.818pt}{0.350pt}}
\put(264,158){\rule[-0.175pt]{0.350pt}{4.818pt}}
\put(264,113){\makebox(0,0){$0$}}
\put(264,767){\rule[-0.175pt]{0.350pt}{4.818pt}}
\put(401,158){\rule[-0.175pt]{0.350pt}{4.818pt}}
\put(401,113){\makebox(0,0){$10$}}
\put(401,767){\rule[-0.175pt]{0.350pt}{4.818pt}}
\put(538,158){\rule[-0.175pt]{0.350pt}{4.818pt}}
\put(538,113){\makebox(0,0){$20$}}
\put(538,767){\rule[-0.175pt]{0.350pt}{4.818pt}}
\put(675,158){\rule[-0.175pt]{0.350pt}{4.818pt}}
\put(675,113){\makebox(0,0){$30$}}
\put(675,767){\rule[-0.175pt]{0.350pt}{4.818pt}}
\put(812,158){\rule[-0.175pt]{0.350pt}{4.818pt}}
\put(812,113){\makebox(0,0){$40$}}
\put(812,767){\rule[-0.175pt]{0.350pt}{4.818pt}}
\put(949,158){\rule[-0.175pt]{0.350pt}{4.818pt}}
\put(949,113){\makebox(0,0){$50$}}
\put(949,767){\rule[-0.175pt]{0.350pt}{4.818pt}}
\put(1086,158){\rule[-0.175pt]{0.350pt}{4.818pt}}
\put(1086,113){\makebox(0,0){$60$}}
\put(1086,767){\rule[-0.175pt]{0.350pt}{4.818pt}}
\put(1223,158){\rule[-0.175pt]{0.350pt}{4.818pt}}
\put(1223,113){\makebox(0,0){$70$}}
\put(1223,767){\rule[-0.175pt]{0.350pt}{4.818pt}}
\put(1360,158){\rule[-0.175pt]{0.350pt}{4.818pt}}
\put(1360,113){\makebox(0,0){$80$}}
\put(1360,767){\rule[-0.175pt]{0.350pt}{4.818pt}}
\put(264,158){\rule[-0.175pt]{264.026pt}{0.350pt}}
\put(1360,158){\rule[-0.175pt]{0.350pt}{151.526pt}}
\put(264,787){\rule[-0.175pt]{264.026pt}{0.350pt}}
\put(-87,472){\makebox(0,0)[l]{\shortstack{$|\rho(p_i)-\rho(\phat)|$}}}
\put(812,68){\makebox(0,0){Step $i$}}
\put(264,158){\rule[-0.175pt]{0.350pt}{151.526pt}}
\put(264,735){\usebox{\plotpoint}}
\put(264,735){\rule[-0.175pt]{0.562pt}{0.350pt}}
\put(266,734){\rule[-0.175pt]{0.562pt}{0.350pt}}
\put(268,733){\rule[-0.175pt]{0.562pt}{0.350pt}}
\put(271,732){\rule[-0.175pt]{0.562pt}{0.350pt}}
\put(273,731){\rule[-0.175pt]{0.562pt}{0.350pt}}
\put(275,730){\rule[-0.175pt]{0.562pt}{0.350pt}}
\put(278,729){\rule[-0.175pt]{3.132pt}{0.350pt}}
\put(291,726){\rule[-0.175pt]{0.350pt}{0.447pt}}
\put(292,724){\rule[-0.175pt]{0.350pt}{0.447pt}}
\put(293,722){\rule[-0.175pt]{0.350pt}{0.447pt}}
\put(294,720){\rule[-0.175pt]{0.350pt}{0.447pt}}
\put(295,718){\rule[-0.175pt]{0.350pt}{0.447pt}}
\put(296,716){\rule[-0.175pt]{0.350pt}{0.447pt}}
\put(297,715){\rule[-0.175pt]{0.350pt}{0.447pt}}
\put(298,713){\rule[-0.175pt]{0.350pt}{0.447pt}}
\put(299,711){\rule[-0.175pt]{0.350pt}{0.447pt}}
\put(300,709){\rule[-0.175pt]{0.350pt}{0.447pt}}
\put(301,707){\rule[-0.175pt]{0.350pt}{0.447pt}}
\put(302,705){\rule[-0.175pt]{0.350pt}{0.447pt}}
\put(303,703){\rule[-0.175pt]{0.350pt}{0.447pt}}
\put(304,702){\rule[-0.175pt]{0.350pt}{0.447pt}}
\put(305,702){\usebox{\plotpoint}}
\put(305,702){\usebox{\plotpoint}}
\put(306,701){\usebox{\plotpoint}}
\put(307,700){\usebox{\plotpoint}}
\put(308,699){\usebox{\plotpoint}}
\put(310,698){\usebox{\plotpoint}}
\put(311,697){\usebox{\plotpoint}}
\put(312,696){\usebox{\plotpoint}}
\put(313,695){\usebox{\plotpoint}}
\put(315,694){\usebox{\plotpoint}}
\put(316,693){\usebox{\plotpoint}}
\put(317,692){\usebox{\plotpoint}}
\put(319,691){\usebox{\plotpoint}}
\put(320,690){\usebox{\plotpoint}}
\put(321,689){\usebox{\plotpoint}}
\put(322,688){\usebox{\plotpoint}}
\put(323,687){\usebox{\plotpoint}}
\put(324,686){\usebox{\plotpoint}}
\put(325,685){\usebox{\plotpoint}}
\put(327,684){\usebox{\plotpoint}}
\put(328,683){\usebox{\plotpoint}}
\put(329,682){\usebox{\plotpoint}}
\put(330,681){\usebox{\plotpoint}}
\put(331,680){\usebox{\plotpoint}}
\put(332,679){\rule[-0.175pt]{0.391pt}{0.350pt}}
\put(334,678){\rule[-0.175pt]{0.391pt}{0.350pt}}
\put(336,677){\rule[-0.175pt]{0.391pt}{0.350pt}}
\put(337,676){\rule[-0.175pt]{0.391pt}{0.350pt}}
\put(339,675){\rule[-0.175pt]{0.391pt}{0.350pt}}
\put(341,674){\rule[-0.175pt]{0.391pt}{0.350pt}}
\put(342,673){\rule[-0.175pt]{0.391pt}{0.350pt}}
\put(344,672){\rule[-0.175pt]{0.391pt}{0.350pt}}
\put(346,671){\rule[-0.175pt]{0.422pt}{0.350pt}}
\put(347,670){\rule[-0.175pt]{0.422pt}{0.350pt}}
\put(349,669){\rule[-0.175pt]{0.422pt}{0.350pt}}
\put(351,668){\rule[-0.175pt]{0.422pt}{0.350pt}}
\put(353,667){\rule[-0.175pt]{0.422pt}{0.350pt}}
\put(354,666){\rule[-0.175pt]{0.422pt}{0.350pt}}
\put(356,665){\rule[-0.175pt]{0.422pt}{0.350pt}}
\put(358,664){\rule[-0.175pt]{0.422pt}{0.350pt}}
\put(360,663){\rule[-0.175pt]{0.843pt}{0.350pt}}
\put(363,662){\rule[-0.175pt]{0.843pt}{0.350pt}}
\put(367,661){\rule[-0.175pt]{0.843pt}{0.350pt}}
\put(370,660){\rule[-0.175pt]{0.843pt}{0.350pt}}
\put(374,659){\rule[-0.175pt]{1.566pt}{0.350pt}}
\put(380,658){\rule[-0.175pt]{1.566pt}{0.350pt}}
\put(387,657){\rule[-0.175pt]{1.686pt}{0.350pt}}
\put(394,656){\rule[-0.175pt]{1.686pt}{0.350pt}}
\put(401,655){\rule[-0.175pt]{3.373pt}{0.350pt}}
\put(415,654){\rule[-0.175pt]{0.522pt}{0.350pt}}
\put(417,653){\rule[-0.175pt]{0.522pt}{0.350pt}}
\put(419,652){\rule[-0.175pt]{0.522pt}{0.350pt}}
\put(421,651){\rule[-0.175pt]{0.522pt}{0.350pt}}
\put(423,650){\rule[-0.175pt]{0.522pt}{0.350pt}}
\put(425,649){\rule[-0.175pt]{0.522pt}{0.350pt}}
\put(427,648){\usebox{\plotpoint}}
\put(429,647){\usebox{\plotpoint}}
\put(430,646){\usebox{\plotpoint}}
\put(431,645){\usebox{\plotpoint}}
\put(432,644){\usebox{\plotpoint}}
\put(433,643){\usebox{\plotpoint}}
\put(434,642){\usebox{\plotpoint}}
\put(436,641){\usebox{\plotpoint}}
\put(437,640){\usebox{\plotpoint}}
\put(438,639){\usebox{\plotpoint}}
\put(439,638){\usebox{\plotpoint}}
\put(440,637){\usebox{\plotpoint}}
\put(441,636){\usebox{\plotpoint}}
\put(442,634){\rule[-0.175pt]{0.350pt}{0.361pt}}
\put(443,633){\rule[-0.175pt]{0.350pt}{0.361pt}}
\put(444,631){\rule[-0.175pt]{0.350pt}{0.361pt}}
\put(445,630){\rule[-0.175pt]{0.350pt}{0.361pt}}
\put(446,628){\rule[-0.175pt]{0.350pt}{0.361pt}}
\put(447,627){\rule[-0.175pt]{0.350pt}{0.361pt}}
\put(448,625){\rule[-0.175pt]{0.350pt}{0.361pt}}
\put(449,624){\rule[-0.175pt]{0.350pt}{0.361pt}}
\put(450,622){\rule[-0.175pt]{0.350pt}{0.361pt}}
\put(451,621){\rule[-0.175pt]{0.350pt}{0.361pt}}
\put(452,619){\rule[-0.175pt]{0.350pt}{0.361pt}}
\put(453,618){\rule[-0.175pt]{0.350pt}{0.361pt}}
\put(454,616){\rule[-0.175pt]{0.350pt}{0.361pt}}
\put(455,615){\rule[-0.175pt]{0.350pt}{0.361pt}}
\put(456,613){\usebox{\plotpoint}}
\put(457,612){\usebox{\plotpoint}}
\put(458,611){\usebox{\plotpoint}}
\put(459,610){\usebox{\plotpoint}}
\put(460,608){\usebox{\plotpoint}}
\put(461,607){\usebox{\plotpoint}}
\put(462,606){\usebox{\plotpoint}}
\put(463,605){\usebox{\plotpoint}}
\put(464,604){\usebox{\plotpoint}}
\put(465,602){\usebox{\plotpoint}}
\put(466,601){\usebox{\plotpoint}}
\put(467,600){\usebox{\plotpoint}}
\put(468,599){\usebox{\plotpoint}}
\put(469,598){\usebox{\plotpoint}}
\put(470,596){\usebox{\plotpoint}}
\put(471,595){\usebox{\plotpoint}}
\put(472,594){\usebox{\plotpoint}}
\put(473,593){\usebox{\plotpoint}}
\put(474,592){\usebox{\plotpoint}}
\put(475,591){\usebox{\plotpoint}}
\put(476,589){\usebox{\plotpoint}}
\put(477,588){\usebox{\plotpoint}}
\put(478,587){\usebox{\plotpoint}}
\put(479,586){\usebox{\plotpoint}}
\put(480,585){\usebox{\plotpoint}}
\put(481,584){\usebox{\plotpoint}}
\put(482,583){\usebox{\plotpoint}}
\put(483,581){\rule[-0.175pt]{0.350pt}{0.361pt}}
\put(484,580){\rule[-0.175pt]{0.350pt}{0.361pt}}
\put(485,578){\rule[-0.175pt]{0.350pt}{0.361pt}}
\put(486,577){\rule[-0.175pt]{0.350pt}{0.361pt}}
\put(487,575){\rule[-0.175pt]{0.350pt}{0.361pt}}
\put(488,574){\rule[-0.175pt]{0.350pt}{0.361pt}}
\put(489,572){\rule[-0.175pt]{0.350pt}{0.361pt}}
\put(490,571){\rule[-0.175pt]{0.350pt}{0.361pt}}
\put(491,569){\rule[-0.175pt]{0.350pt}{0.361pt}}
\put(492,568){\rule[-0.175pt]{0.350pt}{0.361pt}}
\put(493,566){\rule[-0.175pt]{0.350pt}{0.361pt}}
\put(494,565){\rule[-0.175pt]{0.350pt}{0.361pt}}
\put(495,563){\rule[-0.175pt]{0.350pt}{0.361pt}}
\put(496,562){\rule[-0.175pt]{0.350pt}{0.361pt}}
\put(497,562){\usebox{\plotpoint}}
\put(498,561){\usebox{\plotpoint}}
\put(499,560){\usebox{\plotpoint}}
\put(500,559){\usebox{\plotpoint}}
\put(501,558){\usebox{\plotpoint}}
\put(502,557){\usebox{\plotpoint}}
\put(503,556){\usebox{\plotpoint}}
\put(505,555){\usebox{\plotpoint}}
\put(506,554){\usebox{\plotpoint}}
\put(507,553){\usebox{\plotpoint}}
\put(508,552){\usebox{\plotpoint}}
\put(509,551){\usebox{\plotpoint}}
\put(510,550){\usebox{\plotpoint}}
\put(511,548){\usebox{\plotpoint}}
\put(512,547){\usebox{\plotpoint}}
\put(513,546){\usebox{\plotpoint}}
\put(514,545){\usebox{\plotpoint}}
\put(515,544){\usebox{\plotpoint}}
\put(516,543){\usebox{\plotpoint}}
\put(517,542){\usebox{\plotpoint}}
\put(518,541){\usebox{\plotpoint}}
\put(519,540){\usebox{\plotpoint}}
\put(520,539){\usebox{\plotpoint}}
\put(521,538){\usebox{\plotpoint}}
\put(522,537){\usebox{\plotpoint}}
\put(523,536){\usebox{\plotpoint}}
\put(524,534){\usebox{\plotpoint}}
\put(525,533){\usebox{\plotpoint}}
\put(526,532){\usebox{\plotpoint}}
\put(527,531){\usebox{\plotpoint}}
\put(528,529){\usebox{\plotpoint}}
\put(529,528){\usebox{\plotpoint}}
\put(530,527){\usebox{\plotpoint}}
\put(531,526){\usebox{\plotpoint}}
\put(532,525){\usebox{\plotpoint}}
\put(533,523){\usebox{\plotpoint}}
\put(534,522){\usebox{\plotpoint}}
\put(535,521){\usebox{\plotpoint}}
\put(536,520){\usebox{\plotpoint}}
\put(537,519){\usebox{\plotpoint}}
\put(538,517){\usebox{\plotpoint}}
\put(539,516){\usebox{\plotpoint}}
\put(540,515){\usebox{\plotpoint}}
\put(541,514){\usebox{\plotpoint}}
\put(542,513){\usebox{\plotpoint}}
\put(543,512){\usebox{\plotpoint}}
\put(544,510){\usebox{\plotpoint}}
\put(545,509){\usebox{\plotpoint}}
\put(546,508){\usebox{\plotpoint}}
\put(547,507){\usebox{\plotpoint}}
\put(548,506){\usebox{\plotpoint}}
\put(549,505){\usebox{\plotpoint}}
\put(550,504){\usebox{\plotpoint}}
\put(551,503){\usebox{\plotpoint}}
\put(552,503){\usebox{\plotpoint}}
\put(553,502){\usebox{\plotpoint}}
\put(554,501){\usebox{\plotpoint}}
\put(556,500){\usebox{\plotpoint}}
\put(557,499){\usebox{\plotpoint}}
\put(559,498){\usebox{\plotpoint}}
\put(560,497){\usebox{\plotpoint}}
\put(562,496){\usebox{\plotpoint}}
\put(563,495){\usebox{\plotpoint}}
\put(565,494){\usebox{\plotpoint}}
\put(566,493){\usebox{\plotpoint}}
\put(567,492){\usebox{\plotpoint}}
\put(568,491){\usebox{\plotpoint}}
\put(569,490){\usebox{\plotpoint}}
\put(570,489){\usebox{\plotpoint}}
\put(571,488){\usebox{\plotpoint}}
\put(572,487){\usebox{\plotpoint}}
\put(573,486){\usebox{\plotpoint}}
\put(574,485){\usebox{\plotpoint}}
\put(575,484){\usebox{\plotpoint}}
\put(576,483){\usebox{\plotpoint}}
\put(577,482){\usebox{\plotpoint}}
\put(578,481){\usebox{\plotpoint}}
\put(579,480){\usebox{\plotpoint}}
\put(580,479){\usebox{\plotpoint}}
\put(581,478){\usebox{\plotpoint}}
\put(582,477){\usebox{\plotpoint}}
\put(583,476){\usebox{\plotpoint}}
\put(584,475){\usebox{\plotpoint}}
\put(586,474){\usebox{\plotpoint}}
\put(587,473){\usebox{\plotpoint}}
\put(588,472){\usebox{\plotpoint}}
\put(589,471){\usebox{\plotpoint}}
\put(590,470){\usebox{\plotpoint}}
\put(591,469){\usebox{\plotpoint}}
\put(593,468){\rule[-0.175pt]{0.422pt}{0.350pt}}
\put(594,467){\rule[-0.175pt]{0.422pt}{0.350pt}}
\put(596,466){\rule[-0.175pt]{0.422pt}{0.350pt}}
\put(598,465){\rule[-0.175pt]{0.422pt}{0.350pt}}
\put(600,464){\rule[-0.175pt]{0.422pt}{0.350pt}}
\put(601,463){\rule[-0.175pt]{0.422pt}{0.350pt}}
\put(603,462){\rule[-0.175pt]{0.422pt}{0.350pt}}
\put(605,461){\rule[-0.175pt]{0.422pt}{0.350pt}}
\put(607,460){\rule[-0.175pt]{0.522pt}{0.350pt}}
\put(609,459){\rule[-0.175pt]{0.522pt}{0.350pt}}
\put(611,458){\rule[-0.175pt]{0.522pt}{0.350pt}}
\put(613,457){\rule[-0.175pt]{0.522pt}{0.350pt}}
\put(615,456){\rule[-0.175pt]{0.522pt}{0.350pt}}
\put(617,455){\rule[-0.175pt]{0.522pt}{0.350pt}}
\put(620,454){\rule[-0.175pt]{0.562pt}{0.350pt}}
\put(622,453){\rule[-0.175pt]{0.562pt}{0.350pt}}
\put(624,452){\rule[-0.175pt]{0.562pt}{0.350pt}}
\put(626,451){\rule[-0.175pt]{0.562pt}{0.350pt}}
\put(629,450){\rule[-0.175pt]{0.562pt}{0.350pt}}
\put(631,449){\rule[-0.175pt]{0.562pt}{0.350pt}}
\put(633,448){\rule[-0.175pt]{0.843pt}{0.350pt}}
\put(637,447){\rule[-0.175pt]{0.843pt}{0.350pt}}
\put(641,446){\rule[-0.175pt]{0.843pt}{0.350pt}}
\put(644,445){\rule[-0.175pt]{0.843pt}{0.350pt}}
\put(648,444){\rule[-0.175pt]{0.783pt}{0.350pt}}
\put(651,443){\rule[-0.175pt]{0.783pt}{0.350pt}}
\put(654,442){\rule[-0.175pt]{0.783pt}{0.350pt}}
\put(657,441){\rule[-0.175pt]{0.783pt}{0.350pt}}
\put(661,440){\rule[-0.175pt]{0.675pt}{0.350pt}}
\put(663,439){\rule[-0.175pt]{0.675pt}{0.350pt}}
\put(666,438){\rule[-0.175pt]{0.675pt}{0.350pt}}
\put(669,437){\rule[-0.175pt]{0.675pt}{0.350pt}}
\put(672,436){\rule[-0.175pt]{0.675pt}{0.350pt}}
\put(674,435){\rule[-0.175pt]{1.124pt}{0.350pt}}
\put(679,434){\rule[-0.175pt]{1.124pt}{0.350pt}}
\put(684,433){\rule[-0.175pt]{1.124pt}{0.350pt}}
\put(689,432){\rule[-0.175pt]{0.447pt}{0.350pt}}
\put(690,431){\rule[-0.175pt]{0.447pt}{0.350pt}}
\put(692,430){\rule[-0.175pt]{0.447pt}{0.350pt}}
\put(694,429){\rule[-0.175pt]{0.447pt}{0.350pt}}
\put(696,428){\rule[-0.175pt]{0.447pt}{0.350pt}}
\put(698,427){\rule[-0.175pt]{0.447pt}{0.350pt}}
\put(700,426){\rule[-0.175pt]{0.447pt}{0.350pt}}
\put(701,425){\rule[-0.175pt]{0.562pt}{0.350pt}}
\put(704,424){\rule[-0.175pt]{0.562pt}{0.350pt}}
\put(706,423){\rule[-0.175pt]{0.562pt}{0.350pt}}
\put(708,422){\rule[-0.175pt]{0.562pt}{0.350pt}}
\put(711,421){\rule[-0.175pt]{0.562pt}{0.350pt}}
\put(713,420){\rule[-0.175pt]{0.562pt}{0.350pt}}
\put(715,419){\rule[-0.175pt]{0.843pt}{0.350pt}}
\put(719,418){\rule[-0.175pt]{0.843pt}{0.350pt}}
\put(723,417){\rule[-0.175pt]{0.843pt}{0.350pt}}
\put(726,416){\rule[-0.175pt]{0.843pt}{0.350pt}}
\put(730,415){\rule[-0.175pt]{0.482pt}{0.350pt}}
\put(732,414){\rule[-0.175pt]{0.482pt}{0.350pt}}
\put(734,413){\rule[-0.175pt]{0.482pt}{0.350pt}}
\put(736,412){\rule[-0.175pt]{0.482pt}{0.350pt}}
\put(738,411){\rule[-0.175pt]{0.482pt}{0.350pt}}
\put(740,410){\rule[-0.175pt]{0.482pt}{0.350pt}}
\put(742,409){\rule[-0.175pt]{0.482pt}{0.350pt}}
\put(744,408){\rule[-0.175pt]{0.447pt}{0.350pt}}
\put(745,407){\rule[-0.175pt]{0.447pt}{0.350pt}}
\put(747,406){\rule[-0.175pt]{0.447pt}{0.350pt}}
\put(749,405){\rule[-0.175pt]{0.447pt}{0.350pt}}
\put(751,404){\rule[-0.175pt]{0.447pt}{0.350pt}}
\put(753,403){\rule[-0.175pt]{0.447pt}{0.350pt}}
\put(755,402){\rule[-0.175pt]{0.447pt}{0.350pt}}
\put(756,401){\rule[-0.175pt]{0.562pt}{0.350pt}}
\put(759,400){\rule[-0.175pt]{0.562pt}{0.350pt}}
\put(761,399){\rule[-0.175pt]{0.562pt}{0.350pt}}
\put(763,398){\rule[-0.175pt]{0.562pt}{0.350pt}}
\put(766,397){\rule[-0.175pt]{0.562pt}{0.350pt}}
\put(768,396){\rule[-0.175pt]{0.562pt}{0.350pt}}
\put(770,395){\usebox{\plotpoint}}
\put(772,394){\usebox{\plotpoint}}
\put(773,393){\usebox{\plotpoint}}
\put(774,392){\usebox{\plotpoint}}
\put(776,391){\usebox{\plotpoint}}
\put(777,390){\usebox{\plotpoint}}
\put(778,389){\usebox{\plotpoint}}
\put(779,388){\usebox{\plotpoint}}
\put(781,387){\usebox{\plotpoint}}
\put(782,386){\usebox{\plotpoint}}
\put(783,385){\usebox{\plotpoint}}
\put(784,384){\rule[-0.175pt]{0.447pt}{0.350pt}}
\put(786,383){\rule[-0.175pt]{0.447pt}{0.350pt}}
\put(788,382){\rule[-0.175pt]{0.447pt}{0.350pt}}
\put(790,381){\rule[-0.175pt]{0.447pt}{0.350pt}}
\put(792,380){\rule[-0.175pt]{0.447pt}{0.350pt}}
\put(794,379){\rule[-0.175pt]{0.447pt}{0.350pt}}
\put(796,378){\rule[-0.175pt]{0.447pt}{0.350pt}}
\put(797,377){\rule[-0.175pt]{0.375pt}{0.350pt}}
\put(799,376){\rule[-0.175pt]{0.375pt}{0.350pt}}
\put(801,375){\rule[-0.175pt]{0.375pt}{0.350pt}}
\put(802,374){\rule[-0.175pt]{0.375pt}{0.350pt}}
\put(804,373){\rule[-0.175pt]{0.375pt}{0.350pt}}
\put(805,372){\rule[-0.175pt]{0.375pt}{0.350pt}}
\put(807,371){\rule[-0.175pt]{0.375pt}{0.350pt}}
\put(808,370){\rule[-0.175pt]{0.375pt}{0.350pt}}
\put(810,369){\rule[-0.175pt]{0.375pt}{0.350pt}}
\put(811,368){\usebox{\plotpoint}}
\put(812,368){\rule[-0.175pt]{0.350pt}{4.061pt}}
\put(813,384){\rule[-0.175pt]{0.350pt}{4.061pt}}
\put(814,401){\rule[-0.175pt]{0.350pt}{4.061pt}}
\put(815,418){\rule[-0.175pt]{0.350pt}{4.061pt}}
\put(816,435){\rule[-0.175pt]{0.350pt}{4.061pt}}
\put(817,452){\rule[-0.175pt]{0.350pt}{4.061pt}}
\put(818,469){\rule[-0.175pt]{0.350pt}{4.061pt}}
\put(819,486){\rule[-0.175pt]{0.350pt}{4.061pt}}
\put(820,502){\rule[-0.175pt]{0.350pt}{4.061pt}}
\put(821,519){\rule[-0.175pt]{0.350pt}{4.061pt}}
\put(822,536){\rule[-0.175pt]{0.350pt}{4.061pt}}
\put(823,553){\rule[-0.175pt]{0.350pt}{4.061pt}}
\put(824,570){\rule[-0.175pt]{0.350pt}{4.061pt}}
\put(825,587){\rule[-0.175pt]{0.350pt}{4.061pt}}
\put(826,603){\usebox{\plotpoint}}
\put(826,596){\rule[-0.175pt]{0.350pt}{1.816pt}}
\put(827,588){\rule[-0.175pt]{0.350pt}{1.816pt}}
\put(828,581){\rule[-0.175pt]{0.350pt}{1.816pt}}
\put(829,573){\rule[-0.175pt]{0.350pt}{1.816pt}}
\put(830,566){\rule[-0.175pt]{0.350pt}{1.816pt}}
\put(831,558){\rule[-0.175pt]{0.350pt}{1.816pt}}
\put(832,551){\rule[-0.175pt]{0.350pt}{1.816pt}}
\put(833,543){\rule[-0.175pt]{0.350pt}{1.816pt}}
\put(834,536){\rule[-0.175pt]{0.350pt}{1.816pt}}
\put(835,528){\rule[-0.175pt]{0.350pt}{1.816pt}}
\put(836,521){\rule[-0.175pt]{0.350pt}{1.816pt}}
\put(837,513){\rule[-0.175pt]{0.350pt}{1.816pt}}
\put(838,506){\rule[-0.175pt]{0.350pt}{1.816pt}}
\put(839,506){\usebox{\plotpoint}}
\put(839,506){\rule[-0.175pt]{0.562pt}{0.350pt}}
\put(841,505){\rule[-0.175pt]{0.562pt}{0.350pt}}
\put(843,504){\rule[-0.175pt]{0.562pt}{0.350pt}}
\put(845,503){\rule[-0.175pt]{0.562pt}{0.350pt}}
\put(848,502){\rule[-0.175pt]{0.562pt}{0.350pt}}
\put(850,501){\rule[-0.175pt]{0.562pt}{0.350pt}}
\put(852,500){\rule[-0.175pt]{0.562pt}{0.350pt}}
\put(855,499){\rule[-0.175pt]{0.562pt}{0.350pt}}
\put(857,498){\rule[-0.175pt]{0.562pt}{0.350pt}}
\put(859,497){\rule[-0.175pt]{0.562pt}{0.350pt}}
\put(862,496){\rule[-0.175pt]{0.562pt}{0.350pt}}
\put(864,495){\rule[-0.175pt]{0.562pt}{0.350pt}}
\put(866,494){\rule[-0.175pt]{0.422pt}{0.350pt}}
\put(868,493){\rule[-0.175pt]{0.422pt}{0.350pt}}
\put(870,492){\rule[-0.175pt]{0.422pt}{0.350pt}}
\put(872,491){\rule[-0.175pt]{0.422pt}{0.350pt}}
\put(874,490){\rule[-0.175pt]{0.422pt}{0.350pt}}
\put(875,489){\rule[-0.175pt]{0.422pt}{0.350pt}}
\put(877,488){\rule[-0.175pt]{0.422pt}{0.350pt}}
\put(879,487){\rule[-0.175pt]{0.422pt}{0.350pt}}
\put(881,486){\rule[-0.175pt]{0.522pt}{0.350pt}}
\put(883,485){\rule[-0.175pt]{0.522pt}{0.350pt}}
\put(885,484){\rule[-0.175pt]{0.522pt}{0.350pt}}
\put(887,483){\rule[-0.175pt]{0.522pt}{0.350pt}}
\put(889,482){\rule[-0.175pt]{0.522pt}{0.350pt}}
\put(891,481){\rule[-0.175pt]{0.522pt}{0.350pt}}
\put(894,480){\rule[-0.175pt]{0.843pt}{0.350pt}}
\put(897,479){\rule[-0.175pt]{0.843pt}{0.350pt}}
\put(901,478){\rule[-0.175pt]{0.843pt}{0.350pt}}
\put(904,477){\rule[-0.175pt]{0.843pt}{0.350pt}}
\put(908,476){\rule[-0.175pt]{1.124pt}{0.350pt}}
\put(912,475){\rule[-0.175pt]{1.124pt}{0.350pt}}
\put(917,474){\rule[-0.175pt]{1.124pt}{0.350pt}}
\put(922,473){\rule[-0.175pt]{1.044pt}{0.350pt}}
\put(926,472){\rule[-0.175pt]{1.044pt}{0.350pt}}
\put(930,471){\rule[-0.175pt]{1.044pt}{0.350pt}}
\put(934,470){\rule[-0.175pt]{0.675pt}{0.350pt}}
\put(937,469){\rule[-0.175pt]{0.675pt}{0.350pt}}
\put(940,468){\rule[-0.175pt]{0.675pt}{0.350pt}}
\put(943,467){\rule[-0.175pt]{0.675pt}{0.350pt}}
\put(946,466){\rule[-0.175pt]{0.675pt}{0.350pt}}
\put(948,465){\rule[-0.175pt]{0.562pt}{0.350pt}}
\put(951,464){\rule[-0.175pt]{0.562pt}{0.350pt}}
\put(953,463){\rule[-0.175pt]{0.562pt}{0.350pt}}
\put(955,462){\rule[-0.175pt]{0.562pt}{0.350pt}}
\put(958,461){\rule[-0.175pt]{0.562pt}{0.350pt}}
\put(960,460){\rule[-0.175pt]{0.562pt}{0.350pt}}
\put(962,459){\rule[-0.175pt]{0.626pt}{0.350pt}}
\put(965,458){\rule[-0.175pt]{0.626pt}{0.350pt}}
\put(968,457){\rule[-0.175pt]{0.626pt}{0.350pt}}
\put(970,456){\rule[-0.175pt]{0.626pt}{0.350pt}}
\put(973,455){\rule[-0.175pt]{0.626pt}{0.350pt}}
\put(975,454){\rule[-0.175pt]{0.675pt}{0.350pt}}
\put(978,453){\rule[-0.175pt]{0.675pt}{0.350pt}}
\put(981,452){\rule[-0.175pt]{0.675pt}{0.350pt}}
\put(984,451){\rule[-0.175pt]{0.675pt}{0.350pt}}
\put(987,450){\rule[-0.175pt]{0.675pt}{0.350pt}}
\put(989,449){\rule[-0.175pt]{0.675pt}{0.350pt}}
\put(992,448){\rule[-0.175pt]{0.675pt}{0.350pt}}
\put(995,447){\rule[-0.175pt]{0.675pt}{0.350pt}}
\put(998,446){\rule[-0.175pt]{0.675pt}{0.350pt}}
\put(1001,445){\rule[-0.175pt]{0.675pt}{0.350pt}}
\put(1003,444){\rule[-0.175pt]{0.482pt}{0.350pt}}
\put(1006,443){\rule[-0.175pt]{0.482pt}{0.350pt}}
\put(1008,442){\rule[-0.175pt]{0.482pt}{0.350pt}}
\put(1010,441){\rule[-0.175pt]{0.482pt}{0.350pt}}
\put(1012,440){\rule[-0.175pt]{0.482pt}{0.350pt}}
\put(1014,439){\rule[-0.175pt]{0.482pt}{0.350pt}}
\put(1016,438){\rule[-0.175pt]{0.482pt}{0.350pt}}
\put(1018,437){\rule[-0.175pt]{0.391pt}{0.350pt}}
\put(1019,436){\rule[-0.175pt]{0.391pt}{0.350pt}}
\put(1021,435){\rule[-0.175pt]{0.391pt}{0.350pt}}
\put(1022,434){\rule[-0.175pt]{0.391pt}{0.350pt}}
\put(1024,433){\rule[-0.175pt]{0.391pt}{0.350pt}}
\put(1026,432){\rule[-0.175pt]{0.391pt}{0.350pt}}
\put(1027,431){\rule[-0.175pt]{0.391pt}{0.350pt}}
\put(1029,430){\rule[-0.175pt]{0.391pt}{0.350pt}}
\put(1031,429){\usebox{\plotpoint}}
\put(1032,428){\usebox{\plotpoint}}
\put(1033,427){\usebox{\plotpoint}}
\put(1034,426){\usebox{\plotpoint}}
\put(1035,425){\usebox{\plotpoint}}
\put(1036,424){\usebox{\plotpoint}}
\put(1037,423){\usebox{\plotpoint}}
\put(1038,422){\usebox{\plotpoint}}
\put(1039,421){\usebox{\plotpoint}}
\put(1040,420){\usebox{\plotpoint}}
\put(1041,419){\usebox{\plotpoint}}
\put(1042,418){\usebox{\plotpoint}}
\put(1043,417){\usebox{\plotpoint}}
\put(1044,416){\usebox{\plotpoint}}
\put(1045,413){\usebox{\plotpoint}}
\put(1046,412){\usebox{\plotpoint}}
\put(1047,411){\usebox{\plotpoint}}
\put(1048,410){\usebox{\plotpoint}}
\put(1049,409){\usebox{\plotpoint}}
\put(1050,408){\usebox{\plotpoint}}
\put(1051,407){\usebox{\plotpoint}}
\put(1052,405){\usebox{\plotpoint}}
\put(1053,404){\usebox{\plotpoint}}
\put(1054,403){\usebox{\plotpoint}}
\put(1055,402){\usebox{\plotpoint}}
\put(1056,401){\usebox{\plotpoint}}
\put(1057,400){\usebox{\plotpoint}}
\put(1058,399){\usebox{\plotpoint}}
\put(1059,399){\usebox{\plotpoint}}
\put(1059,399){\usebox{\plotpoint}}
\put(1060,398){\usebox{\plotpoint}}
\put(1061,397){\usebox{\plotpoint}}
\put(1062,396){\usebox{\plotpoint}}
\put(1063,395){\usebox{\plotpoint}}
\put(1064,394){\usebox{\plotpoint}}
\put(1065,393){\usebox{\plotpoint}}
\put(1066,392){\usebox{\plotpoint}}
\put(1067,391){\usebox{\plotpoint}}
\put(1068,390){\usebox{\plotpoint}}
\put(1069,389){\usebox{\plotpoint}}
\put(1070,388){\usebox{\plotpoint}}
\put(1071,387){\usebox{\plotpoint}}
\put(1072,384){\usebox{\plotpoint}}
\put(1073,383){\usebox{\plotpoint}}
\put(1074,382){\usebox{\plotpoint}}
\put(1075,380){\usebox{\plotpoint}}
\put(1076,379){\usebox{\plotpoint}}
\put(1077,378){\usebox{\plotpoint}}
\put(1078,377){\usebox{\plotpoint}}
\put(1079,375){\usebox{\plotpoint}}
\put(1080,374){\usebox{\plotpoint}}
\put(1081,373){\usebox{\plotpoint}}
\put(1082,371){\usebox{\plotpoint}}
\put(1083,370){\usebox{\plotpoint}}
\put(1084,369){\usebox{\plotpoint}}
\put(1085,368){\usebox{\plotpoint}}
\put(1086,368){\usebox{\plotpoint}}
\put(1086,368){\usebox{\plotpoint}}
\put(1087,367){\usebox{\plotpoint}}
\put(1088,366){\usebox{\plotpoint}}
\put(1090,365){\usebox{\plotpoint}}
\put(1091,364){\usebox{\plotpoint}}
\put(1093,363){\usebox{\plotpoint}}
\put(1094,362){\usebox{\plotpoint}}
\put(1095,361){\usebox{\plotpoint}}
\put(1097,360){\usebox{\plotpoint}}
\put(1098,359){\usebox{\plotpoint}}
\put(1100,356){\usebox{\plotpoint}}
\put(1101,355){\usebox{\plotpoint}}
\put(1102,354){\usebox{\plotpoint}}
\put(1103,353){\usebox{\plotpoint}}
\put(1104,352){\usebox{\plotpoint}}
\put(1105,351){\usebox{\plotpoint}}
\put(1106,350){\usebox{\plotpoint}}
\put(1107,349){\usebox{\plotpoint}}
\put(1108,348){\usebox{\plotpoint}}
\put(1109,347){\usebox{\plotpoint}}
\put(1110,346){\usebox{\plotpoint}}
\put(1111,345){\usebox{\plotpoint}}
\put(1112,344){\usebox{\plotpoint}}
\put(1113,344){\rule[-0.175pt]{1.686pt}{0.350pt}}
\put(1120,343){\rule[-0.175pt]{1.686pt}{0.350pt}}
\put(1127,342){\rule[-0.175pt]{1.686pt}{0.350pt}}
\put(1134,341){\rule[-0.175pt]{1.686pt}{0.350pt}}
\put(1141,340){\usebox{\plotpoint}}
\put(1142,339){\usebox{\plotpoint}}
\put(1143,338){\usebox{\plotpoint}}
\put(1144,337){\usebox{\plotpoint}}
\put(1145,336){\usebox{\plotpoint}}
\put(1146,335){\usebox{\plotpoint}}
\put(1147,334){\usebox{\plotpoint}}
\put(1148,333){\usebox{\plotpoint}}
\put(1149,332){\usebox{\plotpoint}}
\put(1150,331){\usebox{\plotpoint}}
\put(1151,330){\usebox{\plotpoint}}
\put(1152,329){\usebox{\plotpoint}}
\put(1153,328){\usebox{\plotpoint}}
\put(1154,327){\rule[-0.175pt]{0.447pt}{0.350pt}}
\put(1156,326){\rule[-0.175pt]{0.447pt}{0.350pt}}
\put(1158,325){\rule[-0.175pt]{0.447pt}{0.350pt}}
\put(1160,324){\rule[-0.175pt]{0.447pt}{0.350pt}}
\put(1162,323){\rule[-0.175pt]{0.447pt}{0.350pt}}
\put(1164,322){\rule[-0.175pt]{0.447pt}{0.350pt}}
\put(1166,321){\rule[-0.175pt]{0.447pt}{0.350pt}}
\put(1168,320){\usebox{\plotpoint}}
\put(1169,319){\usebox{\plotpoint}}
\put(1170,318){\usebox{\plotpoint}}
\put(1172,317){\usebox{\plotpoint}}
\put(1173,316){\usebox{\plotpoint}}
\put(1175,315){\usebox{\plotpoint}}
\put(1176,314){\usebox{\plotpoint}}
\put(1177,313){\usebox{\plotpoint}}
\put(1179,312){\usebox{\plotpoint}}
\put(1180,311){\usebox{\plotpoint}}
\put(1182,310){\rule[-0.175pt]{1.124pt}{0.350pt}}
\put(1186,309){\rule[-0.175pt]{1.124pt}{0.350pt}}
\put(1191,308){\rule[-0.175pt]{1.124pt}{0.350pt}}
\put(1195,307){\usebox{\plotpoint}}
\put(1196,305){\rule[-0.175pt]{0.350pt}{0.389pt}}
\put(1197,303){\rule[-0.175pt]{0.350pt}{0.389pt}}
\put(1198,302){\rule[-0.175pt]{0.350pt}{0.389pt}}
\put(1199,300){\rule[-0.175pt]{0.350pt}{0.389pt}}
\put(1200,298){\rule[-0.175pt]{0.350pt}{0.389pt}}
\put(1201,297){\rule[-0.175pt]{0.350pt}{0.389pt}}
\put(1202,295){\rule[-0.175pt]{0.350pt}{0.389pt}}
\put(1203,294){\rule[-0.175pt]{0.350pt}{0.389pt}}
\put(1204,292){\rule[-0.175pt]{0.350pt}{0.389pt}}
\put(1205,290){\rule[-0.175pt]{0.350pt}{0.389pt}}
\put(1206,289){\rule[-0.175pt]{0.350pt}{0.389pt}}
\put(1207,287){\rule[-0.175pt]{0.350pt}{0.389pt}}
\put(1208,286){\rule[-0.175pt]{0.350pt}{0.389pt}}
\put(1209,284){\usebox{\plotpoint}}
\put(1210,283){\usebox{\plotpoint}}
\put(1211,281){\usebox{\plotpoint}}
\put(1212,280){\usebox{\plotpoint}}
\put(1213,279){\usebox{\plotpoint}}
\put(1214,277){\usebox{\plotpoint}}
\put(1215,276){\usebox{\plotpoint}}
\put(1216,275){\usebox{\plotpoint}}
\put(1217,273){\usebox{\plotpoint}}
\put(1218,272){\usebox{\plotpoint}}
\put(1219,271){\usebox{\plotpoint}}
\put(1220,269){\usebox{\plotpoint}}
\put(1221,268){\usebox{\plotpoint}}
\put(1222,267){\usebox{\plotpoint}}
\put(1223,265){\rule[-0.175pt]{0.350pt}{0.413pt}}
\put(1224,263){\rule[-0.175pt]{0.350pt}{0.413pt}}
\put(1225,261){\rule[-0.175pt]{0.350pt}{0.413pt}}
\put(1226,260){\rule[-0.175pt]{0.350pt}{0.413pt}}
\put(1227,258){\rule[-0.175pt]{0.350pt}{0.413pt}}
\put(1228,256){\rule[-0.175pt]{0.350pt}{0.413pt}}
\put(1229,254){\rule[-0.175pt]{0.350pt}{0.413pt}}
\put(1230,253){\rule[-0.175pt]{0.350pt}{0.413pt}}
\put(1231,251){\rule[-0.175pt]{0.350pt}{0.413pt}}
\put(1232,249){\rule[-0.175pt]{0.350pt}{0.413pt}}
\put(1233,248){\rule[-0.175pt]{0.350pt}{0.413pt}}
\put(1234,246){\rule[-0.175pt]{0.350pt}{0.413pt}}
\put(1235,244){\rule[-0.175pt]{0.350pt}{0.413pt}}
\put(1236,243){\rule[-0.175pt]{0.350pt}{0.413pt}}
\put(1237,241){\rule[-0.175pt]{0.350pt}{0.426pt}}
\put(1238,239){\rule[-0.175pt]{0.350pt}{0.426pt}}
\put(1239,237){\rule[-0.175pt]{0.350pt}{0.426pt}}
\put(1240,235){\rule[-0.175pt]{0.350pt}{0.426pt}}
\put(1241,234){\rule[-0.175pt]{0.350pt}{0.426pt}}
\put(1242,232){\rule[-0.175pt]{0.350pt}{0.426pt}}
\put(1243,230){\rule[-0.175pt]{0.350pt}{0.426pt}}
\put(1244,228){\rule[-0.175pt]{0.350pt}{0.426pt}}
\put(1245,227){\rule[-0.175pt]{0.350pt}{0.426pt}}
\put(1246,225){\rule[-0.175pt]{0.350pt}{0.426pt}}
\put(1247,223){\rule[-0.175pt]{0.350pt}{0.426pt}}
\put(1248,221){\rule[-0.175pt]{0.350pt}{0.426pt}}
\put(1249,220){\rule[-0.175pt]{0.350pt}{0.426pt}}
\put(1250,220){\usebox{\plotpoint}}
\put(1250,220){\rule[-0.175pt]{0.350pt}{0.809pt}}
\put(1251,223){\rule[-0.175pt]{0.350pt}{0.809pt}}
\put(1252,226){\rule[-0.175pt]{0.350pt}{0.809pt}}
\put(1253,230){\rule[-0.175pt]{0.350pt}{0.809pt}}
\put(1254,233){\rule[-0.175pt]{0.350pt}{0.809pt}}
\put(1255,236){\rule[-0.175pt]{0.350pt}{0.809pt}}
\put(1256,240){\rule[-0.175pt]{0.350pt}{0.809pt}}
\put(1257,243){\rule[-0.175pt]{0.350pt}{0.809pt}}
\put(1258,246){\rule[-0.175pt]{0.350pt}{0.809pt}}
\put(1259,250){\rule[-0.175pt]{0.350pt}{0.809pt}}
\put(1260,253){\rule[-0.175pt]{0.350pt}{0.809pt}}
\put(1261,256){\rule[-0.175pt]{0.350pt}{0.809pt}}
\put(1262,260){\rule[-0.175pt]{0.350pt}{0.809pt}}
\put(1263,263){\rule[-0.175pt]{0.350pt}{0.809pt}}
\put(1264,267){\rule[-0.175pt]{0.422pt}{0.350pt}}
\put(1265,268){\rule[-0.175pt]{0.422pt}{0.350pt}}
\put(1267,269){\rule[-0.175pt]{0.422pt}{0.350pt}}
\put(1269,270){\rule[-0.175pt]{0.422pt}{0.350pt}}
\put(1271,271){\rule[-0.175pt]{0.422pt}{0.350pt}}
\put(1272,272){\rule[-0.175pt]{0.422pt}{0.350pt}}
\put(1274,273){\rule[-0.175pt]{0.422pt}{0.350pt}}
\put(1276,274){\rule[-0.175pt]{0.422pt}{0.350pt}}
\put(1278,275){\usebox{\plotpoint}}
\put(1279,276){\usebox{\plotpoint}}
\put(1280,277){\usebox{\plotpoint}}
\put(1281,278){\usebox{\plotpoint}}
\put(1282,279){\usebox{\plotpoint}}
\put(1283,280){\usebox{\plotpoint}}
\put(1284,281){\usebox{\plotpoint}}
\put(1285,282){\usebox{\plotpoint}}
\put(1286,284){\usebox{\plotpoint}}
\put(1287,285){\usebox{\plotpoint}}
\put(1288,286){\usebox{\plotpoint}}
\put(1289,287){\usebox{\plotpoint}}
\put(1290,288){\usebox{\plotpoint}}
\put(1291,289){\usebox{\plotpoint}}
\put(1292,290){\usebox{\plotpoint}}
\put(1292,285){\rule[-0.175pt]{0.350pt}{1.316pt}}
\put(1293,280){\rule[-0.175pt]{0.350pt}{1.316pt}}
\put(1294,274){\rule[-0.175pt]{0.350pt}{1.316pt}}
\put(1295,269){\rule[-0.175pt]{0.350pt}{1.316pt}}
\put(1296,263){\rule[-0.175pt]{0.350pt}{1.316pt}}
\put(1297,258){\rule[-0.175pt]{0.350pt}{1.316pt}}
\put(1298,252){\rule[-0.175pt]{0.350pt}{1.316pt}}
\put(1299,247){\rule[-0.175pt]{0.350pt}{1.316pt}}
\put(1300,241){\rule[-0.175pt]{0.350pt}{1.316pt}}
\put(1301,236){\rule[-0.175pt]{0.350pt}{1.316pt}}
\put(1302,230){\rule[-0.175pt]{0.350pt}{1.316pt}}
\put(1303,225){\rule[-0.175pt]{0.350pt}{1.316pt}}
\put(1304,220){\rule[-0.175pt]{0.350pt}{1.316pt}}
\end{picture}

%% file: trk2eig-V100,3.tex
\setlength{\unitlength}{0.240900pt}
\ifx\plotpoint\undefined\newsavebox{\plotpoint}\fi
\sbox{\plotpoint}{\rule[-0.175pt]{0.350pt}{0.350pt}}%
\begin{picture}(1424,900)(0,0)
\tenrm
\sbox{\plotpoint}{\rule[-0.175pt]{0.350pt}{0.350pt}}%
\put(264,158){\rule[-0.175pt]{0.350pt}{151.526pt}}
\put(264,158){\rule[-0.175pt]{4.818pt}{0.350pt}}
\put(242,158){\makebox(0,0)[r]{$80$}}
\put(1340,158){\rule[-0.175pt]{4.818pt}{0.350pt}}
\put(264,284){\rule[-0.175pt]{4.818pt}{0.350pt}}
\put(242,284){\makebox(0,0)[r]{$85$}}
\put(1340,284){\rule[-0.175pt]{4.818pt}{0.350pt}}
\put(264,410){\rule[-0.175pt]{4.818pt}{0.350pt}}
\put(242,410){\makebox(0,0)[r]{$90$}}
\put(1340,410){\rule[-0.175pt]{4.818pt}{0.350pt}}
\put(264,535){\rule[-0.175pt]{4.818pt}{0.350pt}}
\put(242,535){\makebox(0,0)[r]{$95$}}
\put(1340,535){\rule[-0.175pt]{4.818pt}{0.350pt}}
\put(264,661){\rule[-0.175pt]{4.818pt}{0.350pt}}
\put(242,661){\makebox(0,0)[r]{$100$}}
\put(1340,661){\rule[-0.175pt]{4.818pt}{0.350pt}}
\put(264,787){\rule[-0.175pt]{4.818pt}{0.350pt}}
\put(242,787){\makebox(0,0)[r]{$105$}}
\put(1340,787){\rule[-0.175pt]{4.818pt}{0.350pt}}
\put(264,158){\rule[-0.175pt]{0.350pt}{4.818pt}}
\put(264,113){\makebox(0,0){$0$}}
\put(264,767){\rule[-0.175pt]{0.350pt}{4.818pt}}
\put(401,158){\rule[-0.175pt]{0.350pt}{4.818pt}}
\put(401,113){\makebox(0,0){$10$}}
\put(401,767){\rule[-0.175pt]{0.350pt}{4.818pt}}
\put(538,158){\rule[-0.175pt]{0.350pt}{4.818pt}}
\put(538,113){\makebox(0,0){$20$}}
\put(538,767){\rule[-0.175pt]{0.350pt}{4.818pt}}
\put(675,158){\rule[-0.175pt]{0.350pt}{4.818pt}}
\put(675,113){\makebox(0,0){$30$}}
\put(675,767){\rule[-0.175pt]{0.350pt}{4.818pt}}
\put(812,158){\rule[-0.175pt]{0.350pt}{4.818pt}}
\put(812,113){\makebox(0,0){$40$}}
\put(812,767){\rule[-0.175pt]{0.350pt}{4.818pt}}
\put(949,158){\rule[-0.175pt]{0.350pt}{4.818pt}}
\put(949,113){\makebox(0,0){$50$}}
\put(949,767){\rule[-0.175pt]{0.350pt}{4.818pt}}
\put(1086,158){\rule[-0.175pt]{0.350pt}{4.818pt}}
\put(1086,113){\makebox(0,0){$60$}}
\put(1086,767){\rule[-0.175pt]{0.350pt}{4.818pt}}
\put(1223,158){\rule[-0.175pt]{0.350pt}{4.818pt}}
\put(1223,113){\makebox(0,0){$70$}}
\put(1223,767){\rule[-0.175pt]{0.350pt}{4.818pt}}
\put(1360,158){\rule[-0.175pt]{0.350pt}{4.818pt}}
\put(1360,113){\makebox(0,0){$80$}}
\put(1360,767){\rule[-0.175pt]{0.350pt}{4.818pt}}
\put(264,158){\rule[-0.175pt]{264.026pt}{0.350pt}}
\put(1360,158){\rule[-0.175pt]{0.350pt}{151.526pt}}
\put(264,787){\rule[-0.175pt]{264.026pt}{0.350pt}}
\put(-43,472){\makebox(0,0)[l]{\shortstack{${\displaystyle{(p_i^\T\!A_ip_i)_{jj}\atop1\le j\le3}}$}}}
\put(812,68){\makebox(0,0){Step $i$}}
\put(264,158){\rule[-0.175pt]{0.350pt}{151.526pt}}
\put(306,158){\rule[-0.175pt]{0.350pt}{3.150pt}}
\put(307,171){\rule[-0.175pt]{0.350pt}{3.150pt}}
\put(308,184){\rule[-0.175pt]{0.350pt}{3.150pt}}
\put(309,197){\rule[-0.175pt]{0.350pt}{3.150pt}}
\put(310,210){\rule[-0.175pt]{0.350pt}{3.150pt}}
\put(311,223){\rule[-0.175pt]{0.350pt}{3.150pt}}
\put(312,236){\rule[-0.175pt]{0.350pt}{3.150pt}}
\put(313,249){\rule[-0.175pt]{0.350pt}{3.150pt}}
\put(314,262){\rule[-0.175pt]{0.350pt}{3.150pt}}
\put(315,275){\rule[-0.175pt]{0.350pt}{3.150pt}}
\put(316,288){\rule[-0.175pt]{0.350pt}{3.150pt}}
\put(317,301){\rule[-0.175pt]{0.350pt}{3.150pt}}
\put(318,314){\rule[-0.175pt]{0.350pt}{3.150pt}}
\put(319,328){\rule[-0.175pt]{0.350pt}{1.032pt}}
\put(320,332){\rule[-0.175pt]{0.350pt}{1.032pt}}
\put(321,336){\rule[-0.175pt]{0.350pt}{1.032pt}}
\put(322,340){\rule[-0.175pt]{0.350pt}{1.032pt}}
\put(323,345){\rule[-0.175pt]{0.350pt}{1.032pt}}
\put(324,349){\rule[-0.175pt]{0.350pt}{1.032pt}}
\put(325,353){\rule[-0.175pt]{0.350pt}{1.032pt}}
\put(326,357){\rule[-0.175pt]{0.350pt}{1.032pt}}
\put(327,362){\rule[-0.175pt]{0.350pt}{1.032pt}}
\put(328,366){\rule[-0.175pt]{0.350pt}{1.032pt}}
\put(329,370){\rule[-0.175pt]{0.350pt}{1.032pt}}
\put(330,375){\rule[-0.175pt]{0.350pt}{1.032pt}}
\put(331,379){\rule[-0.175pt]{0.350pt}{1.032pt}}
\put(332,383){\rule[-0.175pt]{0.350pt}{1.032pt}}
\put(333,387){\rule[-0.175pt]{0.350pt}{0.797pt}}
\put(334,391){\rule[-0.175pt]{0.350pt}{0.797pt}}
\put(335,394){\rule[-0.175pt]{0.350pt}{0.797pt}}
\put(336,397){\rule[-0.175pt]{0.350pt}{0.797pt}}
\put(337,401){\rule[-0.175pt]{0.350pt}{0.797pt}}
\put(338,404){\rule[-0.175pt]{0.350pt}{0.797pt}}
\put(339,407){\rule[-0.175pt]{0.350pt}{0.797pt}}
\put(340,411){\rule[-0.175pt]{0.350pt}{0.797pt}}
\put(341,414){\rule[-0.175pt]{0.350pt}{0.797pt}}
\put(342,417){\rule[-0.175pt]{0.350pt}{0.797pt}}
\put(343,421){\rule[-0.175pt]{0.350pt}{0.797pt}}
\put(344,424){\rule[-0.175pt]{0.350pt}{0.797pt}}
\put(345,427){\rule[-0.175pt]{0.350pt}{0.797pt}}
\put(346,430){\usebox{\plotpoint}}
\put(346,431){\rule[-0.175pt]{0.482pt}{0.350pt}}
\put(348,432){\rule[-0.175pt]{0.482pt}{0.350pt}}
\put(350,433){\rule[-0.175pt]{0.482pt}{0.350pt}}
\put(352,434){\rule[-0.175pt]{0.482pt}{0.350pt}}
\put(354,435){\rule[-0.175pt]{0.482pt}{0.350pt}}
\put(356,436){\rule[-0.175pt]{0.482pt}{0.350pt}}
\put(358,437){\rule[-0.175pt]{0.482pt}{0.350pt}}
\put(360,438){\rule[-0.175pt]{0.422pt}{0.350pt}}
\put(361,439){\rule[-0.175pt]{0.422pt}{0.350pt}}
\put(363,440){\rule[-0.175pt]{0.422pt}{0.350pt}}
\put(365,441){\rule[-0.175pt]{0.422pt}{0.350pt}}
\put(367,442){\rule[-0.175pt]{0.422pt}{0.350pt}}
\put(368,443){\rule[-0.175pt]{0.422pt}{0.350pt}}
\put(370,444){\rule[-0.175pt]{0.422pt}{0.350pt}}
\put(372,445){\rule[-0.175pt]{0.422pt}{0.350pt}}
\put(374,446){\usebox{\plotpoint}}
\put(375,447){\usebox{\plotpoint}}
\put(376,448){\usebox{\plotpoint}}
\put(377,449){\usebox{\plotpoint}}
\put(379,450){\usebox{\plotpoint}}
\put(380,451){\usebox{\plotpoint}}
\put(381,452){\usebox{\plotpoint}}
\put(383,453){\usebox{\plotpoint}}
\put(384,454){\usebox{\plotpoint}}
\put(385,455){\usebox{\plotpoint}}
\put(386,456){\rule[-0.175pt]{0.675pt}{0.350pt}}
\put(389,457){\rule[-0.175pt]{0.675pt}{0.350pt}}
\put(392,458){\rule[-0.175pt]{0.675pt}{0.350pt}}
\put(395,459){\rule[-0.175pt]{0.675pt}{0.350pt}}
\put(398,460){\rule[-0.175pt]{0.675pt}{0.350pt}}
\put(400,461){\usebox{\plotpoint}}
\put(402,462){\usebox{\plotpoint}}
\put(403,463){\usebox{\plotpoint}}
\put(404,464){\usebox{\plotpoint}}
\put(405,465){\usebox{\plotpoint}}
\put(406,466){\usebox{\plotpoint}}
\put(407,467){\usebox{\plotpoint}}
\put(408,468){\usebox{\plotpoint}}
\put(409,469){\usebox{\plotpoint}}
\put(410,470){\usebox{\plotpoint}}
\put(411,471){\usebox{\plotpoint}}
\put(412,472){\usebox{\plotpoint}}
\put(413,473){\usebox{\plotpoint}}
\put(415,474){\rule[-0.175pt]{0.350pt}{1.038pt}}
\put(416,478){\rule[-0.175pt]{0.350pt}{1.038pt}}
\put(417,482){\rule[-0.175pt]{0.350pt}{1.038pt}}
\put(418,486){\rule[-0.175pt]{0.350pt}{1.038pt}}
\put(419,491){\rule[-0.175pt]{0.350pt}{1.038pt}}
\put(420,495){\rule[-0.175pt]{0.350pt}{1.038pt}}
\put(421,499){\rule[-0.175pt]{0.350pt}{1.038pt}}
\put(422,504){\rule[-0.175pt]{0.350pt}{1.038pt}}
\put(423,508){\rule[-0.175pt]{0.350pt}{1.038pt}}
\put(424,512){\rule[-0.175pt]{0.350pt}{1.038pt}}
\put(425,517){\rule[-0.175pt]{0.350pt}{1.038pt}}
\put(426,521){\rule[-0.175pt]{0.350pt}{1.038pt}}
\put(427,525){\rule[-0.175pt]{0.350pt}{1.038pt}}
\put(428,529){\rule[-0.175pt]{0.350pt}{1.428pt}}
\put(429,535){\rule[-0.175pt]{0.350pt}{1.428pt}}
\put(430,541){\rule[-0.175pt]{0.350pt}{1.428pt}}
\put(431,547){\rule[-0.175pt]{0.350pt}{1.428pt}}
\put(432,553){\rule[-0.175pt]{0.350pt}{1.428pt}}
\put(433,559){\rule[-0.175pt]{0.350pt}{1.428pt}}
\put(434,565){\rule[-0.175pt]{0.350pt}{1.428pt}}
\put(435,571){\rule[-0.175pt]{0.350pt}{1.428pt}}
\put(436,577){\rule[-0.175pt]{0.350pt}{1.428pt}}
\put(437,583){\rule[-0.175pt]{0.350pt}{1.428pt}}
\put(438,589){\rule[-0.175pt]{0.350pt}{1.428pt}}
\put(439,595){\rule[-0.175pt]{0.350pt}{1.428pt}}
\put(440,601){\rule[-0.175pt]{0.350pt}{1.428pt}}
\put(441,607){\rule[-0.175pt]{0.350pt}{1.428pt}}
\put(442,613){\rule[-0.175pt]{0.350pt}{0.602pt}}
\put(443,615){\rule[-0.175pt]{0.350pt}{0.602pt}}
\put(444,618){\rule[-0.175pt]{0.350pt}{0.602pt}}
\put(445,620){\rule[-0.175pt]{0.350pt}{0.602pt}}
\put(446,623){\rule[-0.175pt]{0.350pt}{0.602pt}}
\put(447,625){\rule[-0.175pt]{0.350pt}{0.602pt}}
\put(448,628){\rule[-0.175pt]{0.350pt}{0.602pt}}
\put(449,630){\rule[-0.175pt]{0.350pt}{0.602pt}}
\put(450,633){\rule[-0.175pt]{0.350pt}{0.602pt}}
\put(451,635){\rule[-0.175pt]{0.350pt}{0.602pt}}
\put(452,638){\rule[-0.175pt]{0.350pt}{0.602pt}}
\put(453,640){\rule[-0.175pt]{0.350pt}{0.602pt}}
\put(454,643){\rule[-0.175pt]{0.350pt}{0.602pt}}
\put(455,645){\rule[-0.175pt]{0.350pt}{0.602pt}}
\put(456,648){\rule[-0.175pt]{1.124pt}{0.350pt}}
\put(460,649){\rule[-0.175pt]{1.124pt}{0.350pt}}
\put(465,650){\rule[-0.175pt]{1.124pt}{0.350pt}}
\put(469,651){\rule[-0.175pt]{0.783pt}{0.350pt}}
\put(473,652){\rule[-0.175pt]{0.783pt}{0.350pt}}
\put(476,653){\rule[-0.175pt]{0.783pt}{0.350pt}}
\put(479,654){\rule[-0.175pt]{0.783pt}{0.350pt}}
\put(483,655){\rule[-0.175pt]{1.686pt}{0.350pt}}
\put(490,656){\rule[-0.175pt]{1.686pt}{0.350pt}}
\put(497,657){\rule[-0.175pt]{3.373pt}{0.350pt}}
\put(511,658){\rule[-0.175pt]{3.132pt}{0.350pt}}
\put(524,659){\rule[-0.175pt]{9.877pt}{0.350pt}}
\put(565,660){\rule[-0.175pt]{19.995pt}{0.350pt}}
\put(648,661){\rule[-0.175pt]{39.508pt}{0.350pt}}
\put(812,661){\usebox{\plotpoint}}
\put(813,662){\usebox{\plotpoint}}
\put(814,663){\usebox{\plotpoint}}
\put(815,665){\usebox{\plotpoint}}
\put(816,666){\usebox{\plotpoint}}
\put(817,667){\usebox{\plotpoint}}
\put(818,669){\usebox{\plotpoint}}
\put(819,670){\usebox{\plotpoint}}
\put(820,671){\usebox{\plotpoint}}
\put(821,673){\usebox{\plotpoint}}
\put(822,674){\usebox{\plotpoint}}
\put(823,675){\usebox{\plotpoint}}
\put(824,677){\usebox{\plotpoint}}
\put(825,678){\usebox{\plotpoint}}
\put(826,679){\usebox{\plotpoint}}
\put(826,680){\rule[-0.175pt]{1.566pt}{0.350pt}}
\put(832,681){\rule[-0.175pt]{1.566pt}{0.350pt}}
\put(839,682){\rule[-0.175pt]{3.373pt}{0.350pt}}
\put(853,683){\rule[-0.175pt]{9.877pt}{0.350pt}}
\put(894,684){\rule[-0.175pt]{16.622pt}{0.350pt}}
\put(963,685){\rule[-0.175pt]{3.132pt}{0.350pt}}
\put(976,686){\rule[-0.175pt]{92.506pt}{0.350pt}}
\sbox{\plotpoint}{\rule[-0.350pt]{0.700pt}{0.700pt}}%
\put(295,158){\rule[-0.350pt]{0.700pt}{8.504pt}}
\put(296,193){\rule[-0.350pt]{0.700pt}{8.504pt}}
\put(297,228){\rule[-0.350pt]{0.700pt}{8.504pt}}
\put(298,263){\rule[-0.350pt]{0.700pt}{8.504pt}}
\put(299,299){\rule[-0.350pt]{0.700pt}{8.504pt}}
\put(300,334){\rule[-0.350pt]{0.700pt}{8.504pt}}
\put(301,369){\rule[-0.350pt]{0.700pt}{8.504pt}}
\put(302,405){\rule[-0.350pt]{0.700pt}{8.504pt}}
\put(303,440){\rule[-0.350pt]{0.700pt}{8.504pt}}
\put(304,475){\rule[-0.350pt]{0.700pt}{8.504pt}}
\put(305,510){\usebox{\plotpoint}}
\put(306,513){\usebox{\plotpoint}}
\put(307,515){\usebox{\plotpoint}}
\put(308,517){\usebox{\plotpoint}}
\put(309,519){\usebox{\plotpoint}}
\put(310,521){\usebox{\plotpoint}}
\put(311,523){\usebox{\plotpoint}}
\put(312,526){\usebox{\plotpoint}}
\put(313,528){\usebox{\plotpoint}}
\put(314,530){\usebox{\plotpoint}}
\put(315,532){\usebox{\plotpoint}}
\put(316,534){\usebox{\plotpoint}}
\put(317,536){\usebox{\plotpoint}}
\put(318,538){\usebox{\plotpoint}}
\put(319,541){\rule[-0.350pt]{0.700pt}{0.946pt}}
\put(320,544){\rule[-0.350pt]{0.700pt}{0.946pt}}
\put(321,548){\rule[-0.350pt]{0.700pt}{0.946pt}}
\put(322,552){\rule[-0.350pt]{0.700pt}{0.946pt}}
\put(323,556){\rule[-0.350pt]{0.700pt}{0.946pt}}
\put(324,560){\rule[-0.350pt]{0.700pt}{0.946pt}}
\put(325,564){\rule[-0.350pt]{0.700pt}{0.946pt}}
\put(326,568){\rule[-0.350pt]{0.700pt}{0.946pt}}
\put(327,572){\rule[-0.350pt]{0.700pt}{0.946pt}}
\put(328,576){\rule[-0.350pt]{0.700pt}{0.946pt}}
\put(329,580){\rule[-0.350pt]{0.700pt}{0.946pt}}
\put(330,584){\rule[-0.350pt]{0.700pt}{0.946pt}}
\put(331,588){\rule[-0.350pt]{0.700pt}{0.946pt}}
\put(332,592){\rule[-0.350pt]{0.700pt}{0.946pt}}
\put(333,596){\rule[-0.350pt]{3.132pt}{0.700pt}}
\put(346,596){\usebox{\plotpoint}}
\put(347,597){\usebox{\plotpoint}}
\put(348,599){\usebox{\plotpoint}}
\put(349,600){\usebox{\plotpoint}}
\put(350,602){\usebox{\plotpoint}}
\put(351,603){\usebox{\plotpoint}}
\put(352,605){\usebox{\plotpoint}}
\put(353,606){\usebox{\plotpoint}}
\put(354,608){\usebox{\plotpoint}}
\put(355,610){\usebox{\plotpoint}}
\put(356,611){\usebox{\plotpoint}}
\put(357,613){\usebox{\plotpoint}}
\put(358,614){\usebox{\plotpoint}}
\put(359,616){\usebox{\plotpoint}}
\put(360,617){\usebox{\plotpoint}}
\put(360,618){\usebox{\plotpoint}}
\put(361,619){\usebox{\plotpoint}}
\put(363,620){\usebox{\plotpoint}}
\put(364,621){\usebox{\plotpoint}}
\put(366,622){\usebox{\plotpoint}}
\put(367,623){\usebox{\plotpoint}}
\put(369,624){\usebox{\plotpoint}}
\put(370,625){\usebox{\plotpoint}}
\put(372,626){\usebox{\plotpoint}}
\put(373,627){\rule[-0.350pt]{1.566pt}{0.700pt}}
\put(380,628){\rule[-0.350pt]{1.566pt}{0.700pt}}
\put(387,629){\rule[-0.350pt]{1.686pt}{0.700pt}}
\put(394,630){\rule[-0.350pt]{1.686pt}{0.700pt}}
\put(401,631){\rule[-0.350pt]{3.373pt}{0.700pt}}
\put(415,630){\usebox{\plotpoint}}
\put(416,629){\usebox{\plotpoint}}
\put(418,628){\usebox{\plotpoint}}
\put(420,627){\usebox{\plotpoint}}
\put(422,626){\usebox{\plotpoint}}
\put(424,625){\usebox{\plotpoint}}
\put(426,624){\usebox{\plotpoint}}
\put(428,623){\usebox{\plotpoint}}
\put(430,622){\usebox{\plotpoint}}
\put(432,621){\usebox{\plotpoint}}
\put(435,620){\usebox{\plotpoint}}
\put(437,619){\usebox{\plotpoint}}
\put(439,618){\usebox{\plotpoint}}
\put(442,617){\rule[-0.350pt]{0.843pt}{0.700pt}}
\put(445,618){\rule[-0.350pt]{0.843pt}{0.700pt}}
\put(449,619){\rule[-0.350pt]{0.843pt}{0.700pt}}
\put(452,620){\rule[-0.350pt]{0.843pt}{0.700pt}}
\put(456,621){\usebox{\plotpoint}}
\put(458,622){\usebox{\plotpoint}}
\put(460,623){\usebox{\plotpoint}}
\put(462,624){\usebox{\plotpoint}}
\put(464,625){\usebox{\plotpoint}}
\put(466,626){\usebox{\plotpoint}}
\put(468,627){\usebox{\plotpoint}}
\put(470,628){\rule[-0.350pt]{3.694pt}{0.700pt}}
\put(485,629){\usebox{\plotpoint}}
\put(487,630){\usebox{\plotpoint}}
\put(490,631){\usebox{\plotpoint}}
\put(492,632){\usebox{\plotpoint}}
\put(494,633){\usebox{\plotpoint}}
\put(497,634){\rule[-0.350pt]{3.373pt}{0.700pt}}
\put(511,635){\rule[-0.350pt]{3.132pt}{0.700pt}}
\put(524,636){\rule[-0.350pt]{69.379pt}{0.700pt}}
\put(812,636){\rule[-0.350pt]{0.700pt}{1.273pt}}
\put(813,641){\rule[-0.350pt]{0.700pt}{1.273pt}}
\put(814,646){\rule[-0.350pt]{0.700pt}{1.273pt}}
\put(815,651){\rule[-0.350pt]{0.700pt}{1.273pt}}
\put(816,657){\rule[-0.350pt]{0.700pt}{1.273pt}}
\put(817,662){\rule[-0.350pt]{0.700pt}{1.273pt}}
\put(818,667){\rule[-0.350pt]{0.700pt}{1.273pt}}
\put(819,672){\rule[-0.350pt]{0.700pt}{1.273pt}}
\put(820,678){\rule[-0.350pt]{0.700pt}{1.273pt}}
\put(821,683){\rule[-0.350pt]{0.700pt}{1.273pt}}
\put(822,688){\rule[-0.350pt]{0.700pt}{1.273pt}}
\put(823,694){\rule[-0.350pt]{0.700pt}{1.273pt}}
\put(824,699){\rule[-0.350pt]{0.700pt}{1.273pt}}
\put(825,704){\rule[-0.350pt]{0.700pt}{1.273pt}}
\put(826,709){\usebox{\plotpoint}}
\put(826,710){\rule[-0.350pt]{3.132pt}{0.700pt}}
\put(839,711){\rule[-0.350pt]{56.130pt}{0.700pt}}
\put(1072,712){\rule[-0.350pt]{69.379pt}{0.700pt}}
\sbox{\plotpoint}{\rule[-0.500pt]{1.000pt}{1.000pt}}%
\put(325,158){\rule[-0.500pt]{1.000pt}{2.620pt}}
\put(326,168){\rule[-0.500pt]{1.000pt}{2.620pt}}
\put(327,179){\rule[-0.500pt]{1.000pt}{2.620pt}}
\put(328,190){\rule[-0.500pt]{1.000pt}{2.620pt}}
\put(329,201){\rule[-0.500pt]{1.000pt}{2.620pt}}
\put(330,212){\rule[-0.500pt]{1.000pt}{2.620pt}}
\put(331,223){\rule[-0.500pt]{1.000pt}{2.620pt}}
\put(332,234){\rule[-0.500pt]{1.000pt}{2.620pt}}
\put(333,245){\rule[-0.500pt]{1.000pt}{2.594pt}}
\put(334,255){\rule[-0.500pt]{1.000pt}{2.594pt}}
\put(335,266){\rule[-0.500pt]{1.000pt}{2.594pt}}
\put(336,277){\rule[-0.500pt]{1.000pt}{2.594pt}}
\put(337,288){\rule[-0.500pt]{1.000pt}{2.594pt}}
\put(338,298){\rule[-0.500pt]{1.000pt}{2.594pt}}
\put(339,309){\rule[-0.500pt]{1.000pt}{2.594pt}}
\put(340,320){\rule[-0.500pt]{1.000pt}{2.594pt}}
\put(341,331){\rule[-0.500pt]{1.000pt}{2.594pt}}
\put(342,341){\rule[-0.500pt]{1.000pt}{2.594pt}}
\put(343,352){\rule[-0.500pt]{1.000pt}{2.594pt}}
\put(344,363){\rule[-0.500pt]{1.000pt}{2.594pt}}
\put(345,374){\rule[-0.500pt]{1.000pt}{2.594pt}}
\put(346,384){\rule[-0.500pt]{1.000pt}{1.514pt}}
\put(347,391){\rule[-0.500pt]{1.000pt}{1.514pt}}
\put(348,397){\rule[-0.500pt]{1.000pt}{1.514pt}}
\put(349,403){\rule[-0.500pt]{1.000pt}{1.514pt}}
\put(350,410){\rule[-0.500pt]{1.000pt}{1.514pt}}
\put(351,416){\rule[-0.500pt]{1.000pt}{1.514pt}}
\put(352,422){\rule[-0.500pt]{1.000pt}{1.514pt}}
\put(353,428){\rule[-0.500pt]{1.000pt}{1.514pt}}
\put(354,435){\rule[-0.500pt]{1.000pt}{1.514pt}}
\put(355,441){\rule[-0.500pt]{1.000pt}{1.514pt}}
\put(356,447){\rule[-0.500pt]{1.000pt}{1.514pt}}
\put(357,454){\rule[-0.500pt]{1.000pt}{1.514pt}}
\put(358,460){\rule[-0.500pt]{1.000pt}{1.514pt}}
\put(359,466){\rule[-0.500pt]{1.000pt}{1.514pt}}
\put(360,472){\usebox{\plotpoint}}
\put(361,474){\usebox{\plotpoint}}
\put(362,476){\usebox{\plotpoint}}
\put(363,477){\usebox{\plotpoint}}
\put(364,479){\usebox{\plotpoint}}
\put(365,480){\usebox{\plotpoint}}
\put(366,482){\usebox{\plotpoint}}
\put(367,483){\usebox{\plotpoint}}
\put(368,485){\usebox{\plotpoint}}
\put(369,486){\usebox{\plotpoint}}
\put(370,488){\usebox{\plotpoint}}
\put(371,489){\usebox{\plotpoint}}
\put(372,491){\usebox{\plotpoint}}
\put(373,492){\usebox{\plotpoint}}
\put(374,494){\usebox{\plotpoint}}
\put(375,495){\usebox{\plotpoint}}
\put(376,496){\usebox{\plotpoint}}
\put(377,498){\usebox{\plotpoint}}
\put(378,499){\usebox{\plotpoint}}
\put(379,500){\usebox{\plotpoint}}
\put(380,502){\usebox{\plotpoint}}
\put(381,503){\usebox{\plotpoint}}
\put(382,505){\usebox{\plotpoint}}
\put(383,506){\usebox{\plotpoint}}
\put(384,507){\usebox{\plotpoint}}
\put(385,509){\usebox{\plotpoint}}
\put(386,510){\usebox{\plotpoint}}
\put(387,511){\usebox{\plotpoint}}
\put(387,512){\usebox{\plotpoint}}
\put(389,513){\usebox{\plotpoint}}
\put(392,514){\usebox{\plotpoint}}
\put(395,515){\usebox{\plotpoint}}
\put(398,516){\usebox{\plotpoint}}
\put(400,517){\usebox{\plotpoint}}
\put(403,518){\usebox{\plotpoint}}
\put(406,519){\usebox{\plotpoint}}
\put(409,520){\usebox{\plotpoint}}
\put(412,521){\usebox{\plotpoint}}
\put(414,522){\usebox{\plotpoint}}
\put(415,520){\usebox{\plotpoint}}
\put(416,519){\usebox{\plotpoint}}
\put(417,518){\usebox{\plotpoint}}
\put(418,517){\usebox{\plotpoint}}
\put(419,516){\usebox{\plotpoint}}
\put(420,515){\usebox{\plotpoint}}
\put(421,513){\usebox{\plotpoint}}
\put(422,512){\usebox{\plotpoint}}
\put(423,511){\usebox{\plotpoint}}
\put(424,510){\usebox{\plotpoint}}
\put(425,509){\usebox{\plotpoint}}
\put(426,508){\usebox{\plotpoint}}
\put(427,507){\usebox{\plotpoint}}
\put(428,505){\usebox{\plotpoint}}
\put(429,504){\usebox{\plotpoint}}
\put(430,503){\usebox{\plotpoint}}
\put(431,502){\usebox{\plotpoint}}
\put(432,501){\usebox{\plotpoint}}
\put(433,500){\usebox{\plotpoint}}
\put(434,499){\usebox{\plotpoint}}
\put(435,497){\usebox{\plotpoint}}
\put(436,496){\usebox{\plotpoint}}
\put(437,495){\usebox{\plotpoint}}
\put(438,494){\usebox{\plotpoint}}
\put(439,493){\usebox{\plotpoint}}
\put(440,492){\usebox{\plotpoint}}
\put(441,491){\usebox{\plotpoint}}
\put(442,491){\usebox{\plotpoint}}
\put(442,491){\usebox{\plotpoint}}
\put(443,494){\usebox{\plotpoint}}
\put(444,498){\usebox{\plotpoint}}
\put(445,502){\usebox{\plotpoint}}
\put(446,506){\usebox{\plotpoint}}
\put(447,509){\usebox{\plotpoint}}
\put(448,513){\usebox{\plotpoint}}
\put(449,517){\usebox{\plotpoint}}
\put(450,521){\usebox{\plotpoint}}
\put(451,525){\usebox{\plotpoint}}
\put(452,528){\usebox{\plotpoint}}
\put(453,532){\usebox{\plotpoint}}
\put(454,536){\usebox{\plotpoint}}
\put(455,540){\usebox{\plotpoint}}
\put(456,543){\usebox{\plotpoint}}
\put(457,545){\usebox{\plotpoint}}
\put(458,546){\usebox{\plotpoint}}
\put(459,547){\usebox{\plotpoint}}
\put(460,548){\usebox{\plotpoint}}
\put(461,549){\usebox{\plotpoint}}
\put(462,550){\usebox{\plotpoint}}
\put(463,551){\usebox{\plotpoint}}
\put(464,552){\usebox{\plotpoint}}
\put(465,553){\usebox{\plotpoint}}
\put(466,554){\usebox{\plotpoint}}
\put(467,555){\usebox{\plotpoint}}
\put(468,556){\usebox{\plotpoint}}
\put(469,557){\usebox{\plotpoint}}
\put(470,558){\usebox{\plotpoint}}
\put(471,560){\usebox{\plotpoint}}
\put(472,562){\usebox{\plotpoint}}
\put(473,564){\usebox{\plotpoint}}
\put(474,566){\usebox{\plotpoint}}
\put(475,568){\usebox{\plotpoint}}
\put(476,570){\usebox{\plotpoint}}
\put(477,571){\usebox{\plotpoint}}
\put(478,573){\usebox{\plotpoint}}
\put(479,575){\usebox{\plotpoint}}
\put(480,577){\usebox{\plotpoint}}
\put(481,579){\usebox{\plotpoint}}
\put(482,581){\usebox{\plotpoint}}
\put(483,582){\usebox{\plotpoint}}
\put(483,583){\usebox{\plotpoint}}
\put(484,584){\usebox{\plotpoint}}
\put(485,585){\usebox{\plotpoint}}
\put(487,586){\usebox{\plotpoint}}
\put(488,587){\usebox{\plotpoint}}
\put(489,588){\usebox{\plotpoint}}
\put(491,589){\usebox{\plotpoint}}
\put(492,590){\usebox{\plotpoint}}
\put(494,591){\usebox{\plotpoint}}
\put(495,592){\usebox{\plotpoint}}
\put(496,593){\usebox{\plotpoint}}
\put(500,594){\usebox{\plotpoint}}
\put(504,595){\usebox{\plotpoint}}
\put(507,596){\usebox{\plotpoint}}
\put(511,597){\usebox{\plotpoint}}
\put(513,598){\usebox{\plotpoint}}
\put(515,599){\usebox{\plotpoint}}
\put(517,600){\usebox{\plotpoint}}
\put(519,601){\usebox{\plotpoint}}
\put(521,602){\usebox{\plotpoint}}
\put(524,603){\usebox{\plotpoint}}
\put(527,604){\usebox{\plotpoint}}
\put(531,605){\usebox{\plotpoint}}
\put(534,606){\usebox{\plotpoint}}
\put(538,607){\rule[-0.500pt]{1.686pt}{1.000pt}}
\put(545,608){\rule[-0.500pt]{1.686pt}{1.000pt}}
\put(552,609){\rule[-0.500pt]{3.132pt}{1.000pt}}
\put(565,610){\rule[-0.500pt]{3.373pt}{1.000pt}}
\put(579,611){\rule[-0.500pt]{56.130pt}{1.000pt}}
\put(812,611){\rule[-0.500pt]{1.000pt}{2.134pt}}
\put(813,619){\rule[-0.500pt]{1.000pt}{2.134pt}}
\put(814,628){\rule[-0.500pt]{1.000pt}{2.134pt}}
\put(815,637){\rule[-0.500pt]{1.000pt}{2.134pt}}
\put(816,646){\rule[-0.500pt]{1.000pt}{2.134pt}}
\put(817,655){\rule[-0.500pt]{1.000pt}{2.134pt}}
\put(818,664){\rule[-0.500pt]{1.000pt}{2.134pt}}
\put(819,672){\rule[-0.500pt]{1.000pt}{2.134pt}}
\put(820,681){\rule[-0.500pt]{1.000pt}{2.134pt}}
\put(821,690){\rule[-0.500pt]{1.000pt}{2.134pt}}
\put(822,699){\rule[-0.500pt]{1.000pt}{2.134pt}}
\put(823,708){\rule[-0.500pt]{1.000pt}{2.134pt}}
\put(824,717){\rule[-0.500pt]{1.000pt}{2.134pt}}
\put(825,726){\rule[-0.500pt]{1.000pt}{2.134pt}}
\put(826,734){\usebox{\plotpoint}}
\put(826,735){\rule[-0.500pt]{6.504pt}{1.000pt}}
\put(853,736){\rule[-0.500pt]{6.745pt}{1.000pt}}
\put(881,737){\rule[-0.500pt]{115.391pt}{1.000pt}}
\end{picture}

%% file: trk3eig-V100,3.tex
\setlength{\unitlength}{0.240900pt}
\ifx\plotpoint\undefined\newsavebox{\plotpoint}\fi
\sbox{\plotpoint}{\rule[-0.175pt]{0.350pt}{0.350pt}}%
\begin{picture}(1424,900)(0,0)
\tenrm
\sbox{\plotpoint}{\rule[-0.175pt]{0.350pt}{0.350pt}}%
\put(264,158){\rule[-0.175pt]{0.350pt}{151.526pt}}
\put(264,158){\rule[-0.175pt]{4.818pt}{0.350pt}}
\put(242,158){\makebox(0,0)[r]{$80$}}
\put(1340,158){\rule[-0.175pt]{4.818pt}{0.350pt}}
\put(264,284){\rule[-0.175pt]{4.818pt}{0.350pt}}
\put(242,284){\makebox(0,0)[r]{$85$}}
\put(1340,284){\rule[-0.175pt]{4.818pt}{0.350pt}}
\put(264,410){\rule[-0.175pt]{4.818pt}{0.350pt}}
\put(242,410){\makebox(0,0)[r]{$90$}}
\put(1340,410){\rule[-0.175pt]{4.818pt}{0.350pt}}
\put(264,535){\rule[-0.175pt]{4.818pt}{0.350pt}}
\put(242,535){\makebox(0,0)[r]{$95$}}
\put(1340,535){\rule[-0.175pt]{4.818pt}{0.350pt}}
\put(264,661){\rule[-0.175pt]{4.818pt}{0.350pt}}
\put(242,661){\makebox(0,0)[r]{$100$}}
\put(1340,661){\rule[-0.175pt]{4.818pt}{0.350pt}}
\put(264,787){\rule[-0.175pt]{4.818pt}{0.350pt}}
\put(242,787){\makebox(0,0)[r]{$105$}}
\put(1340,787){\rule[-0.175pt]{4.818pt}{0.350pt}}
\put(264,158){\rule[-0.175pt]{0.350pt}{4.818pt}}
\put(264,113){\makebox(0,0){$0$}}
\put(264,767){\rule[-0.175pt]{0.350pt}{4.818pt}}
\put(401,158){\rule[-0.175pt]{0.350pt}{4.818pt}}
\put(401,113){\makebox(0,0){$10$}}
\put(401,767){\rule[-0.175pt]{0.350pt}{4.818pt}}
\put(538,158){\rule[-0.175pt]{0.350pt}{4.818pt}}
\put(538,113){\makebox(0,0){$20$}}
\put(538,767){\rule[-0.175pt]{0.350pt}{4.818pt}}
\put(675,158){\rule[-0.175pt]{0.350pt}{4.818pt}}
\put(675,113){\makebox(0,0){$30$}}
\put(675,767){\rule[-0.175pt]{0.350pt}{4.818pt}}
\put(812,158){\rule[-0.175pt]{0.350pt}{4.818pt}}
\put(812,113){\makebox(0,0){$40$}}
\put(812,767){\rule[-0.175pt]{0.350pt}{4.818pt}}
\put(949,158){\rule[-0.175pt]{0.350pt}{4.818pt}}
\put(949,113){\makebox(0,0){$50$}}
\put(949,767){\rule[-0.175pt]{0.350pt}{4.818pt}}
\put(1086,158){\rule[-0.175pt]{0.350pt}{4.818pt}}
\put(1086,113){\makebox(0,0){$60$}}
\put(1086,767){\rule[-0.175pt]{0.350pt}{4.818pt}}
\put(1223,158){\rule[-0.175pt]{0.350pt}{4.818pt}}
\put(1223,113){\makebox(0,0){$70$}}
\put(1223,767){\rule[-0.175pt]{0.350pt}{4.818pt}}
\put(1360,158){\rule[-0.175pt]{0.350pt}{4.818pt}}
\put(1360,113){\makebox(0,0){$80$}}
\put(1360,767){\rule[-0.175pt]{0.350pt}{4.818pt}}
\put(264,158){\rule[-0.175pt]{264.026pt}{0.350pt}}
\put(1360,158){\rule[-0.175pt]{0.350pt}{151.526pt}}
\put(264,787){\rule[-0.175pt]{264.026pt}{0.350pt}}
\put(-43,472){\makebox(0,0)[l]{\shortstack{${\displaystyle{(p_i^\T\!A_ip_i)_{jj}\atop1\le j\le3}}$}}}
\put(812,68){\makebox(0,0){Step $i$}}
\put(264,158){\rule[-0.175pt]{0.350pt}{151.526pt}}
\put(306,158){\rule[-0.175pt]{0.350pt}{3.150pt}}
\put(307,171){\rule[-0.175pt]{0.350pt}{3.150pt}}
\put(308,184){\rule[-0.175pt]{0.350pt}{3.150pt}}
\put(309,197){\rule[-0.175pt]{0.350pt}{3.150pt}}
\put(310,210){\rule[-0.175pt]{0.350pt}{3.150pt}}
\put(311,223){\rule[-0.175pt]{0.350pt}{3.150pt}}
\put(312,236){\rule[-0.175pt]{0.350pt}{3.150pt}}
\put(313,249){\rule[-0.175pt]{0.350pt}{3.150pt}}
\put(314,262){\rule[-0.175pt]{0.350pt}{3.150pt}}
\put(315,275){\rule[-0.175pt]{0.350pt}{3.150pt}}
\put(316,288){\rule[-0.175pt]{0.350pt}{3.150pt}}
\put(317,301){\rule[-0.175pt]{0.350pt}{3.150pt}}
\put(318,314){\rule[-0.175pt]{0.350pt}{3.150pt}}
\put(319,328){\rule[-0.175pt]{0.350pt}{1.032pt}}
\put(320,332){\rule[-0.175pt]{0.350pt}{1.032pt}}
\put(321,336){\rule[-0.175pt]{0.350pt}{1.032pt}}
\put(322,340){\rule[-0.175pt]{0.350pt}{1.032pt}}
\put(323,345){\rule[-0.175pt]{0.350pt}{1.032pt}}
\put(324,349){\rule[-0.175pt]{0.350pt}{1.032pt}}
\put(325,353){\rule[-0.175pt]{0.350pt}{1.032pt}}
\put(326,357){\rule[-0.175pt]{0.350pt}{1.032pt}}
\put(327,362){\rule[-0.175pt]{0.350pt}{1.032pt}}
\put(328,366){\rule[-0.175pt]{0.350pt}{1.032pt}}
\put(329,370){\rule[-0.175pt]{0.350pt}{1.032pt}}
\put(330,375){\rule[-0.175pt]{0.350pt}{1.032pt}}
\put(331,379){\rule[-0.175pt]{0.350pt}{1.032pt}}
\put(332,383){\rule[-0.175pt]{0.350pt}{1.032pt}}
\put(333,387){\rule[-0.175pt]{0.350pt}{0.797pt}}
\put(334,391){\rule[-0.175pt]{0.350pt}{0.797pt}}
\put(335,394){\rule[-0.175pt]{0.350pt}{0.797pt}}
\put(336,397){\rule[-0.175pt]{0.350pt}{0.797pt}}
\put(337,401){\rule[-0.175pt]{0.350pt}{0.797pt}}
\put(338,404){\rule[-0.175pt]{0.350pt}{0.797pt}}
\put(339,407){\rule[-0.175pt]{0.350pt}{0.797pt}}
\put(340,411){\rule[-0.175pt]{0.350pt}{0.797pt}}
\put(341,414){\rule[-0.175pt]{0.350pt}{0.797pt}}
\put(342,417){\rule[-0.175pt]{0.350pt}{0.797pt}}
\put(343,421){\rule[-0.175pt]{0.350pt}{0.797pt}}
\put(344,424){\rule[-0.175pt]{0.350pt}{0.797pt}}
\put(345,427){\rule[-0.175pt]{0.350pt}{0.797pt}}
\put(346,430){\usebox{\plotpoint}}
\put(346,431){\rule[-0.175pt]{0.482pt}{0.350pt}}
\put(348,432){\rule[-0.175pt]{0.482pt}{0.350pt}}
\put(350,433){\rule[-0.175pt]{0.482pt}{0.350pt}}
\put(352,434){\rule[-0.175pt]{0.482pt}{0.350pt}}
\put(354,435){\rule[-0.175pt]{0.482pt}{0.350pt}}
\put(356,436){\rule[-0.175pt]{0.482pt}{0.350pt}}
\put(358,437){\rule[-0.175pt]{0.482pt}{0.350pt}}
\put(360,438){\rule[-0.175pt]{0.422pt}{0.350pt}}
\put(361,439){\rule[-0.175pt]{0.422pt}{0.350pt}}
\put(363,440){\rule[-0.175pt]{0.422pt}{0.350pt}}
\put(365,441){\rule[-0.175pt]{0.422pt}{0.350pt}}
\put(367,442){\rule[-0.175pt]{0.422pt}{0.350pt}}
\put(368,443){\rule[-0.175pt]{0.422pt}{0.350pt}}
\put(370,444){\rule[-0.175pt]{0.422pt}{0.350pt}}
\put(372,445){\rule[-0.175pt]{0.422pt}{0.350pt}}
\put(374,446){\usebox{\plotpoint}}
\put(375,447){\usebox{\plotpoint}}
\put(376,448){\usebox{\plotpoint}}
\put(377,449){\usebox{\plotpoint}}
\put(379,450){\usebox{\plotpoint}}
\put(380,451){\usebox{\plotpoint}}
\put(381,452){\usebox{\plotpoint}}
\put(383,453){\usebox{\plotpoint}}
\put(384,454){\usebox{\plotpoint}}
\put(385,455){\usebox{\plotpoint}}
\put(386,456){\rule[-0.175pt]{0.675pt}{0.350pt}}
\put(389,457){\rule[-0.175pt]{0.675pt}{0.350pt}}
\put(392,458){\rule[-0.175pt]{0.675pt}{0.350pt}}
\put(395,459){\rule[-0.175pt]{0.675pt}{0.350pt}}
\put(398,460){\rule[-0.175pt]{0.675pt}{0.350pt}}
\put(400,461){\usebox{\plotpoint}}
\put(402,462){\usebox{\plotpoint}}
\put(403,463){\usebox{\plotpoint}}
\put(404,464){\usebox{\plotpoint}}
\put(405,465){\usebox{\plotpoint}}
\put(406,466){\usebox{\plotpoint}}
\put(407,467){\usebox{\plotpoint}}
\put(408,468){\usebox{\plotpoint}}
\put(409,469){\usebox{\plotpoint}}
\put(410,470){\usebox{\plotpoint}}
\put(411,471){\usebox{\plotpoint}}
\put(412,472){\usebox{\plotpoint}}
\put(413,473){\usebox{\plotpoint}}
\put(415,474){\rule[-0.175pt]{0.350pt}{1.038pt}}
\put(416,478){\rule[-0.175pt]{0.350pt}{1.038pt}}
\put(417,482){\rule[-0.175pt]{0.350pt}{1.038pt}}
\put(418,486){\rule[-0.175pt]{0.350pt}{1.038pt}}
\put(419,491){\rule[-0.175pt]{0.350pt}{1.038pt}}
\put(420,495){\rule[-0.175pt]{0.350pt}{1.038pt}}
\put(421,499){\rule[-0.175pt]{0.350pt}{1.038pt}}
\put(422,504){\rule[-0.175pt]{0.350pt}{1.038pt}}
\put(423,508){\rule[-0.175pt]{0.350pt}{1.038pt}}
\put(424,512){\rule[-0.175pt]{0.350pt}{1.038pt}}
\put(425,517){\rule[-0.175pt]{0.350pt}{1.038pt}}
\put(426,521){\rule[-0.175pt]{0.350pt}{1.038pt}}
\put(427,525){\rule[-0.175pt]{0.350pt}{1.038pt}}
\put(428,529){\rule[-0.175pt]{0.350pt}{1.428pt}}
\put(429,535){\rule[-0.175pt]{0.350pt}{1.428pt}}
\put(430,541){\rule[-0.175pt]{0.350pt}{1.428pt}}
\put(431,547){\rule[-0.175pt]{0.350pt}{1.428pt}}
\put(432,553){\rule[-0.175pt]{0.350pt}{1.428pt}}
\put(433,559){\rule[-0.175pt]{0.350pt}{1.428pt}}
\put(434,565){\rule[-0.175pt]{0.350pt}{1.428pt}}
\put(435,571){\rule[-0.175pt]{0.350pt}{1.428pt}}
\put(436,577){\rule[-0.175pt]{0.350pt}{1.428pt}}
\put(437,583){\rule[-0.175pt]{0.350pt}{1.428pt}}
\put(438,589){\rule[-0.175pt]{0.350pt}{1.428pt}}
\put(439,595){\rule[-0.175pt]{0.350pt}{1.428pt}}
\put(440,601){\rule[-0.175pt]{0.350pt}{1.428pt}}
\put(441,607){\rule[-0.175pt]{0.350pt}{1.428pt}}
\put(442,613){\rule[-0.175pt]{0.350pt}{0.602pt}}
\put(443,615){\rule[-0.175pt]{0.350pt}{0.602pt}}
\put(444,618){\rule[-0.175pt]{0.350pt}{0.602pt}}
\put(445,620){\rule[-0.175pt]{0.350pt}{0.602pt}}
\put(446,623){\rule[-0.175pt]{0.350pt}{0.602pt}}
\put(447,625){\rule[-0.175pt]{0.350pt}{0.602pt}}
\put(448,628){\rule[-0.175pt]{0.350pt}{0.602pt}}
\put(449,630){\rule[-0.175pt]{0.350pt}{0.602pt}}
\put(450,633){\rule[-0.175pt]{0.350pt}{0.602pt}}
\put(451,635){\rule[-0.175pt]{0.350pt}{0.602pt}}
\put(452,638){\rule[-0.175pt]{0.350pt}{0.602pt}}
\put(453,640){\rule[-0.175pt]{0.350pt}{0.602pt}}
\put(454,643){\rule[-0.175pt]{0.350pt}{0.602pt}}
\put(455,645){\rule[-0.175pt]{0.350pt}{0.602pt}}
\put(456,648){\rule[-0.175pt]{1.124pt}{0.350pt}}
\put(460,649){\rule[-0.175pt]{1.124pt}{0.350pt}}
\put(465,650){\rule[-0.175pt]{1.124pt}{0.350pt}}
\put(469,651){\rule[-0.175pt]{0.783pt}{0.350pt}}
\put(473,652){\rule[-0.175pt]{0.783pt}{0.350pt}}
\put(476,653){\rule[-0.175pt]{0.783pt}{0.350pt}}
\put(479,654){\rule[-0.175pt]{0.783pt}{0.350pt}}
\put(483,655){\rule[-0.175pt]{1.686pt}{0.350pt}}
\put(490,656){\rule[-0.175pt]{1.686pt}{0.350pt}}
\put(497,657){\rule[-0.175pt]{3.373pt}{0.350pt}}
\put(511,658){\rule[-0.175pt]{3.132pt}{0.350pt}}
\put(524,659){\rule[-0.175pt]{9.877pt}{0.350pt}}
\put(565,660){\rule[-0.175pt]{19.995pt}{0.350pt}}
\put(648,661){\rule[-0.175pt]{39.508pt}{0.350pt}}
\put(812,655){\rule[-0.175pt]{0.350pt}{1.291pt}}
\put(813,650){\rule[-0.175pt]{0.350pt}{1.291pt}}
\put(814,644){\rule[-0.175pt]{0.350pt}{1.291pt}}
\put(815,639){\rule[-0.175pt]{0.350pt}{1.291pt}}
\put(816,634){\rule[-0.175pt]{0.350pt}{1.291pt}}
\put(817,628){\rule[-0.175pt]{0.350pt}{1.291pt}}
\put(818,623){\rule[-0.175pt]{0.350pt}{1.291pt}}
\put(819,618){\rule[-0.175pt]{0.350pt}{1.291pt}}
\put(820,612){\rule[-0.175pt]{0.350pt}{1.291pt}}
\put(821,607){\rule[-0.175pt]{0.350pt}{1.291pt}}
\put(822,602){\rule[-0.175pt]{0.350pt}{1.291pt}}
\put(823,596){\rule[-0.175pt]{0.350pt}{1.291pt}}
\put(824,591){\rule[-0.175pt]{0.350pt}{1.291pt}}
\put(825,586){\rule[-0.175pt]{0.350pt}{1.291pt}}
\put(826,586){\usebox{\plotpoint}}
\put(826,586){\rule[-0.175pt]{40.351pt}{0.350pt}}
\put(993,585){\rule[-0.175pt]{0.843pt}{0.350pt}}
\put(997,584){\rule[-0.175pt]{0.843pt}{0.350pt}}
\put(1000,583){\rule[-0.175pt]{0.843pt}{0.350pt}}
\put(1004,582){\rule[-0.175pt]{1.686pt}{0.350pt}}
\put(1011,581){\rule[-0.175pt]{1.686pt}{0.350pt}}
\put(1018,580){\rule[-0.175pt]{3.694pt}{0.350pt}}
\put(1033,581){\rule[-0.175pt]{0.562pt}{0.350pt}}
\put(1035,582){\rule[-0.175pt]{0.562pt}{0.350pt}}
\put(1038,583){\rule[-0.175pt]{0.562pt}{0.350pt}}
\put(1040,584){\rule[-0.175pt]{0.562pt}{0.350pt}}
\put(1042,585){\rule[-0.175pt]{0.562pt}{0.350pt}}
\put(1045,586){\rule[-0.175pt]{1.686pt}{0.350pt}}
\put(1052,587){\rule[-0.175pt]{1.686pt}{0.350pt}}
\put(1059,588){\rule[-0.175pt]{3.132pt}{0.350pt}}
\put(1072,589){\rule[-0.175pt]{1.124pt}{0.350pt}}
\put(1076,590){\rule[-0.175pt]{1.124pt}{0.350pt}}
\put(1081,591){\rule[-0.175pt]{1.124pt}{0.350pt}}
\put(1085,592){\rule[-0.175pt]{1.686pt}{0.350pt}}
\put(1093,593){\rule[-0.175pt]{1.686pt}{0.350pt}}
\put(1100,594){\rule[-0.175pt]{0.783pt}{0.350pt}}
\put(1103,595){\rule[-0.175pt]{0.783pt}{0.350pt}}
\put(1106,596){\rule[-0.175pt]{0.783pt}{0.350pt}}
\put(1109,597){\rule[-0.175pt]{0.783pt}{0.350pt}}
\put(1113,598){\rule[-0.175pt]{0.675pt}{0.350pt}}
\put(1115,599){\rule[-0.175pt]{0.675pt}{0.350pt}}
\put(1118,600){\rule[-0.175pt]{0.675pt}{0.350pt}}
\put(1121,601){\rule[-0.175pt]{0.675pt}{0.350pt}}
\put(1124,602){\rule[-0.175pt]{0.674pt}{0.350pt}}
\put(1127,603){\rule[-0.175pt]{0.482pt}{0.350pt}}
\put(1129,604){\rule[-0.175pt]{0.482pt}{0.350pt}}
\put(1131,605){\rule[-0.175pt]{0.482pt}{0.350pt}}
\put(1133,606){\rule[-0.175pt]{0.482pt}{0.350pt}}
\put(1135,607){\rule[-0.175pt]{0.482pt}{0.350pt}}
\put(1137,608){\rule[-0.175pt]{0.482pt}{0.350pt}}
\put(1139,609){\rule[-0.175pt]{0.482pt}{0.350pt}}
\put(1141,610){\rule[-0.175pt]{0.562pt}{0.350pt}}
\put(1143,611){\rule[-0.175pt]{0.562pt}{0.350pt}}
\put(1145,612){\rule[-0.175pt]{0.562pt}{0.350pt}}
\put(1148,613){\rule[-0.175pt]{0.562pt}{0.350pt}}
\put(1150,614){\rule[-0.175pt]{0.562pt}{0.350pt}}
\put(1152,615){\rule[-0.175pt]{0.562pt}{0.350pt}}
\put(1155,616){\rule[-0.175pt]{1.044pt}{0.350pt}}
\put(1159,617){\rule[-0.175pt]{1.044pt}{0.350pt}}
\put(1163,618){\rule[-0.175pt]{1.044pt}{0.350pt}}
\put(1168,619){\rule[-0.175pt]{1.686pt}{0.350pt}}
\put(1175,620){\rule[-0.175pt]{1.686pt}{0.350pt}}
\put(1182,621){\rule[-0.175pt]{1.124pt}{0.350pt}}
\put(1186,622){\rule[-0.175pt]{1.124pt}{0.350pt}}
\put(1191,623){\rule[-0.175pt]{1.124pt}{0.350pt}}
\put(1195,624){\rule[-0.175pt]{1.566pt}{0.350pt}}
\put(1202,625){\rule[-0.175pt]{1.566pt}{0.350pt}}
\put(1209,626){\rule[-0.175pt]{3.373pt}{0.350pt}}
\put(1223,627){\rule[-0.175pt]{1.686pt}{0.350pt}}
\put(1230,628){\rule[-0.175pt]{1.686pt}{0.350pt}}
\put(1237,629){\rule[-0.175pt]{3.132pt}{0.350pt}}
\put(1250,630){\rule[-0.175pt]{3.373pt}{0.350pt}}
\put(1264,631){\rule[-0.175pt]{3.373pt}{0.350pt}}
\put(1278,632){\rule[-0.175pt]{6.504pt}{0.350pt}}
\put(1305,633){\rule[-0.175pt]{9.877pt}{0.350pt}}
\put(1346,634){\rule[-0.175pt]{3.373pt}{0.350pt}}
\sbox{\plotpoint}{\rule[-0.350pt]{0.700pt}{0.700pt}}%
\put(295,158){\rule[-0.350pt]{0.700pt}{8.504pt}}
\put(296,193){\rule[-0.350pt]{0.700pt}{8.504pt}}
\put(297,228){\rule[-0.350pt]{0.700pt}{8.504pt}}
\put(298,263){\rule[-0.350pt]{0.700pt}{8.504pt}}
\put(299,299){\rule[-0.350pt]{0.700pt}{8.504pt}}
\put(300,334){\rule[-0.350pt]{0.700pt}{8.504pt}}
\put(301,369){\rule[-0.350pt]{0.700pt}{8.504pt}}
\put(302,405){\rule[-0.350pt]{0.700pt}{8.504pt}}
\put(303,440){\rule[-0.350pt]{0.700pt}{8.504pt}}
\put(304,475){\rule[-0.350pt]{0.700pt}{8.504pt}}
\put(305,510){\usebox{\plotpoint}}
\put(306,513){\usebox{\plotpoint}}
\put(307,515){\usebox{\plotpoint}}
\put(308,517){\usebox{\plotpoint}}
\put(309,519){\usebox{\plotpoint}}
\put(310,521){\usebox{\plotpoint}}
\put(311,523){\usebox{\plotpoint}}
\put(312,526){\usebox{\plotpoint}}
\put(313,528){\usebox{\plotpoint}}
\put(314,530){\usebox{\plotpoint}}
\put(315,532){\usebox{\plotpoint}}
\put(316,534){\usebox{\plotpoint}}
\put(317,536){\usebox{\plotpoint}}
\put(318,538){\usebox{\plotpoint}}
\put(319,541){\rule[-0.350pt]{0.700pt}{0.946pt}}
\put(320,544){\rule[-0.350pt]{0.700pt}{0.946pt}}
\put(321,548){\rule[-0.350pt]{0.700pt}{0.946pt}}
\put(322,552){\rule[-0.350pt]{0.700pt}{0.946pt}}
\put(323,556){\rule[-0.350pt]{0.700pt}{0.946pt}}
\put(324,560){\rule[-0.350pt]{0.700pt}{0.946pt}}
\put(325,564){\rule[-0.350pt]{0.700pt}{0.946pt}}
\put(326,568){\rule[-0.350pt]{0.700pt}{0.946pt}}
\put(327,572){\rule[-0.350pt]{0.700pt}{0.946pt}}
\put(328,576){\rule[-0.350pt]{0.700pt}{0.946pt}}
\put(329,580){\rule[-0.350pt]{0.700pt}{0.946pt}}
\put(330,584){\rule[-0.350pt]{0.700pt}{0.946pt}}
\put(331,588){\rule[-0.350pt]{0.700pt}{0.946pt}}
\put(332,592){\rule[-0.350pt]{0.700pt}{0.946pt}}
\put(333,596){\rule[-0.350pt]{3.132pt}{0.700pt}}
\put(346,596){\usebox{\plotpoint}}
\put(347,597){\usebox{\plotpoint}}
\put(348,599){\usebox{\plotpoint}}
\put(349,600){\usebox{\plotpoint}}
\put(350,602){\usebox{\plotpoint}}
\put(351,603){\usebox{\plotpoint}}
\put(352,605){\usebox{\plotpoint}}
\put(353,606){\usebox{\plotpoint}}
\put(354,608){\usebox{\plotpoint}}
\put(355,610){\usebox{\plotpoint}}
\put(356,611){\usebox{\plotpoint}}
\put(357,613){\usebox{\plotpoint}}
\put(358,614){\usebox{\plotpoint}}
\put(359,616){\usebox{\plotpoint}}
\put(360,617){\usebox{\plotpoint}}
\put(360,618){\usebox{\plotpoint}}
\put(361,619){\usebox{\plotpoint}}
\put(363,620){\usebox{\plotpoint}}
\put(364,621){\usebox{\plotpoint}}
\put(366,622){\usebox{\plotpoint}}
\put(367,623){\usebox{\plotpoint}}
\put(369,624){\usebox{\plotpoint}}
\put(370,625){\usebox{\plotpoint}}
\put(372,626){\usebox{\plotpoint}}
\put(373,627){\rule[-0.350pt]{1.566pt}{0.700pt}}
\put(380,628){\rule[-0.350pt]{1.566pt}{0.700pt}}
\put(387,629){\rule[-0.350pt]{1.686pt}{0.700pt}}
\put(394,630){\rule[-0.350pt]{1.686pt}{0.700pt}}
\put(401,631){\rule[-0.350pt]{3.373pt}{0.700pt}}
\put(415,630){\usebox{\plotpoint}}
\put(416,629){\usebox{\plotpoint}}
\put(418,628){\usebox{\plotpoint}}
\put(420,627){\usebox{\plotpoint}}
\put(422,626){\usebox{\plotpoint}}
\put(424,625){\usebox{\plotpoint}}
\put(426,624){\usebox{\plotpoint}}
\put(428,623){\usebox{\plotpoint}}
\put(430,622){\usebox{\plotpoint}}
\put(432,621){\usebox{\plotpoint}}
\put(435,620){\usebox{\plotpoint}}
\put(437,619){\usebox{\plotpoint}}
\put(439,618){\usebox{\plotpoint}}
\put(442,617){\rule[-0.350pt]{0.843pt}{0.700pt}}
\put(445,618){\rule[-0.350pt]{0.843pt}{0.700pt}}
\put(449,619){\rule[-0.350pt]{0.843pt}{0.700pt}}
\put(452,620){\rule[-0.350pt]{0.843pt}{0.700pt}}
\put(456,621){\usebox{\plotpoint}}
\put(458,622){\usebox{\plotpoint}}
\put(460,623){\usebox{\plotpoint}}
\put(462,624){\usebox{\plotpoint}}
\put(464,625){\usebox{\plotpoint}}
\put(466,626){\usebox{\plotpoint}}
\put(468,627){\usebox{\plotpoint}}
\put(470,628){\rule[-0.350pt]{3.694pt}{0.700pt}}
\put(485,629){\usebox{\plotpoint}}
\put(487,630){\usebox{\plotpoint}}
\put(490,631){\usebox{\plotpoint}}
\put(492,632){\usebox{\plotpoint}}
\put(494,633){\usebox{\plotpoint}}
\put(497,634){\rule[-0.350pt]{3.373pt}{0.700pt}}
\put(511,635){\rule[-0.350pt]{3.132pt}{0.700pt}}
\put(524,636){\rule[-0.350pt]{69.379pt}{0.700pt}}
\put(812,630){\rule[-0.350pt]{0.700pt}{1.291pt}}
\put(813,625){\rule[-0.350pt]{0.700pt}{1.291pt}}
\put(814,619){\rule[-0.350pt]{0.700pt}{1.291pt}}
\put(815,614){\rule[-0.350pt]{0.700pt}{1.291pt}}
\put(816,609){\rule[-0.350pt]{0.700pt}{1.291pt}}
\put(817,603){\rule[-0.350pt]{0.700pt}{1.291pt}}
\put(818,598){\rule[-0.350pt]{0.700pt}{1.291pt}}
\put(819,593){\rule[-0.350pt]{0.700pt}{1.291pt}}
\put(820,587){\rule[-0.350pt]{0.700pt}{1.291pt}}
\put(821,582){\rule[-0.350pt]{0.700pt}{1.291pt}}
\put(822,577){\rule[-0.350pt]{0.700pt}{1.291pt}}
\put(823,571){\rule[-0.350pt]{0.700pt}{1.291pt}}
\put(824,566){\rule[-0.350pt]{0.700pt}{1.291pt}}
\put(825,561){\rule[-0.350pt]{0.700pt}{1.291pt}}
\put(826,561){\usebox{\plotpoint}}
\put(826,561){\rule[-0.350pt]{39.508pt}{0.700pt}}
\put(990,561){\usebox{\plotpoint}}
\put(991,562){\usebox{\plotpoint}}
\put(992,563){\usebox{\plotpoint}}
\put(993,565){\usebox{\plotpoint}}
\put(994,566){\usebox{\plotpoint}}
\put(995,567){\usebox{\plotpoint}}
\put(996,569){\usebox{\plotpoint}}
\put(997,570){\usebox{\plotpoint}}
\put(998,571){\usebox{\plotpoint}}
\put(999,573){\usebox{\plotpoint}}
\put(1000,574){\usebox{\plotpoint}}
\put(1001,575){\usebox{\plotpoint}}
\put(1002,577){\usebox{\plotpoint}}
\put(1003,578){\usebox{\plotpoint}}
\put(1004,579){\usebox{\plotpoint}}
\put(1005,581){\usebox{\plotpoint}}
\put(1006,583){\usebox{\plotpoint}}
\put(1007,585){\usebox{\plotpoint}}
\put(1008,587){\usebox{\plotpoint}}
\put(1009,588){\usebox{\plotpoint}}
\put(1010,590){\usebox{\plotpoint}}
\put(1011,592){\usebox{\plotpoint}}
\put(1012,594){\usebox{\plotpoint}}
\put(1013,596){\usebox{\plotpoint}}
\put(1014,597){\usebox{\plotpoint}}
\put(1015,599){\usebox{\plotpoint}}
\put(1016,601){\usebox{\plotpoint}}
\put(1017,603){\usebox{\plotpoint}}
\put(1018,604){\usebox{\plotpoint}}
\put(1019,607){\usebox{\plotpoint}}
\put(1020,609){\usebox{\plotpoint}}
\put(1021,612){\usebox{\plotpoint}}
\put(1022,614){\usebox{\plotpoint}}
\put(1023,616){\usebox{\plotpoint}}
\put(1024,619){\usebox{\plotpoint}}
\put(1025,621){\usebox{\plotpoint}}
\put(1026,624){\usebox{\plotpoint}}
\put(1027,626){\usebox{\plotpoint}}
\put(1028,628){\usebox{\plotpoint}}
\put(1029,631){\usebox{\plotpoint}}
\put(1030,633){\usebox{\plotpoint}}
\put(1031,636){\usebox{\plotpoint}}
\put(1032,637){\usebox{\plotpoint}}
\put(1033,638){\usebox{\plotpoint}}
\put(1034,639){\usebox{\plotpoint}}
\put(1035,640){\usebox{\plotpoint}}
\put(1036,641){\usebox{\plotpoint}}
\put(1037,642){\usebox{\plotpoint}}
\put(1038,643){\usebox{\plotpoint}}
\put(1039,644){\usebox{\plotpoint}}
\put(1040,645){\usebox{\plotpoint}}
\put(1041,646){\usebox{\plotpoint}}
\put(1042,647){\usebox{\plotpoint}}
\put(1043,648){\usebox{\plotpoint}}
\put(1044,649){\usebox{\plotpoint}}
\put(1045,650){\usebox{\plotpoint}}
\put(1047,651){\usebox{\plotpoint}}
\put(1050,652){\usebox{\plotpoint}}
\put(1053,653){\usebox{\plotpoint}}
\put(1056,654){\usebox{\plotpoint}}
\put(1059,655){\rule[-0.350pt]{3.132pt}{0.700pt}}
\put(1072,656){\rule[-0.350pt]{1.686pt}{0.700pt}}
\put(1079,657){\rule[-0.350pt]{1.686pt}{0.700pt}}
\put(1086,658){\rule[-0.350pt]{9.877pt}{0.700pt}}
\put(1127,657){\rule[-0.350pt]{9.877pt}{0.700pt}}
\put(1168,658){\rule[-0.350pt]{6.745pt}{0.700pt}}
\put(1196,659){\rule[-0.350pt]{9.877pt}{0.700pt}}
\put(1237,660){\rule[-0.350pt]{19.754pt}{0.700pt}}
\put(1319,661){\rule[-0.350pt]{9.877pt}{0.700pt}}
\sbox{\plotpoint}{\rule[-0.500pt]{1.000pt}{1.000pt}}%
\put(325,158){\rule[-0.500pt]{1.000pt}{2.620pt}}
\put(326,168){\rule[-0.500pt]{1.000pt}{2.620pt}}
\put(327,179){\rule[-0.500pt]{1.000pt}{2.620pt}}
\put(328,190){\rule[-0.500pt]{1.000pt}{2.620pt}}
\put(329,201){\rule[-0.500pt]{1.000pt}{2.620pt}}
\put(330,212){\rule[-0.500pt]{1.000pt}{2.620pt}}
\put(331,223){\rule[-0.500pt]{1.000pt}{2.620pt}}
\put(332,234){\rule[-0.500pt]{1.000pt}{2.620pt}}
\put(333,245){\rule[-0.500pt]{1.000pt}{2.594pt}}
\put(334,255){\rule[-0.500pt]{1.000pt}{2.594pt}}
\put(335,266){\rule[-0.500pt]{1.000pt}{2.594pt}}
\put(336,277){\rule[-0.500pt]{1.000pt}{2.594pt}}
\put(337,288){\rule[-0.500pt]{1.000pt}{2.594pt}}
\put(338,298){\rule[-0.500pt]{1.000pt}{2.594pt}}
\put(339,309){\rule[-0.500pt]{1.000pt}{2.594pt}}
\put(340,320){\rule[-0.500pt]{1.000pt}{2.594pt}}
\put(341,331){\rule[-0.500pt]{1.000pt}{2.594pt}}
\put(342,341){\rule[-0.500pt]{1.000pt}{2.594pt}}
\put(343,352){\rule[-0.500pt]{1.000pt}{2.594pt}}
\put(344,363){\rule[-0.500pt]{1.000pt}{2.594pt}}
\put(345,374){\rule[-0.500pt]{1.000pt}{2.594pt}}
\put(346,384){\rule[-0.500pt]{1.000pt}{1.514pt}}
\put(347,391){\rule[-0.500pt]{1.000pt}{1.514pt}}
\put(348,397){\rule[-0.500pt]{1.000pt}{1.514pt}}
\put(349,403){\rule[-0.500pt]{1.000pt}{1.514pt}}
\put(350,410){\rule[-0.500pt]{1.000pt}{1.514pt}}
\put(351,416){\rule[-0.500pt]{1.000pt}{1.514pt}}
\put(352,422){\rule[-0.500pt]{1.000pt}{1.514pt}}
\put(353,428){\rule[-0.500pt]{1.000pt}{1.514pt}}
\put(354,435){\rule[-0.500pt]{1.000pt}{1.514pt}}
\put(355,441){\rule[-0.500pt]{1.000pt}{1.514pt}}
\put(356,447){\rule[-0.500pt]{1.000pt}{1.514pt}}
\put(357,454){\rule[-0.500pt]{1.000pt}{1.514pt}}
\put(358,460){\rule[-0.500pt]{1.000pt}{1.514pt}}
\put(359,466){\rule[-0.500pt]{1.000pt}{1.514pt}}
\put(360,472){\usebox{\plotpoint}}
\put(361,474){\usebox{\plotpoint}}
\put(362,476){\usebox{\plotpoint}}
\put(363,477){\usebox{\plotpoint}}
\put(364,479){\usebox{\plotpoint}}
\put(365,480){\usebox{\plotpoint}}
\put(366,482){\usebox{\plotpoint}}
\put(367,483){\usebox{\plotpoint}}
\put(368,485){\usebox{\plotpoint}}
\put(369,486){\usebox{\plotpoint}}
\put(370,488){\usebox{\plotpoint}}
\put(371,489){\usebox{\plotpoint}}
\put(372,491){\usebox{\plotpoint}}
\put(373,492){\usebox{\plotpoint}}
\put(374,494){\usebox{\plotpoint}}
\put(375,495){\usebox{\plotpoint}}
\put(376,496){\usebox{\plotpoint}}
\put(377,498){\usebox{\plotpoint}}
\put(378,499){\usebox{\plotpoint}}
\put(379,500){\usebox{\plotpoint}}
\put(380,502){\usebox{\plotpoint}}
\put(381,503){\usebox{\plotpoint}}
\put(382,505){\usebox{\plotpoint}}
\put(383,506){\usebox{\plotpoint}}
\put(384,507){\usebox{\plotpoint}}
\put(385,509){\usebox{\plotpoint}}
\put(386,510){\usebox{\plotpoint}}
\put(387,511){\usebox{\plotpoint}}
\put(387,512){\usebox{\plotpoint}}
\put(389,513){\usebox{\plotpoint}}
\put(392,514){\usebox{\plotpoint}}
\put(395,515){\usebox{\plotpoint}}
\put(398,516){\usebox{\plotpoint}}
\put(400,517){\usebox{\plotpoint}}
\put(403,518){\usebox{\plotpoint}}
\put(406,519){\usebox{\plotpoint}}
\put(409,520){\usebox{\plotpoint}}
\put(412,521){\usebox{\plotpoint}}
\put(414,522){\usebox{\plotpoint}}
\put(415,520){\usebox{\plotpoint}}
\put(416,519){\usebox{\plotpoint}}
\put(417,518){\usebox{\plotpoint}}
\put(418,517){\usebox{\plotpoint}}
\put(419,516){\usebox{\plotpoint}}
\put(420,515){\usebox{\plotpoint}}
\put(421,513){\usebox{\plotpoint}}
\put(422,512){\usebox{\plotpoint}}
\put(423,511){\usebox{\plotpoint}}
\put(424,510){\usebox{\plotpoint}}
\put(425,509){\usebox{\plotpoint}}
\put(426,508){\usebox{\plotpoint}}
\put(427,507){\usebox{\plotpoint}}
\put(428,505){\usebox{\plotpoint}}
\put(429,504){\usebox{\plotpoint}}
\put(430,503){\usebox{\plotpoint}}
\put(431,502){\usebox{\plotpoint}}
\put(432,501){\usebox{\plotpoint}}
\put(433,500){\usebox{\plotpoint}}
\put(434,499){\usebox{\plotpoint}}
\put(435,497){\usebox{\plotpoint}}
\put(436,496){\usebox{\plotpoint}}
\put(437,495){\usebox{\plotpoint}}
\put(438,494){\usebox{\plotpoint}}
\put(439,493){\usebox{\plotpoint}}
\put(440,492){\usebox{\plotpoint}}
\put(441,491){\usebox{\plotpoint}}
\put(442,491){\usebox{\plotpoint}}
\put(442,491){\usebox{\plotpoint}}
\put(443,494){\usebox{\plotpoint}}
\put(444,498){\usebox{\plotpoint}}
\put(445,502){\usebox{\plotpoint}}
\put(446,506){\usebox{\plotpoint}}
\put(447,509){\usebox{\plotpoint}}
\put(448,513){\usebox{\plotpoint}}
\put(449,517){\usebox{\plotpoint}}
\put(450,521){\usebox{\plotpoint}}
\put(451,525){\usebox{\plotpoint}}
\put(452,528){\usebox{\plotpoint}}
\put(453,532){\usebox{\plotpoint}}
\put(454,536){\usebox{\plotpoint}}
\put(455,540){\usebox{\plotpoint}}
\put(456,543){\usebox{\plotpoint}}
\put(457,545){\usebox{\plotpoint}}
\put(458,546){\usebox{\plotpoint}}
\put(459,547){\usebox{\plotpoint}}
\put(460,548){\usebox{\plotpoint}}
\put(461,549){\usebox{\plotpoint}}
\put(462,550){\usebox{\plotpoint}}
\put(463,551){\usebox{\plotpoint}}
\put(464,552){\usebox{\plotpoint}}
\put(465,553){\usebox{\plotpoint}}
\put(466,554){\usebox{\plotpoint}}
\put(467,555){\usebox{\plotpoint}}
\put(468,556){\usebox{\plotpoint}}
\put(469,557){\usebox{\plotpoint}}
\put(470,558){\usebox{\plotpoint}}
\put(471,560){\usebox{\plotpoint}}
\put(472,562){\usebox{\plotpoint}}
\put(473,564){\usebox{\plotpoint}}
\put(474,566){\usebox{\plotpoint}}
\put(475,568){\usebox{\plotpoint}}
\put(476,570){\usebox{\plotpoint}}
\put(477,571){\usebox{\plotpoint}}
\put(478,573){\usebox{\plotpoint}}
\put(479,575){\usebox{\plotpoint}}
\put(480,577){\usebox{\plotpoint}}
\put(481,579){\usebox{\plotpoint}}
\put(482,581){\usebox{\plotpoint}}
\put(483,582){\usebox{\plotpoint}}
\put(483,583){\usebox{\plotpoint}}
\put(484,584){\usebox{\plotpoint}}
\put(485,585){\usebox{\plotpoint}}
\put(487,586){\usebox{\plotpoint}}
\put(488,587){\usebox{\plotpoint}}
\put(489,588){\usebox{\plotpoint}}
\put(491,589){\usebox{\plotpoint}}
\put(492,590){\usebox{\plotpoint}}
\put(494,591){\usebox{\plotpoint}}
\put(495,592){\usebox{\plotpoint}}
\put(496,593){\usebox{\plotpoint}}
\put(500,594){\usebox{\plotpoint}}
\put(504,595){\usebox{\plotpoint}}
\put(507,596){\usebox{\plotpoint}}
\put(511,597){\usebox{\plotpoint}}
\put(513,598){\usebox{\plotpoint}}
\put(515,599){\usebox{\plotpoint}}
\put(517,600){\usebox{\plotpoint}}
\put(519,601){\usebox{\plotpoint}}
\put(521,602){\usebox{\plotpoint}}
\put(524,603){\usebox{\plotpoint}}
\put(527,604){\usebox{\plotpoint}}
\put(531,605){\usebox{\plotpoint}}
\put(534,606){\usebox{\plotpoint}}
\put(538,607){\rule[-0.500pt]{1.686pt}{1.000pt}}
\put(545,608){\rule[-0.500pt]{1.686pt}{1.000pt}}
\put(552,609){\rule[-0.500pt]{3.132pt}{1.000pt}}
\put(565,610){\rule[-0.500pt]{3.373pt}{1.000pt}}
\put(579,611){\rule[-0.500pt]{56.130pt}{1.000pt}}
\put(812,605){\rule[-0.500pt]{1.000pt}{1.308pt}}
\put(813,600){\rule[-0.500pt]{1.000pt}{1.308pt}}
\put(814,594){\rule[-0.500pt]{1.000pt}{1.308pt}}
\put(815,589){\rule[-0.500pt]{1.000pt}{1.308pt}}
\put(816,583){\rule[-0.500pt]{1.000pt}{1.308pt}}
\put(817,578){\rule[-0.500pt]{1.000pt}{1.308pt}}
\put(818,572){\rule[-0.500pt]{1.000pt}{1.308pt}}
\put(819,567){\rule[-0.500pt]{1.000pt}{1.308pt}}
\put(820,562){\rule[-0.500pt]{1.000pt}{1.308pt}}
\put(821,556){\rule[-0.500pt]{1.000pt}{1.308pt}}
\put(822,551){\rule[-0.500pt]{1.000pt}{1.308pt}}
\put(823,545){\rule[-0.500pt]{1.000pt}{1.308pt}}
\put(824,540){\rule[-0.500pt]{1.000pt}{1.308pt}}
\put(825,535){\rule[-0.500pt]{1.000pt}{1.308pt}}
\put(826,535){\rule[-0.500pt]{39.508pt}{1.000pt}}
\put(990,533){\usebox{\plotpoint}}
\put(991,532){\usebox{\plotpoint}}
\put(992,531){\usebox{\plotpoint}}
\put(993,529){\usebox{\plotpoint}}
\put(994,528){\usebox{\plotpoint}}
\put(995,527){\usebox{\plotpoint}}
\put(996,526){\usebox{\plotpoint}}
\put(997,524){\usebox{\plotpoint}}
\put(998,523){\usebox{\plotpoint}}
\put(999,522){\usebox{\plotpoint}}
\put(1000,520){\usebox{\plotpoint}}
\put(1001,519){\usebox{\plotpoint}}
\put(1002,518){\usebox{\plotpoint}}
\put(1003,517){\usebox{\plotpoint}}
\put(1004,515){\usebox{\plotpoint}}
\put(1005,514){\usebox{\plotpoint}}
\put(1006,512){\usebox{\plotpoint}}
\put(1007,511){\usebox{\plotpoint}}
\put(1008,509){\usebox{\plotpoint}}
\put(1009,508){\usebox{\plotpoint}}
\put(1010,507){\usebox{\plotpoint}}
\put(1011,505){\usebox{\plotpoint}}
\put(1012,504){\usebox{\plotpoint}}
\put(1013,502){\usebox{\plotpoint}}
\put(1014,501){\usebox{\plotpoint}}
\put(1015,499){\usebox{\plotpoint}}
\put(1016,498){\usebox{\plotpoint}}
\put(1017,497){\usebox{\plotpoint}}
\put(1018,497){\usebox{\plotpoint}}
\put(1018,497){\usebox{\plotpoint}}
\put(1019,498){\usebox{\plotpoint}}
\put(1020,499){\usebox{\plotpoint}}
\put(1021,500){\usebox{\plotpoint}}
\put(1022,502){\usebox{\plotpoint}}
\put(1023,503){\usebox{\plotpoint}}
\put(1024,504){\usebox{\plotpoint}}
\put(1025,506){\usebox{\plotpoint}}
\put(1026,507){\usebox{\plotpoint}}
\put(1027,508){\usebox{\plotpoint}}
\put(1028,510){\usebox{\plotpoint}}
\put(1029,511){\usebox{\plotpoint}}
\put(1030,512){\usebox{\plotpoint}}
\put(1031,513){\usebox{\plotpoint}}
\put(1031,514){\usebox{\plotpoint}}
\put(1033,515){\usebox{\plotpoint}}
\put(1035,516){\usebox{\plotpoint}}
\put(1038,517){\usebox{\plotpoint}}
\put(1040,518){\usebox{\plotpoint}}
\put(1042,519){\usebox{\plotpoint}}
\put(1045,520){\usebox{\plotpoint}}
\put(1047,521){\usebox{\plotpoint}}
\put(1049,522){\usebox{\plotpoint}}
\put(1052,523){\usebox{\plotpoint}}
\put(1054,524){\usebox{\plotpoint}}
\put(1056,525){\usebox{\plotpoint}}
\put(1059,526){\usebox{\plotpoint}}
\put(1061,527){\usebox{\plotpoint}}
\put(1063,528){\usebox{\plotpoint}}
\put(1065,529){\usebox{\plotpoint}}
\put(1067,530){\usebox{\plotpoint}}
\put(1069,531){\usebox{\plotpoint}}
\put(1071,532){\rule[-0.500pt]{1.686pt}{1.000pt}}
\put(1079,533){\rule[-0.500pt]{1.686pt}{1.000pt}}
\put(1086,534){\rule[-0.500pt]{1.124pt}{1.000pt}}
\put(1090,535){\rule[-0.500pt]{1.124pt}{1.000pt}}
\put(1095,536){\rule[-0.500pt]{1.124pt}{1.000pt}}
\put(1099,537){\usebox{\plotpoint}}
\put(1102,538){\usebox{\plotpoint}}
\put(1104,539){\usebox{\plotpoint}}
\put(1106,540){\usebox{\plotpoint}}
\put(1108,541){\usebox{\plotpoint}}
\put(1110,542){\usebox{\plotpoint}}
\put(1112,543){\usebox{\plotpoint}}
\put(1116,544){\usebox{\plotpoint}}
\put(1120,545){\usebox{\plotpoint}}
\put(1123,546){\usebox{\plotpoint}}
\put(1127,547){\rule[-0.500pt]{1.124pt}{1.000pt}}
\put(1131,548){\rule[-0.500pt]{1.124pt}{1.000pt}}
\put(1136,549){\rule[-0.500pt]{1.124pt}{1.000pt}}
\put(1140,550){\usebox{\plotpoint}}
\put(1143,551){\usebox{\plotpoint}}
\put(1146,552){\usebox{\plotpoint}}
\put(1149,553){\usebox{\plotpoint}}
\put(1152,554){\usebox{\plotpoint}}
\put(1155,555){\usebox{\plotpoint}}
\put(1158,556){\usebox{\plotpoint}}
\put(1161,557){\usebox{\plotpoint}}
\put(1164,558){\usebox{\plotpoint}}
\put(1168,559){\rule[-0.500pt]{1.124pt}{1.000pt}}
\put(1172,560){\rule[-0.500pt]{1.124pt}{1.000pt}}
\put(1177,561){\rule[-0.500pt]{1.124pt}{1.000pt}}
\put(1181,562){\usebox{\plotpoint}}
\put(1184,563){\usebox{\plotpoint}}
\put(1187,564){\usebox{\plotpoint}}
\put(1190,565){\usebox{\plotpoint}}
\put(1193,566){\usebox{\plotpoint}}
\put(1196,567){\usebox{\plotpoint}}
\put(1198,568){\usebox{\plotpoint}}
\put(1200,569){\usebox{\plotpoint}}
\put(1202,570){\usebox{\plotpoint}}
\put(1204,571){\usebox{\plotpoint}}
\put(1206,572){\usebox{\plotpoint}}
\put(1208,573){\usebox{\plotpoint}}
\put(1211,574){\usebox{\plotpoint}}
\put(1214,575){\usebox{\plotpoint}}
\put(1217,576){\usebox{\plotpoint}}
\put(1220,577){\usebox{\plotpoint}}
\put(1223,578){\usebox{\plotpoint}}
\put(1225,579){\usebox{\plotpoint}}
\put(1227,580){\usebox{\plotpoint}}
\put(1230,581){\usebox{\plotpoint}}
\put(1232,582){\usebox{\plotpoint}}
\put(1234,583){\usebox{\plotpoint}}
\put(1237,584){\usebox{\plotpoint}}
\put(1239,585){\usebox{\plotpoint}}
\put(1242,586){\usebox{\plotpoint}}
\put(1244,587){\usebox{\plotpoint}}
\put(1247,588){\usebox{\plotpoint}}
\put(1249,589){\usebox{\plotpoint}}
\put(1253,590){\usebox{\plotpoint}}
\put(1257,591){\usebox{\plotpoint}}
\put(1260,592){\usebox{\plotpoint}}
\put(1264,593){\rule[-0.500pt]{1.124pt}{1.000pt}}
\put(1268,594){\rule[-0.500pt]{1.124pt}{1.000pt}}
\put(1273,595){\rule[-0.500pt]{1.124pt}{1.000pt}}
\put(1277,596){\rule[-0.500pt]{1.686pt}{1.000pt}}
\put(1285,597){\rule[-0.500pt]{1.686pt}{1.000pt}}
\put(1292,598){\rule[-0.500pt]{3.132pt}{1.000pt}}
\put(1305,599){\rule[-0.500pt]{3.373pt}{1.000pt}}
\put(1319,600){\rule[-0.500pt]{3.373pt}{1.000pt}}
\put(1333,601){\rule[-0.500pt]{3.132pt}{1.000pt}}
\put(1346,602){\rule[-0.500pt]{3.373pt}{1.000pt}}
\end{picture}

%% file: conclusions.tex
\chapter{Conclusions}

In this thesis a geometric framework for optimization problems and
their application in adaptive signal processing is established.  Many
approaches to the subspace tracking problem encountered in adaptive
filtering depend upon its formulation as an optimization problem,
namely optimizing a generalized form of the Rayleigh quotient defined
on a set of orthonormal vectors.  However, previous algorithms do not
exploit the natural geometric structure of this constraint manifold.
These algorithms are extrinsically defined in that they depend upon
the choice of an isometric imbedding of the constraint surface in a
higher dimensional Euclidean space.  Furthermore, the algorithms that
use a projected version of the classical conjugate gradient algorithm
on Euclidean space do not account for the curvature of the constraint
surface, and therefore achieve only linear convergence.

There exists a special geometric structure in the type of constraint
surfaces found in the subspace tracking problem.  The geometry of Lie
groups and homogeneous spaces, reviewed in Chapter~\ref{chap:geom},
provides analytic expressions for many fundamental objects of interest
in these spaces, such as geodesics and parallel translation along
geodesics.  While such objects may be computationally unfeasible for
application to general constrained optimization problems, there is an
important class of manifolds which have sufficient structure to yield
potentially practical algorithms.

The subspace tracking problem can be expressed as a gradient flow on a
Lie group or homogeneous space.  This idea, discussed in
Chapter~\ref{chap:grad}, covers several examples of gradient flows on
Lie groups and homogeneous spaces.  All of these gradient flows solve
the eigenvalue or singular value problem of numerical linear algebra.
The gradient flows considered demonstrate how understanding the
differential geometric structure of a problem in numerical linear
algebra can illuminate algorithms used to solve that problem.
Specifically, the gradient flow of the function $\tr\Theta^\T Q\Theta
N$ defined on the special orthogonal group $\SO(n)$ is reviewed.  This
flow yields an ordered eigenvalue decomposition of the matrix~$Q$.
The gradient flow of the generalized Rayleigh quotient $\tr p^\T\!ApN$
defined on the Stiefel manifold $\stiefel(n,k)$ is analyzed and its
stationary points classified.  Finally the gradient flow of the
function $\tr\Sigma^\T\!N$ defined on the set of matrices with fixed
singular values is analyzed.  This gradient flow and a related
gradient flow on the homogeneous space $\svd$ yield the singular value
decomposition of an arbitrary matrix.  A numerical experiment
demonstrating this gradient flow is provided and it is shown that the
experimental convergence rates are close to the predicted convergence
rates.

Because using gradient flows to solve problems in numerical linear
algebra is computationally impractical, the theory of large step
optimization methods on Riemannian manifolds is developed in
Chapter~\ref{chap:orm}.  The first method analyzed---the method of
steepest descent on a Riemannian manifold---is already well-known.  A
thorough treatment of this algorithm employing techniques from
Riemannian geometry is provided to fix ideas for the development of
improved methods.  A proof of linear convergence is given.  Next, a
version of Newton's method on Riemannian manifolds is developed and
analyzed.  It is shown that quadratic convergence may be obtained, and
that this method inherits several properties from the classical
version of Newton's method on a flat space.  Finally, the conjugate
gradient method on Riemannian manifolds is developed and analyzed, and
a proof of superlinear convergence is provided.  Several examples that
demonstrate the predicted convergence rates are given throughout this
chapter.  The Rayleigh quotient on the sphere is optimized using all
three algorithms.  It is shown that the Riemannian version of Newton's
method applied to this function is efficiently approximated by the
Rayleigh quotient iteration.  The conjugate gradient algorithm applied
to the Rayleigh quotient on the sphere yields a new algorithm for
computing the eigenvectors corresponding to the extreme eigenvalues of
a symmetric matrix.  This superlinearly convergent algorithm requires
two matrix-vector multiplications and $O(n)$ operations per iteration.

In Chapter~\ref{chap:af} these ideas are brought to bear on the
subspace tracking problem of adaptive filtering.  The subspace
tracking problem is reviewed and it is shown how this problem may be
viewed as an optimization problem on a Stiefel manifold.  The
Riemannian version of the conjugate gradient method is applied to the
generalized Rayleigh quotient.  By exploiting the homogeneous space
structure of the Stiefel manifold, an efficient superlinearly
convergent algorithm for computing the eigenvectors corresponding to
the $k$ extreme eigenvalues of a symmetric matrix is developed.  This
algorithm requires $O(k)$ matrix-vector multiplications per iteration
and $O(nk^2)$ operations.  This algorithm has the advantage of
maintaining orthonormality of the estimated eigenvectors at every
step.  However, it is important to note that the algorithm is only
efficient if~$2k\le n$. The results of a numerical experiment of this
algorithm which confirm the predicted convergence properties are
shown.  In the experiment, the conjugate gradient algorithm
on~$\stiefel(100,3)$, a manifold of dimension~$294$, converged to
machine accuracy within 50 steps.  Good estimates of the eigenvalues
are obtained in less than 25 steps.  A similar algorithm for computing
the largest left singular vectors corresponding to the extreme
singular values of an arbitrary matrix is discussed.

A new algorithm for subspace tracking based upon this conjugate
gradient algorithm is given.  To test the algorithm's tracking
properties, the algorithm is used to track several time varying
symmetric matrices, each of which has a discontinuous step of some
type.  The following examples are considered: the principal invariant
subspace rotating in its own plane with fixed eigenvalues, rotating
out of its plane with changing eigenvalues, and rotating
instantaneously to an orthogonal plane. Two versions of the algorithm
were tested: one version that reset the conjugate gradient algorithm
at the step, and one version that did not. In the first test, the
reset version reconverged to machine accuracy in less than 20 steps,
and provided accurate estimates of the eigenvalues in less than 10
steps. In the second test, the algorithm reconverged in 30 steps, and
provided accurate estimates of the eigenvalues in 5 iterations.  The
third and final experiment demonstrates how the algorithm behaves when
it has converged to a maximimum point that suddenly becomes a saddle
point.  The algorithm stayed close to the saddle point for about 15
iterations.

This thesis has only considered a few Riemannian manifolds which are
found in certain types of applications and have sufficient structure
to yield efficient algorithms. There are other useful examples which
have not yet been mentioned.  For example, many applications do not
require the eigenvectors and corresponding eigenvalues of principal
invariant subspace, but only an arbitrary orthonormal basis for this
subspace.  In this context, an optimization problem posed on the
Grassmann manifold $\grassmann(n,k)$ of $k\hyphen$planes in~$\R^n$
would be appropriate.  This manifold possesses the structure of a
symmetric space and therefore geodesics and parallel translation along
geodesics may be computed with matrix exponentiation.  Furthermore,
the tangent plane of~$\grassmann(n,k)$ at the origin as a vector
subspace of the Lie algebra of its Lie transformation group contains
large zero blocks that could be exploited to yield an efficient
algorithm.  This thesis also considered only real-valued cases; the
unitary version of these algorithms that would be necessary for many
signal processing contexts have not been explored.

The subspace tracking methods presented in this thesis have not been
applied to particular examples in adaptive filtering, so there is an
opportunity to explore the extent of their usefulness in this area.
There is a broad range of adaptive filtering applications which have
diverse computational requirements, dimensionality, and assumptions
about the signal properties and background noise.  The strengths and
weaknesses of subspace tracking techniques must be evaluated in the
context of the application's requirements.

%% file: references.tex
\let\bibliographysize=\footnotesize  